\documentclass[reqno,10pt]{amsart}

\usepackage{amsmath,mathrsfs}
\usepackage{amssymb, amsmath,amsthm}
\usepackage{graphicx}
\usepackage{enumerate,color,xcolor}
\usepackage[colorlinks,
citecolor=green
]{hyperref}

\usepackage[normalem]{ulem}

\newtheorem{defi}{Definition}[section]
\newtheorem{thm}{Theorem}[section]
\newtheorem{lem}{Lemma}[section]
\newtheorem{rmk}{Remark}[section]

\newtheorem{prop}{Proposition}[section]

\newcommand{\vv}[1]{\boldsymbol{#1}}

\numberwithin{equation}{section}

\newcommand{\beq}{\begin{equation}}
\newcommand{\eeq}{\end{equation}}
\newcommand{\ben}{\begin{eqnarray}}
\newcommand{\een}{\end{eqnarray}}
\newcommand{\beno}{\begin{eqnarray*}}
\newcommand{\eeno}{\end{eqnarray*}}
\let\f=\frac
\let\vth=\vartheta

\newcommand{\be}{\begin{equation} \label}
	\newcommand{\ee}{\end{equation}}
\newcommand{\bea}{\begin{eqnarray}\label}
	\newcommand{\eea}{\end{eqnarray}}
\newcommand{\bas}{\begin{eqnarray*}}
	\newcommand{\eas}{\end{eqnarray*}}
\newcommand{\bit}{\begin{itemize}}
	\newcommand{\eit}{\end{itemize}}

\newcommand{\N}{{\mathbb N}}
\newcommand{\Z}{{\mathbb Z}}
\newcommand{\R}{{\mathbb R}}

\newcommand{\pa}{\partial}
\newcommand{\eps}{\varepsilon}

\newcommand{\rr}[1]{\left( #1 \right)}

 \def\aa{{\mathsf a}}
 \newcommand{\wei}{\langle v \rangle}
\newcommand{\ba}{\begin{aligned}}
\newcommand{\ea}{\end{aligned}}
 \def\na{\nabla}
 \newcommand{\lr}[1]{\langle #1 \rangle}
\def\eqdefa{\buildrel\hbox{\footnotesize def}\over =}
\let\pa=\partial

\let\Ga=\Gamma
\let\z=\zeta
\let\lam=\lambda

\let\f=\frac

\let\om=\omega

\let\D=\Delta
\let\Lam=\Lambda
\let\Om=\Omega

\let\ka=\kappa
\def\mP{\mathbb{P}}
\def\a{\mathfrak{a}}
\def\c{\mathfrak{c}}
\def\mfR{\mathfrak{R}}
\def\mfS{\mathfrak{S}}

\def\ma{\mathbf{a}}
\def\mb{\mathbf{b}}
\def\mc{\mathbf{c}}

\def\pa{\partial}

\def\msA{\mathscr A}

\def\bb{\mathfrak b}
\def\virgp{\raise 2pt\hbox{,}}
\def\cdotpv{\raise 2pt\hbox{;}}

\def\eqdefa{\buildrel\hbox{\footnotesize def}\over =}

\def\C{\mathop{\mathbb C\kern 0pt}\nolimits}
\def\DD{\mathop{\mathbb D\kern 0pt}\nolimits}
\def\EE{\mathop{{\mathbb E \kern 0pt}}\nolimits}
\def\K{\mathop{\mathbb K\kern 0pt}\nolimits}
\def\N{\mathop{\mathbb N\kern 0pt}\nolimits}
\def\Q{\mathop{\mathbb Q\kern 0pt}\nolimits}
\def\R{\mathop{\mathbb R\kern 0pt}\nolimits}
\def\SS{\mathop{\mathbb S\kern 0pt}\nolimits}
\def\sfT{\mathsf{T}}

\def\<{\langle}
\def\>{\rangle}

\def\th{\theta}

\def\al{\alpha}
\def\be{\beta}
\def\de{\delta}

\def\gs{\gtrsim}
\def\ls{\lesssim}
\def\S{\mathbb{S}}
\def\vep{\varepsilon}
\begin{document}

\title[The Landau equation with harmonic potential]{On Landau equation with harmonic potential: nonlinear stability of time-periodic Maxwell-Boltzmann distributions 
}

\author[C. Cao, L.-B. He and J. Ji]{Chuqi Cao, Ling-Bing He and Jie Ji}
\address[C. Cao]{Department of Applied Mathematics, The Hong Kong Polytechnic University, 
Hong Kong,  P. R.  China.} \email{chuqicao@gmail.com}
\address[L.-B. He]{Department of Mathematical Sciences, Tsinghua University\\
	Beijing 100084,  P. R.  China.} \email{hlb@tsinghua.edu.cn}
\address[J. Ji]{Beijing International Center for Mathematical Research, Peking University\\
	Beijing 100871,  P. R.  China.} \email{jij22@pku.edu.cn}

\begin{abstract}   We provide the first and rigorous confirmations of the hypotheses by  Ludwig Boltzmann in his   seminal paper \cite{Boltzmann} within the context of the Landau equation in the presence of a harmonic potential. We prove that
(i) Each {\it entropy-invariant solution} can be identified as a {\it time-periodic Maxwell-Boltzmann distribution}. Moreover, these distributions can be characterized by thirteen conservation laws, which sheds light on the global dynamics. 
(ii) Each {\it time-periodic Maxwell-Boltzmann distribution} is nonlinearly stable,  including neutral asymptotic stability and Lyapunov stability. Furthermore, the convergence rate is entirely reliant on the thirteen conservation laws and is optimal when compared to the linear scenario.
\end{abstract}

\maketitle

\setcounter{tocdepth}{1}
\tableofcontents




\section{Introduction} In 1876, Boltzmann published a paper titled ``\"Uber die Aufstellung und Integration der Gleichungen, welche die Molekularbewegung in Gasen bestimmen''. This paper constitutes a significant contribution to kinetic theory and statistical mechanics. In \cite{Boltzmann}, Boltzmann formulated the kinetic equation, which describes the statistical behavior of particles and their evolution over time due to collisions. Mathematically, it can be expressed as follows:
\ben\label{kineticwithexternalpotential}
\pa_tF+v\cdot\na_x F-\na_x\Phi\cdot\na_vF=\mathcal{C}_{coll}(F,F).
\een

 Several explanations are in order:

\noindent $\bullet$ $F(t,x,v)\geq0$ stands for the distribution of particles that at time $t\in\R^+$ with position $x\in\R^3_x$ and velocity $v\in\R^3_v$. In \eqref{kineticwithexternalpotential}, $(-\na_x\Phi)$ is the acceleration due to the external potential $\Phi=\Phi(x)$, and $\mathcal{C}_{coll}(F,F)$ denotes the collision effect due to  binary collisions.  Typically, $\mathcal{C}_{coll}(F,F)$ represents the Boltzmann or Landau collision operator.

\noindent $\bullet$ The streaming term of \eqref{kineticwithexternalpotential}:
 $S(F):=-v\cdot\na_x F+\na_x\Phi\cdot\na_vF$ is the local change of $F$ per second due to the independent
motion of the molecules in phase space.   $S(F)$  can also be regarded as the one particle Liouville operation $\{H,F\}$ with $H=H(x,v):=|v|^2/2+\Phi(x)$ and $\{H,\cdot\}:=(\na_xH)\cdot\na_v-(\na_vH)\cdot\na_x$. This implies that $(\pa_t-S)F=0$ is governed by the Hamiltonian system and can be  solved by   trajectory. 
\subsection{Boltzmann's argument on the approach to equilibrium } Now we turn to the basic problem of Boltzmann: why and how is the equilibrium state reached in time. Since Boltzmann's arguments are well-known, let us sketch his proof. The first step relies on the famous $H$-theorem. If the entropy $\mathscr{H}(F)$ and the entropy dissipation $\mathscr{D}(F)$ are defined by
\begin{equation}\label{entropy}
\mathscr{H}(F)(t):=\int_{\R^3_x\times\R^3_v}F\log Fdvdx,\quad
\mathscr{D}(F)(t):=-\int_{\R^3_x\times\R^3_v}\mathcal{C}_{coll}(F,F)\log F dvdx,
\end{equation} then the $H$-theorem states that
\begin{equation}\label{H-theorem}
\f d{dt}\mathscr{H}(F)(t)+\mathscr{D}(F)(t)=0.
\end{equation}
By the change of variables for the collision term $\mathcal{C}_{coll}$, one may easily derive that $\mathscr{D}(F)\ge0$. 
This implies the assertions that $\mathscr{H}(F)(t)$ decreases in time and 
\ben\label{entropyinvariant} &&\mbox{$F$ is an entropy-invariant-in-time solution of \eqref{kineticwithexternalpotential}}\Leftrightarrow \mathscr{H}(F)(t)=\mathscr{H}(F)(0), \quad \forall t\ge0\notag \\ &&\Leftrightarrow \mathscr{D}(F)(t)=0,\quad \forall t\ge0 \Leftrightarrow \mathcal{C}_{coll}(F,F)(t)=0, \quad \forall t\ge0. \een 
The detailed proof will be given in the Appendix. Since entropy-invariant solutions represent states of the system that are in some sense equilibrium solutions, let us give a more precise definition as follows:

\begin{defi}\label{Defoentropy-invariant}
	 $M = M(t, x, v)\ge 0$ is a {\it entropy-invariant solution} of \eqref{kineticwithexternalpotential}, if 
	 \ben\label{integcondi}\int_{\R^3_x\times\R^3_v}(|x|^2 + |v|^2 + 1)M(t, x, v)dxdv < +\infty,\een and for any $t\ge0$,
	\begin{equation}\label{equationforeisol}
		 \partial_tM + v \cdot \nabla_xM - \na_x\Phi \cdot \nabla_vM = 0 =\mathcal{C}_{coll}(M, M).
	\end{equation}
	\end{defi}
 \begin{rmk} The  condition \eqref{integcondi} is to ensure the solution satisfies the basic conservation laws(see Lemma \ref{CLaw} for details). While \eqref{equationforeisol} stems from  \eqref{kineticwithexternalpotential} and \eqref{entropyinvariant} directly.
\end{rmk}

It is not difficult to deduce that if  $\mathcal{C}_{coll}(M, M)=0$, then $M$ should be the local Maxwellian, i.e.,  
\ben\label{localMaxwellian}M=A\exp\bigg\{-\f{\beta}2|v-u|^2\bigg\},\een
where $A,\beta$ and the average velocity $u$ can still be functions of $x$ and $t$.  Since the local Maxwellian $M$ has to fulfill the first equation in \eqref{equationforeisol}, by some careful computation, we may finally   
derive that $M$ should be the Maxwell-Boltzmann distribution, i.e.,
\ben\label{Maxwell-Boltzmann distrubution} M=A_0\exp\bigg\{-\beta\big(|v|^2/2+\Phi(x)\big)\bigg\}, \een
where  $\beta$ equals a universal constant independent of $x$ and $t$,  $A_0$ and $\beta$ are determined  by the given total number and total energy of the molecules. This would be justified rigorously for  a sufficiently general external potential $\Phi(x)$ (which include the shape of the vessel or of the wall potential).  As a result, Boltzmann showed that any initial distribution $F_0$ approaches in the course of time the Maxwell-Boltzmann distribution. However, Boltzmann also pointed out that 
for some special potentials for instance a harmonic potential(i.e., $\Phi(x)=\omega_0^2|x|^2/2$), the  Maxwell-Boltzmann distribution will not be reached in time. This is mainly due to the fact that 
\[
M= \exp \bigg\{   -|v\cos t +x\sin t|^2 -|x|^2-|v|^2 \bigg\} 
\] 
is a special solution \eqref{equationforeisol} with harmonic potential $\Phi(x)= |x|^2/2$.  Uhlenbeck  commented in the book \cite{uhlenbeck-ford} that ``for such special potentials(refer to harmonic potential) there are a host of special solutions of the
Boltzmann equation, in which the dependence on the velocity always has the form \eqref{localMaxwellian} but where the $A, \be$ and $u$ can be functions of space and time. Boltzmann himself gave a detailed discussion of these solutions (see \cite{Boltzmann}). They have however only a limited interest.''  
\medskip

Let us conclude some remarks on Boltzmann's argument on the approach to   equilibrium which are due to George E. Uhlenbeck and G.W. Ford(see \cite{uhlenbeck-ford}):

\noindent $\bullet$ Clearly Boltzmann does not actually solve the initial value
problem stated in \cite{Boltzmann}. In fact he supposes in the first place that the initial value problem has an unique solution.

\noindent $\bullet$ The argument due to Boltzmann does not mean, that the approach to
equilibrium happens in two sharply separated successive stages: first the approach in velocity space to the local Maxwellian  \eqref{localMaxwellian} and then the approach in phase space to  the ``barometric'' distribution $A_0\exp\{-\beta\Phi(x)\}$. In particular, both approaches are coupled to each other. 

\noindent $\bullet$ Definitely  the approach to equilibrium in velocity space is quite different from the approach to equilibrium in phase
space. The picture of the approach to the complete Maxwell-Boltzmann distribution \eqref{Maxwell-Boltzmann distrubution}  must and can be confirmed by a deeper study of \eqref{kineticwithexternalpotential}.

\subsection{Landau equation in the presence of a harmonic potential}\label{Q12} The primary objective of this study is to offer {\it a comprehensive mathematical validation} of Boltzmann's argument concerning the approach to equilibrium within the framework of the Landau equation in the presence of a harmonic potential. From a physical standpoint, this study remains significant as comprehending the Landau equation in the context of a harmonic potential is pivotal for modeling and forecasting the behavior of confined charged particles in plasmas. Such understanding holds relevance in diverse fields including fusion research, plasma processing, and the manipulation and trapping of particles.
 Our focus will be on elucidating the following questions:

\smallskip

\noindent \underline{$(Q1).$ Characterization of Entropy-Invariant Solutions.} Since the equation \eqref{kineticwithexternalpotential} enjoys the $H$-theorem, the  basic question would be as follows:

 `` When the initial data satisfying \eqref{integcondi} is given, what kind of state will the solution approach from the perspective of long-time dynamics? ''

From the classical results on the  kinetic equations, the  ``final''  state should be the Gaussian function which is determined by the initial data and its relevant conservation laws. Thus, the natural question arises: does this hold for equation \eqref{kineticwithexternalpotential} as well? Obviously, the entropy-invariant solutions would be the good candidates.
Therefore, the central questions are:

\underline{(i).} Can we provide a full characterization of the entropy-invariant solutions?

\underline{(ii).} What is the connection between the entropy-invariant solutions and conservation laws?

\noindent Answering these questions is crucial for gaining valuable insights into the underlying dynamics and constraints of the system. In particular, these will explain the emergence of a host of time-periodic solutions and give the mathematical response to Uhlenbeck's comment.

\smallskip

\noindent \underline{$(Q2).$ Nonlinear Stability and Global Dynamics.} Can we prove the nonlinear stability of the entropy-invariant solution? What about the mathematical depiction of the interaction between the approach to equilibrium in velocity space and the approach to equilibrium in phase space? Can we determine the optimal convergence rate to equilibrium?

 Addressing these questions would provide crucial insights into the global dynamics of the Landau equation in the presence of a harmonic potential with general initial data. In particular, establishing the nonlinear stability of the entropy-invariant solution would characterize its robustness. Understanding the interplay between local Maxwellian equilibrium in velocity space and  the ``barometric'' distribution  in phase space, as well as determining optimal convergence rates, would further elucidate the system's long-term behavior. 

\smallskip 

Mathematically we will consider the Cauchy problem of 
\begin{equation}\label{CauchyLandau}\left\{ \begin{aligned}
&\pa_t F+v\cdot\na_x F-x\cdot \na_v F=Q(F,F);\\
&F|_{t=0}=F_0(x,v),
\end{aligned} \right.
\end{equation}
where $Q$ is the Landau collision operator defined by
\beno
Q(G,F)(v)=\na_v\cdot \int_{\R^3} a(v-v_*)\big\{G(v_*)\na_v F(v)-F(v)\na_v G(v_*)\big\}dv_*.
\eeno
Here $a$ is a $3 \times 3$ matrix-valued function that is symmetric and non-negative defined by 
\beno
a_{ij}(z)=|z|^{-1}\Pi_{ij}(z)=|z|^{-1}\Big(\de_{ij}-\f{z_iz_j}{|z|^2}\Big).
\eeno
Comparing to \eqref{kineticwithexternalpotential}, here we choose  $\Phi(x):=|x|^2/2$  and   $\mathcal{C}_{coll}(F,F):=Q(F,F)$.

\subsubsection{Landau collision operator} 
   Landau collision operator is derived by Landau through the weak-coupling limit of the Boltzmann collision operator with cutoff Rutherford cross-section  in 1936. It is used to model collisions between charged particles in plasma physics. The Landau collision operator is a bi-linear operator acting  only on the velocity variable $v$. 
If we introduce
\ben\label{defibc}
b_i(z)=\sum_{j=1}^3\pa_j a_{ij}(z)=-2|z|^{-3}z_i,\quad c(z)=\sum_{i=1}^3\sum_{j=1}^3\pa_{ij}a_{ij}(z)=-8\pi \de_0(z),
\een
 Landau collision operator can be also written as
\beno
Q(G,F)=\sum_{i,j=1}^3 (a_{ij}*G)\pa_{ij} F+8\pi GF.
\eeno

\subsubsection{Conservation laws of \eqref{CauchyLandau} and $H$-theorem} We have the following lemma: 
	\begin{lem}\label{CLaw}
	If $F = F(t, x, v)$ is a global solution to equation \eqref{CauchyLandau}, then for all $t>0$,
	\begin{equation*}
		\frac{d}{dt}\int_{\R^3_x\times\R^3_v}F(t, x, v)\phi(t,x,v)dxdv = 0,
	\end{equation*}
if $\phi(t,x,v)=1, X(t), V(t), |X(t)|^2, X(t)\cdot V(t), |V(t)|^2$ and  $X(t)\wedge V(t)$. Here $(X(t),V(t)):=(x\cos t - v\sin t, x\sin t + v\cos t)$  and  $a \wedge b: = ab^\tau - ba^\tau=(a_ib_j-b_ia_j)_{3\times3}$ if $a,b\in\R^3$. 
\end{lem}

The proof of the lemma is omitted here as it can be easily verified.
Thanks to Lemma \ref{CLaw}, we can conveniently introduce a mapping $\Psi$ that maps the solutions to \eqref{CauchyLandau} to the corresponding conserved quantities. In other words, $\Psi$ establishes a relationship between the solution and the conserved quantities, allowing us to analyze and understand the system's behavior based on these conserved quantities.

\begin{defi}\label{ConservMap}
	Let $F(t,x,v)$ be a global solution to \eqref{CauchyLandau} with the initial data $F_0=F_0(x,v)$.  $\Psi$ is called a conserved mapping if  
	\beno
	\Psi(F(t))=(\Psi_1(F)(t),\cdots,\Psi_7(F)(t))^\tau:=\int_{\R^3_x\times\R^3_v} F(t, x, v)\begin{bmatrix} 1 \\ X(t) \\ V(t) \\ |X(t)|^2 \\ X(t) \cdot V(t) \\ |V(t)|^2 \\ X(t)\wedge V(t) \end{bmatrix}dxdv. 
	\eeno
Again $(X(t),V(t))=(X_1(t),X_2(t),X_3(t),V_1(t),V_2(t), V_3(t))=( x\cos t - v\sin t, x\sin t + v\cos t)$. Thanks to the conservation laws, we have $\Psi(F(t)) = \Psi(F_0)$ for all $t>0$. Therefore, it is common practice to use the notation $\Psi(F)$ or $\Psi(F_0)$ instead of $\Psi(F(t))$. 
\end{defi}

\begin{rmk} \textnormal{ This shorthand notation  signifies that the conserved quantities obtained through $\Psi$ are independent of time and can be directly associated with the initial state $F_0$. We may think that 
 $\Psi_1(F), \cdots, \Psi_6(F)$ and $\Psi_7(F)$  correspond to the laws of conservation for mass, initial center of mass, momentum, scalar inertial moment, scalar momentum moment, energy and angular momentum, respectively. It is crucial to highlight that the reason why $\Psi_7(F)$ can be recognized as the conservation of angular momentum is because the components of the skew-symmetric matrix $X(t)\wedge V(t)$ are derived from $X(t)\times V(t)$.  } 
\end{rmk}

\begin{rmk}  \textnormal{ To compare the differences between different initial data, we introduce 2-norm  for $\Psi(F_0)$, i.e., if $\Psi(F_0)=\mathsf{V}=(\mathsf{V}_1,\cdots,\mathsf{V}_7)\in\R^+\times\R^3\times\R^3\times \R^+\times \R\times \R^+\times\mathbb{M}_3(\R)$, then
\ben\label{normofpsif} \|\Psi(F_0)\|_2^2=\|\mathsf{V}\|_2^2:=\sum_{i=1}^7 \|\mathsf{V}_i\|_2^2.  \een 
Here $\mathbb{M}_3(\R)$  represents the set of $3\times3$ matrices with all components in $\R$ and if $A=(a_{ij})_{m\times n}$, then $\|A\|_2^2:=\sum_{i=1}^m\sum_{j=1}^n |a_{ij}|^2$.}
\end{rmk}

Of course, equation \eqref{CauchyLandau} also enjoys the $H$-theorem \eqref{H-theorem}. In this situation, we have 
\ben\label{entropydisspationoflandau}
\mathscr{D}(F)(t)=\f12\sum_{i,j=1}^3\int_{\R^3_x\times\R^6}a_{ij}(v-v_*) \left (\f{\pa_i F}F(v)-\f{\pa_i F}F(v_*) \right) \left(\f{\pa_j F}F(v)-\f{\pa_j F}F(v_*) \right)FF_*dvdv_*dx\geq0.
 \een  
\subsubsection{Harmonic potential} The harmonic potential refers to a specific type of potential energy function used in physics, particularly in the field of classical mechanics and quantum mechanics. It describes the behavior of a particle or system in a potential field that follows Hooke's law. Generally, it is given by 
\ben \label{HarmonicP}\Phi(x):=\f12 \omega_0^2 |x|^2,\een where $\omega_0^2$ is a force constant.
The harmonic potential exhibits specific characteristics, such as a restoring force proportional to the displacement and a symmetric potential energy curve centered at the equilibrium position. Let us conclude its impact on the equation \eqref{kineticwithexternalpotential}.  
\smallskip

\noindent$\bullet$ The equation \eqref{kineticwithexternalpotential} yields a greater number of conservation laws compared to the general external potential $\Phi(x)$ (refer to Lemma \ref{CLaw} for further information).
\smallskip

\noindent$\bullet$ This abundance of conservation laws gives rise to numerous entropy-invariant solutions for the equation \eqref{kineticwithexternalpotential}. Consequently, the global dynamics of \eqref{kineticwithexternalpotential} exhibits a diverse range of phenomena.

\noindent$\bullet$ Let \eqref{HarmonicP} hold  and set the   notation for the averages of a general dynamical quantity $\chi=\chi(x,v)$ by
\ben\label{averagechi} \lr{\chi}:=\int_{\R^3_x\times\R^3_v} \chi(x,v)F(t,x,v)dxdv.\een
By \eqref{kineticwithexternalpotential}, one may easily derive that
\begin{equation}\label{monopole}\left\{ \begin{aligned}
&\f{d}{dt}\lr{|x|^2}=2\lr{x\cdot v};\\
&\f{d}{dt}\lr{x\cdot v}=\lr{|v|^2}-\omega_0^2\lr{|x|^2};\\
&\f{d}{dt}\lr{|v|^2}=-2\omega_0^2\lr{x\cdot v}.
\end{aligned} \right.
\end{equation}
If we seek a solution of the equation \eqref{monopole} in the form of $e^{i\omega t}$, it can be verified that for this solution, $\omega=2\omega_0$. Thus, the explicit solution is given by $(-\frac{e^{i2\omega_0t}}{\omega_0^2}, -\frac{ie^{i2\omega_0t}}{\omega_0}, e^{i2\omega_0t})$. Therefore there is no damping for the breathing mode of a classical   gas confined in a harmonic  trap.  The occurrence of this monopole undamped solution was also pointed out by Boltzmann(see \cite{Boltzmann}). 

\noindent$\bullet$ The streaming term of \eqref{CauchyLandau} is governed by the following trajectory:
\ben\label{Hamiltonian}  \left\{ \begin{aligned}
&\f{d}{dt}\mathsf{X}(t)=\mathsf{V}(t), \quad\f{d}{dt}\mathsf{V}(t)=-\mathsf{X}(t),\\ 
&\mathsf{X}(t)|_{t=0}=x,\,\, \mathsf{V}(t)|_{t=0}=v,
\end{aligned} \right.
\een
then  $(\mathsf{X}(t),\mathsf{V}(t))=(x\cos t+v\sin t, -x\sin t+v\cos t)$ and
\ben\label{trjecforlandau} (\pa_t+v\cdot\na_x-x\cdot\na_v)F=0\Leftrightarrow \f{d}{dt}F(t,\mathsf{X}(t),\mathsf{V}(t))=0\Leftrightarrow F(t,x,v)=F(0,X(t),V(t)), \een
where $(X(t), V(t))=(x\cos t-v\sin t, x\sin t+v\cos t)$.

\subsection{Short review and difficulties}  
The stability and  convergence rate  towards the {\it entropy-invariant  solution} to \eqref{kineticwithexternalpotential} is an important and old problem.  Given the extensive body of literature on this subject, we will provide a brief overview of previous works closely connected to our own research.

$\bullet$ Without an external potential, since the question $(Q1)$ proposed in the last subsection is quite clear, 
previous studies have focused on various approaches to the global dynamics. In the context of perturbation theory, we direct readers to \cite{AMUXY,Guo1,Guo2,GS,LYY} and references therein, which delve into the micro-macro decomposition inspired by Grad’s 13 moments method (see \cite{Grad}). Additionally, \cite{DV2} presents the entropy-production method, analyzing a suitable set of ordinary differential systems, while \cite{GMM,V2} explores the hypocoercivity method in exponentially weighted spaces and discusses the theory of space enlargement in polynomially weighted spaces (see also developments in \cite{CHJ,KS}). When the spatial variable $x$ is in $\R^3$, we refer readers to \cite{BGGL, Luk} and references therein for the stability  results thanks to the dispersion effect.

$\bullet$ When an external potential is present, in particular for the harmonic potential,  much of the research has concentrated on \eqref{kineticwithexternalpotential} within a linear framework.  Let us introduce a normalized static Maxwell-Boltzmann distribution and Maxwell distribution:
 \ben\label{normalstaMBd}\mathcal{M}:=(2\pi)^{-3}e^{-\frac{1}{2}(|x|^2+|v|^2)},\quad \mu:=(2\pi)^{-\f32}e^{-\f12|v|^2}. \een Then the linear framework can be expressed as
\begin{equation}\label{Linearinkineticeqs}
\frac{\partial F}{\partial t}+v\cdot\nabla_xF-x\cdot\nabla_vF=\mathcal{L}(F):= \mathcal{C}_{coll}(\mu,F)+\mathcal{C}_{coll}(F,\mu).
\end{equation}
 Comparing it  to nonlinear equation \eqref{kineticwithexternalpotential}(see also \eqref{pertubofcauchylandau}), there exist two major differences:

\underline{(1).} The ``barometric'' distribution    $(2\pi)^{-\frac{3}{2}}e^{-\frac{1}{2}|x|^2}$, inducing strong degeneration, is absent. This introduces the primary challenge for the nonlinear stability and alters the convergence rate from the exponential rate to the polynomial rate(see Theorem \ref{Thmnonstability} for details). 

\underline{(2).} The linear equation \eqref{Linearinkineticeqs} does not enjoy $H$-theorem but the conservation laws described in Lemma \ref{CLaw} are still valid. Moreover, {\it entropy-invariant} solution can also be defined for \eqref{Linearinkineticeqs} in the same spirit of Definition \ref{Defoentropy-invariant}. In \cite{KJFSCS}, the authors call them by {\it special macroscopic modes}.    For nonlinear equation \eqref{kineticwithexternalpotential}, by Proposition \ref{periodicmax},
if  $(X(t), V(t))=(x\cos t - v\sin t, x\sin t + v\cos t)$, then the {\it entropy-invariant} solution $M$   will satisfy that
 \beno \ln M \in \mathrm{span}\{1, X(t), V(t), |X(t)|^2, X(t) \cdot V(t), |V(t)|^2, X(t) \wedge V(t)\}. \eeno
 On the other hand, in \cite{KJFSCS},  {\it special macroscopic mode} $M$ for the  linear equation \eqref{Linearinkineticeqs} will satisfy that
   \beno    M/\mathcal{M} \in \mathrm{span}\{1, X(t), V(t), |X(t)|^2, X(t) \cdot V(t), |V(t)|^2, X(t) \wedge V(t)\}. \eeno

 For the linear stability on the BGK model, see \cite{BC}, and for the Fokker-Planck and linear Boltzmann models, refer to \cite{Tabata Decay 1993, Tabata Decay 1994,DV1,DMS,HN1,HN2,KJFSCS} and references therein. Let us specify two results on the  hypocoercivity and the time decay estimate.  

 \underline{(1).}  In \cite{Tabata Decay 1993}, Tabata investigated the linearized version of the equation presented in \eqref{kineticwithexternalpotential} centered around the state $\mathcal{M}$. This linearized form is expressed as
 \ben\label{symlinearBeq} \frac{\partial F}{\partial t}=\mathbf{B}F:=-v\cdot\nabla_xF+x\cdot\nabla_vF+\mathcal{M}^{-\f12}\big(\mathcal{C}_{coll}(\mathcal{M}, \mathcal{M}^{\f12}F)+\mathcal{C}_{coll}(\mathcal{M}^{\f12}F,\mathcal{M})\big).\een 
 He established the temporal decay pattern of the semigroup $e^{t\mathbf{B}}$ within spaces like $L^2(\mathbb{R}^3_x \times \mathbb{R}^3_v) \cap (\text{Ker}(\mathbf{B}))^\perp$, employing distinct exponential weights for both   $x$ and $v$ variables. This approach heavily relies on two key points: (a)  The operator $\mathbf{B}$ exhibits favorable properties regarding its spectrum, which arise from the Boltzmann collision operator with the hard sphere model.
(b) The   exponential weight but with negative index can be propagated over time for the linear equation \eqref{symlinearBeq}.
While point (a) suggests that extending this to soft potentials requires additional effort, the validity of point (b) for the nonlinear equation remains unclear. These observations imply that a robust method for achieving linear stability necessitates further exploration, and the resolution of the nonlinear problem remains a distant goal.

 \underline{(2).} It is worth noting that in recent research by \cite{KJFSC,KJFSCS},  special macroscopic modes beyond the static Maxwell-Boltzmann distribution(see \eqref{Maxwell-Boltzmann distrubution}) are considered. New methodologies are introduced to classify all special macroscopic modes and prove hypocoercivity results with constructive convergence rates in exponentially weighted spaces. Furthermore, they investigate general confining potentials and discuss their geometric properties' consequences in terms of symmetry, partial symmetry, or lack of symmetry under rotations.

\smallskip

  The primary barrier to achieving nonlinear stability can be attributed to three factors:
\smallskip

$\bullet$ The inclusion of the streaming term $v \cdot \nabla_x - x \cdot \nabla_v$ induced by the external harmonic potential establishes the Hamiltonian structure, resulting in the absence of dispersion effects and the presence of a breathing mode, indicating the lack of damping. This implies that the interaction between the streaming term and the collision term acts as the primary mechanism for system stabilization. The initial question now revolves around the mathematical representation of this interaction.

$\bullet$  Linearizing equation \eqref{CauchyLandau} around $\mathcal{M}$ is a straightforward process. Letting $f:=F-\mathcal{M}$, we obtain the following equation:   
\ben\label{pertubofcauchylandau} \pa_tf+v\cdot \na_x f-x\cdot \na_v f=(2\pi)^{-\f32}e^{-\f12|x|^2}(\underbrace{Q(\mu,f)+Q(f,\mu)}_{:=L(f)})+Q(f,f).\een
  It is important to note that the degeneration factor $e^{-\frac{1}{2}|x|^2}$ introduces challenges in two aspects: (i). It significantly impacts the decay mechanism(see Subsection \ref{ansatz} for details).
(ii). It creates a strong imbalance between the linear term and the nonlinear term.  To be more specific, let us consider the following illustrative toy model:  
\ben\label{toynoneq} 
\partial_t f + v \cdot \nabla_x f - x \cdot \nabla_v f = - e^{-\frac{1}{2} |x|^2} f + f^2. 
\een  
It is apparent that the linear term \(e^{-\frac{1}{2} |x|^2}f\) surpasses the nonlinear term  \(f^2\) only when \(f \lesssim e^{-\frac{1}{2} |x|^2}\).  This necessitates that \(f\) resides in a weighted space with a weight exceeding \(e^{\frac{1}{2} |x|^2}\). However, as outlined in the section on weight selection (see Subsection \ref{weightchoice} below), we are restricted to working within a weighted space characterized by \(\mathcal{M}^{-2a}\), where \(a \in [1/2,  1)\). This limitation implies that a direct estimation of the nonlinear term using the linear term is unfeasible. Consequently, this discrepancy stands as a primary reason why this issue has persisted as an open problem for an extended period.

$\bullet$ The general form  of the {\it time-periodic Maxwell-Boltzmann distribution}(see \eqref{expliciformeisolu}) deviates from the factorization property observed in the static distribution  where $\mathcal{M}=(2\pi)^{-\frac{3}{2}}e^{-\frac{1}{2}|x|^2}\times (2\pi)^{-\frac{3}{2}}e^{-\frac{1}{2}|v|^2}$. This separation of variables allows us to obtain the standard linear collision operator $L(f)$, which simplifies the analysis by providing comprehensive information through the lens of the standard linear theory. However, in the general case, the mixing of $x$ and $v$ variables in the form \eqref{expliciformeisolu} complicates the problem even at the linear level, presenting challenges in analyzing the collision operator.

\subsection{Notations and main results}  

 \subsubsection{Notations}
$(i)$ We utilize the notation $a\ls b(a\gs b)$ to signify the existence of a uniform constant $C$, which may vary across different contexts, ensuring $a\leq Cb(a\geq Cb)$. We denote $a\sim b$ when $a\ls b$ and $b\ls a$. 

$(ii).$ We represent $C_{a_1,a_2,\cdots,a_n}$(or $C(a_1,a_2,\cdots,a_n)$) by a constant dependent on parameters $a_1,a_2,\cdots,a_n$. Additionally, the parameter $\varepsilon,\epsilon$ and $\epsilon_i,i\in\N$ denote various positive numbers significantly less than 1, determined in different scenarios.

$(iii).$ $\R^+$ and $\Z_+^m$ with $m\ge1$ are used to denote the set $\{x\ge0|x\in\R\}$ and the set $\{z=(z_1,\cdots,z_m)\in\Z^m|z_i\ge0, i=1,\cdots,m\}$ respectively. We also specify the set $\R^+_{>0}:=\{x>0|x\in\R\}$.
$\mathbf{1}_\Om$ is the characteristic function of the set $\Om$. We use $(\cdot,\cdot)_{\R^3},(\cdot,\cdot)_{L^2_v}$, and $(\cdot,\cdot)_{L^2_{x, v}}$ to denote the inner product in $\R^3$, $L^2_{\R^3_v}$, and $L^2_{\R^3_v\times\R^3_x}$, respectively. Occasionally, we employ $(\cdot,\cdot)$ to represent the inner product briefly without ambiguity.

$(iv).$ We denote $\pa^\al_\be$ by   $\pa_{x_1}^{\al_1}\pa_{x_2}^{\al_2}\pa_{x_3}^{\al_3}\pa_{v_1}^{\be_1}\pa_{v_2}^{\be_2}\pa_{v_3}^{\be_3}$ with $\al=(\al_1,\al_2,\al_3),\be=(\be_1,\be_2,\be_3)\in \Z^3_+$. We also set: $\al\le\be$ if $\al_i\le\be_i$ for all $1\le i\le 3$. If $T,S$ are two operators, then  the commutator $[T,S]$ is defined by $[T,S]=TS-ST$.

$(v).$ $\mathbb{M}_3(\R)$ represents the set of $3\times3$ matrices with all components in $\R$. We also use $\mathbb{I}$ to represent the unit matrix or identity operator. $\mathbb{I}_{3\times 3}$ is used to emphasize that it is a $3\times3$ unit matrix.  If $A=(a_{ij})_{m\times n}$, then $\|A\|_2^2:=\sum_{i=1}^m\sum_{j=1}^n |a_{ij}|^2$ and $|A|:=\sqrt{A^\tau A}$. Here $A^\tau$ denotes the transpose of $A$.

 $(vi).$    For the convenience, if $\<x\>:=(1+|x|^2)^{1/2}$, we define
\ben\label{exponentialfunction} \mathcal{P}_x^\a:=\lr{x}^\a, \quad \mathcal{Z}_x^{\a}:=\exp\{\a|x|^2\},\quad \mbox{for}\quad \a\in\R. \een
 We further define
\ben\label{exponentialfunctionxv}
\mathcal{P}_{x,v}^\a:=\mathcal{P}_{x}^\a\mathcal{P}_{v}^\a , \quad  \mathcal{Z}_{x,v}^\a:= \mathcal{Z}_x^\a \mathcal{Z}_v^\a,\quad \mbox{for}\quad \a\in \R.
\een  

\subsubsection{Function spaces} We provide definitions for spaces involving different variables.
\smallskip

$(1)$ \textit{Function Spaces in   $v$ variable.} Let $f=f(v)$. For $m\in\Z_+,l\in\R$, we define the weighted Sobolev space $H_l^m$ as follows:
\beno H^m_l:=\Big\{f(v)|\|f\|_{H^m_l}^2=\sum_{|\al|\le m}\|\mathcal{P}_v^{l}\pa^\al_v f\|_{L^2}^2<+\infty\Big\}.\eeno

$(2)$ \textit{Function Spaces in $v$ and $x$ variables with polynomial weights.} Let $f=f(x,v)$. For $n\geq0$, $m\in\Z_+,l\in\R$, the weighted Sobolev space $H^n_xH^m_l$ is defined by:
\beno
H^n_xH^m_l:=\Big\{f(x,v)|\|f\|^2_{H^n_xH^m_l}=\sum_{|\al|\leq n}\|\pa^\al_xf\|^2_{H^m_l}<+\infty \Big\}.
\eeno 
For simplicity, we set $\|f\|_{H^n_xH^m_l}:=\|f\|_{H^n_xL^2_v}$ when $m=l=0$ and $\|f\|_{H^n_xH^m_l}:=\|f\|_{L^2_xH^m_l}$ when $n=0$. Also, for $N\geq 0$,
$H^N_{x,v}:=\Big\{f(x,v)|\|f\|^2_{H^N_{x,v}}=\sum_{|\al|+|\be|\leq N}\|\pa^\al_x\pa^\be_v f\|^2_{L^2_{x,v}}<+\infty\Big\}.$
 
\smallskip

$(3)$ \textit{Function Spaces in $v$ and $x$ variables with exponential weights.} For $\de\geq0$ and $\eta\ll1$, we introduce  several function spaces with exponential weight defined as follows:
\ben
 &&\mathcal{H}^{N,\de}_{x}:=\Big\{h\in H^N_{x} |\|h\|^2_{\mathcal{H}^{N,\de}_{x}}:=\sum_{|\al|\leq N}\int_{\R^3_x}|\mathcal{Z}_x^\delta\pa^\al_x h|^2dx<+\infty\Big\};\label{xepwspace}\\
&&
\mathcal{H}^{N,\de}_{x}L^2_v:=\Big\{f\in H^N_{x}L^2_v |\|f\|^2_{\mathcal{H}^{N,\de}_{x}L^2_v}:=\sum_{|\al|\leq N}\int_{\R^3_x\times\R^3_v}|\mathcal{Z}_x^\delta\pa^\al_xf(x,v)|^2dvdx<+\infty\Big\}.\label{Pfspace}\\
&&\mathcal{H}^N_{x,v}:=\Big\{f\in H^N_{x,v} |\|f\|^2_{\mathcal{H}^N_{x,v}}:=\sum_{|\al|+|\be|\leq N}\|\mathcal{Z}_{x,v}^{\f12( 1-2\de -\eta(|\al|+|\be|))}\pa^\al_\be f\|^2_{L^2_{x,v}}<+\infty\Big\}.
\een

$(4)$ \textit{Function Spaces in $t,x,v$ variables.} Let $f=f(t,x,v)$ and $X$ be a function space in $x,v$ variables. Then $L^p([0,T],X)$ and $ L^\infty([0,T],X)$ are defined as follows:
\[ L^p([0,T],X):=\bigg\{f(t,x,v)\big|\|f\|^p_{ L^p([0,T],X)}=\int_0^T\|f(t)\|^p_{X}dt<+\infty\bigg\},\quad 1\leq p<\infty,\]
\[L^\infty([0,T],X):=\Big\{f(t,x,v)|\|f\|_{ L^\infty([0,T],X)}=\mathrm{esssup}_{t\in [0,T]}\|f(t)\|_{X}<+\infty\Big\}. \]
 
\subsubsection{Main results} The first two theorems focus on the answering to $(Q1)$ listed in  Section \ref{Q12}.  We begin with  the characterization of the entropy-invariant solutions. 
\begin{thm}[Characterization of entropy-invariant solution]\label{Classificationofeiit} 
Let \begin{multline}\label{omegadomain} \mathcal{O}:=\{(a,b,c,\mathsf{R})\in \R^+\times\R\times\R^+\times\mathbb{M}_3(\R)|\, ac-b^2>0, \mathsf{R}\, \mbox{is skew-symmetric}, \mathrm{tr}(|\mathsf{R}|)<2\sqrt{ac-b^2}  \}, \end{multline}
where $|\mathsf{R}|:=\sqrt{\mathsf{R}^\tau\mathsf{R}}$. $M$ is an entropy-invariant solution to \eqref{CauchyLandau} if and only if  
\begin{multline}\label{expliciformeisolu}
		M(t, x, v) = m\frac{\sqrt{\det \mathsf{Q}}}{(2\pi)^3}\exp\bigg\{-\frac{1}{2}\big[a|X(t)-y|^2 + 2b(X(t)-y) \cdot (V(t)-z) + c|V(t)-z|^2 + 2(X(t)-y)^\tau \mathsf{R}(V(t)-z)\big]\bigg\},
	\end{multline}
	where $(X(t),V(t))=(x\cos t - v\sin t, x\sin t + v\cos t)$,  $(m,y,z)\in \R^+\times\R^3\times\R^3$, $(a,b,c,\mathsf{R})\in \mathcal{O}$ and $\mathsf{Q}:=(ac-b^2)\mathbb{I}+\mathsf{R}^2$. 
\end{thm}

\begin{rmk}  The explicit form  \eqref{expliciformeisolu} provides a comprehensive characterization    of the entropy-invariant solution. This characterization  is entirely determined by thirteen conserved quantities and the integrable condition \eqref{integcondi}. 
\end{rmk}

\begin{rmk} The static Maxwell-Boltzmann distribution \eqref{Maxwell-Boltzmann distrubution} can be derived by setting $(m,y,z)=(1,0,0)$ and $(a,b,c,\mathsf{R})=(1,0,1,0)$ in the explicit form  \eqref{expliciformeisolu}. This implies that equilibrium(refer to the static Maxwell-Boltzmann distribution) is not necessarily achieved in an harmonic field(see \cite{Ce}). 
 \end{rmk}

\begin{rmk} It is not difficult to check that \eqref{omegadomain} is equivalent to 
\begin{multline}\label{omegadomain1} \mathcal{O}:=\{(a,b,c,\mathsf{R})\in \R^+\times\R\times\R^+\times\mathbb{M}_3(\R)|\, ac-b^2>0, \mathsf{R}\, \mbox{is skew-symmetric}, 	(ac-b^2)\mathbb{I}+\mathsf{R}^2>0  \}. \end{multline}
One may prove it following facts \eqref{deofU}, \eqref{deofR} and \eqref{deofS}. \end{rmk}

In fact, by utilizing the notation introduced in {\it Definition \ref{ConservMap}}, we have $(\Psi_1(M),\Psi_2(M),\Psi_3(M))=m(1,y,z)$, and the constraints for $(\Psi_4(M),\Psi_5(M),\Psi_6(M),\Psi_7(M))$ precisely align with the constraints specified in the definition of $\mathcal{O}$ in \eqref{omegadomain}. 
Because of \eqref{expliciformeisolu}, the entropy-invariant solution can also be identified as follows:
\begin{defi}\label{TPMBdis} $M$ is called a {\it time-periodic Maxwell-Boltzmann distribution} if it takes the form \eqref{expliciformeisolu}, with constrains that $(X(t),V(t))=(x\cos t - v\sin t, x\sin t + v\cos t)$,  $(m,y,z)\in \R^+\times\R^3\times\R^3$, $(a,b,c,\mathsf{R})\in \mathcal{O}$ and $\mathsf{Q}:=(ac-b^2)\mathbb{I}+\mathsf{R}^2$. 
\end{defi}

 The next theorem elucidates the relationship between $(m, y, z, a, b, c, \mathsf{R})$ in \eqref{expliciformeisolu} and the thirteen conservation laws, with a focus on unveiling the connection between $(\Psi_4(M), \Psi_5(M), \Psi_6(M), \Psi_7(M))$ and $(a, b, c, \mathsf{R})$. This is crucial for gaining valuable insights into the global dynamics.

\begin{thm}\label{biinjectionofMF0}  Let 
\ben
 &&\mathscr{M}:=\bigg\{ M(t,x,v)\big| M\,\, \mbox{a {\it time-periodic Maxwell-Boltzmann distribution} with $m>0$}\bigg\},\label{setofM} \\
&& \mathscr{I}:=\bigg\{\mathsf{V}= \Psi(F_0)\bigg|F_0(x,v)\ge0, 0<\int_{\R^6} F_0(1+|x|^2+|v|^2)dxdv<\infty\bigg\}.\label{setofF0}
 \een
 There exists a bijection mapping  $\widetilde{\mathscr{F}}:\mathscr{I}\rightarrow \mathscr{M}$ such that  $\Psi(\widetilde{\mathscr{F}}(\mathsf{V}))=\mathsf{V}$ for any $\mathsf{V}\in\mathscr{I}$. Moreover, 
if $M_i=\widetilde{\mathscr{F}}(\mathsf{V}_i)$ has  the form   \eqref{expliciformeisolu} with $(m_i,y_i,z_i,a_i,b_i,c_i,\mathsf{R}_i)\in \R^+_{>0}\times\R^6\times\mathcal{O}$, for $i=1,2$,
satisfying that   $\|\mathsf{V}_1-\mathsf{V}_2\|_2\ll1$(see \eqref{normofpsif} for the definition), then exists a constant $C(\mathsf{V}_1)$ such that
\ben\label{stabbijection3}
 \|(m_1-m_2,y_1-y_2,z_1-z_2,a_1-a_2, b_1-b_2,c_1-c_2,\mathsf{R}_1-\mathsf{R}_2)\|_2\le C(\mathsf{V}_1)\|\mathsf{V}_1-\mathsf{V}_2\|_2.
 \een
\ben\label{stabbijection4}
 \f{M_{1}}{M_{2}}+\f{M_{2}}{M_{1}}\le 4\exp\bigg\{C(\mathsf{V}_1)\|\mathsf{V}_1-\mathsf{V}_2\|_2(|x|^2+|v|^2+1)\bigg\}.\een
\end{thm}

 \begin{rmk} In the proof of Theorem \ref{biinjectionofMF0}, we offer a more thorough elucidation of the bijection mapping $\widetilde{\mathscr{F}}$. Precisely, $(a, b, c, \mathsf{R})$ is entirely determined by $(\Psi_4(M), \Psi_5(M), \Psi_6(M), \Psi_7(M))$ through equation \eqref{equalcondition1} and the corresponding translation operation. Our approach is inspired by the methodology outlined in \cite{Levermore1}.
  \end{rmk}

 \begin{rmk} Estimates \eqref{stabbijection3} and \eqref{stabbijection4} yield stability results concerning the bijection mapping $\widetilde{\mathscr{F}}$, thereby enabling us to establish the overall Lyapunov stability of the Maxwell-Boltzmann distribution.
\end{rmk}

\begin{rmk} Utilizing Theorem \ref{biinjectionofMF0}, the relative entropy associated with \eqref{CauchyLandau} can be rigorously defined as follows:
\ben\label{relativeentorpy} 
\mathscr{H}(F|\widetilde{\mathscr{F}}(\Psi(F_0)))(t):=\int_{\R^3_x\times\R^3_v}\bigg(F\log \f{F}{\widetilde{\mathscr{F}}(\Phi(F_0))}-F+\widetilde{\mathscr{F}}(\Psi(F_0))\bigg)dvdx. \een
Subsequently, the $H$-theorem can be reformulated as: 
\begin{equation}\label{Hthmrevised}
\f d{dt}\mathscr{H}(F|\widetilde{\mathscr{F}}(\Psi(F_0)))(t)+\mathscr{D}(F)(t)=0.
\end{equation}
This constitutes the primary dissipative mechanism for equation \eqref{CauchyLandau}. Moreover, this indicates that time-periodic Maxwell-Boltzmann distribution $\widetilde{\mathscr{F}}(\Psi(F_0))$ is the only candidate of the final state of the solution $F(t)$ generated by the initial data $F_0$ from the perspective of the long-time dynamics. 
\end{rmk}

\begin{rmk} In the proof process, we only rely on the properties of the collision operator rather than its specific expression. Therefore, the conclusions in Theorem \ref{Classificationofeiit} and Theorem \ref{biinjectionofMF0} are also valid for the Boltzmann equation with harmonic potential.
\end{rmk}

  Now we can state our  nonlinear stability result which gives the affirmative answer to $(Q2)$:
\begin{thm}[Nonlinear stability]\label{Thmnonstability}
	Let $M$ be a time-periodic Maxwell-Boltzmann distribution   in \eqref{expliciformeisolu} with $(m,y,z)\in \R^+_{>0}\times\R^3\times\R^3$ and  $(a,b,c,\mathsf{R})\in \mathcal{O}$. Let  $r:=\f{\mathrm{tr}(|\mathsf{R}|)}{2\sqrt{(ac-b^2)}}$ with $|\mathsf{R}|=\sqrt{\mathsf{R}^\tau\mathsf{R}}$ and $\c:=\f{1+r}4$.  
\begin{itemize}
\item[(i).](Asymptotic stability) Suppose that  $F_0=F_0(x,v)$ satisfies that $F_0\ge0$ and $\Psi(F_0)=\Psi(M)$(see Definition \ref{ConservMap}).   
 For any $\delta\in (0,\f14-\f{\c}2)$ and $\eta_0\in(0,\f{1-4\delta}{4\c}-\f12)$, if $q_0:=\f{1-4\delta}{4\c}-\eta_0$, $\eta<\min\{\eta_0/1200,1/5-1/(10q_0)\}$,   there exists a constant $\vep=\vep(M,\delta, \eta_0,r)\ll1$ such that if
	\beno
	\sum_{|\al|+|\be|\leq 5}\|M^{-1+2\de+\eta(|\al|+|\be|)}(0)\pa^\al_\be( F_0-M(0))\|_{L^2_{x,v}}\leq \vep ,
	\eeno
then equation \eqref{CauchyLandau} admits a unique global solution $F=F(t,x,v)$ verifying that $F\ge0$ and
	\begin{align*} &\sum_{|\al|+|\be|\leq 5}\big(\|M^{-1+2\de+\eta(|\al|+|\be|)}(t)\pa^\al_\be (F(t)-M(t))\|_{L^2_{x,v}}+\lr{t}^{q_0}\|M^{-\f12}(t)\pa^\al_\be (F(t)-M(t))\|_{L^2_{x,v}}\big)\\
		&\lesssim \sum_{|\al|+|\be|\leq 5}\|M^{-1+2\de+\eta(|\al|+|\be|)}(0)\pa^\al_\be (F_0-M(0))\|_{L^2_{x,v}}.
\end{align*}

\item[(ii).](Lyapunov stability) Suppose that  $F_0=F_0(x,v)$ satisfies that $F_0\ge0$ and $\Psi_1(F_0)>0$. 
 For any $\delta\in (0,\f16-\f{\c}3)$, there exists a constant $\vep=\vep(M,\delta,r)\ll1$ such that if
	\beno
	\sum_{|\al|+|\be|\leq 5}\|M^{-1+2\de}(0)\pa^\al_\be( F_0-M(0))\|_{L^2_{x,v}}\leq \vep,
	\eeno
equation \eqref{CauchyLandau} admits a unique global solution $F=F(t,x,v)$ verifying that $F\ge0$ and
	\beno \sum_{|\al|+|\be|\leq 5} \|M^{-1+4\de}(t)\pa^\al_\be (F(t)-M(t))\|_{L^2_{x,v}} \lesssim \vep.
\eeno
\end{itemize}
\end{thm}

We begin with the comments on asymptotic stability:

\noindent\underline{(1).} As far as we know, this signifies a comprehensive and initial understanding of the asymptotic stability pertaining to the time-periodic Maxwell-Boltzmann distribution. Moreover, our employment of the micro-macro decomposition (as delineated in equation \eqref{deofP}) strengthens the concept that the route to equilibrium intertwines the evolution towards the local Maxwellian in velocity space and the progression towards a barometric distribution in phase space. However, we omit the vacuum scenario from our findings due to our methodology's heavy reliance on collision dynamics.
\smallskip 

\noindent\underline{(2).} The quantitative assessment of the convergence rate illustrates its dependency solely on $(a, b, c, \mathsf{R}) \in \mathcal{O}$ and highlights its optimality in comparison to the toy model described in equation \eqref{toymodelofNSlandau}. The constraints on $(\delta, \eta_0)$ stem from the weight-transfer lemma (refer to Lemma \ref{mixturegainweight}) and the condition $\frac{1-4\delta}{4\c}-\eta_0>\frac{1}{2}$, which are pivotal in our approach. This also implies that as the rotational effect intensifies (i.e., $r\rightarrow1$), the weight function $M^{-1+2\delta}$ will converge towards $M^{-1}$, approaching the threshold, considering $M$ as the background solution.
 \smallskip

\noindent\underline{(3).} From the result on the asymptotic stability(more precisely, see \eqref{anstzenergy}), one may easily check that 
\beno
 \Big|\f{F(t,x,v)}{M(t,x,v)}-1\Big|\leq C \<t\>^{-q_0}M^{-\f12}(t,x,v),\quad (x,v)\in \R^3_x\times\R^3_v.
\eeno
Thus we derive that
\begin{align*}
0\le \mathscr{H}(F)(t)- \mathscr{H}(M)(t)= \mathscr{H}(F|M)(t) \leq& \int_{\R^3_x\times\R^3_v} F\log (1+C\<t\>^{-q_0}M^{-\f12}) dxdv\\
\leq & C\<t\>^{-q_0} \int_{\R^3_x\times\R^3_v} F(t)M(t)^{-\f12} dxdv\leq C\<t\>^{-q_0}.
\end{align*}
This, together with Theorem \ref{biinjectionofMF0}, indicates that   the {\it time-periodic   Maxwell-Boltzmann distribution} $M$  is the unique minimizer of $\mathscr{H}(F)$ over the set of all solutions to equation \eqref{CauchyLandau} that satisfy the same conservation laws as $M$. Furthermore, the entropy of each solution in this set will asymptotically converge to $\mathscr{H}(M)$ at an explicit rate.
These findings suggest that Uhlenbeck's statement in \cite{uhlenbeck-ford} claiming that ``the time-periodic Maxwell-Boltzmann distributions have only limited interest'' is incorrect.

\smallskip

\noindent\underline{(4).} Our findings, as stated in Theorem \ref{biinjectionofMF0} and Theorem \ref{Thmnonstability}, together with the approach used in \cite{DV2}, provide a comprehensive understanding of the nonlinear problem with general initial data: if the solution maintains boundedness in the Sobolev spaces with exponential weights,  the solution $F$ will gradually approach $\widetilde{\mathscr{F}}(\Psi(F_0))$. Crucially, the convergence rate is determined solely by the thirteen conserved quantities associated with the initial data $F_0$.   Our findings demonstrate that the long-term behavior of the solutions is in fact intimately tied to the conserved quantities, which gives them significant theoretical importance.
 \smallskip

\noindent\underline{(5).} Relaxing the requirement for rapid decay of the initial data in both the $x$ and $v$ variables is also quite an interesting problem, even at the linear level. Undoubtedly, the relative entropy, $\mathscr{H}(F|\widetilde{\mathscr{F}}(\Psi(F_0)))$, will play a central role in this further investigation.
 
\smallskip
 
Some comments on Lyapunov stability are in order:
 \smallskip

\noindent\underline{(1).} Lyapunov stability hinges greatly on two critical factors: the emergence of a new time-periodic Maxwell-Boltzmann distribution, dictated by the initial perturbation, which remains proximate to the background solution; and the neutral asymptotic stability ensuring the convergence of the solution towards this new distribution.  
 \smallskip

\noindent\underline{(2).} The decrease in weight in the final estimate primarily stems from the stability outcome \eqref{stabbijection4} regarding the bijection mapping $\widetilde{\mathscr{F}}$.

\subsection{Ideas and strategies} 
 
 Our main ideas and strategies can be concluded as four aspects:
(i) reducing the investigation of the general case to a specific static Maxwell-Boltzmann distribution through a change of coordinates; (ii)  constructing a suitable weighted function to transfer the weight from the velocity variable to the phase variable and  capturing the decay machanism; 
(iii)  employing the variant Poincar\'e inequality to establish the macroscopic estimates;  (iv) devising a time-phase splitting method to close the energy argument.

 \subsubsection{Reduction to a normalized static Maxwell-Boltzmann distribution} 
To restore the factorization property of the general Maxwell-Boltzmann distribution \eqref{expliciformeisolu}, we will utilize a change of coordinates to convert \eqref{expliciformeisolu} into the normalized static Maxwell-Boltzmann distribution $\mathcal{M} =(2\pi)^{-3}e^{-\frac{1}{2}(|x|^2+|v|^2)}$. Concurrently, we will meticulously track all transformations and apply them to the equation to achieve the desired reduction. Following this reduction, the analysis of the Maxwell-Boltzmann distribution \eqref{expliciformeisolu} to \eqref{CauchyLandau} will be reduced to some equation involving the normalized static Maxwell-Boltzmann distribution $\mathcal{M}$. The key point utilized here is to fully exploit the Galilean invariance property of the collision operator.
\smallskip

Let $M$ be a {\it time-periodic Maxwell-Boltzmann distribution} represented by the form \eqref{expliciformeisolu}. In the subsequent steps, we will achieve the reduction. Through a straightforward calculation, we find
\ben\label{psi123}
(\Psi_1(M),\Psi_2(M),\Psi_3(M))=m(1,y,z).
\een


 \noindent\underline{\textit{Step 1. Normalization of mass and momentum.}} Recall that $m=\Psi_1(M)$, we introduce
\begin{equation}\label{FtoF1}
F_1(t, x, v) := m^{-1}F(t, x, v), \quad M_1(t, x, v) := m^{-1}M(t, x, v).
\end{equation}
Then, one can verify that
\begin{equation*}
\partial_tF_1 + v \cdot \nabla_xF_1 - x \cdot \nabla_vF_1 = mQ(F_1, F_1).
\end{equation*}
Next, since $(\Psi_2(M_1),\Psi_3(M_1))=(y,z)$, by setting
\begin{equation}\label{F1toF2}
\begin{aligned}
F_2(t, x, v) &:= F_1(t, x + y\cos t + z\sin t, v - y\sin t + z \cos t),\\
M_2(t, x, v) &:= M_1(t, x + y\cos t + z \sin t, v - y\sin t + z \cos t),
\end{aligned}
\end{equation}
we obtain the following equation, leveraging the Galilean invariance property of the collision operator:
\begin{equation*}
\partial_tF_2 + v \cdot \nabla_xF_2 - x \cdot \nabla_vF_2 = mQ(F_2, F_2).
\end{equation*}

\noindent\underline{\textit{Step 2. Maxwell-Boltzmann distribution with quadratic form only in $(x,v)$ variables.}} It is straightforward to verify that $(\Psi_1(M_2),\Psi_2(M_2),\Psi_3(M_2))=(1,0,0)$. Thus, the form of \eqref{expliciformeisolu} will be reduced to
\beno
	M_2(t, x, v) = \frac{\sqrt{\det \mathsf{Q}}}{(2\pi)^3}\exp\bigg\{-\frac{1}{2}\big[a|X(t)|^2 + 2bX(t) \cdot V(t) + c|V(t)|^2 + 2X(t)^\tau \mathsf{R}V(t)\big]\bigg\},
\eeno
where $(a,b,c,\mathsf{R})\in \mathcal{O}$ (see the definition \eqref{omegadomain}) and $\mathsf{Q} = (ac - b^2)\mathbb{I} + \mathsf{R}^2 > 0$. Direct computation yields
\begin{equation*}
	a|X(t)|^2 + 2bX(t) \cdot V(t) + c|V(t)|^2 + 2X(t)^\tau \mathsf{R}V(t) = A(t)|x|^2 + 2B(t)x \cdot v + C(t)|v|^2 + 2x^\tau \mathsf{R}v,
\end{equation*}
where\begin{equation}\label{deofCt}
	\begin{aligned}
	A(t) :=& a\cos^2t + 2b\sin t\cos t + c\sin^2t;\\
	B(t) :=& (c - a)\sin t\cos t + b(\cos^2t - \sin^2t);\\
C(t) :=& a\sin^2t - 2b\sin t\cos t + c\cos^2t.
	\end{aligned}
\end{equation}  Thanks to the constraints in \eqref{omegadomain}, for any $t \ge 0$, we deduce that
\begin{equation}\label{Ctproperty}
	0 < \frac{a + c}{2} - \sqrt{( \frac{a - c}{2}  ) ^2 + b^2} \le A(t), ~~C(t) \le \frac{a + c}{2} + \sqrt{(\frac{a - c}{2} )^2 + b^2},
\end{equation}
and
\begin{equation*}
	C'(t)=-2B(t),\quad B'(t)=C(t)-A(t),\quad A(t)C(t) - B^2(t) = ac - b^2 > 0.
\end{equation*}
Let $\mathsf{S}:=\mathsf{Q}^{\frac{1}{2}}$, the above relations imply that
\begin{equation*}
	a|X(t)|^2 + 2bV(t) \cdot X(t) + c|V(t)|^2 + 2X(t)^\tau \mathsf{R}V(t) = C(t)\Big|v + \frac{B(t)x - \mathsf{R}x}{C(t)}\Big|^2 + \frac{1}{C(t)}|\mathsf{S}x|^2.
\end{equation*} Now we introduce
\begin{equation}\label{F2toF3}
	\begin{aligned}
	M_3(t, x, v) &:= M_2(t, \sqrt{C(t)}x, \frac{v - B(t)x + \mathsf{R}x}{\sqrt{C(t)}})= \frac{\det \mathsf{S}}{(2\pi)^3}\exp\{-\frac{1}{2}(|v|^2 + |\mathsf{S}x|^2)\},\\
	F_3(t, x, v) &:= F_2(t, \sqrt{C(t)}x, \frac{v - B(t)x + \mathsf{R}x}{\sqrt{C(t)}}).
		\end{aligned}
\end{equation}
To derive the equation for $F_3$, we observe that
\begin{equation*}
	\begin{aligned}
		\partial_tF_3 &= \partial_tF_2 - \frac{B(t)}{\sqrt{C(t)}}x \cdot \nabla_xF_2 + \frac{A(t) - C(t)}{\sqrt{C(t)}}x \cdot \nabla_vF_2 + \frac{B(t)(v - B(t)x + \mathsf{R}x)}{C(t)\sqrt{C(t)}} \cdot \nabla_vF_2;\\
		\nabla_xF_3 &= \sqrt{C(t)}\nabla_xF_2 - \frac{B(t)}{\sqrt{C(t)}}\nabla_vF_2 - \frac{\mathsf{R}\nabla_vF_2}{\sqrt{C(t)}}; \:\:\:\:\:\:\:\:\:\:\:\: \nabla_vF_3 = \frac{\nabla_vF_2}{\sqrt{C(t)}},\\
	\end{aligned}
\end{equation*}
which immediately implies that
\begin{equation*}
\begin{aligned}
\nabla_vF_2 &= \sqrt{C(t)}\nabla_vF_3, \quad \nabla_xF_2 = \frac{\nabla_xF_3 + B(t)\nabla_vF_3 + \mathsf{R}\nabla_vF_3}{\sqrt{C(t)}},\\
\partial_tF_2 &= \partial_tF_3 + [C(t) - A(t)]x \cdot \nabla_vF_3 - \frac{2B(t)}{C(t)}\mathsf{R}x \cdot \nabla_vF_3 + \frac{2B^2(t)}{C(t)}x \cdot \nabla_vF_3 + \frac{B(t)}{C(t)}(x \cdot \nabla_x - v \cdot \nabla_v)F_3.
\end{aligned}
\end{equation*}
On one hand, we have
\beno
	&&(\partial_t + v \cdot \nabla_x - x \cdot \nabla_v)F_2(t, \sqrt{C(t)}x, \frac{v - B(t)x + \mathsf{R}x}{\sqrt{C(t)}}) \\
	&&= (\partial_t + \frac{v \cdot \nabla_x - \mathsf{S}x \cdot \mathsf{S}\nabla_v + \mathsf{R}x \cdot \nabla_x - \mathsf{R}v \cdot \nabla_v}{C(t)})F_3(t, x, v).
\eeno
On the other hand, since the collision operator $Q$ is Galilean invariant, one may have 
\begin{equation*}
	mQ(F_2, F_2)(t, \sqrt{C(t)}x, \frac{v - B(t)x + \mathsf{R}x}{\sqrt{C(t)}}) = mQ(F_3, F_3)(t, x, v).
\end{equation*}
These lead to the equation for $F_3$:
\begin{equation*}
	\partial_t F_3 + \frac{v \cdot \nabla_x - \mathsf{S}x \cdot \mathsf{S}\nabla_v + \mathsf{R}x \cdot \nabla_x - \mathsf{R}v \cdot \nabla_v}{C(t)}F_3 = mQ(F_3, F_3).
\end{equation*}

\noindent\underline{\it Step 3. Maxwell-Boltzmann distribution with normal form of $(x,v)$ variables.}  Let
\begin{equation}\label{F3toF4}
	\begin{aligned}
	M_4(t, x, v) &:= (\det \mathsf{S})^{-1}M_3(t, \mathsf{S}^{-1}x, v) = (2\pi)^{-3}e^{-\f12 (|x|^2+|v|^2)},\\
	F_4(t, x, v) &:= (\det \mathsf{S})^{-1}F_3(t, \mathsf{S}^{-1}x, v).
		\end{aligned}
\end{equation}
Then we derive that 
\beno
	\partial_t F_4+ \frac{\mathsf{S}v \cdot \nabla_x - \mathsf{S}x \cdot \nabla_v + \mathsf{R}x \cdot \nabla_x - \mathsf{R}v \cdot \nabla_v}{C(t)}F_4 = m(\det \mathsf{S})Q(F_4, F_4).
\eeno

\noindent\underline{\it  Step 4. Normalization of $\mathsf{R}$ and $\mathsf{S}$.} Since $\mathsf{R}$ is skew-symmetric,  there exists an orthogonal matrix $\mathsf{U}$ such that if $|\mathsf{R}|:=\sqrt{\mathsf{R}^\tau\mathsf{R}}$, then
\ben\label{deofU} \mathsf{U}^\tau\mathsf{R}\mathsf{U}= 
\begin{pmatrix}
	0 &  \f12 \mathrm{tr}|\mathsf{R}| & 0\\
	-\f12 \mathrm{tr}|\mathsf{R}| & 0& 0\\
	0 & 0 & 0
\end{pmatrix}.\een 
 Since $\mathsf{S}=\sqrt{(ac-b^2)\mathbb{I}+\mathsf{R}^2}>0$, we derive that $\mathrm{tr}|\mathsf{R}|\in [0,2\sqrt{ac-b^2})$, which implies that     
\ben
&\label{deofR} \mathfrak{R}:=(\sqrt{ac-b^2})^{-1}\mathsf{U}^\tau\mathsf{R}\mathsf{U}= 
\begin{pmatrix}
	0 & r & 0\\
	-r & 0& 0\\
	0 & 0 & 0
\end{pmatrix}, \quad\mbox{where}\,\, r:=\f{\mathrm{tr}|\mathsf{R}|}{2\sqrt{ac-b^2}}\in[0,1),
\\ \label{deofS}
 & \mathfrak{S}:= (\sqrt{ac-b^2})^{-1}\mathsf{U}^\tau\mathsf{S}\mathsf{U}= 
\begin{pmatrix}
	\sqrt{1-r^2} & 0 & 0\\
	0 & \sqrt{1-r^2} & 0\\
	0 & 0 & 1
\end{pmatrix}.
\een
 Now we set
\begin{equation}\label{F4toF5}
	\begin{aligned}
M_5(t, x, v) &:=M_4(t, \mathsf{U}x, \mathsf{U}v) = (2\pi)^{-3}e^{-\f12 (|x|^2+|v|^2)},\\
F_5(t, x, v) &:=F_4(t, \mathsf{U}x, \mathsf{U}v).
		\end{aligned}
\end{equation}
It is easy to see that $F_5$ satisfies that
\ben\label{reduceeq}
\partial_t F_5 + \f{\sqrt{ac-b^2}}{C(t)}(\mfS v \cdot \nabla_x - \mfS x \cdot \nabla_v + \mfR x \cdot \nabla_x - \mfR v \cdot \nabla_v)F_5 = m(ac-b^2)^{3/2}(\det \mfS)Q(F_5, F_5).
\een
   We can introduce the following definition:
\begin{defi}[Normalization-Scaled Landau equation in the presence of harmonic potential]
Let $M$ be a time-periodic Maxwell-Boltzmann distribution defined from \eqref{expliciformeisolu} to the equation \eqref{CauchyLandau}, with $(m,y,z)\in \mathbb{R}^+_{>0}\times\mathbb{R}^6$, $(a,b,c,\mathsf{R})\in\mathcal{O}$, as defined in \eqref{omegadomain}, and $\mathsf{Q}:=(ac-b^2)\mathbb{I}+\mathsf{R}^2$. Following arguments from previous steps, $\mathcal{M}$, defined in \eqref{normalstaMBd}, is the steady solution to the so-called Normalization-Scaled Landau equation in the presence of harmonic potential, which is given by
\begin{equation}\label{NSlandauCauchy}
\begin{aligned}
\partial_t G &+ \mathscr{C}_l(t)\left( \mfS v \cdot \nabla_x -  \mfS x \cdot \nabla_v + \mfR x \cdot \nabla_x -\mfR v \cdot \nabla_v\right)G =\mathscr{C} Q(G, G),
\end{aligned}
\end{equation}
where  $\mathscr{C}_l(t):=\frac{\sqrt{ac-b^2}}{C(t)}$ and $\mathscr{C} :=m(ac-b^2)^{3/2}\left(\det \mfS\right)$. Here $C(t)$, $\mfR$, and $\mfS$ are defined in \eqref{deofCt}, \eqref{deofR}, and \eqref{deofS}, respectively. 
\end{defi}

Several remarks are in order:
\begin{rmk}
It is easy to verify that the coefficients $\mathscr{C}_l(t)$ and $\mathscr{C}$ are bounded from below and above, thanks to \eqref{Ctproperty} and the fact that $\det\mathfrak{S}=(ac-b^2)^{-\frac{3}{2}}\sqrt{\det \mathsf{Q}}$. Thus, they can be regarded as universal constants when applying the energy method.  
\end{rmk}

\begin{rmk}
The relationship between solutions to \eqref{CauchyLandau} and \eqref{NSlandauCauchy} can be summarized as follows:
\begin{equation}\label{RelationFG}
\begin{aligned}
G(t,x,v) &= m^{-1}\left(\det \mathsf{S}\right)^{-1} F\left(t,\sqrt{C(t)}\mathsf{S}^{-1}\mathsf{U}x+y\cos t+z\sin t,\frac{\mathsf{U}v-B(t)\mathsf{S}^{-1}\mathsf{U}x+\mathsf{R}\mathsf{S}^{-1}\mathsf{U}x}{\sqrt{C(t)}}-y\sin t+z\cos t\right),\\
\mathcal{M}(x,v) &= m^{-1}\left(\det \mathsf{S}\right)^{-1} M\left(t,\sqrt{C(t)}\mathsf{S}^{-1}\mathsf{U}x+y\cos t+z\sin t,\frac{\mathsf{U}v-B(t)\mathsf{S}^{-1}\mathsf{U}x+\mathsf{R}\mathsf{S}^{-1}\mathsf{U}x}{\sqrt{C(t)}}-y\sin t+z\cos t\right),
\end{aligned}
\end{equation}
where $\mathsf{S}=\sqrt{(ac-b^2)\mathbb{I}+\mathsf{R}^2}$, and $\mathsf{U}$ is defined through \eqref{deofU}.
\end{rmk}

  The equation \eqref{NSlandauCauchy} still has  thirteen conservation laws.  Similar to the mapping  $\Psi$ for \eqref{CauchyLandau}(see Definition \ref{ConservMap}), we can  introduce the conserved mapping $\bar{\Psi}$ for these thirteen conservation laws of \eqref{NSlandauCauchy}.  

\begin{defi}\label{ConservMap1}
	Let $G(t,x,v)$ be a global solution to \eqref{NSlandauCauchy} with the initial data $G_0=G_0(x,v)$.  $\bar{\Psi}$ is called a conserved mapping of \eqref{NSlandauCauchy} if 
	\ben\label{Gconserved}
	\bar{\Psi}(G(t)):=\int_{\R^3_x\times\R^3_v} G(t, x, v)\begin{bmatrix} 1 \\  v  \\ \mfS^{-1}x \\ |v + \mfR\mfS^{-1}x|^2  \\ \mfS^{-1}x \cdot v  \\  |\mfS^{-1}x|^2 \\ (v + \mfR\mfS^{-1}x) \wedge \mfS^{-1}x \end{bmatrix}dxdv. 
	\een 
\end{defi}
\begin{rmk}\label{v2Rx}
	Using \eqref{equality for G -1 x R G -1 x} and \eqref{equality for R G -1 x} it is easy to verify that $|v|^2+(v, \mfR\mfS^{-1}x)=|v+ \mfR\mfS^{-1}x|^2-r((v+ \mfR\mfS^{-1}x)\wedge(\mfS^{-1}x))_{12}$, thus $|v + \mfR\mfS^{-1}x|^2$ in \eqref{Gconserved} can be replaced by $|v|^2+(v, \mfR\mfS^{-1}x)$.
\end{rmk}

\begin{rmk}\label{rmk140} In the appendix, we shall give a detailed proof to the fact that
\ben\label{conservationofNSlandau}
\f{d}{dt}\int_{\R^3_x\times\R^3_v} G(t, x, v)\begin{bmatrix} 1 \\ v   \\ \mfS^{-1}x \\ |v + \mfR\mfS^{-1}x|^2 \\ \mfS^{-1}x \cdot v \\ |\mfS^{-1}x|^2 \\   (v + \mfR\mfS^{-1}x) \wedge \mfS^{-1}x   \end{bmatrix} dxdv=0.
\een
This justifies the validity of the definition of $\bar{\Psi}$.
 \end{rmk} 
Now the investigation of the stability of $M$ to the equation \eqref{CauchyLandau} is transformed into that of the stability analysis of $\mathcal{M}$ to the equation \eqref{NSlandauCauchy}. Let $G:=\mathcal{M}+f$. We will focus on the following equation:
\ben\label{pertubNSlandaucauchy} 
&&\partial_t f + \mathscr{C}_l(t)\left( \mfS v \cdot \nabla_x -  \mfS x \cdot \nabla_v + \mfR x \cdot \nabla_x -\mfR v \cdot \nabla_v\right)f=  \mathscr{C}_1e^{-\f12|x|^2}L(f)+\mathscr{C}_2Q(f,f),
\een
where  $L(f):=Q(\mu, f)+Q(f,\mu)$ with $\mu=(2\pi)^{-\f32}e^{-\f12|v|^2}$ and $(\mathscr{C}_1, \mathscr{C}_2):=((2\pi)^{-\f32}\mathscr{C}, \mathscr{C})$. 

Thanks to \eqref{RelationFG},  the asymptotic stability in Theorem \ref{Thmnonstability} can be reduced to prove:
\begin{thm}[Stability of \eqref{NSlandauCauchy} or Global wellposedness of \eqref{pertubNSlandaucauchy}]\label{ThmGstofNSlandau} Let $(a,b,c,\mathsf{R})\in \mathcal{O}$(see the definition \eqref{omegadomain}), $r:=\f{\mathrm{tr}(|\mathsf{R}|)}{2\sqrt{(ac-b^2)}}$ with $|\mathsf{R}|=\sqrt{\mathsf{R}^\tau\mathsf{R}}$ and $\c:=\f{1+r}4$. Suppose   $f_0$ satisfies that $\mathcal{M}+f_0\geq0$ and  $\bar{\Psi}(f_0)=0$, where  $\bar{\Psi}$ is the conserved mapping for \eqref{NSlandauCauchy}(see Definition \ref{ConservMap1}). Then for any $\delta\in (0,\f14-\f{\c}2)$ and $\eta_0\in(0,\f{1-4\delta}{4\c}-\f12)$, if $q_0:=\f{1-4\delta}{4\c}-\eta_0$, $\eta<\min\{\eta_0/1200,1/5-1/(10q_0)\}$, and  \beno
	\sum_{|\al|+|\be|\leq 5}\| \mathcal{M}^{-1+2\delta+\eta(|\al|+\be|)}\pa^\al_\be f_0\|_{L^2_{x,v}}\leq \vep\ll1,
	\eeno
the equation \eqref{pertubNSlandaucauchy} admits a global solution $f=f(t,x,v)$ such that $\mathcal{M}+f\ge0$ and
\beno 
\sum_{|\al|+|\be|\leq 5}(\| \mathcal{M}^{-1+2\delta+\eta(|\al|+\be|)}\pa^\al_\be f\|_{L^2_{x,v}}+\lr{t}^{q_0}\| \mathcal{M}^{-\f12}\pa^\al_\be f\|_{L^2_{x,v}})\lesssim \sum_{|\al|+|\be|\leq 5}\| \mathcal{M}^{-1+2\delta+\eta(|\al|+\be|)}\pa^\al_\be f_0\|_{L^2_{x,v}} .
\eeno
\end{thm}

  This concludes the reduction step. In the following section, we will delve into various facets of the statement and elucidate the core concept of Theorem \ref{ThmGstofNSlandau}. This will encompass discussions on the selection of the weight $\mathcal{M}^{-1+2\delta}$, the ansatz regarding the time factor $\<t\>^{q_0}$, the adjustment involving micro-macro decomposition, and the fundamental approach underpinning the proof of nonlinear stability.

\subsubsection{Selection of weight function}\label{weightchoice}  As elucidated in the explanation for the toy model \eqref{toynoneq}, ensuring the dominance of the linear term over the nonlinear term mandates the multiplication of a weight surpassing $e^{\frac{1}{2}|x|^2}$. Simultaneously, if we mandate that the weight function commutes with the stream operator $ \mfS v \cdot \nabla_x -  \mfS x \cdot \nabla_v + \mfR x \cdot \nabla_x -\mfR v \cdot \nabla_v$, the optimal choice for the weight function becomes $\mathcal{M}^{-2a}$ with $a \in [1/2,1)$. This implies that the solution $f$ will be endowed with the weight $\mathcal{M}^{-a}$.
This configuration offers two key benefits:
(a). By selecting $a=\f12$, we can  recover the coercivity estimate from the symmetric linearized theory.
(b). For $a \in (1/2,1)$, we achieve the inequality: \[
(L f  , f e^{a|v|^2}    ) \le -\lambda  \Vert f e^{\f a2  |v|^2}  \Vert_{L^2_{-1/2} }^2 + C   \Vert f  \Vert_{L^2_{-1/2} }^2.
\] 
 For detailed insights, we direct readers to refer to Lemma \ref{coer}.

Finally, given $G = \mathcal{M} + f$, the stipulation that $a < 1$ corresponds to the requirement that $G \in L^1(\mathbb{R}^3_x \times \mathbb{R}^3_v)$.

\subsubsection{Ansatz of the decay estimate of \eqref{pertubNSlandaucauchy}} \label{ansatz} To capture the decay mechanism of \eqref{pertubNSlandaucauchy}, one may first consider the toy model which can be presented as follows:
\ben\label{toymodelofNSlandau}  
\partial_t g +  \mathcal{T}   g =-e^{-\f12|x|^2}g,  \quad \mathcal{T}   g :=  \left( \mfS v \cdot \nabla_x -  \mfS x \cdot \nabla_v + \mfR x \cdot \nabla_x -\mfR v \cdot \nabla_v\right)g.
\een

\noindent$\bullet$ On one hand, by trajectory method, we can get the explicit formula for \eqref{toymodelofNSlandau}, i.e.,
 \beno g(t,\tilde{X}(t),\tilde{V}(t))=e^{-\int_0^t \exp\{-\f12|\tilde{X}(s)|^2\}  ds}g_0(x,v), \eeno 
where $(\tilde{X}(t),\tilde{V}(t))=(x\cos t+(\mfS v+\mfR x)\sin t,  v\cos t-(\mfR v+\mfS x)\sin t)$. Choosing $x:=(x_1,x_2,0)$ and  $v:=\sqrt{\f{1-r}{1+r}}(x_2, -x_1,0)$, where $r$ is defined in \eqref{deofR}, will imply that $|\tilde{X}(s)|^2=\f{1+r}2(|x|^2+|v|^2)$. From this, we derive that
\beno |g(t,\tilde{X}(t),\tilde{V}(t))|=|g_0(x, v)    |\exp    \bigg\{-e^{-\f   {1+r}4(|x|^2+|v|^2)}t     \bigg\}. \eeno 
Recalling that $\mathcal{M}^{-\f12}g$ still satisfies  \eqref{toymodelofNSlandau}, we eventually get that for $\delta\in(0,\f14)$,
\ben\label{examplesharpconverge1} |(g\mathcal{M}^{-\f12})(t,\tilde{X}(t),\tilde{V}(t))|   \sim |g_0(x,v)|\mathcal{M}^{-(1-2\delta)} t^{-\f{1-4\delta}{1+r}},
\een
if $e^{-\f{1+r}4(|x|^2+|v|^2)}t\sim1$.

\noindent$\bullet$ On the other hand,   let $\vartheta\ll1$, $N_1\gg1 $, $a\in[\f{1}{2},1)$. Following the ideas from \cite{Bakry Rate 2008, Douc Subgeometric 2009, Mischler Exponential 2016, Wu Large 2001}, we make full use of the weight function defined by 
\begin{equation}
\label{definition W a}
 W^2_a(x, v):=N_1e^{  a   ( |x|^2 + |v|^2  )   }+\mfS^{-1}x\cdot v e^{(  a   -\f{1+r}{4}- \vartheta)(|x|^2+|v|^2)}\sim \mathcal{M}^{-2a}, \quad  W_a(x, v) \sim \mathcal{M}^{-a},
 \end{equation}
and facts that $\mathcal{T} (\mfS ^{-1}x \cdot v) = -(|x|^2-|v|^2 - 2 (\mfR  \mfS ^{-1} x, v ) )$ as well as Lemma \ref{mixturegainweight} to derive  that 
\ben\label{111} 
\f{d}{dt}\|gW_a\|_{L^2}^2+\|g\mathcal{M}^{-a+\f{1+r}4+\vartheta}\|_{L^2}^2\le 0.  
\een 
 Choosing $a=\f12$ and $a=1-2\delta$ with $\delta\in(0,\f14)$, then we get that $\f{d}{dt}\|gW_{\f12}\|_{L^2}^2+\|g\mathcal{M}^{-\f{1-r}4+\vartheta}\|_{L^2}^2\le 0$ and $\f{d}{dt}\|gW_{1-2\delta}\|_{L^2}^2\le0$. Observing that $\|gW_{\f12}\|_{L^2}\le  C\|gW_{\f{1-r}4-\vartheta}\|_{L^2}^{\f{\f12-2\delta}{\f12-2\delta+\f{1+r}4+\vartheta}}\|gW_{1-2\delta}\|_{L^2}^{\f{\f{1+r}4+\vartheta}{\f12-2\delta+\f{1+r}4+\vartheta}}$, we have
\beno \f{d}{dt}\|gW_{\f12}\|_{L^2}^2+ C\|gW_{\f12}\|_{L^2}^{2(1+\f{\f{1+r}4+\vartheta}{\f12-2\delta})}\le 0,\eeno   
 which implies that
 \ben\label{examplesharpconverge2}  \|g\mathcal{M}^{-\f12}\|^2_{L^2}\sim \|gW_{\f12}\|^2_{L^2}\le Ct^{-\f{2(1-4\delta)}{1+r+4\vartheta}}.\een
Note that for any $r\in(0,1)$, we can choose sufficiently small $\vartheta$ and $\de$ such that the above convergence rate is larger than $(1+t)^{-1}$, which is essential for nonlinear stability of equation \eqref{pertubNSlandaucauchy}(see Subsection \ref{nonlinear stability subsection} for more details).  

Results \eqref{examplesharpconverge1} and \eqref{examplesharpconverge2} imply several key points: 
\noindent\underline{(i).} The trajectory method and the energy method can be employed to establish the optimal decay rate of the toy model \eqref{toymodelofNSlandau}. Indeed, the decay estimate in the space with low-index exponential weight is driven by the high-index exponential weight inherent in the solution itself. Contrasting with \cite{Tabata Decay 1993}, the proof presented here is straightforward and resilient.
\noindent\underline{(ii).} The decay rate is heavily influenced by the scaling matrix $\mfS$, the rotation matrix $\mfR$, and the propagation of the exponential weight. 
\noindent\underline{(iii).} The term $\mfS ^{-1}x \cdot v$ and Lemma \ref{mixturegainweight} is crucial in transferring the weight from the $v$ variable to the $x$ variable.  Additionally, the decay mechanism is solely determined by the exponential weight and is independent of the regularity.
 
\subsubsection{The macroscopic equation}  
To establish the nonlinear results, we adjust the standard micro-macro method, deviating from the symmetrization of the linear collision operator $L(f)$ in \eqref{pertubNSlandaucauchy}. This divergence primarily arises from the streaming term on the left-hand side of \eqref{pertubNSlandaucauchy}.  Mathematically, if $f$ solves \eqref{pertubNSlandaucauchy}, then the expression:
\begin{equation}\label{deofP}
	\begin{aligned}
\mathbb{P}f&:= \left (\int_{\R^3 } f  dv \right)\mu+\sum_{i=1}^3  \left (\int_{\R^3}v_ifdv  \right)v_i\mu+ \left (  \int_{\R^3}   \f{|v|^2-3}6fdv  \right)(|v|^2-3)\mu\\
&:=\mathbf{a}(t,x)\mu+\mathbf{b}(t,x)\cdot v\mu+\mathbf{c}(t,x)(|v|^2-3)\mu,
	\end{aligned} 
\end{equation}
is identified as the macroscopic part of $f$, which is associated with the fluid equations.  
The corresponding macroscopic equation is
\begin{equation}\label{macroeq2}  
	\left\{ 
	\begin{aligned}
&\pa_t(\ma e^{\f12|x|^2})=-\mathscr{C}_l(t)(\mfS\na_x\cdot \mb)e^{\f12|x|^2}-\mathscr{C}_l(t)\mfR x\cdot\na_x(\ma e^{\f12|x|^2})+\sfT_{11};\\
&\pa_t(\mb e^{\f12|x|^2})=-\mathscr{C}_l(t)\mfS\na_x(\ma e^{\f12|x|^2})-2\mathscr{C}_l(t)(\mfS\na_x\mc)e^{\f12|x|^2}-\mathscr{C}_l(t)\mfR\mb e^{\f12|x|^2}-\mathscr{C}_l(t)(\mfR x\cdot\na_x)(\mb e^{\f12|x|^2})+\sfT_{21};\\
&\pa_t(\mc e^{\f12|x|^2})=-\f13\mathscr{C}_l(t) \mfS\na_x\cdot (\mb e^{\f12|x|^2})-\mathscr{C}_l(t)\mfR x\cdot\na_x (\mc e^{\f12|x|^2})+\sfT_{31};\\
&\pa_t(\mc e^{\f12|x|^2}) \mathbb{I}_{3\times 3}=-\mathscr{C}_l(t)\na^{sym}_\mfS(\mb e^{\f12|x|^2})-\mathscr{C}_l(t)\mfR x\cdot\na_x (\mc e^{\f12|x|^2})\mathbb{I}_{3\times 3}+\sfT_{41}+\pa_t \sfT_{42};\\
&\mathscr{C}_l(t)\mfS\na_x (\mc e^{\f12|x|^2})=\sfT_{51}+\pa_t\sfT_{52},
	\end{aligned}\right.
\end{equation}
where $\sfT_{11}$ and $\sfT_{31}$ are scalar functions, $\sfT_{21}$, $\sfT_{51}$, and $\sfT_{52}$ are vector functions, and $\sfT_{41}$ and $\sfT_{42}$ are matrix functions which only depend on microscopic part $(\mathbb{I-P})f$ and nonlinear term $\mathscr{C}_2Q(f,f)$. One may check the precise expression in (\ref{macroeq}-\ref{defofT2}) and the proof of \eqref{macroeq2} in Proposition \ref{abcequa}.

\subsubsection{Macroscopic estimates: zero mode, non-zero mode and variant Poincar\'e inequality.} Through the micro-macro decomposition, the solution $f$ to the actual model \eqref{pertubNSlandaucauchy} is segregated into the microscopic part $(\mathbb{I}-\mathbb{P})f$ and the macroscopic part $\mathbb{P}f$. As per Lemma \ref{coer}, only the microscopic segment mirrors the analogous scenario to the toy model \eqref{toynoneq}. Therefore, for attaining linear or nonlinear stability, controlling the macroscopic component $\mathbb{P}f$ becomes paramount.
To achieve this, our  insights unfold in the subsequent domains:
\smallskip

\noindent $\bullet$ Our initial insight involves partitioning $(\ma, \mb, \mc)$ into two facets: the zero mode $(\ma_0, \mb_0, \mc_0)$ and the non-zero mode $(\ma_{\mathsf{n}}, \mb_{\mathsf{n}}, \mc_{\mathsf{n}})$. We anticipate two outcomes: (1) The zero mode $(\ma_0, \mb_0, \mc_0)$ can be regulated by the non-zero mode; (2) We can attain complete dissipation of the non-zero mode by employing the Poincar\'e inequality since $(\ma_{\mathsf{n}}, \mb_{\mathsf{n}}, \mc_{\mathsf{n}})$ adheres to the condition:
\begin{equation}
	\label{nonzeromodecondition}
	\int_{\R^3} ((\ma_{\mathsf{n}}, \mb_{\mathsf{n}}, \mc_{\mathsf{n}})  e^{\frac 1 2|x|^2} )  e^{ - \frac 1 2 |x|^2} dx =0.
\end{equation}
  This elucidates why we express the macroscopic equation in terms of $((\ma, \mb, \mc)e^{\frac{1}{2}|x|^2})$ rather than $(\ma, \mb, \mc)$.

\noindent $\bullet$ To ensure that the zero-mode component can be assessed through estimations based on the non-zero mode, we extensively utilize the thirteen conservation laws listed in \eqref{conservationofNSlandau}. Although the computation is intricate, it is direct.

\noindent $\bullet$ To align the energy estimates with the microscopic part, it becomes imperative to estimate the non-zero mode $(\ma_{\mathsf{n}}, \mb_{\mathsf{n}}, \mc_{\mathsf{n}})$ in   Sobolev spaces with exponential weights. This necessitates the establishment of a modified Poincar\'e    inequality. Specifically, we adapt the standard Poincar\'e inequality to suit our unique context (refer to Lemma \ref{lemKPI}), i.e.,    
\[
\int_{\R^3} h^2  \langle x\rangle^{2}    e^{ - (1-\delta) |x|^2}    dx   \lesssim  \int_{\R^3}  |\nabla_x h |^2  e^{ - (1-\delta) |x|^2} dx,\quad\mbox{if}\,\, \int_{\R^3}  h e^{-\frac 1 2| x|^2} dx =0,\quad \delta \in (0, \f12). 
\]
In our case, if $h:=(\ma_{\mathsf{n}}, \mb_{\mathsf{n}}, \mc_{\mathsf{n}})e^{\frac{1}{2}|x|^2}$, then the left-hand side of the  inequality turns to be
\[\int_{\R^3}  |(\ma_{\mathsf{n}}, \mb_{\mathsf{n}}, \mc_{\mathsf{n}})|^2  \langle x\rangle^{2}    e^{\delta |x|^2}    dx. \]

These revelations pave the way for achieving complete dissipation of the macroscopic quantities $(\ma, \mb, \mc)$. Coupled with the ansatz outlined in Subsection \ref{ansatz} for the microscopic part, this progression enables us to attain linear stability, a pivotal stride for tackling the  nonlinear conundrum.

\subsubsection{Strategy of nonlinear stability} \label{nonlinear stability subsection} 
To establish nonlinear stability, we must address the significant degeneration induced by the Gaussian distribution $e^{-\frac{|x|^2}{2}}$ relative to the nonlinear term. Fortunately, with the newfound linear stability, we gain two critical tools: the complete dissipation effect denoted by $D(f)$, resulting from both the microscopic and macroscopic parts, and the decay estimate in a space with a low-index exponential weight. Our innovative approach can be summarized as follows:

\noindent$\bullet$ To counteract this degeneration effect, we utilize a time-phase splitting method that divides the time-phase domain into two regions: $\mathscr{D}_1:=\{(t,x)|\frac{|x|^2}{4}\geq q_0\ln(t+e)\}$ and $\mathscr{D}_2:=\{(t,x)|\frac{|x|^2}{4}\leq q_0\ln(t+e)\}$, where $q_0:=\frac{1-4\delta}{1+r}$ originates from the toy model \eqref{toymodelofNSlandau}.

\noindent$\bullet$ In the $\mathscr{D}_1$ region, we treat the equation as nearly vacuum-like. The additional weight induces polynomial decay, aiding in refining the energy estimate crucial for controlling the nonlinear term. Conversely, in the $\mathscr{D}_2$ region, we utilize the full dissipation $D(f)$ to bound the nonlinear term.

\noindent$\bullet$ Let us illustrate our strategy in the context of the toy model \eqref{toynoneq}. Following the notations in Subsection \ref{ansatz} and considering $W_a \sim e^{\frac{a}{2}(|x|^2+|v|^2)}$, we establish the existence of a universal constant $\lambda$ such that:
\beno
&&(-v \cdot \nabla_x f + x \cdot \nabla_v f  - e^{- \frac 1 2 |x|^2} f , W_a^2  f ) \\
&\le& -\lambda ( \Vert e^{(\frac a 2-\frac 1 4 ) |x|^2+\f a2 |v|^2}  f \Vert_{L^2_{x, v}}^2 +  \Vert e^{(\frac a 2-\frac 1 8  -\vartheta) (|x|^2 + |v|^2)}  f \Vert_{L^2_{x, v}}^2 ):= -\lambda D_{a} (f).
\eeno
We are now positioned to bound the nonlinear term. In region $\mathscr{D}_1$, we have $ \langle t \rangle^{q_0} \lesssim e^{\frac{|x|^2}{4}}$. Therefore
\[
 |(f^2, W_a^2  f  \mathbf{1}_{\mathscr{D}_1} ) | \lesssim      \Vert   f \mathbf{1}_{\mathscr{D}_1} \Vert_{L^\infty_{x, v}} \Vert W_a  f \Vert_{L^2_{x, v}}^2    \lesssim   \langle t \rangle^{ -2 q_0} ( \langle t \rangle^{q_0}\sum_{|\al|+|\be|\le 5} \Vert e^{\frac 1 4 |x|^2 } \pa^\alpha_\beta f  \Vert_{L^2_{x, v}}) \Vert W_a  f \Vert_{L^2_{x, v}}^2. 
\]
In region $\mathscr{D}_2$, , we have $e^{\frac {|x|^2} {4} } \lesssim \langle  t \rangle ^{q_0}$, leading to
\[
 |(f^2, W_a^2 f  \mathbf{1}_{\mathscr{D}_2} ) | \lesssim      \Vert e^{\frac 1 2 |x|^2 }  f  \mathbf{1}_{\mathscr{D}_2} \Vert_{L^\infty_{x, v}}    \Vert e^{( \frac a 2-\frac 1 4 ) |x|^2+\f a2 |v|^2}  f \Vert_{L^2_{x, v}}^2  \lesssim   ( \langle t \rangle^{q_0}\sum_{|\al|+|\be|\le 5} \Vert e^{\frac 1 4 |x|^2 }\lr{x}^5 \pa^\alpha_\beta f  \Vert_{L^2_{x, v}})  D_a(f).
\]
If $2q_0>1$, these derivations enable us to finalize the energy estimates by selecting $ \langle t \rangle^{q_0}\sum\limits_{|\al|+|\be|\le 5} \Vert e^{\frac 1 4 |x|^2 } \lr{x}^5\pa^\alpha_\beta f  \Vert_{L^2_{x, v}}^2$ and $\sum\limits_{|\al|+|\be|\le 5}  \Vert W_a  \pa^\alpha_\beta f \Vert_{L^2_{x, v}}^2$ as the energy functionals.

To extend this strategy to the full model \eqref{pertubNSlandaucauchy}, we must carefully balance the energy estimates from the microscopic and macroscopic parts, a task considerably more intricate than in \eqref{toynoneq}. For detailed insights, please refer to Section \ref{Section nonlinear stability}.

\subsection{Organization of the paper} The remainder of this paper is structured as follows:
In Section \ref{Section invariant}, we present the classification of the entropy-invariant solutions.
Section \ref{Section Landau} provides auxiliary lemmas concerning the Landau operator and the weight-transfer lemma.
Section \ref{Section macroscopic} is dedicated to proving a variant of the Poincar\'e and Poincar\'e-Korn inequalities and deriving estimates for the macroscopic quantities. The proof of our main theorems is completed in Section \ref{Section nonlinear stability}. Additional supplementary material is included in the Appendix.

\section{Proof of Theorem \ref{Classificationofeiit} and Theorem \ref{biinjectionofMF0}}\label{Section invariant}
This section is dedicated to proving Theorem \ref{Classificationofeiit} and Theorem \ref{biinjectionofMF0}. To achieve this, we break down the proof into several steps. The first step involves characterizing an entropy-invariant solution through thirteen conservation laws.

\begin{prop}\label{periodicmax}
	$M=M(t, x, v)\geq0$ is an entropy-invariant solution of equation \eqref{CauchyLandau} if and only if \ben &&\int_{\R^3_x\times\R^3_v}(|x|^2 + |v|^2 + 1)M(t, x, v)dxdv < +\infty, \quad\mbox{and}\nonumber\\
	 &&\ln M \in \mathrm{span}\{1, X(t), V(t), |X(t)|^2, X(t) \cdot V(t), |V(t)|^2, X(t) \wedge V(t)\}, \label{characterM}
	\een  
	where $(X(t), V(t))=(x\cos t - v\sin t, x\sin t + v\cos t)$.
\end{prop}
\begin{proof} We observe that $(\partial_t + v \cdot \nabla_x - x \cdot \nabla_v)(\ln M) = M^{-1}(\partial_t + v \cdot \nabla_x - x \cdot \nabla_v)M$, which implies that if $M$ satisfies \eqref{characterM}, then by \eqref{trjecforlandau}, we have $Q(M,M)=0$ and $(\partial_t + v \cdot \nabla_x - x \cdot \nabla_v)M=0$. According to Definition \ref{Defoentropy-invariant}, $M$ is an entropy-invariant solution of \eqref{CauchyLandau}.

	Conversely, let us  assume that $M$ is an entropy-invariant solution of \eqref{CauchyLandau}. We only need to prove \eqref{characterM}. Since $Q(M, M) = 0$, then $\ln M \in \mathrm{span}\{1, v, |v|^2\}$, which implies that
	\begin{equation}\label{lnM}
		\ln M = a(t, x) + v \cdot b(t, x) + |v|^2c(t, x).
	\end{equation}
	By \eqref{equationforeisol}, we deduce that
	\begin{equation*}
		\partial_ta + v \cdot \partial_tb + |v|^2\partial_tc + v \cdot \nabla a + (v \cdot \nabla)(v \cdot b) + v|v|^2 \cdot \nabla c - x \cdot b - 2x \cdot vc = 0,
	\end{equation*}
	which implies that
	\begin{equation}\label{lnM13Eqs}
		\begin{aligned}
			&\partial_ta - x \cdot b = 0;\quad
			\partial_tb_i + \partial_ia - 2x_ic = 0;\quad
			\partial_tc + \partial_ib_i = 0;\\
			&\partial_ib_j + \partial_jb_i = 0, i \neq j; \quad
			\partial_ic= 0.\\
		\end{aligned}
	\end{equation}
	These will imply that (1).  $c(t, x) = c(t)$; (2).  Applying $\partial_i$ to the forth equation in \eqref{lnM13Eqs} will yield that
	\begin{equation*}
		0 = \partial_i^2b_j + \partial_i\partial_jb_i = \partial_i^2b_j - \partial_j\partial_tc = \partial_i^2b_j.
	\end{equation*}
	(3). Taking $\partial_i$ in the third equation in \eqref{lnM13Eqs},  we have  $0 = \partial_i\partial_tc + \partial_i^2b_i = \partial_i^2b_i$, which mean that $\partial_i^2b_j = 0, \:\:\: \forall 1 \le i, j \le 3.$
	From this, we set
	\begin{equation*}
		b_i(t, x) = b_{i0}(t) + \sum_{j = 1}^3b_{ij}(t)x_j + \sum_{j < k}b_{ijk}(t)x_jx_k + b_{i4}(t)x_1x_2x_3.
	\end{equation*}
	Put the above form in the forth equation in \eqref{lnM13Eqs}, one may derive that
	\begin{equation*}
		b_{ij}(t) + b_{ji}(t) = 0, i \neq j, \:\:\: b_{ijk}(t) = 0, 1 \le i \le 3, j <k, \:\:\: b_{i4}(t) = 0, 1 \le i \le 3.
	\end{equation*}
	The third equation leads that $\partial_tc(t) + b_{ii}(t) = 0, 1 \le i \le 3$. In general, the form of $b(t, x)$ is
	\begin{equation}\label{lnMbform}
		b_i(t, x) = b_{i0}(t) + \sum_{j = 1}^3b_{ij}(t)x_j.
	\end{equation}
	With properties that
	\begin{equation}\label{lnMeqbc}
		b_{ij}(t) + b_{ji}(t) = 0, i \neq j; \:\:\: \partial_tc(t) + b_{ii}(t) = 0, 1 \le i \le 3.
	\end{equation}
	(4). Applying  $\partial_i$ to the second equation in \eqref{lnM13Eqs}, we get that $\partial_i^2a = 2c(t) - \partial_tb_{ii}(t)$. From \eqref{lnMeqbc} we have $\partial_1^2a = \partial_2^2a = \partial_3^2a$ which depends only on $t$. Besides, applying  $\partial_j$ to the second equation in \eqref{lnM13Eqs}, we get that $\partial_t\partial_jb_i + \partial_i\partial_ja = 0$. By symmetric property, it holds that $\partial_t\partial_ib_j + \partial_i\partial_ja = 0$. Again using \eqref{lnMeqbc} we have $\partial_i\partial_ja = 0$. Now we may write
	\begin{equation}\label{lnMaform}
		a(t, x) = a_0(t) + \sum_{j = 1}^3a_j(t)x_j + a_4(t)|x|^2.
	\end{equation}

	Using \eqref{lnMbform} and \eqref{lnMaform}, the first and second equations in \eqref{lnM13Eqs} turn to be 
	\begin{equation*}
		\begin{gathered} \partial_ta_0 + \sum_{i = 1}^3x_i\partial_ta_i + |x|^2\partial_ta_4 - \sum_{i = 1}^3x_i(b_{i0} + \sum_{j = 1}^3b_{ij}x_j) = 0;\\
			\partial_tb_{i0} + \sum_{j = 1}^3\partial_tb_{ij}x_j + a_i + 2x_ia_4 - 2x_ic = 0 
			. 
		\end{gathered}
	\end{equation*}
	By comparing the coefficients, we get that
	\begin{equation}\label{lnMeqs}
		\begin{gathered} \partial_ta_0 = 0(\Rightarrow a_0(t) \equiv a_0), \:\:\: \partial_ta_i - b_{i0}=0, \:\:\: \partial_ta_4 - b_{ii} = 0; \\
			a_4 = c - \frac{1}{2}\partial_tb_{ii}, \:\:\: \partial_tb_{ij} = 0(\Rightarrow b_{ij}(t) \equiv b_{ij}, i \neq j), \:\:\: \partial_tb_{i0} + a_i = 0. 
		\end{gathered}
	\end{equation}
	From these, we first derive that   $\partial_t^2b_{i0} + b_{i0} = 0$, which implies that
	\begin{equation*}
		b_{i0}(t)= \alpha_i\cos t + \beta_i\sin t, \:\:\: a_i(t)= \alpha_i\sin t - \beta_i\cos t.
	\end{equation*}
	Secondly, thanks to the second equation in \eqref{lnMeqbc}, we have
	\begin{equation*}
		a_4 = c - \frac{1}{2}\partial_tb_{ii}, \:\:\: \partial_ta_4 - b_{ii} = 0, \:\:\: \partial_tc + b_{ii}= 0,
	\end{equation*}
	which yield  that $\partial_t^3c + 4\partial_tc = 0$. From this, we deduce that 
	\begin{equation*}
		\begin{aligned}
			c(t)= c_0 + c_1\sin2t - c_2\cos2t;\,\, 
			a_4(t)= c_0 - c_1\sin2t + c_2\cos2t;\,\, 
			b_{ii}(t) = -2c_1\cos2t - 2c_2\sin2t. 
		\end{aligned}
	\end{equation*}
	
	Now  substituting all the estimates into \eqref{lnMaform}, \eqref{lnMbform} and $c(t,x)=c(t)$, we have
	\begin{equation*}
		\begin{aligned}
			a(t, x) &= a_0 + (\alpha\sin t - \beta\cos t) \cdot x + (c_0 - c_1\sin2t + c_2\cos2t)|x|^2;\\
			b(t, x) &= \alpha\cos t +\beta\sin t + \begin{pmatrix} -2c_1\cos2t - 2c_2\sin2t & b_{12} & b_{13} \\ -b_{12} & -2c_1\cos2t - 2c_2\sin2t & b_{23} \\ -b_{13} & -b_{23} & -2c_1\cos2t - 2c_2\sin2t \end{pmatrix} \begin{pmatrix} x_1 \\ x_2 \\ x_3 \end{pmatrix};\\
			c(t, x) &= c(t) = c_0 + c_1\sin2t - c_2\cos2t.
		\end{aligned}
	\end{equation*}
	
	By \eqref{lnM}, it is not difficult to check that
	\beno
	\ln M &=& a_0 - \beta \cdot X(t) + \alpha \cdot V(t) + b_{12}(x_2v_1 - x_1v_2) + b_{13}(x_3v_1 - x_1v_3) + b_{23}(x_3v_2 - x_2v_3) \\
	&&+ (c_0 + c_2)|X(t)|^2 - 2c_1X(t)\cdot V(t) + (c_0 - c_2)|V(t)|^2.
	\eeno
	Noting that for any $i\neq j$, $x_iv_j-x_jv_i=X_i(t)V_j(t)-X_j(t)V_i(t)$, we conclude the result \eqref{characterM} and then end the proof.
\end{proof}

Now we are in a position to prove Theorem \ref{Classificationofeiit}.

\begin{proof}[Proof of Theorem \ref{Classificationofeiit}] Thanks to Proposition \ref{periodicmax}, Theorem \ref{Classificationofeiit} can be reduced to proving that    \eqref{integcondi} and \eqref{characterM} is equivalent to \eqref{expliciformeisolu} with $(m,y,z)\in \mathbb{R}^+\times\mathbb{R}^3\times\mathbb{R}^3$, $(a,b,c,\mathsf{R})\in \mathcal{O}$(see \eqref{omegadomain}), and $\mathsf{Q}:=(ac-b^2)\mathbb{I}+\mathsf{R}^2$. Starting from \eqref{expliciformeisolu}, we can readily verify that \eqref{integcondi} and \eqref{characterM} hold. Thus, we only need to prove the inverse result.

 Suppose that $M$ is an entropy-invariant solution of \eqref{CauchyLandau} satisfying  \eqref{integcondi} and \eqref{characterM}. In the next, we will show that  $M$ really takes the form \eqref{expliciformeisolu} with constrains stated in the above. By \eqref{trjecforlandau}, we know that if $M_0(x,v):=M(0,x,v)$, then
 \beno M(t,x,v)=M_0(X(t),V(t)),\eeno where $(X(t), V(t))=(x\cos t - v\sin t, x\sin t + v\cos t)$.
Observing that $(X(t), V(t))$ is involved in \eqref{characterM} and \eqref{expliciformeisolu}, we reduce the desired result to proving that if  
\ben
\label{M0charcter} &&\ln M_0\in \mathrm{span}\{1,v,x,|v|^2, x \cdot v, |x|^2, v \wedge x\},\quad \mbox{and}\\
\label{M0integralcondition} &&\int_{\R^3_x\times \R^3_v} M_0(1+|x|^2+|v|^2)dxdv<\infty, \een then  there exist $(m,y,z)\in\R^+\times\R^3\times\R^3$ and $(a,b,c,\mathsf{R})\in \mathcal{O}$ such that
\ben\label{explicitformM0} M_0(x, v) = m\frac{\sqrt{\det \mathsf{Q}}}{(2\pi)^3}\exp\bigg\{-\frac{1}{2}\big[a|x-y|^2 + 2b(x-y) \cdot (v-z) + c|v-z|^2 + 2(x-y)^\tau \mathsf{R}(v-z)\big]\bigg\}. \een

$\bullet$ We first handle the case that $(\Psi_1(M_0),\Psi_2(M_0),\Psi_3(M_0))=(1,0,0)$. By \eqref{M0charcter}, we may assume that \ 
\ben\label{roughM0} M_0=  \exp\bigg\{-\frac{1}{2}\big[a|x|^2 + 2bx \cdot v + c|v|^2 + 2x^\tau \mathsf{R} v\big] + \mathsf{p} \cdot (x, v) + d\bigg\},\een 
where  $a, b, c, d \in \R, \mathsf{p}\in \R^6$ and $\mathsf{R}\in \mathbb{M}_3(\R)$ is skew-symmetric. Using the notations that
\begin{equation}\label{DefizP}
		\mathsf{z} := \begin{pmatrix} x \\ v \end{pmatrix} \in \mathbb{R}^6, \:\:\: \mathsf{P} := \begin{pmatrix} a\mathbb{I} & b\mathbb{I} + \mathsf{R} \\ b\mathbb{I} - \mathsf{R} & c\mathbb{I} \end{pmatrix},
	\end{equation}
	we have $a|x|^2 + 2bx \cdot v + c|v|^2 + 2x^\tau \mathsf{R}v = \mathsf{z}^\tau \mathsf{P}\mathsf{z}$. Since $(\Psi_1(M_0),\Psi_2(M_0),\Psi_3(M_0))=(1,0,0)$, we deduce that 
\beno
		&&\int_{\R^6} e^{-\frac{1}{2}\mathsf{z}^\tau \mathsf{P}\mathsf{z} + \mathsf{p} \cdot \mathsf{z} + d}d\mathsf{z} = 1, \:\:\: \int_{\R^6} \mathsf{z}e^{-\frac{1}{2} \mathsf{z}^\tau  \mathsf{P} \mathsf{z} +  \mathsf{p} \cdot  \mathsf{z} + d}d \mathsf{z} = 0, \quad\mbox{and}\\
			&& 0 = \int_{\mathsf{p} \cdot \mathsf{z} \ge 0}\mathsf{p} \cdot \mathsf{z}e^{-\frac{1}{2}\mathsf{z}^\tau \mathsf{P}\mathsf{z} + \mathsf{p} \cdot \mathsf{z} + d}d\mathsf{z} + \int_{\mathsf{p} \cdot \mathsf{z} \le 0}\mathsf{p} \cdot \mathsf{z}e^{-\frac{1}{2}\mathsf{z}^\tau \mathsf{P}\mathsf{z} + \mathsf{p} \cdot \mathsf{z} + d}d\mathsf{z}\\
			&&= \int_{\mathsf{p} \cdot \mathsf{z} \ge 0}\mathsf{p} \cdot \mathsf{z}e^{-\frac{1}{2}\mathsf{z}^\tau \mathsf{P}\mathsf{z} + d}(e^{\mathsf{p} \cdot \mathsf{z}} - e^{-\mathsf{p} \cdot \mathsf{z}})d\mathsf{z} \ge 0,\eeno
	which implies that $\mathsf{p} \cdot \mathsf{z} = 0, a.e. \, \mathsf{z} \in \mathbb{R}^6$. Thus $\mathsf{p}= 0$.
	
	By definition \eqref{DefizP}, $\mathsf{P}$ is a symmetric matrix. Then there exists a orthogonal matrix $\mathsf{E}$ and the diagonal matrix $\mathsf{D} := \mathrm{diag}\{\lambda_i\}_{i = 1}^6$ such that $\mathsf{P}= \mathsf{E}^\tau \mathsf{D}\mathsf{E}$. Suppose that $\{\lambda_i| 1 \le i \le 6\}$ are eigenvalues of $\mathsf{P}$, then
	\begin{equation*}
		1=\int_{\R^6} e^{-\frac{1}{2}\mathsf{z}^\tau \mathsf{P}\mathsf{z} + d}d\mathsf{z}  = e^d\prod_{i = 1}^6\int e^{-\frac{1}{2}\lambda_i\mathsf{z}_i^2}d\mathsf{z}_i.
	\end{equation*}
	This forces $\lambda_i > 0$ for $1 \le i \le 6$ and thus $\mathsf{P}$ is a positive definite matrix. Moreover, $e^d = \frac{\sqrt{\det \mathsf{P}}}{(2\pi)^3}$.

	Since $\mathsf{P}$ is positive definite, $\mathsf{z}^\tau \mathsf{P}\mathsf{z} = a|x|^2 + 2bx \cdot v + c|v|^2 + 2x^\tau \mathsf{R}v > 0$ for all $\mathsf{z} \neq 0$. Choosing $x = kv$ with $|v| = 1$ will lead to  $ak^2 + 2bk + c > 0$ for all $k \in \mathbb{R}$. Therefore $a, c > 0$ and $ac-b^2>0$. Moreover, it is not difficult to check that
	\ben\label{abcRdiagonal} 
	\begin{pmatrix} \mathbb{I} & 0  \\ -a^{-1}(b\mathbb{I}-\mathsf{R})  &  \mathbb{I} \end{pmatrix} \begin{pmatrix} a\mathbb{I} & b\mathbb{I} + \mathsf{R} \\ b\mathbb{I} - \mathsf{R} & c\mathbb{I} \end{pmatrix}\begin{pmatrix} \mathbb{I} & 0  \\ -a^{-1}(b\mathbb{I}-\mathsf{R})  &  \mathbb{I} \end{pmatrix}^\tau=\begin{pmatrix} a\mathbb{I} & 0  \\  0 & a^{-1}((ac-b^2)\mathbb{I}+\mathsf{R}^2) \end{pmatrix}.   \een 
    Thus  $\mathsf{Q}=(ac-b^2)\mathbb{I}+\mathsf{R}^2$ is also positive definite and $\det \mathsf{Q}=\det \mathsf{P}   = (e^d(2\pi)^3)^2$.   Patching together all the estimates,  by \eqref{roughM0}, we arrive at 
\ben\label{explicitformM0100}
		M_0(x, v) = \frac{\sqrt{\det \mathsf{Q}}}{(2\pi)^3}\exp\bigg\{-\frac{1}{2}\big[a|x|^2 + 2bx \cdot v + c|v|^2 + 2x^\tau \mathsf{R}v\big]\bigg\},
	\een 
	where $(a,b,c,\mathsf{R})\in\mathcal{O}$. This proves the claim \eqref{explicitformM0}.

	$\bullet$ For the general case, if $\Psi_1(M_0)=0$, then $M_0(x,v)=0$. For the case $\Psi_1(M_0)\neq0$, let $(\Psi_1(M_0),\Psi_2(M_0),$\\$\Psi_3(M_0))=m(1,y,z)$, by the change of coordinates, if 
	\beno \widetilde{M}_0(t, x, v) := m^{-1}M_0(t, x + y, v  + z ), \eeno 
    then $(\Psi_1(\widetilde{M}_0),\Psi_2(\widetilde{M}_0),\Psi_3(\widetilde{M}_0))=(1,0,0)$. Following the previous step, we easily conclude  \eqref{explicitformM0} and then complete the proof of the theorem.
\end{proof}

The following two propositions aim to elucidate additional properties of $\mathscr{I}$ which defined in \eqref{setofF0}.

\begin{prop}\label{properV100} Suppose that $\mathsf{V}\in\mathscr{I}$ with $\mathsf{V}=\Psi(F_0)=(1,0,0,\mathsf{a},\mathsf{b},\mathsf{c},\mathcal{R})$. Then $(\mathsf{a},\mathsf{b},\mathsf{c},\mathcal{R})$ satisfies the constrains that $\mathsf{a},\mathsf{c}>0$ and $\f12\mathrm{tr}(|\mathcal{R}|)<\sqrt{\mathsf{a}\mathsf{c}-\mathsf{b}^2}$.  
\end{prop}
\begin{proof}  
By \eqref{deofU}, \eqref{deofR} and \eqref{deofS}, the desired result can be reduced to prove that $(\mathsf{a}\mathsf{c}-\mathsf{b}^2)\mathbb{I}+\mathcal{R}^2$ is positive definite. Thanks to \eqref{abcRdiagonal}, it can be further reduced to prove 
 \beno 
\begin{pmatrix} \mathsf{a}\mathbb{I} & \mathsf{b}\mathbb{I} + \mathcal{R} \\ \mathsf{b}\mathbb{I} - \mathcal{R} & \mathsf{c}\mathbb{I} \end{pmatrix}=\int_{\R^6} \begin{pmatrix} |x|^2\mathbb{I} & (x\cdot v)\mathbb{I} + x\wedge v \\ (x\cdot v) \mathbb{I} -x\wedge v& |v|^2\mathbb{I} \end{pmatrix} F_0dxdv
 \eeno is positive definite. It is not difficult to check that 
  \ben\label{postivargu1}  &&(z_1^\tau,z_2^\tau)\begin{pmatrix} \mathsf{a}\mathbb{I} & \mathsf{b}\mathbb{I} + \mathcal{R} \\ \mathsf{b}\mathbb{I} - \mathcal{R} & \mathsf{c}\mathbb{I} \end{pmatrix}\begin{pmatrix} z_1\\z_2\end{pmatrix}=\int_{\R^6}|x|^2\bigg|z_1+\f{(x\cdot v)z_2+(x\wedge v)z_2}{|x|^2}\bigg|^2F_0dxdv\nonumber \\
 &&+ \int_{\R^6} |x|^{-2}z_2^\tau \big((|x|^2|v|^2-(x\cdot v)^2)\mathbb{I}+(x\wedge v)^2\big)z_2F_0dxdv.\een 
Thus the matrix is semi-positive definite since $(|x|^2|v|^2-(x\cdot v)^2)\mathbb{I}+(x\wedge v)^2$ is semi-positive definite.
Let $(z_1,z_2)\in \R^3\times\R^3$ and suppose that \ben\label{0} (z_1^\tau,z_2^\tau)\begin{pmatrix} \mathsf{a}\mathbb{I} & \mathsf{b}\mathbb{I} + \mathcal{R} \\ \mathsf{b}\mathbb{I} - \mathcal{R} & \mathsf{c}\mathbb{I} \end{pmatrix}\begin{pmatrix} z_1\\z_2\end{pmatrix}=0.\een
 Note that  $x$ and $v$ are the only two eigenvectors corresponding to the eigenvalue $0$ of the matrix $(|x|^2|v|^2-(x\cdot v)^2)\mathbb{I}+(x\wedge v)^2$.
 From the fact that $\Psi_1(F_0)=1$, we may assume that there exists a measurable set $E\subset \R^6$ such that $|E|>0$ and $F_0(x,v)>\eta>0$ for any $(x,v)\in E$. 
  Then \eqref{postivargu1} and \eqref{0} imply that for any $(x,v)\in E$, it holds that $(x\times v)\perp z_2$, and thus $(x\times z_2)\perp v$. Then we obtain that 
\beno 0<|E|<\int_{\R^3_x}\int_{(x\times z_2)\perp v} 1dvdx. \eeno
This yields that $z_2=0$. By \eqref{postivargu1} and the fact that $z_2=0$, we can also get that $z_1=0$. 
Thus we  conclude that if $(z_1^\tau,z_2^\tau)\begin{pmatrix} \mathsf{a}\mathbb{I} & \mathsf{b}\mathbb{I} + \mathcal{R} \\ \mathsf{b}\mathbb{I} - \mathcal{R} & \mathsf{c}\mathbb{I} \end{pmatrix}\begin{pmatrix} z_1\\z_2\end{pmatrix}=0$, then $z_1=z_2=0$. This implies that  $\begin{pmatrix} \mathsf{a}\mathbb{I} & \mathsf{b}\mathbb{I} + \mathcal{R} \\ \mathsf{b}\mathbb{I} - \mathcal{R} & \mathsf{c}\mathbb{I} \end{pmatrix}$ is positive definite and then we complete the proof.
\end{proof}

\begin{prop}\label{properM0100} Suppose $M_0$ has the form \eqref{explicitformM0100} with $(a,b,c,\mathsf{R})\in\mathcal{O}$ and $\mathsf{Q}=(ac-b^2)\mathbb{I}+\mathsf{R}^2$. Then
$\Psi(M_0)=(1,0,0,\mathrm{tr}(\mathsf{Q}^{-1})c,-\mathrm{tr}(\mathsf{Q}^{-1})b, \mathrm{tr}(\mathsf{Q}^{-1})a,-2\mathsf{Q}^{-1}\mathsf{R})$.
 \end{prop}
\begin{proof} Following the notations used in the proof of Theorem \ref{Classificationofeiit}, we may write $M_0=\f{\sqrt{\mathrm{det}(\mathsf{P})}}{(2\pi)^{3}}e^{-\f12z^\tau \mathsf{P}z}$ where $z:=(x,v)$ and $\mathsf{P}=\begin{pmatrix} a\mathbb{I} & b\mathbb{I} + \mathsf{R} \\ b\mathbb{I} - \mathsf{R} & c\mathbb{I} \end{pmatrix}$ is positive definite.  Let $\mathsf{P}=\mathsf{A}^2$ where $\mathsf{A}$ is also positive definite. Then one may check that  
 \beno 
\int_{\R^6} zz^\tau e^{-\f12z^\tau \mathsf{P}z}dz=(\mathrm{det}(\mathsf{A}))^{-1}\mathsf{A}^{-1}\bigg(\int_{\R^6} zz^\tau e^{-\f12|z|^2}dz\bigg)\mathsf{A}^{-1}=\f{(2\pi)^{3}}{\sqrt{\mathrm{det}(\mathsf{P})}}(\mathsf{A}^{-1})^2=\f{(2\pi)^{3}}{\sqrt{\mathrm{det}(\mathsf{P})}}\mathsf{P}^{-1},\eeno 
from which together with \eqref{abcRdiagonal}, we have
\beno 
\int_{\R^6} \begin{pmatrix} x\\v\end{pmatrix}(x^\tau,v^\tau)M_0dxdv=\mathsf{P}^{-1}=\begin{pmatrix} c\mathsf{Q}^{-1}& - b\mathsf{Q}^{-1}- \mathsf{Q}^{-1}\mathsf{R} \\ - b\mathsf{Q}^{-1}+\mathsf{Q}^{-1}\mathsf{R} & a\mathsf{Q}^{-1}  \end{pmatrix}.
\eeno
This is enough to conclude our results.
\end{proof} 

\begin{lem}\label{normalbijection} Let   \ben
 &&\mathscr{M}_0:=\bigg\{ M_0(x,v):=M(0,x,v)\big|M(t,x,v)\in \mathscr{M} \bigg\},  \mathscr{I}_{\mathsf{N}}:=\big\{\mathsf{V}\in\mathscr{I}\big|(\mathsf{V}_1,\mathsf{V}_2,\mathsf{V}_3)=(1,0,0)\big\},\label{SetofM0}
 \\&& \label{NormalizesetofM0O0}
 \mathscr{M}_{0,\mathsf{N}}:=\big\{M_0\in \mathscr{M}_0\big| (\Phi_1(M_0),\Phi_2(M_0),\Phi_3(M_0))=(1,0,0) \big\}.\label{NormalizesetofO0}
\een
There exists a bijection mapping $\mathscr{F}: \mathscr{I}_{\mathsf{N}}\rightarrow \mathscr{M}_{0,\mathsf{N}}$ 
such that  $\Phi(\mathscr{F}(\mathsf{V}))=\mathsf{V}$ for any $\mathsf{V}\in\mathscr{I}_{\mathsf{N}}$.  
\end{lem} 
\begin{proof} Suppose $\mathsf{V}=\Psi(F_0)=(1,0,0,\mathsf{a},\mathsf{b},\mathsf{c},\mathcal{R})\in  \mathscr{I}_{\mathsf{N}}$. To prove the existence of $M_0\in \mathscr{M}_{0,\mathsf{N}}$ such that $\Psi(M_0)=\mathsf{V}$, we will apply Levermore's argument used in \cite{Levermore1}. 

 Thanks to Proposition \ref{properV100}, if $\mathsf{d}:=\sqrt{\mathsf{a}\mathsf{c}-\mathsf{b}^2}$, then $\beta:=\f12\mathrm{tr}(|\mathcal{R}|)<\mathsf{d}$. Let $M_0$ has the form \eqref{explicitformM0100} with $(a,b,c,\mathsf{R})\in \mathcal{O}$ with $\mathsf{Q}=(ac-b^2)\mathbb{I}+\mathsf{R}^2$. By Proposition \ref{properM0100}, then $\Psi(M_0)=\mathsf{V}$ is reduced to show that 
\ben\label{equalcondition1} (\mathsf{a},\mathsf{b},\mathsf{c},\mathcal{R})=(\mathrm{tr}(\mathsf{Q}^{-1})c,-\mathrm{tr}(\mathsf{Q}^{-1})b, \mathrm{tr}(\mathsf{Q}^{-1})a,-2\mathsf{Q}^{-1}\mathsf{R}). \een  
Assume that \ben\label{abcrassume} c:=\f{\nu \mathsf{a}}{\mathsf{d}^2},\quad  b:=-\f{\nu \mathsf{b}}{\mathsf{d}^2}, \quad a:=\f{\nu \mathsf{c}}{\mathsf{d}^2}, \quad \mathsf{R}:=-\f{\nu}{\mathsf{d}}\mathcal{N}, \een
where $\nu>0$ and the skew-symmetric matrix $\mathcal{N}$ will be determined later.  Then $\mathsf{Q}=\f{\nu^2}{\mathsf{d}^2}(\mathbb{I}+\mathcal{N}^2)$ and \eqref{equalcondition1} is equivalent to 
\beno  1=\nu^{-1}\mathrm{tr}((\mathbb{I}+\mathcal{N}^2)^{-1}), \quad \mathcal{R}=2\f{\mathsf{d}}{\nu}(\mathbb{I}+\mathcal{N}^2)^{-1}\mathcal{N}.\eeno  From the second equation, we have
 \ben\label{formulaofN}\mathcal{N}=\nu\big(\mathsf{d}\mathbb{I}+(\mathsf{d}^2\mathbb{I}-\nu^2\mathcal{R}^2)^{\f12}\big)^{-1}\mathcal{R}.\een It is easy to see that $\mathcal{N}$ is skew-symmetric and  if $\alpha:=\f12\mathrm{tr}(|\mathcal{N}|)$, then 
  $\alpha=\f{\nu\beta}{\mathsf{d}+(\mathsf{d}^2+\nu^2\beta^2)^{\f12}}<1$, 
 which implies that  $\mathbb{I}+\mathcal{N}^2$ is positive definite and moreover,
 \beno \nu=\mathrm{tr}((\mathbb{I}+\mathcal{N}^2)^{-1})=1+\f{2}{1-\alpha^2}=2+(1+(\f{\beta}{\mathsf{d}})^2\nu^2)^{\f12}.\eeno
 From this, we get that 
 \ben\label{formulaofmu}  \nu=\f{2+(1+3(\beta/\mathsf{d})^2)^{\f12}}{1-(\beta/\mathsf{d})^2}.\een 
 Now substituting \eqref{formulaofN} and \eqref{formulaofmu} into \eqref{abcrassume}, we obtain that $\Psi(M_0)=\mathsf{V}$. This proves the existence.

 To justify the mapping from $\mathsf{V}$ to $M_0$, we have to show the uniqueness in the sense that if  $M_{0,1}$ and $M_{0,2}$,  having the same form \eqref{explicitformM0100} but with $(a_1,b_1,c_1,\mathsf{R}_1)\in \eqref{omegadomain}$ and $(a_2,b_2,c_2,\mathsf{R}_2)\in \eqref{omegadomain}$ respectively, satisfy that $\Psi(M_{0,1})=\Psi(M_{0,2})$, then $M_{0,1}=M_{0,2}$.  

 Let $d_i:=a_ic_i-b_i^2$, $\mathsf{Q}_i=d_i\mathbb{I}+\mathsf{R}_i^2$, $\gamma_i:=\f12\mathrm{tr}{|\mathsf{R}_i|}$ and $0<e_i:=d_i-\gamma_i^2$. Then by Proposition \ref{properM0100}, we have
 \ben\label{uniq1}  (\mathrm{tr}(\mathsf{Q}_1^{-1})c_1,-\mathrm{tr}(\mathsf{Q}_1^{-1})b_1, \mathrm{tr}(\mathsf{Q}_1^{-1})a_1,-2\mathsf{Q}_1^{-1}\mathsf{R}_1)= (\mathrm{tr}(\mathsf{Q}_2^{-1})c_2,-\mathrm{tr}(\mathsf{Q}_2^{-1})b_2, \mathrm{tr}(\mathsf{Q}_2^{-1})a_2,-2\mathsf{Q}_2^{-1}\mathsf{R}_2).\een
Since $\mathrm{tr}(\mathsf{Q}_i^{-1})=d_i^{-1}+2e_i^{-1}$, if assume that $k(a_1,b_1,c_1)=(a_2,b_2,c_2)$ with $k>0$, then from \eqref{uniq1}, we get that $(d_1^{-1}+2e_1^{-1})(a_1,b_1,c_1)=(d_2^{-1}+2e_2^{-1})(a_2,b_2,c_2)$ and $d_2=k^2d_1$,
which implies \ben\label{uniq2} 1/d_1+2/e_1=1/(kd_1)+(2k)/e_2.\een
Since $\mathsf{Q}_1^{-1}\mathsf{R}_1=\mathsf{Q}_2^{-1}\mathsf{R}_2$ and $\mathsf{Q}_i\mathsf{R}_i=\mathsf{R}_i\mathsf{Q}_i$, we also have $\mathsf{Q}_1^{-1}-d_1\mathsf{Q}_1^{-2}=\mathsf{Q}_1^{-2}\mathsf{R}_1^2=\mathsf{Q}_2^{-2}\mathsf{R}_2^2=\mathsf{Q}_2^{-1}-d_2\mathsf{Q}_2^{-2}$. Since $\mathrm{tr}(\mathsf{Q}_i^{-1}-d_i\mathsf{Q}_i^{-2})=2/e_i-(2d_i)/e_i^2$, then
\ben\label{uniq3} 2/e_1-(2d_1)/e_1^2= 2/e_2-(2d_2)/e_2^2.\een 
From \eqref{uniq2}, we derive that $e_2=(2k^2d_1e_1)/((k-1)e_1+2kd_1)$. By substituting it and $d_2=k^2d_1$ into \eqref{uniq3}, we finally get that
$(k-1)((k-3)e_1+8kd_1)e_1=0$. We will show that $k=1$ by contradiction argument. Suppose that  $(k-3)e_1+8kd_1=0$, which implies that $k=3e_1/(8d_1+e_1)$. Then we get that 
\beno 0<e_2=(2k^2d_1e_1)/((k-1)e_1+2kd_1)=\f{d_1(3e_1)^2}{(e_1-d_1)(8d_1+e_1)}<0. \eeno
This proves the claim and then $(a_1,b_1,c_1)=(a_2,b_2,c_2)$. From \eqref{uniq2}, we deduce that  $e_1=e_2$, which implies that $\mathrm{tr}(|\mathsf{R}_1|)=\mathrm{tr}(|\mathsf{R}_2|)$. From this together with \eqref{deofU}, there exist unitary matrices $U_1$ and $U_2$ such that 
\beno 
\mathsf{R_1}=U_1^\tau \begin{pmatrix} 0&\gamma_1&0 \\ -\gamma_1&0&0\\ 0&0&0 \end{pmatrix}U_1; \mathsf{R_2}=U_2^\tau \begin{pmatrix} 0&\gamma_1&0 \\ -\gamma_1&0&0\\ 0&0&0 \end{pmatrix}U_2,
\eeno
which yields  that 
\beno 
\mathsf{Q_1}=U_1^\tau \begin{pmatrix} e_1&0&0 \\ 0&e_1&0\\ 0&0&1\end{pmatrix}U_1; \mathsf{Q_2}=U_2^\tau \begin{pmatrix} e_1&0&0 \\ 0&e_1&0\\ 0&0&1 \end{pmatrix}U_2.
\eeno
From the fact that $\mathsf{Q}_1^{-1}\mathsf{R}_1=\mathsf{Q}_2^{-1}\mathsf{R}_2$, we get that 
\beno U_1^\tau \begin{pmatrix} 0&\gamma_1e_1^{-1}&0 \\ -\gamma_1e_1^{-1}&0&0\\ 0&0&0 \end{pmatrix}U_1=U_2^\tau \begin{pmatrix} 0&\gamma_1e_1^{-1}&0 \\ -\gamma_1e_1^{-1}&0&0\\ 0&0&0 \end{pmatrix}U_2. \eeno  
Since $e_1\neq0$, we conclude that $\mathsf{R_1}=\mathsf{R_2}$. This ends the proof of the uniqueness.

We conclude that for any $\mathsf{V}\in \mathscr{I}_{\mathsf{N}}$, there exists a unique $M_0\in \mathscr{M}_{0,\mathsf{N}}$ such that $\Psi(M_0)=\mathsf{V}$. Now we can define a mapping $\mathscr{F}: \mathscr{I}_{\mathsf{N}}\rightarrow \mathscr{M}_{0,\mathsf{N}}$ such that  $\mathscr{F}(\mathsf{V})=M_0$. It is easy to verify that $\mathscr{F}$ is a bijection. Moreover, from the construction in \eqref{equalcondition1}, \eqref{formulaofN} and \eqref{formulaofmu}, it is easy to check the following stability result: Suppose that  $M_{0,i}=\mathscr{F}(\mathsf{V}_i)$, where $\mathsf{V}_i=(1,0,0,\mathsf{a}_i,\mathsf{b}_i,\mathsf{c}_i, \mathcal{R}_i)$ and $M_{0,i}$ is determined by $(a_i,b_i,c_i,\mathsf{R}_i)$, for $i=1,2$, then by \eqref{normofpsif}, if $\|(\mathsf{a}_1-\mathsf{a}_2,\mathsf{b}_1-\mathsf{b}_2,\mathsf{c}_1-\mathsf{c}_2, \mathcal{R}_1-\mathcal{R}_2)\|_2\ll1$,
\ben\label{stabbijection1}
 \|(a_1-a_2, b_1-b_2,c_1-c_2,\mathsf{R}_1-\mathsf{R}_2)\|_2\le C(\mathsf{a}_1,\mathsf{b}_1,\mathsf{c}_1, \mathcal{R}_1)\|(\mathsf{a}_1-\mathsf{a}_2,\mathsf{b}_1-\mathsf{b}_2,\mathsf{c}_1-\mathsf{c}_2, \mathcal{R}_1-\mathcal{R}_2)\|_2.
 \een
\ben\label{stabbijection2}
 \f{M_{0,1}}{M_{0,2}}+\f{M_{0,2}}{M_{0,1}}\le \exp\bigg\{C(\mathsf{a}_1,\mathsf{b}_1,\mathsf{c}_1, \mathcal{R}_1)\|(\mathsf{a}_1-\mathsf{a}_2,\mathsf{b}_1-\mathsf{b}_2,\mathsf{c}_1-\mathsf{c}_2, \mathcal{R}_1-\mathcal{R}_2)\|_2(|x|^2+|v|^2+1)\bigg\}.\een
We end the proof of the lemma. 
\end{proof}

Now we are in a position to prove Theorem \ref{biinjectionofMF0}.
 
 \begin{proof}[Proof of Theorem \ref{biinjectionofMF0}] Before providing the proof for the theorem, we will first introduce two important mappings and then demonstrate the favorable properties of these two mappings.

  \noindent $\bullet$ Let $\mathscr{F}_1:\mathscr{M}_0\rightarrow \mathscr{M}$
 verify that $\mathscr{F}_1(M_0):=M_0(X(t),V(t))$, where $\mathscr{M}_0$ is defined in \eqref{SetofM0}. One may verify that (i). For any $M\in \mathscr{M}$, $\mathscr{F}_1^{-1}(M)=M(0,x,v)$. (ii). $\Psi(\mathscr{F}_1^{-1}(M))=\Psi(M)$. 

 \noindent $\bullet$ Suppose that $\mathsf{V}=\Psi(F_0)=(m,my,mz,\mathsf{a},\mathsf{b},\mathsf{c},\mathcal{R})\in \mathscr{I}$ and let
\beno \mathscr{D}_0:=\big\{\bar{F}_0(x,v)\ge0| 0<\int_{\R^6} \bar{F}_0(1+|x|^2+|v|^2)dxdv<\infty\big\}.\eeno 
 We define  a mapping  $\mathscr{F}_2:\mathscr{D}_0\rightarrow \mathscr{D}_0$, which verifies that
 $\mathscr{F}_2(\bar{F}_0):=m^{-1}\bar{F}_0(x+y,v+z)$. One may verify that (iii).   $\mathscr{F}_2^{-1}(\bar{F}_0)=m\bar{F}_0(x-y,v-z)$ and $\Psi(\mathscr{F}_2(F_0))\in \mathscr{I}_{\mathsf{N}}$(see \eqref{SetofM0} for definition). (iv). Thanks to Lemma \ref{normalbijection}, if $M_{0,\mathsf{N}}:=\mathscr{F}(\Psi(\mathscr{F}_2(F_0)))$, then $\Psi(M_{0,\mathsf{N}})=\Psi(\mathscr{F}_2(F_0))$. (v). Let $M_0=M_0(x,v):=\mathscr{F}_2^{-1}(M_{0,\mathsf{N}})$. Then it is easy to verify that $\Psi(M_0)=\Psi(F_0)=\mathsf{V}$. In fact, by change of variable, we have
 \beno &&\int_{\R^6} M_0(x,v)\begin{bmatrix}  |v|^2 \\ x \cdot v \\ |x|^2 \\ x\wedge v \end{bmatrix}dxdv=m\int_{\R^6} M_{0,\mathsf{N}}(x,v)\begin{bmatrix}  |v+z|^2 \\ (x+y) \cdot (v+z)  \\ |x+y|^2 \\ (x+y)\wedge (v+z) \end{bmatrix}dxdv\\
\mbox{and} &&\int_{\R^6} F_0(x,v)\begin{bmatrix}  |v|^2 \\ x \cdot v \\ |x|^2 \\ x\wedge v \end{bmatrix}dxdv=m\int_{\R^6} \mathscr{F}_2(F_0)(x,v)\begin{bmatrix}  |v+z|^2 \\ (x+y) \cdot (v+z) \\ |x+y|^2 \\ (x+y)\wedge (v+z) \end{bmatrix}dxdv, 
 \eeno 
 which implies that $\Psi(M_0)=\Psi(F_0)=\mathsf{V}$. We emphasize here that $M_0=(\mathscr{F}_2^{-1}\circ\mathscr{F}\circ\Psi\circ\mathscr{F}_2)(F_0)$.

 To prove the main theorem, we first observe that if $M:=(\mathscr{F}_1\circ\mathscr{F}_2^{-1}\circ\mathscr{F}\circ\Psi\circ\mathscr{F}_2)(F_0)=\mathscr{F}_1(M_0)\in \mathscr{M}$, then  $\Psi(M)=\mathsf{V}$ thanks to point (ii) and point (v). 
Suppose that there exists another $\widetilde{M}\in \mathscr{M}$ such that $\Psi(\widetilde{M})=\mathsf{V}=\Psi(F_0)$. Then one may easily prove that $\Psi((\mathscr{F}_2\circ\mathscr{F}_1^{-1})(\widetilde{M}))=\Psi(\mathscr{F}_2(F_0))\in \mathscr{I}_{\mathsf{N}}$ since
\[ \Psi(\mathscr{F}_2(F_0))=m^{-1}\int_{\R^6} F_0(x,v)(1,x-y,v-z,|x-y|^2, (x-y)\cdot (v-z),|v-z|^2, (x-y) \wedge (v-z) )^\tau dxdv. \] 
Thus by Lemma \ref{normalbijection} and point (iv), we have  $(\mathscr{F}_2\circ\mathscr{F}_1^{-1})(\widetilde{M})=M_{0,\mathsf{N}}=(\mathscr{F}\circ\Psi\circ\mathscr{F}_2)(F_0)$. This implies that $\widetilde{M}=(\mathscr{F}_1\circ\mathscr{F}_2^{-1}\circ\mathscr{F}\circ\Psi\circ\mathscr{F}_2)(F_0)=M$. 

 Now we can define a mapping $\widetilde{\mathscr{F}}:\mathscr{I}\rightarrow \mathscr{M}$ such that $\widetilde{\mathscr{F}}(\mathsf{V})=\widetilde{\mathscr{F}}(\Psi(F_0))=(\mathscr{F}_1\circ\mathscr{F}_2^{-1}\circ\mathscr{F}\circ\Psi\circ\mathscr{F}_2)(F_0)\in \mathscr{M}$ if $\mathsf{V}=\Psi(F_0)\in \mathscr{I}$. Then $\widetilde{\mathscr{F}}$ is  bijection. And the stability results \eqref{stabbijection3} and \eqref{stabbijection4} follow from \eqref{psi123}, \eqref{stabbijection1} and \eqref{stabbijection2} by observing that $M_{0,\mathsf{N}}$ is the profile of $M$. We complete the proof.
\end{proof}

\section{Landau collision operator and the weight-transfer lemma}\label{Section Landau}
In this section, we will derive various estimates for the Landau collision operator with exponential weights, including upper and lower bounds, and the weight-transfer lemma,  which will be frequently utilized later on.  We commence with the coercivity estimates:

\begin{lem}\label{coer}
	For any $a\in (\f12,1)$, there exist constants $\lam_a>0$ and $C_a$ such that
	\begin{equation}\label{coer1}
	(Q(\mu,f),fe^{a|v|^2})_{L^2_v}\leq -\lam_a\big(\|fe^{\f a 2|v|^2}\|^2_{H^1_{-\f32}}+\|\mathbb{T}_{\S^2}fe^{\f a 2|v|^2}\|^2_{L^2_{-\f32}}+\|fe^{\f a 2|v|^2}\|^2_{L^2_{-\f12}}\big)+C_a\|f\|^2_{L^2},
	\end{equation}
where $\mathbb{T}_{\S^2}$ is defined by: $\mathbb{T}_{\S^2}f:=v\times\na_v f$. When $a=\f12$, we have 
	\begin{equation}\label{i-p}
	(Lf,f\mu^{-1})_{L^2_v}\leq -\lam\big(\|\mu^{-1/2}(\mathbb{I}-\mathbb{P})f\|^2_{H^1_{-\f32}}+\| \mu^{-1/2} \mathbb{T}_{\S^2}(\mathbb{I}-\mathbb{P})f\|^2_{L^2_{-\f32}}+\|\mu^{-1/2}(\mathbb{I}-\mathbb{P})f\|^2_{L^2_{-\f12}}\big),
  \end{equation}
	where $L:=Q(\mu,\cdot)+Q(\cdot,\mu)$ and $\mathbb{P}$ is defined in \eqref{deofP}.
 
\end{lem}
\begin{proof}
	We first prove \eqref{i-p} by  the linearized theory. Indeed, by \cite{PM}, if $L_\mu g:=\mu^{-\f12}(Q(\mu,\mu^{\f12}g)+Q(\mu^{\f12}g,\mu))$,   we have
	\ben\label{ip}
	(L_\mu g,g)_{L^2_v}\leq -\lam \big(\|(\mathbb{I}-\mP_\mu)g\|^2_{H^1_{-3/2}}+\|\mathbb{T}_{\S^2}(\mathbb{I}-\mP_\mu)g\|^2_{L^2_{-3/2}}+\|(\mathbb{I}-\mP_\mu)g\|^2_{L^2_{-1/2}}\big),
	\een
	where $\lambda>0$ and $\mP_\mu$ is defined by 
	\begin{equation*}
	\mP_{\mu} g=\Big(\int_{\R^3}g\mu^{\f12}dv\Big)\mu^{\f12}+\sum_{i=1}^3\Big(\int_{\R^3}v_i\mu^{\f12}gdv\Big)v_i\mu^{\f12}+\Big(\int_{\R^3}\f{|v|^2-3}6\mu^{\f12}gdv\Big)(|v|^2-3)\mu^{\f12}.
	\end{equation*}
 It is easy to see that \eqref{ip} implies \eqref{i-p} by choosing $g:=\mu^{-\f12}f$.
 
 For \eqref{coer1}, by definition, we have 
 \begin{equation*}
 (Q(\mu,f),fe^{a|v|^2})_{L^2_v}=\int_{\R^3} (a_{ij}*\mu)\pa_{ij} f fe^{a|v|^2}dv+8\pi \int_{\R^3} \mu f^2 e^{a|v|^2}dv.
 \end{equation*}
 Let $g=fe^{\f a 2|v|^2}$,  integration by parts yields that
 \begin{equation*}
 \begin{aligned}
 &(Q(\mu,f),fe^{a|v|^2})_{L^2_v}=-\sum_{i,j=1}^3\int_{\R^3}(a_{ij}*\mu)\pa_i g\pa_j gdv +\Big(8\pi \int_{\R^3} \mu f^2 e^{a|v|^2}dv-\sum_{j=1}^3\int_{\R^6} (b_j*\mu) g \pa_j gdv\Big)\\
 &+\sum_{i,j=1}^3\int_{\R^3}\Big((a_{ij}*\mu)(\pa_i e^{\f a 2|v|^2})(\pa_j e^{\f a 2|v|^2})e^{-a|v|^2}+\sum_{j=1}^3(b_j*\mu)(\pa_j e^{\f a 2|v|^2})e^{-\f a 2|v|^2}\Big)g^2dv
 =\mathcal{T}_1+\mathcal{T}_2+\mathcal{T}_3.
\end{aligned}
 \end{equation*}
 
 By Proposition 2.3 and Proposition 2.4 in \cite{PM}, the matrix $a*\mu$ has a simple eigenvalue $l_1(v)\sim 2\<v\>^{-3}$ associated with the the eigenvector $v$ and a double eigenvalue $l_2(v)\sim \<v\>^{-1}$ associated with eigenspace $v^\perp$, where $l_1(v),l_2(v)$ are defined as follows:
 \begin{equation*}
	\begin{aligned}
 l_1(v)&=\int_{\R^3} \Big(1-\big(\f v{|v|}\cdot\f {v_*}{|v_*|}\big)\Big)|v_*|^{-1}\mu(v-v_*)dv_*;\\
l_2(v)&=\int_{\R^3} \Big(1-\f12\big(\f v{|v|}\times\f {v_*}{|v_*|}\big)\Big)|v_*|^{-1}\mu(v-v_*)dv_*.
\end{aligned}
\end{equation*}
Moreover, one can easily derive from above that
\begin{equation}\label{abv}
\sum_{i,j}a_{ij}*\mu v_iv_j=|v|^2l_1(v),\quad \sum_{j}b_j*\mu v_j=-l_1(v) |v|^2.
\end{equation}

Thus for $\mathcal{T}_1$, if the projection operator $P_v$ defined by $P_v\xi:=(\xi\cdot \f v{|v|})\f v{|v|}$, then there exists a constant $k>0$ such that
\begin{equation*}
\mathcal{T}_1\leq -k(\|\<v\>^{-\f32}P_v\na_v g\|^2_{L^2_v}+\|\<v\>^{-\f12}(I-P_v)\na_vg\|^2_{L^2_v})
\le-k/4(\|\na g\|^2_{L^2_{-\f 32}}+\|\mathbb{T}_{\S^2}g\|^2_{L^2_{-\f32}}).
\end{equation*}
 For $\mathcal{T}_2$ and $\mathcal{T}_3$, thanks to integration by parts, we have 
\begin{equation*}
\mathcal{T}_2=8\pi \int_{\R^3} \mu f^2 e^{a|v|^2}dv+\f12 \int_{\R^3} c*\mu g^2 dv=4\pi \int_{\R^3}\mu g^2dv.
\end{equation*}
By \eqref{abv} and the fact that $a\in (\f12,1)$, one may check that
\begin{equation*}
	\begin{aligned}
\mathcal{T}_2+\mathcal{T}_3&=\int_{\R^3}\big(\sum_{i,j=1}^3a^2(a_{ij}*\mu)v_iv_j+\sum_{j=1}^3a(b_j*\mu)v_j+4\pi\mu  \big) g^2dv\\
&= \int_{\R^3} (a^2|v|^2l_1(v)-al_1(v)|v|^2+4\pi\mu)g^2dv\leq -(a-a^2)\int_{\R^3} \<v\>^{-1}g^2dv+C_a\int_{\R^3} \<v\>^{-3}g^2dv. 
	\end{aligned}
\end{equation*}

Now patching together all the  estimates, we derive  the desired result
by using the fact that $\|g\|^2_{L^2_{-\f32}}\leq \vep \|g\|^2_{L^2_{-\f12}}+C_\vep \|e^{-\f a2 |v|^2}g\|^2_{L^2_v}$. We complete the proof of the lemma.
\end{proof}

\begin{lem}\label{Fff} Let  $0\leq m^2(x,v):=\mathcal{N}e^{a (|x|^2+|v|^2)}+e^{b(|x|^2+|v|^2)}Sx\cdot v$, where $a\in (\f12,1),b\in(0,a)$, $\mathcal{N}>0$ and $S\in\mathbb{M}_3(\R)$.  Then for  $F\geq 0$,  we have 
	\beno
	(Q(F,f),fm^2)_{L^2_{x,v}}\leq C_a\|\mathcal{P}_v^2 F\|_{H^4_{x,v}}\|fm\|^2_{L^2_{x,v}}.
	\eeno
\end{lem}
\begin{proof}
	It is easy to see that 
\begin{equation*}
	(Q(F,f),fm^2)_{L^2_{x,v}}=\sum_{i,j=1}^3\int_{\R^6}(a_{ij}*F)\pa_{ij}f fm^2dvdx+8\pi\int_{\R^6}Ff^2m^2dvdx.
\end{equation*}
	We only need to handle  the first term in the r.h.s. since the second term can be bounded as
	\beno
	8\pi\int_{\R^6}Ff^2m^2dvdx\ls \|F\|_{H^4_{x,v}}\|fm\|^2_{L^2_{x,v}}.
	\eeno
 Let $g=mf$ and this implies that $\pa_{ij}ffm^2=\pa_{ij}(m^{-1}g)gm$. By integration by parts, we get that
\beno 
	&&\int_{\R^6}(a_{ij}*F)\pa_{ij}f fm^2dvdx=-\int_{\R^6}\Big((b_j*F)gm+(a_{ij}*F)\pa_i gm+(a_{ij}*F)g\pa_im\Big)\pa_j(m^{-1}g)dvdx\\
&&=-\int_{\R^6}(a_{ij}*F)\pa_i g\pa_j gdvdx -\int_{\R^6} (b_j*F) g \pa_j gdvdx +\int_{\R^6}\Big((a_{ij}*F)\f{\pa_i m}m\f{\pa_jm}m+(b_j*F)\f{\pa_j m}m\Big)g^2dvdx,
\eeno
where we use the fact that $\pa_j(m^{-1}g)=m^{-1}\pa_j g-m^{-2}\pa_jmg$. Let us denote the four terms in r.h.s by  $\mathcal{I}_1,\mathcal{I}_2,\mathcal{I}_3,\mathcal{I}_4$. Obviously $\mathcal{I}_1$  is negative since $a_{ij}*F$ is positive. For $\mathcal{I}_2$, integration by parts implies that 
\begin{equation*}
	\sum_{j=1}^3\int_{\R^6} (b_j*F) g \pa_j gdvdx=4\pi\int_{\R^6} F g^2dvdx\ls \|F\|_{H^4_{x,v}}\|g\|^2_{L^2_{x,v}}.
\end{equation*}
 For $\mathcal{I}_3$, we observe that $m=\sqrt{\mathcal{N}}e^{\f a2 (|x|^2+|v|^2)}+(m-\sqrt{\mathcal{N}}e^{\f a2 (|x|^2+|v|^2)})$ and thus
\begin{equation*}
    \f{\pa_i m}m=\pa_i(e^{\f a 2|v|^2})e^{-\f a2|v|^2}+\tilde{m}_i
\end{equation*}
    with $|\tilde{m}_i|\ls 1$. For the reminder term $\tilde{m}_i$, it is easy to see that
\begin{align*}
    \sum_{i,j=1}^3\int_{\R^6}(a_{ij}*F)(\tilde{m}_i\tilde{m}_j+a\tilde{m}_iv_j+a\tilde{m}_jv_i) g^2dvdx
    \ls \|\mathcal{P}_vF\|_{H^4_{x,v}}\|g\|^2_{L^2_{x,v}},
\end{align*}
where we use the fact that $\sum_{i=1}^3a_{ij}(v-v_*)(v_i-(v_*)_i)=0$. For the main term $\pa_i(e^{\f a 2|v|^2})e^{-\f a2|v|^2}$,  we have
\beno
	 &&\sum_{i,j=1}^3\int_{\R^6}\Big((a_{ij}*F)\pa_i(e^{\f a 2|v|^2})\pa_j(e^{\f a 2|v|^2})e^{- a |v|^2} g^2dvdx\\
	 &=&\sum_{i,j=1}^3a^2\int_{\R^6}\Big((a_{ij}*F)v_iv_j g^2dvdx
	 =\sum_{i,j=1}^3a^2\int_{\R^6}\Big((a_{ij}*F)(v_*)_i(v_*)_j g^2dvdx.
\eeno	
	It can be bounded by 
	$\int_{\R^9}|v-v_*|^{-1}\<v_*\>^2F_*g^2 dvdv_*dx\ls \|\mathcal{P}_v^2 F\|_{H^4_{x,v}}\|g\|^2_{L^2_{x,v}}$.

	For $\mathcal{I}_4$, similar argument can be applied to get that $\mathcal{I}_4\ls\|\mathcal{P}_v^2F\|_{H^4_{x,v}}\|g\|^2_{L^2_{x,v}}$.  We prove the desired result by putting together all the estimates and then complete the proof of this lemma.
\end{proof}

\begin{lem}\label{ghf}
	  For functions $g,h$ and $f$, it holds that 
\begin{equation*}
	\begin{aligned}
	&|(Q(g, h),f)_{L^2_v}| 
	\ls\min\Big\{ \|g\|_{L^2_5}\|h \|_{H^{a_3}_{\om_3}}\|f \|_{H^{a_4}_{\om_4}},\, \|g\|_{L^2_5}\big(\|\mathbb{T}_{\S^2}h \|_{L^2_{-\f32}}+\|h \|_{H^1_{-\f 32}}\big)\\&\times\big(\|\mathbb{T}_{\S^2}f \|_{L^2_{-\f32}}+\|f\|_{H^1_{-\f 32}}\big) +\|g\|_{L^2_5}\|h \|_{H^{a_1}_{\om_1}}\|f \|_{H^{a_2}_{\om_2}}\Big\},
	\end{aligned}
\end{equation*}
where $a_i\geq0,i=1,\cdots,4$ satisfy that $a_1+a_2=1,a_3+a_4=2$ and $\om_1+\om_2=-2,\om_3+\om_4=-1$.
\end{lem}
\begin{proof} One may prove it by following the proof of  Theorem 1.5 in \cite{HL} and we  omit   the details here.
\end{proof}

\begin{lem}\label{expQ}
	Let $\aa\in(0,\f12)$. For function $g,h$ and $f$, it holds that
\begin{equation*}
	\begin{aligned}
	 |( e^{\aa|v|^2}Q(g,h)-Q(g, e^{\aa|v|^2}h),f)_{L^2_v}|\ls&\min\Big\{\|g\|_{L^2_6}\|e^{\aa|v|^2}h\|_{L^2}\|f\|_{L^2}+\|g\|_{L^2_6}\|e^{\aa|v|^2}\na_v h\|_{L^2}\|f\|_{L^2},\\
	&\|g\|_{L^2_6}\|e^{\aa|v|^2}h\|_{L^2_{-\f12}}\big(\|\mathbb{T}_{\S^2}f\|_{L^2_{-\f32}}+\|f\|_{H^1_{-\f 32}}+\|f\|_{L^2_{-\f12}}\big)\Big\}.
	\end{aligned}
\end{equation*}
\end{lem}
\begin{proof} We split the proof into two steps. 

$\bullet$ \underline{Step 1: The second upper bound.} By integration by parts, we derive that 
\begin{equation*}
	\begin{aligned}
	&( e^{\aa|v|^2}Q(g,h)-Q(g, e^{\aa|v|^2}h),f)=2\sum_{i=1}^3\int_{\R^3}(b_i*g)\pa_i( e^{\aa|v|^2}) hfdv\\
	&+\sum_{i,j=1}^3\int_{\R^3} (a_{ij}*g) \pa_{ij}(e^{\aa|v|^2}) hfdv+2\sum_{i,j=1}^3\int_{\R^3}(a_{ij}*g) \pa_i ( e^{\aa|v|^2}) \pa_j fhdv\eqdefa\sum_{k=1}^3\mathcal{I}_k.
	\end{aligned}
\end{equation*}

\noindent We first have
	\[\ba 
	&|\mathcal{I}_1| =\bigg|-4\aa \iint\f{(v-v_*)\cdot v}{|v-v_*|^3}g_*he^{\aa\wei^2}f(\mathrm{1}_{|v-v_*|\le1}+\mathrm{1}_{|v-v_*|\ge1})dv_*dv\bigg|\\
	&\ls\|g\|_{L^{2}_6}\|he^{\aa\wei^2}\|_{L^2_{-1/2}}\|f\|_{H^1_{-3/2}}+\|g\|_{L^2_6}\|he^{\aa\wei^2}\|_{L^2_{-1/2}}\|f\|_{L^2_{-1/2}}.
    \ea\]
	Here we use Hardy-Littlewood-Sobolev inequality and the facts that $\<v_*\>\sim\wei$ in the regime $\{|v-v_*|\leq 1\}$ and $|v-v_*|\geq \<v_*\>^{-1}\<v\>$ in the regime $\{|v-v_*|\geq 1\}$.  For $\mathcal{I}_2$, we observe that 
	\[\ba
	&|\mathcal{I}_2|=\bigg|4\aa^2\iint\f{|(v-v_*)\times v_*|^2}{|v-v_*|^3}g_*he^{\aa\wei^2} fdv_*dv+4\aa\iint|v-v_*|^{-1}g_*he^{\aa\wei^2}f dv_*dv\bigg|\ls \iint|v-v_*|^{-1}\\
	&\times(\vv{1}_{|v-v_*|\le1}+\vv{1}_{|v-v_*|\geq1})\lr{v_*}^2|g_*||he^{\aa\wei^2}||f|dv_*dv \ls \|g\|_{L^2_5}\|he^{\aa\wei^2}\|_{L^2_{-1/2}}\|f\|_{L^2_{-1/2}}.
	\ea\]

	\noindent For $\mathcal{I}_3$, direct computation will give that
	\[\ba
	&\mathcal{I}_3=4\aa\bigg[\iint\f{\rr{(v-v_*)\times \na f}\cdot\rr{(v-v_*)\times v_*}}{|v-v_*|^3}\vv{1}_{|v-v_*|<1}g_*he^{\aa\wei^2}dv_*dv+  \iint\f{\rr{(v-v_*)\times\na f}\cdot\rr{(v-v_*)\times v_*}}{|v-v_*|^3}\\
	&\times\vv{1}_{\substack{|v-v_*|\geq1\\|v-v_*|\leq\f{1}{2}|v|} }g_*he^{\aa\wei^2} dv_*dv +   \iint\f{\rr{(v-v_*)\times\na f}\cdot\rr{(v-v_*)\times v_*}}{|v-v_*|^3}\vv{1}_{\substack{|v-v_*|\geq1\\|v-v_*|>\f{1}{2}|v|} }g_*he^{\aa\wei^2}dv_*dv\bigg]\\
	&:=\mathcal{I}_{3,1}+\mathcal{I}_{3,2}+\mathcal{I}_{3,3}.
	\ea\]
	Let us give the estimates term by term. One may easily check that 
	
	\[
    |\mathcal{I}_{3,1}| \ls \iint|v-v_*|^{-1}\vv{1}_{|v-v_*|< 1}|v_*||g_*||\na f||he^{\aa\wei^2}| dv_*dv
    \ls  \|g\|_{L^2_5}\|he^{\aa\wei^2}\|_{L^2_{-1/2}}\|\na f\|_{L^2_{-3/2}}.
	\]
   In the regime $\{|v-v_*|\leq\f{1}{2}|v|\}$, we have  $|v|\sim |v_*|$, which implies that
    \[
    |\mathcal{I}_{3,2}|\ls\iint|v-v_*|^{-1}\vv{1}_{\substack{|v-v_*|\geq1\\|v-v_*|\leq\f{1}{2}|v|}}|v_*|g_*|\na f||he^{\aa\wei^2}|dv_*dv
    \ls \|g\|_{L^2_5}\|he^{\aa\wei^2}\|_{L^2_{-1/2}}\|\na f\|_{L^2_{-3/2}}.
    \]
    While in the regime $\{|v-v_*|\geq1, |v-v_*|> \f{1}{2}|v|\}$, we have  $|v-v_*|^{-1}\ls\wei^{-1}$. Then we get that
    \[\ba
    &|\mathcal{I}_{3,3}| \ls\iint\f{|v\times\na f|}{\wei^2}|v_*||g_*||he^{\aa\wei^2}|dv_*dv+\iint\f{|\na f||v_*|^2}{\wei^2}g_*|he^{\aa\wei^2}|dv_*dv\\&
\ls\|g\|_{L^2_5}\|he^{\aa\wei^2}\|_{L^2_{-1/2}}\|\mathbb{T}_{\S^2}f\|_{L^2_{-3/2}}+\|g\|_{L^2_5}\|he^{\aa\wei^2}\|_{L^2_{-1/2}}\|\na f\|_{L^2_{-3/2}}.
    \ea\]

$\bullet$ \underline{Step 2: The first upper bound.}
By direct computation, we derive that
\begin{equation*}
	\begin{aligned}
	&( e^{\aa|v|^2}Q(g,h)-Q(g, e^{\aa|v|^2}h),f)=
	-\sum_{i,j=1}^3\int_{\R^3} (a_{ij}*g) \pa_{ij}(e^{\aa|v|^2}) hfdv\\
	&-2\sum_{i,j=1}^3\int_{\R^3}(a_{ij}*g) \pa_i ( e^{\aa|v|^2}) \pa_j hfdv\eqdefa\sum_{k=4}^5\mathcal{I}_k.
	\end{aligned}
\end{equation*}
	Similar to the estimates of $\mathcal{I}_2$ and  $\mathcal{I}_3$, we  get that
$|\mathcal{I}_4|+|\mathcal{I}_5|\ls \|g\|_{L^2_6}(\| e^{\aa|v|^2}h\|_{L^2}+\|e^{\aa|v|^2}\na h\|_{L^2})\|f\|_{L^2}.$

     We end the proof by patching together all the estimates.
\end{proof}

\begin{lem}\label{IBP}
	Let $\pa=\pa_x$ or $\pa_v$ and $\aa\in (\f12,1),\bb=0$ or $\bb=x$, then
	\begin{multline*}
		|\big((Q(g, f),(1+\bb\cdot v)e^{\aa(|x|^2+|v|^2)}\pa f\big)_{L^2_{x,v}}|\ls (\|\mathcal{P}_{x,v}^3g\|_{H^3_{x,v}}+\|\mathcal{P}_{v}^3g\|_{H^4_{x,v}})\big(\|(1+|\bb|+|\bb \cdot v|)^{\f12}e^{\f \aa 2(|x|^2+|v|^2)}\na_v f\|^2_{L^2_{x,v}}\\
		+\|(1+|\bb|+|\bb \cdot v|)^{\f12}e^{\f \aa 2(|x|^2+|v|^2)} f\|^2_{L^2_{x,v}}+\|(1+|\bb|+|\bb\cdot v|)^{\f12}e^{\f \aa 2(|x|^2+|v|^2)}\pa f\|^2_{L^2_{x,v}}\big).
	\end{multline*}
\end{lem}
\begin{proof}
	By definition, we have  
	\ben\label{IBPine1}
		(Q(g, f),(1+\bb\cdot v)e^{\aa(|x|^2+|v|^2)}\pa f)_{L^2_{x,v}}=-\int_{\R^6} (a*g)\na f \na ((1+\bb\cdot v)e^{\aa(|x|^2+|v|^2)}\pa f)dvdx\notag\\
		-\int_{\R^6}(b*g) \na f (1+\bb\cdot v)e^{\aa(|x|^2+|v|^2)}\pa f dvdx+8\pi\int_{\R^6}gf(1+\bb\cdot v)e^{\aa(|x|^2+|v|^2)}\pa fdvdx.
	\een
	For the first term in r.h.s. of \eqref{IBPine1}, we have  
	\begin{equation*}
		\begin{aligned}
			&\sum_{i,j=1}^3\int_{\R^6} (a_{ij}*g)\pa_i f \pa_j ((1+\bb\cdot v)e^{\aa(|x|^2+|v|^2)}\pa f)dvdx=\sum_{i,j=1}^3\int_{\R^6} (a_{ij}*g)\pa_i f \pa_j ((1+\bb\cdot v)e^{\aa(|x|^2+|v|^2)})\pa fdvdx\\
			&+\f12\sum_{i,j=1}^3\int_{\R^6} (a_{ij}*g)\pa(\pa_i f\pa_j f) (1+\bb\cdot v)e^{\aa(|x|^2+|v|^2)}dvdx
			= \sum_{i,j=1}^3\int_{\R^6} (a_{ij}*g)\pa_i f \pa_j ((1+\bb\cdot v)e^{\aa(|x|^2+|v|^2)})\pa fdvdx\\
			&-\f12\sum_{i,j=1}^3\int_{\R^6}\big(( a_{ij}*\pa g)(\pa_i f\pa_j f)(1+\bb\cdot v)e^{\aa(|x|^2+|v|^2)}+(a_{ij}*g)(\pa_i f\pa_j f) \pa((1+\bb\cdot v)e^{\aa(|x|^2+|v|^2)}) \big)dvdx,
		\end{aligned}
	\end{equation*}	
	where   integration by parts is used. From this, we can derive that
	\beno
	&&\big|\int_{\R^6} (a*g)\na f \na (e^{\aa(|x|^2+|v|^2)}\pa f)dvdx\big|\\
	&\ls& \|\mathcal{P}_{x,v}^3g\|_{H^3_{x,v}}(\|(1+|\bb|+|\bb \cdot v|)^{\f12}e^{\f \aa 2(|x|^2+|v|^2)}\na_v f\|^2_{L^2_{x,v}}+\|(1+|\bb|+|\bb \cdot v|)^{\f12}e^{\f \aa 2(|x|^2+|v|^2)}\pa f\|^2_{L^2_{x,v}}).
	\eeno
	The second and the third terms in r.h.s. of \eqref{IBPine1} can be bounded similarly as follows:
	\beno
	&&\big|\int_{\R^6}(b*g) \na f (1+\bb\cdot v)e^{\aa(|x|^2+|v|^2)}\pa f dvdx\big|+\big|8\pi\int_{\R^6}gf(1+\bb\cdot v)e^{\aa(|x|^2+|v|^2)}\pa fdvdx\big|\\
	&\ls& \|\mathcal{P}_{v}^3g\|_{H^4_{x,v}}(\|(1+|\bb|+|\bb \cdot v|)^{\f12}e^{\f \aa 2(|x|^2+|v|^2)}|\na_v f|+|f|\|^2_{L^2_{x,v}}+\|(1+|\bb|+|\bb \cdot v|)^{\f12}e^{\f \aa 2(|x|^2+|v|^2)}\pa f\|^2_{L^2_{x,v}}).
	\eeno
	We end the proof of this lemma.
\end{proof}

Next we give a proof for the weight-transfer lemma which is crucial to prove the decay estimate.
\begin{lem}\label{mixturegainweight} Let $\mfR$ and $\mfS$ be defined in \eqref{deofR} and \eqref{deofS} with $r\in[0,1)$ and $\c=\f{1+r}4$. Then for any $\a\in(\f12,1)$ and  $0<\vartheta< \min\{\f1 8,\f12-\c\}$, there exist  large constants $\mathcal{N}$ and $C_\vth$ such that
\beno
	&&\mathcal{N}e^{ (\a-\frac 1 2) |x|^2+ \a |v|^2  }\<v\>^{-1} + (|x|^2-|v|^2-2(\mfR\mfS^{-1}x,v)) e^{(\a-\c-\vth)  (|x|^2+ |v|^2)   } \\ &&\ge C_\vth (|x|^2+ |v|^2)e^{(\a-\c-\vth) (|x|^2+ |v|^2)   } 
	\ge C_\vth (\langle x\rangle^2+\langle v\rangle^2) e^{(\a-\c-\vth)  (|x|^2+ |v|^2)   } -C_\vth.
\eeno
\end{lem}
\begin{proof}
By \eqref{deofR} and \eqref{deofS}, we only need to show 
\beno
\mathcal{N}e^{ (\a-\frac 1 2) |x|^2+ \a |v|^2  }\<v\>^{-1} + (|x|^2-|v|^2-\f{2r}{\sqrt{1-r^2}}(x_2v_1-x_1v_2)) e^{(\a-\c-\vth)  (|x|^2+ |v|^2)} 
 \ge C_\vth (|x|^2+ |v|^2)e^{(\a-\c-\vth) (|x|^2+ |v|^2)}.
\eeno

$\bullet$ \underline{Case 1: $(\f12-\c-\f\vth 2)|x|^2\leq (\c+\f\vth2)|v|^2$}. In this case, $(\a-\f12)|x|^2+\a|v|^2\geq (\a-\c-\f\vth 2)(|x|^2+|v|^2)$,  thus
\beno
\mathcal{N} e^{(\a-\c-\f \vth 2)(|x|^2+|v|^2)}\<v\>^{-1}-2(|x|^2+|v|^2)e^{(\a-\c-\vth )(|x|^2+|v|^2)}\geq C_\vth (|x|^2+ |v|^2)e^{(\a-\c-\vth) (|x|^2+ |v|^2) },
\eeno
since $\mathcal{N}$ is large. We prove the desired result.

$\bullet$ \underline{Case 2: $(\f12-\c-\f\vth 2)|x|^2> (\c+\f\vth 2)|v|^2$}.  In this situation, we easily verify that 
\beno 
&&|x|^2-|v|^2-2\f r{\sqrt{1-r^2}}(x_2v_1-x_1v_2)\geq |x|^2-|v|^2-2\f{r}{\sqrt{1-r^2}}|x||v|\\
&&\geq \Big(1-\f{1-2\c-\vth}{2\c+\vth}-2\f{r}{\sqrt{1-r^2}}\sqrt{\f{1-2\c-\vth}{2\c+\vth}}\Big)|x|^2\geq C_\vth(|x|^2+|v|^2),
\eeno
where  
$1-\f{1-2\c}{2\c}-2\f{r}{\sqrt{1-r^2}}\sqrt{\f{1-2\c}{2\c}}=0$ since $\c=\f{1+r}4$.
We complete the proof of this lemma. \end{proof}
	

Before concluding this section, we will demonstrate the communication between the transport operator and the derivatives with respect to the variables $x$ and $v$. This is essential for the energy method.
\begin{lem}\label{trans}
Let $N\in \N$ and $m$ be a weight function depending on $(x,v)$ variable. If $T:=Sv\cdot \na_x-Sx\cdot \na_v+Rx\cdot \na_x-Rv\cdot \na_v$, where $S$ and $R$ are square matrices that satisfy the properties of being symmetric and skew-symmetric, respectively. Then
\begin{equation*}
\sum_{|\al|+|\be|= N}(\pa^\al_\be(Tf),\pa^\al_\be f m)_{L^2_{x,v}}=\sum_{|\al|+|\be|=N}(T(\pa^\al_\be f),\pa^\al_\be f m)_{L^2_{x,v}}.
\end{equation*}
\end{lem}

\begin{proof} Let $\om_i\in\Z_+^3$ be a vector such that its $i$-th component is $1$ and the other components are zero. Let $\alpha=(\alpha_1,\alpha_2,\alpha_3)\in \Z^3$. In what follows, we will define $\alpha<0$ to mean that there exists $1\le i\le 3$ such that $\alpha_i<0$. In this way, we will slightly abuse the notation by stating that
$\pa^\alpha_\beta:=0$ if $\al<0$ or $\beta<0$.  Based on these notations, we easily deduce that  
\begin{equation*}
 \pa^\al_\be (Tf)=\sum_{i,j=1}^3(S_{ij}\pa^{\al+\om_i}_{\be-\om_j}f-S_{ij}\pa^{\al-\om_j}_{\be+\om_i}f+R_{ij}\pa^{\al-\om_j+\om_i}_\be f-R_{ij}\pa^\al_{\be-\om_j+\om_i}f)+T(\pa^\al_\be f).
\end{equation*}
By taking inner product, we have
\begin{multline*}
\sum_{|\al|+|\be|= N}(\pa^\al_\be(Tf),\pa^\al_\be f m)_{L^2_{x,v}}=\sum_{|\al|+|\be|= N}(T(\pa^\al_\be f),\pa^\al_\be f m)_{L^2_{x,v}}+\sum_{|\al|+|\be|=N}\Big(\sum_{i,j=1}^3(S_{ij}\pa^{\al+\om_i}_{\be-\om_j}f\\
-S_{ij}\pa^{\al-
\om_j}_{\be+\om_i}f,\pa^\al_\be fm)_{L^2_{x,v}}-\sum_{i,j=1}^3(R_{ij}\pa^{\al-\om_j+\om_i}_\be f-R_{ij}\pa^\al_{\be-\om_j+\om_i}f),\pa^\al_\be fm)_{L^2_{x,v}}\Big).
\end{multline*}

$\bullet$ Since $S_{ij}=S_{ji}$, by change of variable from $(\al-\om_i,\be+\om_i)$ to $(\al,\be)$, we derive that
\beno
\sum_{i,j=1}^3(S_{ij}\pa^{\al+\om_i}_{\be-\om_j}f,\pa^\al_\be fm)=\sum_{i,j=1}^3(S_{ij}\pa^{\al}_{\be}f,\pa^{\al-\om_i}_{\be+\om_j} fm)=\sum_{i,j=1}^3(S_{ij}\pa^{\al-\om_j}_{\be+\om_i}f,\pa^{\al}_{\be}fm),
\eeno
which leads to 
$\sum_{|\al|+|\be|=N}\sum_{i,j=1}^3(S_{ij}\pa^{\al+\om_i}_{\be-\om_j}f-S_{ij}\pa^{\al-\om_j}_{\be+\om_i}f,\pa^\al_\be fm)_{L^2_{x,v}}=0.
 $

$\bullet$ Similarly, by change the variable  from $(\al-\om_j+\om_i,\be-\om_j+\om_i)$ to $(\al,\be)$ and using the fact that $R_{ij}=-R_{ji}$, we obtain that
\beno
&\sum_{i,j=1}^3(R_{ij}\pa^{\al-\om_j+\om_i}_\be f,\pa^\al_\be f m)=\sum_{i,j=1}^3(R_{ij}\pa^{\al}_\be f,\pa^{\al-\om_i+\om_j}_\be f m)=-\sum_{i,j=1}^3(R_{ij}\pa^{\al-\om_j+\om_i}_\be f,\pa^\al_\be f m);\\
&\sum_{i,j=1}^3(R_{ij}\pa^{\al}_{\be-\om_j+\om_i} f,\pa^\al_\be f m)=\sum_{i,j=1}^3(R_{ij}\pa^{\al}_\be f,\pa^{\al}_{\be-\om_i+\om_j} f m)=-\sum_{i,j=1}^3(R_{ij}\pa^{\al}_{\be-\om_j+\om_i} f,\pa^\al_\be f m),
\eeno
 which are enough to conclude the result. This ends the proof of the lemma.
\end{proof}

\section{Variant Poincar\'{e}-Korn inequality and estimates of the macroscopic quantities}\label{Section macroscopic}
In this section, we will establish a significant result that demonstrates the controllability of the macroscopic part of the solution $\mathbb{P}f$ (defined by \eqref{deofP}) through the microscopic part $(\mathbb{I}-\mathbb{P})f$, with the inclusion of necessary damping. To accomplish this, we will start by introducing   variant   Poincar\'{e}  and Poincar\'{e}-Korn inequalities.

\subsection{Variant Poincar\'{e} and Poincar\'{e}-Korn inequalities}
Consider a vector-valued function $h=(h_1,h_2,h_3)$  and an invertible and symmetric matrix $S$, we define symmetric and skew-symmetric derivative w.r.t $S$ as follows:
\begin{equation}\label{denaS}
(\na_S^{sym} h)_{ij}:=\f 12 ((S\na_x)_jh_i+(S\na_x)_i h_j),\quad (\na_S^{skew}h)_{ij}:=\f12((S\na_x)_jh_i-(S\na_x)_i h_j).
\end{equation} 
It is easy to see that
$|\na_x h|^2\sim_S |\na_S^{sym} h|^2+|\na_S^{skew}h|^2.$ In particular, $S=\mathbb{I}$, define
\begin{equation*}
	(\na_x^{sym} h)_{ij}:=\f 12 (\pa_{x_j}h_i+\pa_{x_i} h_j),\quad (\na_x^{skew}h)_{ij}:=\f12(\pa_{x_j}h_i-\pa_{x_i} h_j).
\end{equation*} 
  
\begin{defi}
 We say that $W=W(x)$ satisfies a local Poincar\'{e} inequality on a bounded open set $\Om$ if for any smooth function $h:\R^3\rightarrow \R$, there exists some constant $k_\Om$ such that 
\beno
\int_{\Om} h^2 W dx\leq k_\Om \int_{\Om} |\na h|^2Wdx+\f1{W(\Om)}\Big(\int_{\Om}hWdx\Big)^2,~~\mbox{where}~~W(\Om):=\int_{\Om}Wdx.
\eeno
\end{defi}

It follows easily from Section 12.2 in \cite{Leoni Giovanni (2017)} that
\begin{lem}\label{LP}
	If $W,W^{-1}\in L^\infty_{loc}(\R^3)$, then  $W$ satisfies the local Poincar\'{e} inequality for any ball $\Om\subset \R^3$.
\end{lem}

Now we give a proof to variant Poincar\'{e} and Poincar\'{e}-Korn inequalities.
\begin{lem}[Variant Poincar\'{e} and Poincar\'{e}-Korn inequalities]\label{lemKPI}
	Let $\de\in(0,\f12)$.

\noindent\underline{(1).} If $\int_{\R^3}he^{-\f12|x|^2}dx=0$, then
	\ben\label{PI}
	\int_{\R^3} \<x\>^2|h|^2e^{-(1-\de)|x|^2}dx\leq C_\de \int_{\R^3} |\na_x h|^2e^{-(1-\de)|x|^2}dx.
	\een
 \noindent\underline{(2).}  If  
 $\int_{\R^3}\na_S^{skew}he^{-\f12|x|^2}dx=0$, where $S$ is an invertible and symmetric matrix, then 
	\begin{equation}\label{PKI}
	  \int_{\R^3} |\na_xh|^2 e^{-(1-\de)|x|^2}dx\leq C_{S,\de} \int_{\R^3}|\na^{sym}_Sh|^2e^{-(1-\de)|x|^2}dx.
	\end{equation}
\end{lem}
 
\begin{proof}
To prove \eqref{PI}, we set $V(x):=(1-\de)|x|^2$ and $g:=he^{-\f12 V}$. Then one may check that    $\na g=\na h e^{-\f12 V}-\f12 \na Vh e^{-\f12 V}$, which implies that
	\beno
	0\leq \int_{\R^3} |\na g|^2dx&=&\int_{\R^3} |\na h|^2 e^{-V}dx+\int_{\R^3} h^2 \f14 |\na V|^2 e^{-V}dx-\int_{\R^3} \f12 \na (h^2)\cdot \na V e^{-V}dx\\
	&=& \int_{\R^3} |\na h|^2e^{-V}dx+\int_{\R^3} h^2\big(\f12 \D V-\f14 |\na V|^2   \big)e^{-V}dx.
	\eeno
	Since  $\f 14|\na V|^2-\f 12 \D V\geq(1-\de)^2\<x\>^2-4\geq \f1{8}\<x\>^2-4\times\mathbf{1}_{B_R}$ for $R>6$,
	where $B_R$ is a ball centered at $0$ with radius $R$,  and $e^{-\f12|x|^2}\geq e^{-(1-\de)|x|^2}$, we deduce that
	\ben\label{VPPKinq1}
	\int_{\R^3} |\na h|^2e^{-V}dx\geq \f 1{8}\int_{\R^3}  h^2\<x\>^2 e^{-V}dx-4\int_{B_R}h^2e^{-V}dx\geq \f 1{8}\int_{\R^3}  h^2\<x\>^2 e^{-V}dx-4\int_{B_R}h^2e^{-\f12|x|^2}dx.
	\een 
	Define $\vep_R:=\int _{B_R^c} e^{-\de|x|^2} dx,\quad Z_R:=\int_{B_R}e^{-\f12|x|^2}dx$. Then $\vep_R\rightarrow0$ and $Z_R\rightarrow Z_\infty$ as  $R\rightarrow \infty$. Denote $B_R^c := \R^3   \setminus B_R$, 
	using the condition that $\int_{\R^3}he^{-\f12|x|^2}dx=0$ and H\"{o}lder inequality, we get  that
	\beno
	\Big(\int_{B_R}he^{-\f12|x|^2}dx\Big)^2=\Big(\int_{B^c_R}he^{-\f12|x|^2}dx\Big)^2\leq \int_{B_R^c} h^2 e^{-(1-\de)|x|^2}dx\int_{B_R^c} e^{-\de|x|^2}dx\leq \vep_R \int_{B_R^c} h^2 e^{-V(x)}dx.
	\eeno
	Also, thanks to  Lemma \ref{LP}, we deduce that
	\beno
	\int_{B_R}h^2 e^{-\f12|x|^2}dx&\leq& C_R \int_{B_R} |\na h|^2 e^{-\f12|x|^2}dx+\f 1{Z_R}\Big(\int_{B_R} he^{-\f12|x|^2}dx\Big)^2\\
	&\leq& C_R \int_{B_R} |\na h|^2 e^{-V}dx+\f{\vep_R}{Z_R} \int_{B_R^c} h^2 e^{-V}dx.
	\eeno

	Putting together all the inequalities, we finally get that
	\beno
	\int_{\R^3} \<x\>^2 h^2e^{-V}dx&\leq&8 \int_{\R^3} |\na h|^2 e^{-V}dx+32\int_{B_R} h^2 e^{ - \f12|x|^2}dx\\
	&\leq& 8(1+4C_R) \int_{\R^3}|\na h|^2 e^{-V}dx+32\f{\vep_R}{Z_R} \int_{B_R^c} h^2e^{-V}dx.
	\eeno
	We  establish the desired result by selecting a sufficiently large value for $R$, such that $32\vep_R/Z_R<\frac{1}{2}$.
	\smallskip

	We turn to prove \eqref{PKI}.   We begin with the simple case that $S=\mathbb{I}$. Suppose that 
	\beno
	-\D_Vf:=-\D f+\na V\cdot \na f\quad\mbox{and}\quad \Lam:=-\D_V+\mathbb{I}.
	\eeno
	It is easy to check that $-\D_V$ is nonnegative and symmetric in Hilbert space $L^2(V)$, defined by 
	\beno
	(f,g)_{L^2(V)}:=\int_{\R^3} f\bar{g}e^{-V(x)}dx.
	\eeno
   
    $\bullet$ By \eqref{VPPKinq1}, for any $h\in L^2(V)$, we have $\big((|\na V|+1)^2h,h\big)_{L^2(V)}+(-\D_Vh,h)_{L^2(V)}\ls(\Lam h,h)_{L^2(V)}$ which implies that $\|(|\na V|+1)h\|_{L^2(V)}+\|\na h \|_{L^2(V)}\ls \|\Lam^{\f12}h\|_{L^2(V)}$. Thus for $h\in L^2(V)$, it holds that
\ben\label{VPPKinq2}\|(|\na V|+1)\Lam^{-\f12}h\|_{L^2(V)}+\|\na \Lam^{-\f12}h \|_{L^2(V)}\ls \|h\|_{L^2(V)}.\een
 Observing that $(\Lam ^{-\f12}\pa h,g)_{L^2(V)}=(h,(-\pa+\pa V) \Lam ^{-\f12}g)_{L^2(V)}$, by \eqref{VPPKinq2} and duality, we get that 
 \ben\label{VPPKinq3}\|\Lam^{-\f12}(|\na V|+1)h\|_{L^2(V)}+\|\Lam^{-\f12}\na  h \|_{L^2(V)}\ls \|h\|_{L^2(V)}.\een

	$\bullet$ We claim that if $\int_{\R^3}he^{-\f12|x|^2}dx=0$, then 
	\ben\label{PL1}
	 \|h\|^2_{L^2(V)}\leq C_{V}	\|\Lam ^{-\f12}\na h\|^2_{L^2(V)}.
	\een
   From \eqref{PI} and $\int_{\R^3}he^{-\f12|x|^2}dx=0$,  we have   $(-\D_Vh,h)_{L^2(V)}  \sim (\Lam h,h)_{L^2(V)}$. This in particular implies that  $\|\na \Lam^{-\f12}h\|_{L^2(V)}\sim \|h\|_{L^2(V)}$. From this together with the fact $[\D_V,\Lam]=0$, one may obtain that
 \beno&& 
 \|h\|_{L^2(V)}^2\sim \|\na \Lam^{-\f12}h\|_{L^2(V)}^2=(-\D_V\Lam^{-1}h, h)_{L^2(V)}     =(\Lam^{\f12}\na \Lam^{-1}h, \Lam^{-\f12}\na h)_{L^2(V)}
 \\&& \le \|\Lam^{\f12}\na \Lam^{-1}h\|_{L^2(V)}\|\Lam^{-\f12}\na h\|_{L^2(V)}. \eeno 
 Note that 
 \beno
  &&\|\Lam^{\f12}\na f\|^2_{L^2(V)}    =  \sum_{j  =1}^3((-\sum_{i  =1}^3\pa_i^2+\sum_{k  =1}^3\pa_kV\pa_k+1)\pa_jf,\pa_jf)_{L^2(V)}=-\sum_{j=1}^3\big[(\sum_{k  =1}^3(\pa_j\pa_kV)\pa_kf,\pa_jf)_{L^2(V)}
  \\
  &&+((-\sum_{i  =1}^3\pa_i^2+\sum_{k  =1}^3\pa_kV\pa_k+1)f,(-\pa_j+\pa_jV)\pa_jf)_{L^2(V)}\big]\le
\|\Lam f\|_{L^2(V)}^2+\|\Lam^{\f12} f\|_{L^2(V)}^2.
 \eeno 
  These two estimates imply that 
  \beno  \|h\|_{L^2(V)}^2\sim \|\na \Lam^{-\f12}h\|_{L^2(V)}^2\ls \|h\|_{L^2(V)}\|\Lam^{-\f12}\na h\|_{L^2(V)}. \eeno 
	We conclude the claim \eqref{PL1}.

	$\bullet$ Thanks to \eqref{PL1} and the condition $\int_{\R^3} (\na^{skew}_xh) e^{-\f12|x|^2}dx=0$, we  first have
	\begin{equation*}
	\begin{aligned}
	\|\na_xh\|^2_{L^2(V)}&=\|\na^{sym}_xh\|^2_{L^2(V)}+\|\na^{skew}_xh\|^2_{L^2(V)} 
	\leq \|\na^{sym}_xh\|^2_{L^2(V)}+C_V \|\Lam^{-\f12}\na (\na^{skew}_xh)\|^2_{L^2(V)},
	\end{aligned}
\end{equation*}
	where $\|\Lam^{-\f12}\na (\na^{skew}_xh)\|^2_{L^2(V)}=\sum_{i,j=1}^3 \|\Lam^{-\f12}\na (\na^{skew}_xh)_{ij}\|^2_{L^2(V)}$. The  Schwarz Theorem implies that
	\beno
	\pa_k(\na_x^{skew}h)_{ij}=\pa_j(\na^{sym}_xh)_{ik}-\pa_i(\na^{sym}_xh)_{jk},\quad \forall i,j,k\in \{1,2,3\},
	\eeno
	which leads to that $\|\Lam^{-\f12}\na (\na^{skew}_xh)\|^2_{L^2(V)}\le 4\sum_{i,j}^3 \|\Lam^{-\f12}\na (\na^{sym}_xh)_{ij}\|^2_{L^2(V)}$. 
	Then \eqref{VPPKinq3} yields that
	\beno
	\sum_{i,j}^3 \|\Lam^{-\f12}\na (\na^{sym}_xh)_{ij}\|^2_{L^2(V)}\leq \sum_{i,j}^3\| (\na^{sym}_xh)_{ij}\|^2_{L^2(V)}=\|\na^{sym}_xh\|^2_{L^2(V)}.
	\eeno
	Thus we obtain that $\|\na_xh\|_{L^2(V)}\leq C_V\|\na^{sym}_x h\|_{L^2(V)}$. It ends the proof for \eqref{PKI} in the case $S=\mathbb{I}$.
	\smallskip

	For general case,  we observe that $\na_S^{skew}h=S\na^{skew}_x(S^{-1}h)S$ and $\na_S^{sym}h=S\na^{sym}_x(S^{-1}h)S$, thus
	\beno
    \int_{\R^3}\na^{skew}_x(S^{-1}h)e^{-\f12|x|^2}dx=0\,\,\mbox{if}\quad\int_{\R^3}\na_S^{skew}he^{-\f12|x|^2}dx=0. 
	\eeno
Now we are in a position to complete the proof. We have
\begin{equation*}
	\begin{aligned} &\int_{\R^3} |\na_xh|^2 e^{-(1-\de)|x|^2}dx\le C_S\int_{\R^3} |\na_x(S^{-1}h)|^2 e^{-(1-\de)|x|^2}dx\\& 
 \leq C_{S,\de} \int_{\R^3}|\na^{sym}_x(S^{-1}h)|^2e^{-(1-\de)|x|^2}dx\leq C_{S,\de} \int_{\R^3}|\na^{sym}_Sh|^2e^{-(1-\de)|x|^2}dx.
	\end{aligned}
\end{equation*}
This ends the proof of the lemma.
\end{proof}

\subsection{Estimates of the macroscopic quantities} This subsection is devoted to the proof that  the macroscopic part of the solution $\mathbb{P}f=\mathbf{a}(t,x)\mu+\mathbf{b}(t,x)\cdot v\mu+\mathbf{c}(t,x)(|v|^2-3)\mu$ (see \eqref{deofP}) can be controlled through the microscopic part $(\mathbb{I}-\mathbb{P})f$, with the inclusion of necessary damping. We present the macroscopic equation here:
\begin{equation}\label{macroeq}  \left\{ 
	\begin{aligned}
		&\pa_t(\ma e^{\f12|x|^2})=-\mathscr{C}_l(t)(\mfS\na_x\cdot \mb)e^{\f12|x|^2}-\mathscr{C}_l(t)\mfR x\cdot\na_x(\ma e^{\f12|x|^2})+\sfT_{11};\\
		&\pa_t(\mb e^{\f12|x|^2})=-\mathscr{C}_l(t)\mfS\na_x(\ma e^{\f12|x|^2})-2\mathscr{C}_l(t)(\mfS\na_x\mc)e^{\f12|x|^2}-\mathscr{C}_l(t)\mfR\mb e^{\f12|x|^2}-\mathscr{C}_l(t)(\mfR x\cdot\na_x)(\mb e^{\f12|x|^2})+\sfT_{21};\\
		&\pa_t(\mc e^{\f12|x|^2})=-\f13\mathscr{C}_l(t) \mfS\na_x\cdot (\mb e^{\f12|x|^2})-\mathscr{C}_l(t)\mfR x\cdot\na_x (\mc e^{\f12|x|^2})+\sfT_{31};\\
		&\pa_t(\mc e^{\f12|x|^2}) \mathbb{I}_{3\times 3}=-\mathscr{C}_l(t)\na^{sym}_\mfS(\mb e^{\f12|x|^2})-\mathscr{C}_l(t)\mfR x\cdot\na_x (\mc e^{\f12|x|^2})\mathbb{I}_{3\times 3}+\sfT_{41}+\pa_t \sfT_{42};\\
		&\mathscr{C}_l(t)\mfS\na_x (\mc e^{\f12|x|^2})=\sfT_{51}+\pa_t\sfT_{52},
	\end{aligned}\right.
\end{equation}
where $\sfT_{k1}$ and $\sfT_{l2}$, with $1\le k\le 5$ and $l=4,5$, are defined as follows: for $i,j=1,2,3$, $i\neq j$ and $g:=\mathscr{C}_2Q(f,f)$,
\begin{equation}\label{defofT1}
	\left\{ \begin{aligned}
		&\sfT_{11}=e^{\f12|x|^2}(g,1)_{L^2_v},~\sfT_{31}=-\mathscr{C}_l(t)\f{e^{\f12|x|^2}}6(\mfS v\cdot\na_x((\mathbb{I-P})f),|v|^2)_{L^2_v}+ e^{\f12|x|^2}\Big(g,\f{|v|^2-3}6\Big)_{L^2_v};\\
		&\sfT_{21}=e^{\f12|x|^2}(-\mathscr{C}_l(t)\mfS v\cdot\na_x((\mathbb {I-P})f)+g,v)_{L^2_v},~\sfT_{52}=-e^{\f12|x|^2}\Big((\mathbb{I-P})f,\f{v(|v|^2-5)}{10}\Big)_{L^2_v};\\
		&\sfT_{51}=-e^{\f12|x|^2}\Big(\mathscr{C}_l(t)(\mfS v\cdot\na_x-\mfS x\cdot\na_v+\mfR x\cdot\na_x-\mfR v\cdot\na_v)((\mathbb{I-P})f)\\&\qquad\qquad\qquad+\mathscr{C}_1e^{-\f12|x|^2}L((\mathbb{I-P})f)+g,\f{v(|v|^2-5)}{10}\Big)_{L^2_v};  \end{aligned}\right.
\end{equation} \begin{equation}\label{defofT2}
	\left\{ \begin{aligned}
		&(\sfT_{41})_{ij}=-e^{\f12|x|^2}(\mathscr{C}_l(t)(\mfS v\cdot\na_x-\mfS x\cdot\na_v+\mfR x\cdot\na_x-\mfR v\cdot\na_v)((\mathbb{I-P})f)\\&\qquad\qquad\qquad+\mathscr{C}_1e^{-\f12|x|^2}L((\mathbb{I-P})f)+g,v_iv_j/2)_{L^2_v};\\
		&(\sfT_{41})_{ii}=-e^{\f12|x|^2}\Big(\mathscr{C}_l(t)(\mfS v\cdot\na_x-\mfS x\cdot\na_v+\mfR x\cdot\na_x-\mfR v\cdot\na_v)((\mathbb{I-P})f)\\&\qquad\qquad\qquad+\mathscr{C}_1e^{-\f12|x|^2}L((\mathbb{I-P})f)+g,\f{v_i^2-1}2\Big)_{L^2_v};\\
		&(\sfT_{42})_{ij}=e^{\f12|x|^2}(-(\mathbb{I-P})f,v_iv_j/2)_{L^2_v},~~~(T_{42})_{ii}=e^{\f12|x|^2}\Big(-(\mathbb{I-P})f,\f{v_i^2-1}2\Big)_{L^2_v}.
	\end{aligned}\right.
\end{equation}
The proof of (\ref{macroeq}-\ref{defofT2}) will be given in Proposition \ref{abcequa}.

\smallskip

Now, we address some facts which will be used frequently in what follows.

\underline{(1).}  $\ma(t,x),\mb(t,x)$ and $\mc(t,x)$  satisfy the conservation laws:
\begin{equation}\label{conmabc}
	\begin{aligned}
		\int_{\R^3}(\mathbf{a}, \mfS^{-1}x\ma,\mathbf{b},|\mfS^{-1}x|^2\ma,\mfS^{-1}x\cdot \mb,6\mc+\mb\cdot \mfR\mfS^{-1}x,\mfS^{-1}x\wedge \mb+\mfS^{-1}x\wedge \mfR\mfS^{-1}x\ma)dx=0,
	\end{aligned}
\end{equation}
which are derived from   \eqref{conservationofNSlandau} and Remark \ref{v2Rx} since $G=\mathcal{M}+f$.

\underline{(2).} $\ma(t,x),\mb(t,x)$ and $\mc(t,x)$ will be split into two parts respectively, i.e.,
\ben \label{412}\ma(t,x):=\ma_{0}(t,x)+\ma_{\mathsf{n}}(t,x);\,\, \mb(t,x):=\mb_{0}(t,x)+\mb_{\mathsf{n}}(t,x);\,\, \mc(t,x):=\mc_{0}(t,x)+\mc_{\mathsf{n}}(t,x),   \een 
where $(\ma_{0} ,\mb_{0}, \mc_{0})$ and  $(\ma_{\mathsf{n}}, \mb_{\mathsf{n}},\mc_{\mathsf{n}})$ are  denoted by the zero-mode part and the non zero-mode part respectively. In fact,  $(\ma_{\mathsf{n}}, \mb_{\mathsf{n}},\mc_{\mathsf{n}})$ are defined by 
\begin{equation}\label{deofabcs}
	\begin{aligned}
		&\mc_{\mathsf{n}} \mathcal{Z}_x^{\f12}:=\mc \mathcal{Z}_x^{\f12}-\msA(\mc \mathcal{Z}_x^{\f12});\\
		&\mb_{\mathsf{n}}\mathcal{Z}_x^{\f12}:=\mb \mathcal{Z}_x^{\f12}-\msA(\na^{skew}_\mfS(\mb \mathcal{Z}_x^{\f12})) \mfS^{-1}x-\f13 \msA(\mfS\na_x\cdot(\mb \mathcal{Z}_x^{\f12}))\mfS^{-1}x-   \msA(\mb \mathcal{Z}_x^{\f12});\\
		&\ma_{\mathsf{n}}\mathcal{Z}_x^{\f12}:=\ma \mathcal{Z}_x^{\f12}+\f12\left(\Big(\mfR\msA(\na^{skew}_\mfS(\mb \mathcal{Z}_x^{\f12}))+\msA(\na^{skew}_\mfS(\mb \mathcal{Z}_x^{\f12}))\mfR\Big)\mfS^{-1}x,\mfS^{-1}x\right)_{\R^3}\\
		&-\msA\left(\f12\left(\Big(\mfR\msA(\na^{skew}_\mfS(\mb \mathcal{Z}_x^{\f12}))+\msA(\na^{skew}_\mfS(\mb \mathcal{Z}_x^{\f12}))\mfR\Big)\mfS^{-1}x,\mfS^{-1}x\right)_{\R^3}\right)-\msA(\mc \mathcal{Z}_x^{\f12})(|x|^2-\msA(|x|^2))\\
		&+\f16\Big(-\msA(\D_{\mfS x}(\ma \mathcal{Z}_x^{\f12}))+2\sum_{i=1}^3(\mfS^2)_{ii}\msA(\mc \mathcal{Z}_x^{\f12})-2\msA( \mfR\mb \mathcal{Z}_x^{\f12}\cdot \mfS x)\Big)(|\mfS^{-1}x|^2-\msA(|\mfS^{-1}x|^2)),
	\end{aligned}
\end{equation}
where we recall the definition of $\mathcal{Z}_x^\a$ in \eqref{exponentialfunction},
\begin{equation}\label{RSr<1}
	\mfR=
	\begin{pmatrix}
		0 & r & 0\\
		-r & 0& 0\\
		0 & 0 & 0
	\end{pmatrix}
	\quad\mbox{and}\quad
	\mfS=
	\begin{pmatrix}
		\sqrt{1-r^2} & 0 & 0\\
		0 & \sqrt{1-r^2} & 0\\
		0 & 0 & 1
	\end{pmatrix}
\end{equation}
with $0\le r<1$,  $(\vec{u},\vec{v})_{\R^3}:=\vec{u}^\tau \vec{v}$ for any $\vec{u},\vec{v}\in\R^3$, and
\ben\D_{\mfS x}:=\sum_{i=1}^3(\mfS\na_x)_i^2,\quad \mathscr{A}(h):=\f1{(2\pi)^{3/2}}\int_{\R^3} h(x)\mathcal{Z}_x^{-\f12}dx. \een

\underline{(3).} We list some facts that can be easily verified and will be frequently referenced in the later.  

\noindent$\bullet$ It holds that  \ben\label{msA}
\msA(1)=\msA(x_i^2)=\msA(x_i^2x_j^2)=1,~~\msA(x_i^4)=3,~~\msA(|x|^2)=3,~~\msA(|x|^4)=15,~~i,j=1,2,3,  ~~i\neq j.
\een

\noindent$\bullet$ For any skew-symmetric $A$  and any smooth function $f$, the following facts hold:
\begin{equation}
\label{equation for A G -1 x and A A G^-1 x}
\msA(\mfS^{-1}x ) =\msA(A\mfS^{-1} x ) = 0, \quad \msA ( A x \cdot \nabla_x f )=0,
\end{equation}
and
\[
\nabla^{skew}_\mfS (\mfS^{-1} x)  = 0 ,\quad  \nabla^{skew}_\mfS  (A \mfS^{-1}  x) =  A,\quad \nabla^{sym}_\mfS  (\mfS^{-1} x) = \mathbb{I}_{3 \times 3},\quad \nabla^{sym}_\mfS  (A \mfS^{-1}  x) =  0.
\]
From these   together with \eqref{deofabcs},  we get that
\begin{equation}
\label{skew and symmetric of b n  Z x 1 2}
\nabla^{skew}_\mfS  (\mb_n \mathcal{Z}_{x}^{1/2})=0,\quad \nabla^{sym}_\mfS (\mb_n \mathcal{Z}_x^{1/2})  =  \nabla^{sym}_\mfS (\mb     \mathcal{Z}_x^{1/2}      )        -\frac 1 3 \msA(\mfS \na_x\cdot(\mb \mathcal{Z}^{1/2}_x)  ) \mathbb{I}_{3\times 3}  .
\end{equation}

\noindent$\bullet$ It is easy to check that
\begin{equation}
\label{equality R G -1 x G x}
\mfS\mfR=\mfR\mfS, \quad \mfR \mfS^{-1}=\mfS^{-1}\mfR,\quad  (\mfR\mfS^{-1}x,  \mfS x)_{\R^3} = (\mfR x, x)_{\R^3}=0,\quad (\mfS\na_x) \cdot(\mfR\mfS^{-1}x)=0,
\end{equation}
and
\begin{equation}
\label{equality R x nabla x G -1 x}
\mfR x\cdot \nabla_x (\mfS^{-1} x) =\mfR\mfS^{-1} x,\quad \mfR x \cdot \nabla_x ( A\mfS^{-1}x) = A \mfR \mfS^{-1}x.
\end{equation}
Thanks to  integration by parts, for any $i, j =1, 2, 3$, we may have
\begin{equation}
\label{A integral for G nabla x R b}
\msA( (\mfS\nabla_x )_j   (\mfR \mb)_i \mathcal{Z}^{1/2}_x ) = \msA ((\mfR\mb)_i  (\mfS x)_j   ),\quad \msA( (\mfS\nabla_x )_j   (\mfR x \cdot \nabla_x) (\mb \mathcal{Z}^{1/2}_x) _i) =  - \msA (  \mb_i (\mfR \mfS x )_j \mathcal{Z}^{1/2}_x ),
\end{equation}
\begin{equation}
\label{A integral for G nabla j G nabla i f}
\msA \Big ( (\mfS \nabla_x)_j ((\mfS\nabla_x   f)_i\mathcal{Z}_x^{\frac 12}    )\Big) =  - (\mfS^2)_{ij}  \msA (f \mathcal{Z}_x^{\frac 12}  ).
\end{equation}

\noindent$\bullet$ One may easily verify that  
\ben
\label{equality for R G -1 x}
&& \big( (\mfS^{-1}x \wedge \mfR \mfS^{-1} x )_{12}, (\mfS^{-1}x \wedge \mfR \mfS^{-1} x )_{13},(\mfS^{-1}x \wedge \mfR \mfS^{-1} x )_{23}\big)=(-\frac {r} {1-r^2}  (x_1^2+x_2^2), \\&&  - \frac {r} {\sqrt{1-r^2} } x_2x_3, \frac {r} {\sqrt{1-r^2}} x_1x_3),\quad
\mfR\mfS^{-1} x =(\frac r {\sqrt{1-r^2} } x_2,   -\frac r {\sqrt{1-r^2} } x_1, 0).\label{equality for G -1 x R G -1 x}
\een

\underline{(4).} Note that $\ma_{\mathsf{n}},\mb_{\mathsf{n}}$ and $\mc_{\mathsf{n}}$ have vanishing properties, that is 
\begin{equation}\label{abcspk}
	\int_{\R^3} (\mc_{\mathsf{n}}\mathcal{Z}_x^{\f12},\mb_{\mathsf{n}}\mathcal{Z}_x^{\f12},\ma_{\mathsf{n}}\mathcal{Z}_x^{\f12})\mathcal{Z}_x^{-\f12}dx
	=0,\quad\int_{\R^3}\na^{skew}_\mfS(\mb_{\mathsf{n}}\mathcal{Z}_x^{\f12})\mathcal{Z}_x^{-\f12}dx=0,
\end{equation}
which can be directly verified from  \eqref{conmabc} and  \eqref{deofabcs} by using \eqref{equation for A G -1 x and A A G^-1 x} and \eqref{skew and symmetric of b n  Z x 1 2}.  Since $|\na_x f|\leq |\na_x f+x f|+|xf|$,   by Lemma \ref{lemKPI}, for any $\de\in (0,\f12)$, we have
\begin{equation}\label{PKIforabcs}
	\begin{aligned}
		\int_{\R^3}(\mathcal{P}_x^2|\ma_{\mathsf{n}}|^2+|\na_x\ma_{\mathsf{n}}|^2)\mathcal{Z}_x^{\de}dx&\leq C_\de \int_{\R^3} |\na_x(\ma_{\mathsf{n}}\mathcal{Z}_x^{\f12})|^2 \mathcal{Z}_x^{-(1-\de)}dx,\\
		\int_{\R^3}(\mathcal{P}_x^2|\mc_{\mathsf{n}}|^2+|\na_x\mc_{\mathsf{n}}|^2)\mathcal{Z}_x^{\de}dx&\leq C_\de \int_{\R^3} |\na_x(\mc_{\mathsf{n}}\mathcal{Z}_x^{\f12})|^2 \mathcal{Z}_x^{-(1-\de)}dx,\\
		\int_{\R^3}(\mathcal{P}_x^2|\mb_{\mathsf{n}}|^2+|\na_x\mb_{\mathsf{n}}|^2)\mathcal{Z}_x^{\de}dx&\leq C_\de \int_{\R^3} |\na_x(\mb_{\mathsf{n}}\mathcal{Z}_x^{\f12})|^2 \mathcal{Z}_x^{-(1-\de)}dx\\
		&\leq C_{r,\de} \int_{\R^3}|\na^{sym}_\mfS(\mb_{\mathsf{n}}\mathcal{Z}_x^{\f12})|^2\mathcal{Z}_x^{-(1-\de)}dx.
	\end{aligned}
\end{equation}

\subsubsection{Control of the zero mode} We want to prove that

\begin{lem}\label{abceqtoabcs}
	Under the condition \eqref{conmabc} and the notation \eqref{exponentialfunction}, for any $0<\de_3<\de_2<\de_1<\f18$ and $N\geq0$, there exists a constant $C_{r,\de,N}$ depending of $r,N$ and $\de_i,i=1,2,3$ such that
	\begin{multline}\label{scontrol}
		\sum_{|\al|\leq N}(\|(\pa^\al_x\mc) \mathcal{Z}_x^{\f{\de_1}2}\|_{L^2_x}+\|(\pa^\al_x\mb) \mathcal{Z}_x^{\f{\de_2}2}\|_{L^2_x}+\|(\pa^\al_x\ma) \mathcal{Z}_x^{\f{\de_3}2}\|_{L^2_x})\\
		\leq C_{r,\de,N}\sum_{|\al|\leq N} (\|(\pa^\al_x\mc_{\mathsf{n}})\mathcal{Z}_x^{\f{\de_1}2}\|_{L^2_x}+\|(\pa^\al_x\mb_{\mathsf{n}})\mathcal{Z}_x^{\f{\de_2}2}\|_{L^2_x}+\|(\pa^\al_x\ma_{\mathsf{n}})\mathcal{Z}_x^{\f{\de_3}2}\|_{L^2_x}).
	\end{multline}
\end{lem}
\begin{proof}
	For simplicity, we set
	\begin{equation*}
		\begin{aligned}
			&\mathcal{B}_1:=\msA(\mb \mathcal{Z}_x^{\f12}),~\mathcal{B}_2:=\msA(\mfS\na_x\cdot(\mb \mathcal{Z}_x^{\f12})),~\mathcal{B}:=\msA(\na^{skew}_\mfS(\mb \mathcal{Z}_x^{\f12})),\\
			&\mathcal{C}:=\msA(\mc \mathcal{Z}_x^{\f12}),~\mathcal{R}:=\Big(\mfR\msA(\na^{skew}_\mfS(\mb \mathcal{Z}_x^{\f12}))+\msA(\na^{skew}_\mfS(\mb \mathcal{Z}_x^{\f12}))\mfR\Big),\\
			& \mathcal{D}:=\f16\Big(-\msA(\D_{\mfS x}(\ma \mathcal{Z}_x^{\f12}))+2\sum_{i=1}^3(\mfS^2)_{ii}\msA(\mc \mathcal{Z}_x^{\f12})-2\msA( \mfR\mb \mathcal{Z}_x^{\f12}\cdot \mfS x)\Big).
		\end{aligned}
	\end{equation*}
	Here $\mathcal{B}$ is  skew-symmetric and $\mathcal{R}$ is symmetric. By direct computation, if we  set
	\begin{equation*}
		\mathcal{B}=\begin{pmatrix}
			0 & (\mathcal{B})_{12} & (\mathcal{B})_{13}\\
			-(\mathcal{B})_{12} & 0& (\mathcal{B})_{23}\\
			-(\mathcal{B})_{13}& -(\mathcal{B})_{23} & 0
		\end{pmatrix},\quad\mbox{then}\quad\mathcal{R}=\begin{pmatrix}
			-2r (\mathcal{B})_{12} & 0 & r (\mathcal{B})_{23}\\
			0 & -2r(\mathcal{B})_{12}& -r(\mathcal{B})_{13}\\
			r (\mathcal{B})_{23} & -r(\mathcal{B})_{13} & 0
		\end{pmatrix}.
	\end{equation*}
For $\mathcal{B}$, we compute
\begin{equation}
\label{equality for B G - 1 x}
\mathcal{B} \mfS^{-1} x =   (\frac 1 {\sqrt{1-r^2} }\mathcal{B}_{12} x_2 + \mathcal{B}_{13} x_3, - \frac 1 {\sqrt{1-r^2} } \mathcal{B}_{12} x_1 + \mathcal{B}_{23} x_3,  -\frac 1 {\sqrt{1-r^2} } \mathcal{B}_{13} x_1 - \frac 1 {\sqrt{1-r^2} } \mathcal{B}_{23}   x_2).
\end{equation}
	Since $\mathcal{B}_1,\mathcal{B}_2,\mathcal{B},\mathcal{C},\mathcal{R}$ and $\mathcal{D}$  do not depend on  $x$ variable, one may easily check   that
	\beno
			&&\sum_{|\al|\leq N}(\|(\pa^\al_x\mc) \mathcal{Z}_x^{\f{\de_1}2}\|_{L^2_x}+\|(\pa^\al_x\mb) \mathcal{Z}_x^{\f{\de_2}2}\|_{L^2_x}+\|(\pa^\al_x\ma) \mathcal{Z}_x^{\f{\de_3}2}\|_{L^2_x})
			\leq C_{r,\de,N}\sum_{|\al|\leq N} (\|(\pa^\al_x\mc_{\mathsf{n}})\mathcal{Z}_x^{\f{\de_1}2}\|_{L^2_x}\\&&+\|(\pa^\al_x\mb_{\mathsf{n}})\mathcal{Z}_x^{\f{\de_2}2}\|_{L^2_x}+\|(\pa^\al_x\ma_{\mathsf{n}})\mathcal{Z}_x^{\f{\de_3}2}\|_{L^2_x}) +|\mathcal{B}_1|+|\mathcal{B}_2|+\|\mathcal{B}\|_2+|\mathcal{C}|+\|\mathcal{R}\|_2+|\mathcal{D}|.
	\eeno
	Thus to prove \eqref{scontrol}, it is sufficient to obtain that
	\begin{equation}\label{finitetermcontrol}
		|\mathcal{B}_1|+|\mathcal{B}_2|+|(\mathcal{B})_{12}|+|(\mathcal{B})_{13}|+|(\mathcal{B})_{23}|+|\mathcal{C}|+|\mathcal{D}|\leq C_{r,\de} (\|\mc_{\mathsf{n}}\mathcal{Z}_x^{\f{\de_1}2}\|_{L^2_x}+\|\mb_{\mathsf{n}}\mathcal{Z}_x^{\f{\de_2}2}\|_{L^2_x}+\|\ma_{\mathsf{n}}\mathcal{Z}_x^{\f{\de_3}2}\|_{L^2_x}).
	\end{equation}
	
	\smallskip
	
	We  first observe that if $r=0$, then  $\mfR=0$ and $\mfS=\mathbb{I}$. One may check that in this situation 
	 $\mathcal{B}_1=\mathcal{B}_2=\mathcal{B}=\mathcal{C}=\mathcal{R}=\mathcal{D}=0$ from \eqref{conmabc}. Thus \eqref{scontrol} holds true. In what follows, we assume that   $r\in (0,1)$.
	\medskip
	
	\underline{Step 1: \textit{Estimates of $\mathcal{B}_1$ and $\mathcal{B}_2$.}} 
Thanks to \eqref{conmabc}, we have  $\int_{\R^3}\mb dx=0$ and $\int_{\R^3}\mfS^{-1}x\cdot\mb dx=0$. These imply that   $\mathcal{B}_1=0$ and 
	\beno
	\int_{\R^3} \mfS^{-1}x\cdot \mb_{\mathsf{n}}dx=-\f{1}3 \mathcal{B}_2\int_{\R^3} |\mfS^{-1}x|^2\mathcal{Z}_x^{-\f12}dx=-\f{(2\pi)^{\f32}}3(1+\f2{1-r^2})\mathcal{B}_2,
	\eeno
	where we use \eqref{msA} in the last step. Thus we get that 
	\begin{equation}\label{controlB1B2}
		|\mathcal{B}_1|+|\mathcal{B}_2|\leq \f{3(1-r^2)}{(2\pi)^{\f32}(3-r^2)}\Big|\int_{\R^3} \mfS^{-1}x\cdot \mb_{\mathsf{n}}dx\Big|\leq C_{r,\de} \|\mb_{\mathsf{n}}\mathcal{Z}_x^{\f{\de_2}2}\|_{L^2_x}.
	\end{equation}
	
	\medskip

	  \underline{Step 2: \textit{Estimates of $(\mathcal{B})_{13}$ and $(\mathcal{B})_{23}$.}}
	We begin with the estimate of $(\mathcal{B})_{13}$. Again by \eqref{conmabc}, in particular, $\int_{\R^3}(\mfS^{-1}x\wedge\mb+\mfS^{-1}x\wedge \mfR\mfS^{-1}x\ma)_{13}dx=0$ and using  \eqref{equality for G -1 x R G -1 x}, we have
\[ 
\int_{\R^3}(\f1{\sqrt{1-r^2}}x_1\mb_3-x_3\mb_1-\f r{\sqrt{1-r^2}}x_2x_3\ma)dx=0,
\]
and by\eqref{deofabcs}, using \eqref{equality for B G - 1 x} and the form of $\mathcal{R}$ we obtain that
	\beno
	&&\int_{\R^3}(\f1{\sqrt{1-r^2}}x_1(\mb_{\mathsf{n}})_3-x_3(\mb_{\mathsf{n}})_1)dx=\int_{\R^3}(\f1{\sqrt{1-r^2}}x_1\mb_3-x_3\mb_1)dx+(\mathcal{B})_{13}\int_{\R^3}(\f1{1-r^2}x_1^2+x_3^2)\mathcal{Z}_x^{-\f12}dx\\
	&&\int_{\R^3} \f r{\sqrt{1-r^2}}x_2x_3\ma_{\mathsf{n}}dx=\int_{\R^3} \f r{\sqrt{1-r^2}}x_2x_3\ma dx-\int_{\R^3}\f{r^2}{1-r^2}(\mathcal{B})_{13}x_2^2x_3^2\mathcal{Z}_x^{-\f12}dx.
	\eeno
	Using \eqref{msA}, these facts enable us to get that
	\beno
	\f{2(2\pi)^{\f32}}{1-r^2}(\mathcal{B})_{13}=\int_{\R^3}(\f1{\sqrt{1-r^2}}x_1(\mb_{\mathsf{n}})_3-x_3(\mb_{\mathsf{n}})_1)dx-\int_{\R^3} \f r{\sqrt{1-r^2}}x_2x_3\ma_{\mathsf{n}}dx.
	\eeno
	 Similarly, utilizing the condition that $\int_{\R^3}(\mfS^{-1}x\wedge\mb+\mfS^{-1}x\wedge \mfR\mfS^{-1}x\ma)_{23}dx=0$ and \eqref{equality for G -1 x R G -1 x}, we   obtain that
	\beno
	\f{2(2\pi)^{\f32}}{1-r^2}(\mathcal{B})_{23}=\int_{\R^3}(\f1{\sqrt{1-r^2}}x_2(\mb_{\mathsf{n}})_3-x_3(\mb_{\mathsf{n}})_2)dx+\int_{\R^3} \f r{\sqrt{1-r^2}}x_1x_3\ma_{\mathsf{n}}dx.
	\eeno
	We are led to that
	\begin{equation}\label{controlB1223}
		|(\mathcal{B})_{13}|+|(\mathcal{B})_{23}|\leq C_{r,\de}(\|\mb_{\mathsf{n}} \mathcal{Z}_x^{\f{\de_2}2}\|_{L^2_x}+\|\ma_{\mathsf{n}} \mathcal{Z}_x^{\f{\de_3}2}\|_{L^2_x}).
	\end{equation}

	\medskip

	\underline{Step 3: \textit{Estimates of $(\mathcal{B})_{12},\mathcal{C}$ and $\mathcal{D}$.}}  By \eqref{conmabc}, we have
	\begin{equation}\label{BCD3}
		\begin{aligned}
			\int_{\R^3}|\mfS^{-1}x|^2\ma dx=0,\quad \int_{\R^3}(6\mc+\mb\cdot \mfR\mfS^{-1}x)dx=0,\quad \int_{\R^3}(\mfS^{-1}x\wedge\mb+\mfS^{-1}x\wedge \mfR\mfS^{-1}x\ma)_{12}dx=0.
		\end{aligned}
	\end{equation}
	
	\noindent$\bullet$ From the first equation, due to \eqref{deofabcs}, we can derive that
	\begin{equation}\label{as1}
		\begin{aligned}
			&\int_{\R^3}|\mfS^{-1}x|^2\ma_{\mathsf{n}}dx=\f12\int_{\R^3}|\mfS^{-1}x|^2(\mathcal{R}\mfS^{-1}x,\mfS^{-1}x)\mathcal{Z}_x^{-\f12}dx-\f1{2(2\pi)^{\f32}}\int_{\R^3}(\mathcal{R}\mfS^{-1}x,\mfS^{-1}x)\mathcal{Z}_x^{-\f12}dx\\&\times\int_{\R^3}|\mfS^{-1}x|^2\mathcal{Z}_x^{-\f12}dx
			-\mathcal{C}\Big(\int_{\R^3}|x|^2|\mfS^{-1}x|^2\mathcal{Z}_x^{-\f12}dx-\f1{(2\pi)^{\f32}}\int_{\R^3}|x|^2\mathcal{Z}_x^{-\f12}dx\int_{\R^3}|\mfS^{-1}x|^2\mathcal{Z}_x^{-\f12}dx\Big)\\
			&+\mathcal{D} \Big(\int_{\R^3}|\mfS^{-1}x|^4\mathcal{Z}_x^{-\f12}dx-\f1{(2\pi)^{\f32}}(\int_{\R^3}|\mfS^{-1}x|^2\mathcal{Z}_x^{-\f12}dx)^2\Big).
		\end{aligned}
	\end{equation}
	Basic computation  together with \eqref{msA} imply that
	\beno
	&&\f1{2(2\pi)^{\f32}}\int_{\R^3}|\mfS^{-1}x|^2(\mathcal{R}\mfS^{-1}x,\mfS^{-1}x)\mathcal{Z}_x^{-\f12}dx=\f{-r(10-2r^2)}{(1-r^2)^2}(\mathcal{B})_{12},\\
	&&\f1{2(2\pi)^{\f32}}\int_{\R^3}(\mathcal{R}\mfS^{-1}x,\mfS^{-1}x)\mathcal{Z}_x^{-\f12}dx=\f{-2r}{1-r^2}(\mathcal{B})_{12},\quad \f1{(2\pi)^{\f32}}\int_{\R^3}|\mfS^{-1}x|^4\mathcal{Z}_x^{-\f12}dx=\f{12-4r^2}{(1-r^2)^2}+3,\\
	&& \f1{(2\pi)^{\f32}}\int_{\R^3}|\mfS^{-1}x|^2\mathcal{Z}_x^{-\f12}dx=\f{3-r^2}{1-r^2},\quad \f1{(2\pi)^{\f32}}\int_{\R^3}|x|^2|\mfS^{-1}x|^2\mathcal{Z}_x^{-\f12}dx=\f{15-5r^2}{1-r^2}.
	\eeno
	Substituting them into \eqref{as1}, we get that
	\begin{equation}\label{S-1x}
		 \f1{(2\pi)^{\f32}}\int_{R^3}|\mfS^{-1}x|^2\ma_{\mathsf{n}}dx=\f{-4r}{(1-r^2)^2}(\mathcal{B})_{12}-\f{2(3-r^2)}{1-r^2}\mathcal{C}+ \f{2(r^4-2r^2+3)}{(1-r^2)^2}\mathcal{D}.
	\end{equation}
	
	\noindent$\bullet$ By the second equation in \eqref{BCD3},  we have $(2\pi)^{-\f32}\int_{\R^3}\mb\cdot \mfR\mfS^{-1}xdx=-6\mathcal{C}$, by \eqref{deofabcs} and \eqref{equality for R G -1 x}, we have
	\beno
	\f1{(2\pi)^{\f32}}\int_{\R^3}\mb_{\mathsf{n}}\cdot \mfR\mfS^{-1}xdx=\f1{(2\pi)^{\f32}}\int_{\R^3}\mb\cdot \mfR\mfS^{-1}xdx-\f{(\mathcal{B})_{12}}{(2\pi)^{\f32}}\f r{1-r^2}\int_{\R^3}(x_1^2+x_2^2)\mathcal{Z}_x^{-\f12}dx.
	\eeno
From this, we deduce that
	\begin{equation}\label{v2RS}
	\f1{(2\pi)^{\f32}}\int_{\R^3}\mb_{\mathsf{n}}\cdot \mfR\mfS^{-1}xdx=-6\mathcal{C}-\f{2r}{1-r^2}(\mathcal{B})_{12}.
	\end{equation}
	
	\noindent$\bullet$ From the third equation in \eqref{BCD3} and  \eqref{deofabcs}, using  \eqref{equality for G -1 x R G -1 x} and   \eqref{equality for B G - 1 x}, we derive that 
	\ben\label{ba0}
		&&\f1{(2\pi)^{\f32}}\int_{\R^3}(\f1{\sqrt{1-r^2}}(x_1\mb_2-x_2\mb_1)-\f r{1-r^2}(x_1^2+x_2^2)\ma)dx=0,\\
	  &&\notag
		\f1{(2\pi)^{\f32}}\int_{\R^3}(\f1{\sqrt{1-r^2}}(x_1(\mb_{\mathsf{n}})_2-x_2(\mb_{\mathsf{n}})_1)dx=\f1{(2\pi)^{\f32}}\int_{\R^3}(\f1{\sqrt{1-r^2}}(x_1\mb_2-x_2\mb_1)dx+\f{2}{1-r^2}(\mathcal{B})_{12},\\
		&&\label{ba1}
	\een
	and by  \eqref{deofabcs}  
	\begin{equation}\label{ba2}
		\begin{aligned}
			&\int_{\R^3}\f r{1-r^2}(x_1^2+x_2^2)\ma_{\mathsf{n}}dx=\int_{\R^3}\f r{1-r^2}(x_1^2+x_2^2)\ma dx+\f12\int_{\R^3}\f r{1-r^2}(x_1^2+x_2^2)(\mathcal{R}\mfS^{-1}x,\mfS^{-1}x)\mathcal{Z}_x^{-\f12}dx\\
			&-\f1{2(2\pi)^{\f32}}\int_{\R^3}(\mathcal{R}\mfS^{-1}x,\mfS^{-1}x)\mathcal{Z}_x^{-\f12}dx\int_{\R^3}\f r{1-r^2}(x_1^2+x_2^2)\mathcal{Z}_x^{-\f12}dx\\
			&-\mathcal{C}\Big(\int_{\R^3_x}|x|^2\f r{1-r^2}(x_1^2+x_2^2)\mathcal{Z}_x^{-\f12}dx-\f1{(2\pi)^{\f32}}\int_{\R^3_x}|x|^2\mathcal{Z}_x^{-\f12}dx\int_{\R^3_x}\f r{1-r^2}(x_1^2+x_2^2)\mathcal{Z}_x^{-\f12}dx\Big)\\
			&+\mathcal{D} \Big(\int_{\R^3_x}\f r{1-r^2}(x_1^2+x_2^2)|\mfS^{-1}x|^2\mathcal{Z}_x^{-\f12}dx-\f1{(2\pi)^{\f32}}\int_{\R^3_x}|\mfS^{-1}x|^2\mathcal{Z}_x^{-\f12}dx\int_{\R^3_x}\f r{1-r^2}(x_1^2+x_2^2)\mathcal{Z}_x^{-\f12}dx\Big).
		\end{aligned}
	\end{equation}
In virtue of \eqref{msA}, it is not difficult to check that
	\beno
	&&\f1{2(2\pi)^{\f32}}\int_{\R^3}\f r{1-r^2}(x_1^2+x_2^2)(\mathcal{R}\mfS^{-1}x,\mfS^{-1}x)\mathcal{Z}_x^{-\f12}dx=\f{-8r^2}{(1-r^2)^2}(\mathcal{B})_{12},\\
	&&\f1{(2\pi)^{\f32}}\int_{\R^3}\f r{1-r^2}(x_1^2+x_2^2)|\mfS^{-1}x|^2\mathcal{Z}_x^{-\f12}dx=\f{r(10-2r^2)}{(1-r^2)^2}.
	\eeno
	Plugging them into \eqref{ba2} will give that
	\begin{equation}\label{ba3}
		\f1{(2\pi)^{\f32}}\int_{\R^3}\f r{1-r^2}(x_1^2+x_2^2)\ma_{\mathsf{n}}dx=\f1{(2\pi)^{\f32}}\int_{\R^3}\f r{1-r^2}(x_1^2+x_2^2)\ma dx-\f{4r^2}{(1-r^2)^2}(\mathcal{B})_{12}-\f{4r}{1-r^2}\mathcal{C}+\f{4r}{(1-r^2)^2}\mathcal{D}.
	\end{equation}	
	Patching together  \eqref{ba0}, \eqref{ba1} and \eqref{ba3}, we  have
	\ben\label{xv12}
		&&\f1{(2\pi)^{\f32}}\int_{\R^3}(\f1{\sqrt{1-r^2}}(x_1(\mb_{\mathsf{n}})_2-x_2(\mb_{\mathsf{n}})_1)dx-\f1{(2\pi)^{\f32}}\int_{\R^3}\f r{1-r^2}(x_1^2+x_2^2)\ma_{\mathsf{n}}dx\nonumber\\
		&&=\f{2r^2+2}{(1-r^2)^2}(\mathcal{B})_{12}+\f{4r}{1-r^2}\mathcal{C}-\f{4r}{(1-r^2)^2}\mathcal{D}.
	\een
	
	From \eqref{S-1x}, \eqref{v2RS} and \eqref{xv12}, we get the following system:
	\begin{equation*}
		\underbrace{\begin{pmatrix}
			\f{-4r}{(1-r^2)^2} & -\f{2(3-r^2)}{1-r^2} & \f{2(r^4-2r^2+3)}{(1-r^2)^2}\\
			-\f{2r}{1-r^2} & -6& 0\\
			\f{2r^2+2}{(1-r^2)^2} & \f{4r}{1-r^2} & -\f{4r}{(1-r^2)^2}
		\end{pmatrix}}_{:=\mathsf{A}}\begin{pmatrix}(\mathcal{B})_{12}\\ \mathcal{C}\\\mathcal{D}\end{pmatrix}=R.H.S.,
	\end{equation*}
	where $R.H.S.$ is bounded by $C_{r,\de}(\|\mb_{\mathsf{n}}\mathcal{Z}_x^{\f{\de_2}2}\|_{L^2_x}+\|\ma_{\mathsf{n}}\mathcal{Z}_x^{\f{\de_3}2}\|_{L^2_x})$. It is not difficult to verify that $\det(\mathsf{A})=\f{8(9-r^4)}{(1-r^2)^3}>0$. Since $r\in(0,1)$, this yields that
	\begin{equation}\label{controlB12CD}
		|(\mathcal{B})_{12}|+|\mathcal{C}|+|\mathcal{D}|\leq C_{r,\de}(\|\mb_{\mathsf{n}}\mathcal{Z}_x^{\f{\de_2}2}\|_{L^2_x}+\|\ma_{\mathsf{n}}\mathcal{Z}_x^{\f{\de_3}2}\|_{L^2_x}).
	\end{equation}	
	We complete the proof of the lemma thanks to  \eqref{controlB1B2}, \eqref{controlB1223} and \eqref{controlB12CD}. 
\end{proof}

\subsubsection{Control of non-zero mode}  To control the non-zero mode, we first rewrite the equation \eqref{pertubNSlandaucauchy} as
\ben\label{nonli}
&&\partial_t f + \mathscr{C}_l(t)\left( \mfS v \cdot \nabla_x -  \mfS x \cdot \nabla_v + \mfR x \cdot \nabla_x -\mfR v \cdot \nabla_v\right)f=  \mathscr{C}_1e^{-\f12|x|^2}L(f)+g,
\een
where  $L(f):=Q(\mu, f)+Q(f,\mu)$ with $\mu=(2\pi)^{-\f32}e^{-\f12|v|^2}$ and $g:=\mathscr{C}_2Q(f,f)$. We recall that   $f$ satisfies the conservation laws 
\begin{equation}\label{conlaw2}
	\int_{\R^3_x\times\R^3_v}(1,\mfS^{-1}x,v,|\mfS^{-1}x|^2,\mfS^{-1}x\cdot v,|v+\mfR \mfS^{-1}x|^2,\mfS^{-1}x\wedge (v+\mfR\mfS^{-1}x))^\tau f(t,x,v)dvdx=0.
\end{equation}

Recalling the function space $\mathcal{H}^{N,\de}_{x}$ and $\mathcal{H}^{N,\de}_{x}L^2_v$ defined in \eqref{xepwspace} and \eqref{Pfspace}, we want to prove
\begin{thm}\label{macroestimate}
Let $f=f(t,x,v)$ be a solution to \eqref{nonli}. Then  for any $N\in \N^+$ and $0<\de<\f1{16}$,  there exists a energy functional $\mathcal{F}_{N,\de}$ verifying  $|\mathcal{F}_{N,\delta}|\leq C_{\de,N}\|\mathcal{P}_v^{10}f\|^2_{\mathcal{H}^{N,2\de}_{x}L^2_v}$. Moreover,  it holds that
\begin{equation}
	\f d{dt} \mathcal{F}_{N,\delta}+\|\mP f\|^2_{\mathcal{H}^{N,\de/2}_{x}L^2_v}\leq C_{\de,N}\|\mathcal{P}_v^{10}\mathbb{(I-P)}f\|^2_{\mathcal{H}^{N,2\de}_{x}L^2_v}+C_{\de,N}\sum_{|\al|\leq N-1}\sum_{i=1}^{13}\int_{\R^3_x}\mathcal{Z}_x^{4\de}(\pa^\al_xg,e_i)^2_{L^2_v}dx,
\end{equation}
where $\mathcal{Z}_x^\a$ is defined in \eqref{exponentialfunction} and
\begin{equation}\label{ei}
	\begin{aligned}
&e_1=1,e_2=v_1,e_3=v_2,e_4=v_3,e_5=v_1^2,e_6=v_2^2,e_7=v_3^2,\\
&e_8=v_1v_2,e_9=v_1v_3,e_{10}=v_2v_3,e_{11}=v_1|v|^2,e_{12}=v_2|v|^2,e_{13}=v_3|v|^2.
	\end{aligned}
\end{equation}
\end{thm}

In the sequel, we will provide a detailed proof of Theorem  \ref{macroestimate} in the case of $N=1$. We begin with
 the upper bound for $\sfT_{k1}$ and $\sfT_{l2}$, where $1\le k\le 5$ and $l=4,5$(see the definitions in 
\eqref{defofT1} and \eqref{defofT2}). 
\begin{lem}\label{estimateofT}
For any $0<\de<\f18$ and $\al\in\Z^3_+$, we have 
\begin{equation*}
\begin{aligned}
&\sum_{k=1}^5\|\pa^\al_x(\sfT_{k1}\mathcal{Z}_x^{-\f12})\mathcal{Z}_x^{\f{\de}2}\|^2_{L^2_x}\leq C_{\de,\al}\Big(\|\mathcal{P}_v^{10}(\mathbb{I-P})f\|^2_{\mathcal{H}^{|\al|+1,3\de/5}_{x}L^2_v}+\sum_{i=1}^{13}\int_{\R^3_x}\mathcal{Z}_x^{\de}(\pa^\al_xg,e_i)^2_{L^2_v}dx\Big),\\
&\sum_{l=4,5}\|\pa^\al_x(\sfT_{l2}\mathcal{Z}_x^{-\f12})\mathcal{Z}_x^{\f\de2}\|^2_{L^2_x}\leq C_{\de,\al}\|\mathcal{P}_v^{10}(\mathbb{I-P})f\|^2_{\mathcal{H}^{|\al|,3\de/5}_{x}L^2_v}.
\end{aligned}
\end{equation*}
\end{lem}	
\begin{proof} We only give the proof to $\sfT_{51}$ with $\al\in \Z^3_+$ since the similar argument can be applied for $\sfT_{k1}$ and $\sfT_{l2}$. For $\sfT_{51}$, on one hand, by integration by parts, we have 
\begin{equation*}
	\begin{aligned}
		&\mathscr{C}_l(t)\pa^\al_x\big((\mfS x\cdot\na_v )\big((\mathbb{I}-\mP)f\big),v_i(|v|^2-5)\big)_{L^2_v}\\
		=&-\mathscr{C}_l(t)\big(\sum_{j=1}^3\pa^{\om_j}_x(\mfS x)\pa^{\al-\om_j}_x\big((\mathbb{I}-\mP)f\big)+\mfS x\pa^\al_x\big((\mathbb{I}-\mP)f\big),\na_v(v_i(|v|^2-5))\big)_{L^2_v},\\
		&\mathscr{C}_l(t)\pa^\al_x\big((\mfR v\cdot\na_v )\big((\mathbb{I}-\mP)f\big),v_i(|v|^2-5)\big)_{L^2_v}=-\mathscr{C}_l(t)\sum_{j=1}^3\big(\pa^\al_x\big((\mathbb{I}-\mP)f\big),\pa_{v_j}((\mfR v)_jv_i(|v|^2-5))\big)_{L^2_v},
	\end{aligned}
\end{equation*}
where $i=1,2,3$, $\om_1=(1,0,0),\om_2=(0,1,0)$ and $\om_3=(0,0,1)$.
Recall that $\mathscr{C}_l(t)$ is bounded from above and below. It yields that
\begin{equation}\label{T511}
	\begin{aligned}
		\|\mathscr{C}_l(t)\pa^\al_x((\mfS x\cdot\na_v+\mfR v\cdot\na_v)((\mathbb{I}-\mP)f),v(|v|^2-5)_{L^2_v} \mathcal{Z}_x^{\f\de 2} \|_{L^2_x}\leq C_{\de,\al}\|\mathcal{P}_v^{10}(\mathbb{I-P})f\|^2_{\mathcal{H}^{|\al|,3\de/5}_{x}L^2_v}.
	\end{aligned}
\end{equation}

On the other hand, note that
\begin{equation*}
	\begin{aligned}
&(\mathscr{C}_1\pa^\al_x\big(L((\mathbb{I}-\mP)f)\big),v(|v|^2-5))_{L^2_v}\ls \mathscr{C}_1\|\mathcal{P}_v^{10}\pa^\al_x((\mathbb{I}-\mP)f)\|_{L^2_v};\\
&\mathscr{C}_l(t)\pa^\al_x((\mfS v\cdot\na_x+\mfR x\cdot\na_x)((\mathbb{I}-\mP)f),v(|v|^2-5))_{L^2_v}\leq \mathscr{C}_l(t)\|\mathcal{P}_x\mathcal{P}_v^{6}\pa^\al_x\na_x((\mathbb{I}-\mP)f)\|_{L^2_v},
	\end{aligned}
\end{equation*}
where we use fact that by integration by part, for $l\in\R$, it holds that \beno	|(L(f),\mathcal{P}_v^l)_{L^2_v}|+|(L(f),v\mathcal{P}_v^l)_{L^2_v}|\ls\|\mathcal{P}_v^{l+7}f\|_{L^2_v}+\|\mathcal{P}_v^{10}f\|_{L^2_v}.	\eeno  These imply that
\begin{equation}\label{T512}
	\begin{aligned}
&\|\mathscr{C}_1\pa^\al_x(L((\mathbb{I}-\mP)f),v(|v|^2-5))_{L^2_v}\mathcal{Z}_x^{\f\de 2}\|_{L^2_x}\leq C_{\de,\al} \|\mathcal{P}_v^{10}(\mathbb{I}-\mP)f\|_{\mathcal{H}^{|\al|,3\de/5}_{x}L^2_v};\\	&\|\mathscr{C}_l(t)\pa^\al_x((\mfS v\cdot\na_x+\mfR x\cdot\na_x)((\mathbb{I}-\mP)f),v(|v|^2-5)_{L^2_v} \mathcal{Z}_x^{\f\de 2} \|_{L^2_x}\leq C_{\de,\al} \|\mathcal{P}_v^{10}(\mathbb{I}-\mP)f\|_{\mathcal{H}^{|\al|+1,3\de/5}_{x}L^2_v}.
	\end{aligned}
\end{equation}

It is obvious that
$\|(\pa^\al_xg,v(|v|^2-5))_{L^2_v}\mathcal{Z}_x^{\f\de 2}\|^2_{L^2_x}\leq 5\sum_{i=1}^{13}\int_{\R^3_x}\mathcal{Z}_x^{\de }(\pa^\al_xg,e_i)^2_{L^2_v}dx$. From this together with  \eqref{T511} and \eqref{T512}, we get the desired result for $\sfT_{51}$. It ends  the proof.
\end{proof}

\begin{lem}\label{mapf}
	Recall that $\mathbb{P}f$ and $(\ma,\mb,\mc)$ are defined in  \eqref{deofP}. For any $\de>0$ and $N\in \Z_+$, it holds that
	\beno
	\|\mathbb{P}f\|_{\mathcal{H}^{N,\de}_xL^2_v}\sim 	\|(\ma,\mb,\mc)\|_{\mathcal{H}^{N,\de}_x}\ls\|\mathcal{P}_v^5f\|_{\mathcal{H}^{N,\de}_xL^2_v}.
	\eeno
\end{lem}
\begin{proof} Direct computation yields that
	\begin{equation}\label{pf}
		\begin{aligned}
			\|\mathbb{P}f\|^2_{\mathcal{H}^{N,\de}_xL^2_v}&=\sum_{|\al|\leq N}\int_{\R^3_x\times \R^3_v}(\pa^\al_x\ma+\sum_{i=1}^3\pa^\al_x\mb_iv_i+\pa^\al_x\mc(|v|^2-3))^2\mu^2\mathcal{Z}_x^{2\de}dvdx\\
			&=\sum_{|\al|\leq N}\Big(\int_{\R^3_x}|\pa^\al_x\ma|^2\mathcal{Z}_x^{2\de}dx\int_{\R^3_v}\mu^2dv+\sum_{i=1}^3\int_{\R^3_x}|\pa^\al_x\mb_i|^2\mathcal{Z}_x^{2\de}dx\int_{\R^3_v}v_i^2\mu^2dv\\
			+&\int_{\R^3_x}|\pa^\al_x\mc|^2\mathcal{Z}_x^{2\de}dx\int_{\R^3_v}(|v|^2-3)^2\mu^2dv+2\int_{\R^3_x}(\pa^\al_x\ma) (\pa^\al_x \mc)\mathcal{Z}_x^{2\de}dx\int_{\R^3_v}(|v|^2-3)\mu^2dv\Big).
		\end{aligned}
	\end{equation}
	By \eqref{msA}, we can compute that 
	\beno
	\int_{\R^3_v}\mu^2 dv=\f{\sqrt{2}}{4(2\pi)^{\f32}},~~\int_{\R^3_v}(3-|v|^2)\mu^2 dv=\f{3\sqrt{2}}{8(2\pi)^{\f32}},~~\int_{\R^3_v}(3-|v|^2)^2\mu^2 dv=\f{15\sqrt{2}}{16(2\pi)^{\f32}},
	\eeno
	which imply that 
	\begin{multline*}
		2\Big|\int_{\R^3_x}(\pa^\al_x\ma) (\pa^\al_x \mc)\mathcal{Z}_x^{2\de}dx\int_{\R^3_v}(|v|^2-3)\mu^2dv\Big|\\
		\leq (3/5)^{\f12}\Big(\int_{\R^3_x}|\pa^\al_x\ma|^2\mathcal{Z}_x^{2\de}dx\int_{\R^3_v}\mu^2dv+\int_{\R^3_x}|\pa^\al_x\mc|^2\mathcal{Z}_x^{2\de}dx\int_{\R^3_v}(|v|^2-3)^2\mu^2dv\Big).
	\end{multline*}
	From this together with \eqref{pf}, we can deduce that $\|\mathbb{P}f\|^2_{\mathcal{H}^{N,\de}_xL^2_v}\sim \|(\ma,\mb,\mc)\|^2_{\mathcal{H}^{N,\de}_x}.$
	
	On the other hand, by definition \eqref{deofP} and Minkowski inequality, we have  
	\beno
	\|\ma\|^2_{\mathcal{H}^{N,\de}_x}=\sum_{|\al|\leq N}\int_{\R^3_x}\Big(\int_{\R^3_v} \pa^\al_x fdv\Big)^2\mathcal{Z}_x^{2\de}dx\leq\sum_{|\al|\leq N}\Big( \int_{\R^3_v}\Big(\int_{\R^3_x}|\pa^\al_xf|^2\mathcal{Z}_x^{2\de}dx\Big)^{\f12}dv\Big)^2
	\ls \|\mathcal{P}_v^5\pa^\al_x f\|^2_{\mathcal{H}^{N,\de}_xL^2_v}.
	\eeno
	It is easy to see that the same bound also holds for $\mb$ and $\mc$ and we complete the proof.
\end{proof}

Next we consider the estimates of $(\ma_{\mathsf{n}},\mb_{\mathsf{n}},\mc_{\mathsf{n}})$. We first have 
\begin{lem}\label{cs}
	  Let $\de_1=2\de_3,\de_2=\f54\de_3$ with $0<\de_3<\f18$. Then there exists a positive constant  $C_{\de_3}$ such that
\beno 
	 &&\f d{dt}(\mfS\na_x(\mc_{\mathsf{n}}\mathcal{Z}_x^{\f12}),-\sfT_{52}\mathcal{Z}_x^{-(1-\de_1)})_{L^2_x}\leq -\f{\mathscr{C}_l(t)}2\|\mfS\na_x(\mc_{\mathsf{n}}\mathcal{Z}_x^{\f1 2 })\mathcal{Z}_x^{-\f{1-\de_1}2}\|^2_{L^2_x}\\ &&+C_{\de_3}\Big(\sum_{i=1}^{13}\int_{\R^3_x}\mathcal{Z}_x^{2\de_1}(g,e_i)^2_{L^2_v}dx +\|\mathcal{P}_v^{10}(\mathbb{I}-\mP)f\|^2_{\mathcal{H}^{1,\de_1}_{x}L^2_v}\Big)+C_{\de_3}\|\mathcal{P}_v^{10}(\mathbb{I}-\mP)f\|_{\mathcal{H}^{1,\de_1}_{x}L^2_v} \|\mathbb{P}f\|_{\mathcal{H}^{1,\de_3/2}_{x}L^2_v}.
			\eeno
\end{lem}
\begin{proof} Thanks to \eqref{macroeq} and the definition of $\mc_{\mathsf{n}}$, we deduce that
	\begin{equation*}
		\begin{aligned} 
	&\mathscr{C}_l(t)\mfS\na_x(\mc_{\mathsf{n}}\mathcal{Z}_x^{\f12})=\mathscr{C}_l(t)\mfS\na_x(\mc \mathcal{Z}_x^{\f12})=\sfT_{51}+\pa_t \sfT_{52},\\
	&\pa_t(\mc \mathcal{Z}_x^{\f12})=-\f13 \mathscr{C}_l(t)\mfS\na_x\cdot (\mb \mathcal{Z}_x^{\f12})-\mathscr{C}_l(t)\mfR x\cdot\na_x (\mc \mathcal{Z}_x^{\f12})+\sfT_{31},
		\end{aligned}
	\end{equation*}
which implies that
	\begin{equation*}
		\begin{aligned}
	&\f d{dt}(\mfS\na_x(\mc_{\mathsf{n}}\mathcal{Z}_x^{\f12}),-\sfT_{52}\mathcal{Z}_x^{-(1-\de_1)})_{L^2_x} 
=-\mathscr{C}_l(t)\|\mfS\na_x(\mc_{\mathsf{n}}\mathcal{Z}_x^{\f12})\mathcal{Z}_x^{-\f{1-\de_1}2}\|^2_{L^2_x}+(\mfS\na_x(\mc_{\mathsf{n}}\mathcal{Z}_x^{\f12}),\sfT_{51}\mathcal{Z}_x^{-(1-\de_1)})_{L^2_x}\\
&+\Big(\mfS\na_x\Big(-\f13 \mathscr{C}_l(t)\mfS\na_x\cdot (\mb \mathcal{Z}_x^{\f12})-\mathscr{C}_l(t)\mfR x\cdot\na_x (\mc \mathcal{Z}_x^{\f12})+\sfT_{31}\Big),-\sfT_{52}\mathcal{Z}_x^{-(1-\de_1)}\Big)_{L^2_x}.
\end{aligned}
\end{equation*}

 For the second term in the r.h.s., by Cauchy-Schwarz inequality and Lemma \ref{estimateofT}, we   derive that  
\begin{equation*}
\begin{aligned}
(\mfS\na_x(\mc_{\mathsf{n}}\mathcal{Z}_x^{\f12}),\sfT_{51}\mathcal{Z}_x^{-(1-\de_1)})_{L^2_x}&\leq \f{\mathscr{C}_l(t)}2 \|\mfS\na_x(\mc_{\mathsf{n}}\mathcal{Z}_x^{\f1 2 })\mathcal{Z}_x^{-\f{1-\de_1}2}\|^2_{L^2_x}\\&+C_{\de}\Big(\|\mathcal{P}_v^{10}(\mathbb{I}-\mP)f\|^2_{\mathcal{H}_{x}^{1,\de_1}L^2_v}+\sum_{i=1}^{13}\int_{\R^3_x}\mathcal{Z}_x^{\de_1}(g,e_i)^2_{L^2_v}dx\Big).
\end{aligned}
\end{equation*}
For the remainder term, using integration by parts, we have
\beno 
&&\Big(\mfS\na_x\Big(-\f13 \mathscr{C}_l(t)\mfS\na_x\cdot (\mb \mathcal{Z}_x^{\f12})-\mathscr{C}_l(t)\mfR x\cdot\na_x (\mc \mathcal{Z}_x^{\f12})+\sfT_{31}\Big),-\sfT_{52}\mathcal{Z}_x^{-(1-\de_1)}\Big)_{L^2_x}\\
&&=(-\f13\mathscr{C}_l(t)\mfS\na_x\cdot (\mb \mathcal{Z}_x^{\f12})-\mathscr{C}_l(t)\mfR x\cdot\na_x (\mc \mathcal{Z}_x^{\f12})+\sfT_{31},(\mfS\na_x-2(1-\de_1)\mfS x)\sfT_{52}\mathcal{Z}_x^{-(1-\de_1)})_{L^2_x}\\
&&\leq~C_{\de_1}(\|\mathcal{P}_x^{-1}\mfS\na_x(\mb \mathcal{Z}_x^{\f12})\mathcal{Z}_x^{-\f{1-\de_3}2}\|_{L^2_x}+\|\mathcal{P}_x^{-2}\mfR x\cdot\na_x (\mc \mathcal{Z}_x^{\f12})\mathcal{Z}_x^{-\f{1-\de_3}2} \|_{L^2_x})\\
&&\times(\|\mathcal{P}_x^2(\na_x\sfT_{52})\mathcal{Z}_x^{-\f{1-2\de_1+\de_3}2}\|_{L^2_x}
+\|\mathcal{P}_x^3\sfT_{52}\mathcal{Z}_x^{-\f{1-2\de_1+\de_3}2}\|_{L^2_x})+\|\sfT_{31}\mathcal{Z}_x^{-\f{1-\de_1}2}\|_{L^2_x}(\|\mathcal{P}_x\sfT_{52}\mathcal{Z}_x^{-\f{1-\de_1}2}\|_{L^2_x}\\
&&+\|(\na_x \sfT_{52})\mathcal{Z}_x^{-\f{1-\de_1}2}\|_{L^2_x})
\leq ~C_{\de_1}\Big(\|\mathcal{P}_v^{10}(\mathbb{I}-\mP)f\|^2_{\mathcal{H}^{1,\de_1}_{x}L^2_v}+\sum_{i=1}^{13}\int_{\R^3_x}\mathcal{Z}_x^{2\de_1}(g,e_i)^2_{L^2_v}dx\Big)\\
&&+C_{\de_1} \|\mathcal{P}_v^{10}(\mathbb{I}-\mP)f\|_{\mathcal{H}^{1,\de_1}_{x}L^2_v}\|\mathbb{P}f\|_{\mathcal{H}^{1,\de_3/2}_{x}L^2_v},
\eeno
where we use the facts that $\f35(2\de_1-\de_3)=\f9{10}\de_1<\de_1$, $2\de_1-\de_3<2\de_1$ and 
\[\|\mathcal{P}_x^{-1}\mfS\na_x(\mb \mathcal{Z}_x^{\f12})\mathcal{Z}_x^{-\f{1-\de_3}2}\|_{L^2_x}+\|\mathcal{P}_x^{-2}\mfR x\cdot\na_x (\mc \mathcal{Z}_x^{\f12})\mathcal{Z}_x^{-\f{1-\de_3}2} \|_{L^2_x} 
\leq C_{\de_3} \|\mathbb{P}f\|_{\mathcal{H}^{1,\de_3/2}_{x}L^2_v},\] thanks to Lemma \ref{mapf}. We complete the proof of this lemma.
\end{proof}

  Next we give the  estimate to $\mb_{\mathsf{n}}$.
\begin{lem}\label{bs}
	  Let $\de_1=2\de_3,\de_2=\f54\de_3$ with $0<\de_3<\f18$. Then there exists a positive constant  $C_{\de_3}$ such that
\begin{equation*}
	\begin{aligned}
	&\f d{dt}(\na^{sym}_\mfS(\mb_{\mathsf{n}}\mathcal{Z}_x^{\f12}),(-T_{42}+\mc_{\mathsf{n}} \mathcal{Z}_x^{\f12})\mathcal{Z}_x^{-(1-\de_2)})_{L^2_x}\leq -\f{\mathscr{C}_l(t)}2\|\na^{sym}_\mfS(\mb_{\mathsf{n}}\mathcal{Z}_x^{\f12})\mathcal{Z}_x^{-\f{1-\de_2}2}\|^2_{L^2_x}\\
	+&C_{\de_3}
	\|\mfS\na_x(\mc_{\mathsf{n}}\mathcal{Z}_x^{\f12})\mathcal{Z}_x^{-\f{1-\de_1}2}\|^2_{L^2_x}+C_{\de_3}(\|\mathcal{P}_v^{10}(\mathbb{I}-\mP)f\|^2_{\mathcal{H}^{1,\de_1}_{x}L^2_v}+\sum_{i=1}^{13}\int_{\R^3_x}\mathcal{Z}_x^{2\de_1}(g,e_i)^2_{L^2_v}dx)\\
	+&C_{\de_3}(\|\mfS\na_x(\mc_{\mathsf{n}}\mathcal{Z}_x^{\f12})\mathcal{Z}_x^{-\f{1-\de_1}2}\|_{L^2_x}+\|\mathcal{P}_v^{10}(\mathbb{I}-\mP)f\|_{\mathcal{H}^{1,\de_1}_{x}L^2_v})\|\mathbb{P}f\|_{\mathcal{H}^{1,\de_3/2}_{x}L^2_v}.
	\end{aligned}
\end{equation*}
\end{lem}

\begin{proof}
	By \eqref{deofabcs}, we first recall that $\pa_t(\mc \mathcal{Z}_x^{\f12})=\pa_t(\mc_{\mathsf{n}}\mathcal{Z}_x^{\f12})+\pa_t\msA(\mc \mathcal{Z}_x^{\f12})$. Thus by \eqref{macroeq} and \eqref{equation for A G -1 x and A A G^-1 x}, we have
	\[
	\pa_t \msA(\mc \mathcal{Z}_x^{\f12})=-\f13\mathscr{C}_l(t) \msA(\mfS\na_x \cdot (\mb \mathcal{Z}_x^{\f12}))+\msA(\sfT_{31}).
	\]
	From these together with \eqref{macroeq}, we have 
	\beno 
	\mathscr{C}_l(t)\na^{sym}_\mfS(\mb \mathcal{Z}_x^{\f12})
	&=&-\mathscr{C}_l(t)\mfR x\cdot\na_x (\mc \mathcal{Z}_x^{\f12})\mathbb{I}_{3\times 3}+\sfT_{41}+\pa_t \sfT_{42}-\pa_t(\mc_{\mathsf{n}}\mathcal{Z}_x^{\f12})\mathbb{I}_{3\times 3}\\
	&&+\f13\mathscr{C}_l(t)\msA(\mfS\na_x \cdot (\mb \mathcal{Z}_x^{\f12}))\mathbb{I}_{3\times 3}-\msA(\sfT_{31})\mathbb{I}_{3\times 3},
	\eeno
which implies that 
\ben
&&\f d{dt}(\na^{sym}_\mfS(\mb_{\mathsf{n}}\mathcal{Z}_x^{\f12}),(-\sfT_{42}+\mc_{\mathsf{n}} \mathcal{Z}_x^{\f12}\mathbb{I}_{3\times 3})\mathcal{Z}_x^{-(1-\de_2)})_{L^2_x} 
=\big(\na^{sym}_\mfS(\mb_{\mathsf{n}}\mathcal{Z}_x^{\f12}),(-\mathscr{C}_l(t)\na^{sym}_\mfS(\mb \mathcal{Z}_x^{\f12})\nonumber\\
&&+\f13\mathscr{C}_l(t)\msA( \mfS\na_x \cdot (\mb \mathcal{Z}_x^{\f12}))\mathbb{I}_{3\times 3})\mathcal{Z}_x^{-(1-\de_2)}\big)_{L^2_x}+\big(\na^{sym}_\mfS(\mb_{\mathsf{n}}\mathcal{Z}_x^{\f12}),(-\mathscr{C}_l(t)\mfR x\cdot \na_x(\mc_{\mathsf{n}}\mathcal{Z}_x^{\f12})\mathbb{I}_{3\times 3}\nonumber\\
&&+\sfT_{41}-\msA(\sfT_{31})\mathbb{I}_{3\times 3})\mathcal{Z}_x^{-(1-\de_2)}\big)_{L^2_x}+(\na^{sym}_\mfS(\pa_t\mb_{\mathsf{n}}\mathcal{Z}_x^{\f12}),(-\sfT_{42}+\mc_{\mathsf{n}}\mathcal{Z}_x^{\f12}\mathbb{I}_{3\times 3})\mathcal{Z}_x^{-(1-\de_2)})_{L^2_x}.\label{bs1}
\een
For the first term in  the right-hand side, by \eqref{skew and symmetric of b n  Z x 1 2}, we can easily derive that
\begin{equation}\label{bs2}
	\begin{aligned}
	&\big(\na^{sym}_\mfS(\mb_{\mathsf{n}}\mathcal{Z}_x^{\f12}),(-\mathscr{C}_l(t)\na^{sym}_\mfS(\mb \mathcal{Z}_x^{\f12})+\f13\mathscr{C}_l(t)\msA( \mfS\na_x \cdot (\mb \mathcal{Z}_x^{\f12}))\mathbb{I}_{3\times 3})\mathcal{Z}_x^{-(1-\de_2)}\big)_{L^2_x}\\
	=&-\mathscr{C}_l(t)\|\na^{sym}_\mfS(\mb_{\mathsf{n}}\mathcal{Z}_x^{\f12})\mathcal{Z}_x^{-\f{1-\de_2}2}\|^2_{L^2_x}.
\end{aligned}
\end{equation}
	For the second term, notice that $|\msA(\sfT_{31})|\ls \int_{\R^3}|\sfT_{31}|\mathcal{Z}_x^{-\f12}dx\leq C_{\delta_3}\|\sfT_{31}\mathcal{Z}_x^{-\f{1-\de_2}2}\|_{L^2_x}$, 
	 then Cauchy-Schwarz inequality and Lemma \ref{estimateofT} will yield that
\begin{equation}\label{bs3}
\begin{aligned}
&\Big|\big(\na^{sym}_\mfS(\mb_{\mathsf{n}}\mathcal{Z}_x^{\f12}),(-\mathscr{C}_l(t)\mfR x\cdot \na_x(\mc_{\mathsf{n}}\mathcal{Z}_x^{\f12})\mathbb{I}_{3\times 3}+\sfT_{41}-\msA(\sfT_{31})\mathbb{I}_{3\times 3})\mathcal{Z}_x^{-(1-\de_2)}\big)_{L^2_x}\Big|\\
\leq&\f{\mathscr{C}_l(t)}2\|\na^{sym}_\mfS(\mb_{\mathsf{n}}\mathcal{Z}_x^{\f12})\mathcal{Z}_x^{-\f{1-\de_2}2}\|^2_{L^2_x}+C_{\delta_3}\|\mathcal{P}_x\mfS\na_x(\mc_{\mathsf{n}}\mathcal{Z}_x^{\f12})\mathcal{Z}_x^{-\f{1-\de_2}2}\|^2_{L^2_x}\\
&+C_{\delta_3}(\|\mathcal{P}_v^{10}(\mathbb{I}-\mP)f\|^2_{\mathcal{H}^{1,\de_2}_{x}L^2_v}+\sum_{i=1}^{13}\int_{\R^3_x}\mathcal{Z}_x^{\de_2}(g,e_i)^2_{L^2_v}dx).
\end{aligned}
\end{equation}
	For the third term, by integration by parts, it holds that
\begin{equation}\label{bs4}
	\begin{aligned}
	&(\na^{sym}_\mfS(\pa_t\mb_{\mathsf{n}}\mathcal{Z}_x^{\f12}),(-\sfT_{42}+\mc_{\mathsf{n}}\mathcal{Z}_x^{\f12}\mathbb{I}_{3\times 3})\mathcal{Z}_x^{-(1-\de_2)})_{L^2_x}
	\leq \|\mathcal{P}_x^{-1}\pa_t(\mb_{\mathsf{n}}\mathcal{Z}_x^{\f12})\mathcal{Z}_x^{-\f{1-\de_3}2}\|_{L^2_x}\\
	&\times(\|\mathcal{P}^2_x\sfT_{42}\mathcal{Z}_x^{-\f{1-2\de_2+\de_3}2}\|_{L^2_x}+\|\mathcal{P}_x(\mfS\na_x\sfT_{42})\mathcal{Z}_x^{-\f{1-2\de_2+\de_3}2}\|_{L^2_x}+\|\mathcal{P}_x\mfS\na_x(\mc_{\mathsf{n}}\mathcal{Z}_x^{\f12})\mathcal{Z}_x^{-\f{1-2\de_2+\de_3}2}\|_{L^2_x}\\
	&+\|\mathcal{P}^2_x\mc_{\mathsf{n}}\mathcal{Z}_x^{\f{2\de_2-\de_3}2}\|_{L^2_x})
	\leq C_{\delta_3}\|\mathcal{P}_x^{-1}\pa_t(\mb_{\mathsf{n}}\mathcal{Z}_x^{\f12})\mathcal{Z}_x^{-\f{1-\de_3}2}\|_{L^2_x}\\
	&\times(\|\mfS\na_x(\mc_{\mathsf{n}}\mathcal{Z}_x^{\f12})\mathcal{Z}_x^{-\f{1-\de_1}2}\|_{L^2_x}+\|\mathcal{P}_v^{10}(\mathbb{I}-\mP)f\|_{\mathcal{H}^{1,\de_1}_{x}L^2_v}),
\end{aligned}
\end{equation}
where we use Lemma \ref{estimateofT}, the fact that $2\de_2-\de_3<\de_1$ and Poincar\'{e} inequality \eqref{PKIforabcs} for $\mc_{\mathsf{n}}$ in the final step. Finally,  following  \eqref{deofabcs}  and \eqref{macroeq}, we are led to that  
\begin{equation}\label{bs5}
	\begin{aligned}
	&\|\mathcal{P}_x^{-1}\pa_t(\mb_{\mathsf{n}}\mathcal{Z}_x^{\f12})\mathcal{Z}_x^{-\f{1-\de_3}2}\|^2_{L^2_x}\leq C_{\de_3}(\|(\ma,\mb,\mc)\|^2_{\mathcal{H}_x^{1,\de_3/2}}+\|\sfT_{21}\mathcal{Z}_x^{-\f{1-\de_3}2}\|^2_{L^2_x})\\
	\leq&C_{\delta_3}\|\mathbb{P}f\|^2_{\mathcal{H}^{1,\de_3/2}_{x}L^2_v}+C_{\de_3}(\|\mathcal{P}_v^{10}(\mathbb{I}-\mP)f\|^2_{\mathcal{H}^{1,\de_1}_{x}L^2_v}+\sum_{i=1}^{13}\int_{\R^3_x}\mathcal{Z}_x^{2\de_1}(g,e_i)^2_{L^2_v}dx).
\end{aligned}
\end{equation}
	
Putting together \eqref{bs1},\eqref{bs2},\eqref{bs3},\eqref{bs4} and \eqref{bs5}, we can conclude the desired result  and then end the proof of this lemma.
\end{proof}

 Finally, we give the estimate for $\ma_{\mathsf{n}}$.
\begin{lem}\label{as}
	Let $\de_1=2\de_3,\de_2=\f54\de_3$ with $0<\de_3<\f18$. Then there exists a positive constant  $C_{\de_3}$ such that
\beno
	&&\f d{dt}\bigg[\big(\mfS\na_x(\ma_{\mathsf{n}}\mathcal{Z}_x^{\f12}),(\sfT_{52}+\mb_{\mathsf{n}}\mathcal{Z}_x^{\f12})\mathcal{Z}_x^{-(1-\de_3)}\big)_{L^2_x} 
	 -\big(\ma_{\mathsf{n}} \mathcal{Z}_x^{\f12},((\msA(\sfT_{42}))_{32}x_1x_3-(\msA(\sfT_{42}))_{31}x_2x_3)\mathcal{Z}_x^{-(1-\de_3)}\big)_{L^2_x}\\
	&&\times2(1-\de_3) r(\sqrt{1-r^2}-\f1{\sqrt{1-r^2}}) \bigg]\leq -\f{\mathscr{C}_l(t)}2\|\mfS\na_x(\ma_{\mathsf{n}}\mathcal{Z}_x^{\f12})\mathcal{Z}_x^{-\f{1-\de_3}2}\|^2_{L^2_x}+C_{\delta_3}\big(\|\mfS\na_x(\mc_{\mathsf{n}}\mathcal{Z}_x^{\f1 2 })\mathcal{Z}_x^{-\f{1-\de_1}2}\|^2_{L^2_x}\\
	&&+\|\na^{sym}_\mfS(\mb_{\mathsf{n}}\mathcal{Z}_x^{\f12})\mathcal{Z}_x^{-\f{1-\de_2}2}\|^2_{L^2_x}\big)+C_{\delta_3}\Big(\|\mathcal{P}_v^{10}(\mathbb{I}-\mP)f\|^2_{\mathcal{H}^{1,\de_1}_{x}L^2_v}+\sum_{i=1}^{13}\int_{\R^3_x}\mathcal{Z}_x^{2\de_1}(g,e_i)^2_{L^2_v}dx\Big)+C_{\delta_3}\|\mathbb{P}f\|_{\mathcal{H}^{1,\de_3/2}_{x}L^2_v}\\
	&&\times \big(\|\na^{sym}_\mfS(\mb_{\mathsf{n}}\mathcal{Z}_x^{\f12})\mathcal{Z}_x^{-\f{1-\de_2}2}\|_{L^2_x}+\|\mathcal{P}_v^{10}(\mathbb{I}-\mP)f\|_{\mathcal{H}^{1,\de_1}_{x}L^2_v}\big).
\eeno

\end{lem}	 

\begin{proof}
	Thanks to \eqref{macroeq}, we first have 
	\ben\label{aas1}
	&&\mathscr{C}_l(t)\mfS\na_x(\ma \mathcal{Z}_x^{\f12})-2\mathscr{C}_l(t)\mfS x\mc \mathcal{Z}_x^{\f12}+\mathscr{C}_l(t)\mfR\mb \mathcal{Z}_x^{\f12}+\mathscr{C}_l(t)(\mfR x\cdot\na_x)(\mb \mathcal{Z}_x^{\f12})+\pa_t(\mb \mathcal{Z}_x^{\f12})\nonumber\\&&=\sfT_{21}-2\sfT_{51}-2\pa_t \sfT_{52}.
	\een
	Using \eqref{deofabcs}, the fact that $\msA(\mb \mathcal{Z}_x^{\f12}) =0$ and \eqref{equality R x nabla x G -1 x}, we rewrite  \eqref{aas1} by
	\ben\label{aas}
		 &&\mathscr{C}_l(t)\mfS\na_x(\ma \mathcal{Z}_x^{\f12})-2\mathscr{C}_l(t)\mfS x \mc_{\mathsf{n}}\mathcal{Z}_x^{\f12}-2\mathscr{C}_l(t)\msA(\mc \mathcal{Z}_x^{\f12})\mfS x+\mathscr{C}_l(t)\mfR\mb_{\mathsf{n}}\mathcal{Z}_x^{\f12}+\mathscr{C}_l(t)(\mfR x\cdot\na_x)(\mb_{\mathsf{n}}\mathcal{Z}_x^{\f12})\nonumber\\
&&+\mathscr{C}_l(t)(\mfR\msA(\na^{skew}_\mfS(\mb \mathcal{Z}_x^{\f12}))+\msA(\na^{skew}_\mfS(\mb \mathcal{Z}_x^{\f12})) \mfR)\mfS^{-1}x+\f23\mathscr{C}_l(t) \msA(\mfS\na_x\cdot(\mb \mathcal{Z}_x^{\f12}))\mfR\mfS^{-1}x\nonumber\\
&&+\pa_t\mb \mathcal{Z}_x^{\f12}=\sfT_{21}-2\sfT_{51}-2\pa_t \sfT_{52}.
		\een
	To reduce the term $\pa_t\mb \mathcal{Z}_x^{\f12}$ in the left-hand side, by \eqref{deofabcs}, we obtain that
	\begin{equation}\label{as2}
	\begin{aligned}
	\pa_t\mb \mathcal{Z}_x^{\f12}=\pa_t\mb_{\mathsf{n}}\mathcal{Z}_x^{\f12}+\pa_t\msA(\na^{skew}_\mfS(\mb \mathcal{Z}_x^{\f12}))\mfS^{-1}x+\f13 \pa_t\msA(\mfS\na_x\cdot(\mb \mathcal{Z}_x^{\f12}))\mfS^{-1}x.
		\end{aligned}
\end{equation}
By \eqref{macroeq}, it is not difficult to check that
	\ben\label{as3}
		&&\pa_t (\mfS\na_x)_j(\mb \mathcal{Z}_x^{\f12})_i=-\mathscr{C}_l(t)(\mfS\na_x)_j(\mfS\na_x)_i(\ma \mathcal{Z}_x^{\f12})-2\mathscr{C}_l(t)  (\mfS\na_x)_j\Big((\mfS\na_x\mc)_i\mathcal{Z}_x^{\f12})\Big)\nonumber\\&&\qquad\qquad-\mathscr{C}_l(t)(\mfS\na_x)_j((\mfR\mb)_i\mathcal{Z}_x^{\f12}) 
		-\mathscr{C}_l(t)(\mfS\na_x)_j(\mfR x\cdot\na_x)(\mb \mathcal{Z}_x^{\f12})_i+(\mfS\na_x)_j(\sfT_{21})_i.\nonumber
\een
From this together with the fact that $\msA(\ma \mathcal{Z}_x^{\f12})=0$, \eqref{A integral for G nabla x R b} and \eqref{A integral for G nabla j G nabla i f} imply that 
	\ben\label{as4}
	&&\pa_t\msA(\mfS\na_x\cdot(\mb \mathcal{Z}_x^{\f12}))=\sum_{i=1}^3\pa_t\msA( (\mfS\na_x)_i(\mb \mathcal{Z}_x^{\f12})_i) =-\mathscr{C}_l(t)\msA(\D_{\mfS x}(\ma \mathcal{Z}_x^{\f12}))\\&&\qquad\qquad+2\mathscr{C}_l(t)\sum_{i =1}^3(\mfS^2)_{ii}\msA(\mc \mathcal{Z}_x^{\f12})-2\mathscr{C}_l(t)\msA(\mfR\mb \mathcal{Z}_x^{\f12}\cdot \mfS x)+\msA(\mfS\na_{x}\cdot \sfT_{21}),\nonumber\\
	\label{as5}
	 &&\pa_t\msA(\na^{skew}_\mfS(\mb \mathcal{Z}_x^{\f12}))_{ij}=\f12\pa_t\msA\Big((\mfS\na_x)_j (\mb \mathcal{Z}_x^{\f12})_i-(\mfS\na_x)_i (\mb \mathcal{Z}_x^{\f12})_j\Big)\nonumber\\
	&&=\f12\Big(\mathscr{C}_l(t)\msA( \mfS x\wedge \mfR\mb \mathcal{Z}_x^{\f12}) -\mathscr{C}_l(t)\msA( \mfR\mfS x\wedge \mb \mathcal{Z}_x^{\f12}) +2\msA(\na^{skew}_\mfS \sfT_{21})\Big)_{ij}.
\een
	Here $\D_{\mfS x}:=\sum_{k=1}^3(\mfS\na_x)_k^2$.   Plugging \eqref{as4} and \eqref{as5} into \eqref{as2}, together with \eqref{aas}, we further derive that
	\beno 
	&&\mathscr{C}_l(t)\mfS\na_x(\ma \mathcal{Z}_x^{\f12})-2\mathscr{C}_l(t)\mfS x \mc_{\mathsf{n}}\mathcal{Z}_x^{\f12}+\mathscr{C}_l(t)\mfR\mb_{\mathsf{n}}\mathcal{Z}_x^{\f12}+\mathscr{C}_l(t)(\mfR x\cdot\na_x)(\mb_{\mathsf{n}}\mathcal{Z}_x^{\f12})-2\mathscr{C}_l(t)\mfS x\msA(\mc \mathcal{Z}_x^{\f12})\\
&&+\mathscr{C}_l(t)(\mfR\msA(\na^{skew}_\mfS(\mb \mathcal{Z}_x^{\f12})) +\msA(\na^{skew}_\mfS(\mb \mathcal{Z}_x^{\f12})) \mfR)\mfS^{-1}x+\pa_t(\mb_{\mathsf{n}}\mathcal{Z}_x^{\f12})+\f13\mathscr{C}_l(t)\Big(-\msA(\D_{\mfS x}(\ma \mathcal{Z}_x^{\f12}))\\
	&&+2\sum_{i=1}^3(\mfS^2)_{ii}\msA(\mc \mathcal{Z}_x^{\f12})-2\msA(\mfR\mb \mathcal{Z}_x^{\f12}\cdot \mfS x)\Big)\mfS^{-1}x+\f23 \mathscr{C}_l(t)\msA(\mfS\na_x\cdot(\mb \mathcal{Z}_x^{\f12}))\mfR\mfS^{-1}x+\f12\mathscr{C}_l(t)\\
&&\times\Big(\msA( \mfS x\wedge \mfR\mb \mathcal{Z}_x^{\f12}) -\msA( \mfR \mfS x\wedge \mb \mathcal{Z}_x^{\f12})\Big)\mfS^{-1}x=\sfT_{21}-2\sfT_{51}-2\pa_t \sfT_{52}
	- \msA(\na^{skew}_\mfS \sfT_{21})\mfS^{-1}x\\
	&&-\f13\msA(\mfS\na_{x}\cdot \sfT_{21})\mfS^{-1}x.
\eeno	
Thanks to this reduction, we may arrive at 
\beno
	 &&\f d{dt}\big(\mfS\na_x(\ma_{\mathsf{n}}\mathcal{Z}_x^{\f12}),(2\sfT_{52}+\mb_{\mathsf{n}}e^{\f12|x|^2})\mathcal{Z}_x^{-(1-\de_3)}\big)_{L^2_x}
	=\big(\mfS\na_x(\ma_{\mathsf{n}}\mathcal{Z}_x^{\f12}),\pa_t(2\sfT_{52}+\mb_{\mathsf{n}}\mathcal{Z}_x^{\f12})\mathcal{Z}_x^{-(1-\de_3)}\big)_{L^2_x}
	\\&&+\big(\mfS\na_x(\pa_t\ma_{\mathsf{n}}\mathcal{Z}_x^{\f12}),(2\sfT_{52}+\mb_{\mathsf{n}}\mathcal{Z}_x^{\f12})\mathcal{Z}_x^{-(1-\de_3)}\big)_{L^2_x}=\mathcal{A}_1+\mathcal{A}_2+\mathcal{A}_3,\eeno where 
	\beno  \mathcal{A}_1&:=&\Big(\mfS\na_x(\ma_{\mathsf{n}}\mathcal{Z}_x^{\f12}),\mathscr{C}_l(t)\Big[-\mfS\na_x(\ma\mathcal{Z}_x^{\f12}) 
	+2\msA(\mc \mathcal{Z}_x^{\f12})\mfS x-\big(\mfR\msA(\na^{skew}_\mfS(\mb \mathcal{Z}_x^{\f12}))\\
	&&+\msA(\na^{skew}_\mfS(\mb \mathcal{Z}_x^{\f12}))\mfR\big)\mfS^{-1}x-\f{1}3\big(-\msA(\D_{\mfS x}(\ma \mathcal{Z}_x^{\f12}))+2\sum_{i=1}^3(\mfS^2)_{ii}\msA(\mc \mathcal{Z}_x^{\f12})\\
	&&-2\msA( \mfR\mb \mathcal{Z}_x^{\f12}\cdot \mfS x)\big)\mfS^{-1}x\Big]\mathcal{Z}_x^{-(1-\de_3)}\Big)_{L^2_x}, \\
   \mathcal{A}_2&:=&\Big(\mfS\na_x(\ma_{\mathsf{n}}\mathcal{Z}_x^{\f12}),\mathscr{C}_l(t)\Big[2(\mfS x \mc_{\mathsf{n}}) 
	  \mathcal{Z}_x^{\f12}-\mfR\mb_{\mathsf{n}}\mathcal{Z}_x^{\f12}-(\mfR x\cdot\na_x)(\mb_{\mathsf{n}}\mathcal{Z}_x^{\f12})-\f23 \msA(\mfS\na_x\cdot(\mb \mathcal{Z}_x^{\f12}))\\
	&&\mfR\mfS^{-1}x-\f12\big(\msA( \mfS x\wedge \mfR\mb \mathcal{Z}_x^{\f12}) -\msA( \mfR \mfS x\wedge \mb \mathcal{Z}_x^{\f12}))\mfS^{-1}x+\sfT_{21}-2\sfT_{51}-\msA(\na^{skew}_\mfS \sfT_{21})\mfS^{-1}x\\
	&&-\f13\msA(\mfS\na_{x}\cdot \sfT_{21})\mfS^{-1}x \Big]\mathcal{Z}_x^{-(1-\de_3)}\Big)_{L^2_x},\\
 \mathcal{A}_3&:=&(\mfS\na_x(\pa_t\ma_{\mathsf{n}}\mathcal{Z}_x^{\f12}),(2\sfT_{52}+\mb_{\mathsf{n}}\mathcal{Z}_x^{\f12})\mathcal{Z}_x^{-(1-\de_3)})_{L^2_x}.
\eeno

 \noindent\underline{Step 1: Estimate of $\mathcal{A}_1$.} Following \eqref{deofabcs}, we easily get that  
\begin{equation}\label{as6}
	\mathcal{A}_1=-\mathscr{C}_l(t)\|\mfS\na_x(\ma_{\mathsf{n}}\mathcal{Z}_x^{\f12})\mathcal{Z}_x^{-\f{1-\de_3}2}\|^2_{L^2_x}.
\end{equation}

  \noindent\underline{Step 2: Estimate of $\mathcal{A}_2$.} We split it into three parts: $\mathcal{A}_2=\mathcal{A}_{2,1}+\mathcal{A}_{2,2}+\mathcal{A}_{2,3}$ with
\begin{equation*}
\begin{aligned}
		\mathcal{A}_{2,1}=&\Big(\mfS\na_x(\ma_{\mathsf{n}}\mathcal{Z}_x^{\f12}),\mathscr{C}_l(t)\big(2(\mfS x \mc_{\mathsf{n}})\mathcal{Z}_x^{\f12}-\mfR\mb_{\mathsf{n}}\mathcal{Z}_x^{\f12}-(\mfR x\cdot\na_x)(\mb_{\mathsf{n}}\mathcal{Z}_x^{\f12})\big)\mathcal{Z}_x^{-(1-\de_3)}\Big)_{L^2_x},\\
		\mathcal{A}_{2,2}=&\Big(\mfS\na_x(\ma_{\mathsf{n}}\mathcal{Z}_x^{\f12})\mathcal{Z}_x^{-(1-\de_3)},-\f23\mathscr{C}_l(t) \msA(\mfS\na_x\cdot(\mb \mathcal{Z}_x^{\f12}))\mfR\mfS^{-1}x\\
		&-\f12\mathscr{C}_l(t)(\msA(\mfS x\wedge \mfR\mb \mathcal{Z}_x^{\f12})-\msA(\mfR\mfS x\wedge \mb \mathcal{Z}_x^{\f12}) )\mfS^{-1}x\Big)_{L^2_x},\\
		\mathcal{A}_{2,3}=&\Big(\mfS\na_x(\ma_{\mathsf{n}}\mathcal{Z}_x^{\f12})\mathcal{Z}_x^{-(1-\de_3)},\sfT_{21}-2\sfT_{51}-\msA(\na^{skew}_\mfS \sfT_{21})\mfS^{-1}x-\f13\msA(\mfS\na_{x}\cdot \sfT_{21})\mfS^{-1}x\Big)_{L^2_x}.
\end{aligned}
\end{equation*}

\noindent$\bullet$ For $\mathcal{A}_{2,1}$ and $\mathcal{A}_{2,3}$, by Cauchy-Schwarz inequality and   Lemma \ref{estimateofT}, we have 
\begin{equation*}
\begin{aligned}
&\mathcal{A}_{2,1}\leq \f{\mathscr{C}_l(t)}8 \|\mfS\na_x(\ma_{\mathsf{n}}\mathcal{Z}_x^{\f12})\mathcal{Z}_x^{-\f{1-\de_3}2}\|^2_{L^2_x}+C_{\de_3}(\|\mathcal{P}_x\mc_{\mathsf{n}}\mathcal{Z}_x^{\f{\de_3}2}\|^2_{L^2_x}+\|\mathcal{P}_x\na_x\mb_{\mathsf{n}}\mathcal{Z}_x^{\f{\de_3}2}\|^2_{L^2_x}),\\
		&\mathcal{A}_{2,3}\leq \f{\mathscr{C}_l(t)}8 \|\mfS\na_x(\ma_{\mathsf{n}}\mathcal{Z}_x^{\f12})\mathcal{Z}_x^{-\f{1-\de_3}2}\|^2_{L^2_x}+C_{\de_3}\big(\|\mathcal{P}_v^{10}(\mathbb{I}-\mP)f\|^2_{\mathcal{H}^{1,\de_3}_{x}L^2_v}+\sum_{i=1}^{13}\int_{\R^3_x}\mathcal{Z}_x^{\de_3}(g,e_i)^2_{L^2_v}dx\big).
\end{aligned}
\end{equation*}

\noindent$\bullet$ For  $\mathcal{A}_{2,2}$,  using integration by parts, we have 
\beno
 \Big(\mfS\na_x(\ma_{\mathsf{n}}\mathcal{Z}_x^{\f12})\mathcal{Z}_x^{-(1-\de_3)},\msA(\mfS\na_x\cdot(\mb \mathcal{Z}_x^{\f12}))\mfR\mfS^{-1}x\Big)_{L^2_x}=\Big(2(1-\de_3)(\ma_{\mathsf{n}}\mathcal{Z}_x^{\f12})\mathcal{Z}_x^{-(1-\de_3)}\\
 \msA(\mfS\na_x\cdot(\mb \mathcal{Z}_x^{\f12})),(\mfR\mfS^{-1}x,\mfS x)_{\R^3}\Big)-\Big((\ma_{\mathsf{n}}\mathcal{Z}_x^{\f12})\mathcal{Z}_x^{-(1-\de_3)}\msA(\mfS\na_x\cdot(\mb \mathcal{Z}_x^{\f12})),(\mfS\na_x)\cdot(\mfR\mfS^{-1}x)\Big)_{L^2_x}=0,
\eeno
where   we use \eqref{equality R G -1 x G x}.  Thus we are led to that  
\begin{equation}\label{A22}
	\begin{aligned}
\mathcal{A}_{2,2}=&\Big(\mfS\na_x(\ma_{\mathsf{n}}\mathcal{Z}_x^{\f12})\mathcal{Z}_x^{-(1-\de_3)},\f{\mathscr{C}_l(t)}2(-\msA (\mfS x\wedge \mfR\mb \mathcal{Z}_x^{\f12}) +\msA( \mfR\mfS x\wedge \mb \mathcal{Z}_x^{\f12}) )\mfS^{-1}x\Big)_{L^2_x}\\
 =& \Big(2(1-\de_3)(\ma_{\mathsf{n}} \mathcal{Z}_x^{\f12})\mathcal{Z}_x^{-(1-\de_3)},\f{\mathscr{C}_l(t)}2\big((-\msA (\mfS x\wedge \mfR\mb \mathcal{Z}_x^{\f12}) +\msA( \mfR\mfS x\wedge \mb \mathcal{Z}_x^{\f12}) )\mfS^{-1}x,\mfS x\big)_{\R^3}\Big)_{L^2_x}.
	\end{aligned}
\end{equation}
It is not difficult to see that 
\begin{equation*}
	\begin{aligned}
	&\big((-\msA (\mfS x\wedge \mfR\mb \mathcal{Z}_x^{\f12}) +\msA( \mfR\mfS x\wedge \mb \mathcal{Z}_x^{\f12}) )\mfS^{-1}x,\mfS x\big)_{\R^3}\\
	=&\f r{(2\pi)^{\f32}}(\sqrt{1-r^2}-\f1{\sqrt{1-r^2}})(\int_{\R^3_x}(\mb_2x_3+\sqrt{1-r^2}x_2\mb_3)dxx_1x_3-\int_{\R^3_x}(\mb_1x_3+\sqrt{1-r^2}x_1\mb_3)dxx_2x_3),
	\end{aligned}
\end{equation*}
and by \eqref{macroeq},
\begin{equation*}
	\begin{aligned}
	&\f{\mathscr{C}_l(t)}{2(2\pi)^{\f32}}\int_{\R^3_x}(\mb_2x_3+\sqrt{1-r^2}x_2\mb_3)dx=\mathscr{C}_l(t)\big(\msA(\na^{sym}_{\mfS}(\mb \mathcal{Z}_x^{\f12}))\big)_{32}=\big(\msA(\sfT_{41}+\pa_t \sfT_{42})\big)_{32},\\
	&\f{\mathscr{C}_l(t)}{2(2\pi)^{\f32}}\int_{\R^3_x}(\mb_1x_3+\sqrt{1-r^2}x_1\mb_3)dx=\mathscr{C}_l(t)\big(\msA(\na^{sym}_{\mfS}(\mb \mathcal{Z}_x^{\f12}))\big)_{31}=\big(\msA(\sfT_{41}+\pa_t \sfT_{42})\big)_{31}.
	\end{aligned}
\end{equation*}
These yield  that
\begin{equation*}
	\begin{aligned}
\mathcal{A}_{2,2}&=r(\sqrt{1-r^2}-\f1{\sqrt{1-r^2}})\Big(2(1-\de_3)\ma_{\mathsf{n}}\mathcal{Z}_x^{\f12},\Big(\big(\msA(\sfT_{41}+\pa_t \sfT_{42})\big)_{32}x_1x_3-\big(\msA(\sfT_{41}+\pa_t \sfT_{42})\big)_{31}x_2x_3\Big)\\&\times\mathcal{Z}_x^{-(1-\de_3)}\Big)=r(\sqrt{1-r^2}-\f1{\sqrt{1-r^2}})\Big(2(1-\de_3)\ma_{\mathsf{n}} \mathcal{Z}_x^{\f12},((\msA(\sfT_{41}))_{32}x_1x_3-(\msA(\sfT_{41}))_{31}x_2x_3)\mathcal{Z}_x^{-(1-\de_3)}\Big)\\
+&r(\sqrt{1-r^2}-\f1{\sqrt{1-r^2}})\f d{dt}\Big(2(1-\de_3)\ma_{\mathsf{n}} \mathcal{Z}_x^{\f12},((\msA(\sfT_{42}))_{32}x_1x_3-(\msA(\sfT_{42}))_{31}x_2x_3)\mathcal{Z}_x^{-(1-\de_3)}\Big)\\
-&r(\sqrt{1-r^2}-\f1{\sqrt{1-r^2}})\Big(2(1-\de_3)\pa_t(\ma_{\mathsf{n}} \mathcal{Z}_x^{\f12}),((\msA(T_{42}))_{32}x_1x_3-(\msA(T_{42}))_{31}x_2x_3)\mathcal{Z}_x^{-(1-\de_3)}\Big).
	\end{aligned}
\end{equation*}

We claim that for any sufficiently small $\vep$, it holds that
\begin{equation*}
	\begin{aligned}
		&\mathcal{A}_{2,2}-r(\sqrt{1-r^2}-\f1{\sqrt{1-r^2}})\f d{dt}\Big(2(1-\de_3)\ma_{\mathsf{n}} \mathcal{Z}_x^{\f12},((\msA(\sfT_{42}))_{32}x_1x_3-(\msA(\sfT_{42}))_{31}x_2x_3)\mathcal{Z}_x^{-(1-\de_3)}\Big)\\
		&\leq \vep\|\ma_{\mathsf{n}} \mathcal{Z}_x^{\f{\de_3}2}\|^2_{L^2_x}+C_\vep (\|\mathcal{P}_v^{10}(\mathbb{I}-\mP)f\|^2_{\mathcal{H}^{1,\de_1}_{x}L^2_v}+\sum_{i=1}^{13}\int_{\R^3_x}\mathcal{Z}_x^{2\de_1}(g,e_i)^2_{L^2_v}dx)\\
		&+C_{\de_3} \|\mathcal{P}_v^{10}(\mathbb{I}-\mP)f\|_{\mathcal{H}^{1,\de_1}_{x}L^2_v}\|\mathbb{P}f\|_{\mathcal{H}^{1,\de_3/2}_{x}L^2_v}.
	\end{aligned}
\end{equation*}
In fact, by Cauchy-Schwarz inequality, Lemma \ref{estimateofT} and the argument used for \eqref{bs5}, we get that

\begin{equation}\label{as7}
	\begin{aligned}
	&\|\mathcal{P}_x^{-1}\pa_t(\ma_{\mathsf{n}}\mathcal{Z}_x^{\f12})\mathcal{Z}_x^{-\f{1-\de_3}2}\|^2_{L^2_x}
	\leq C_{\de_3}\|\mathbb{P}f\|^2_{\mathcal{H}^{1,\de_3/2}_{x}L^2_v}
	+C_{\de_3}(\|\mathcal{P}_v^{10}(\mathbb{I}-\mP)f\|^2_{\mathcal{H}^{1,\de_1}_{x}L^2_v}+\sum_{i=1}^{13}\int_{\R^3_x}\mathcal{Z}_x^{2\de_1}(g,e_i)^2_{L^2_v}dx).
	\end{aligned}
\end{equation}
This is enough to conclude the claim.
By Poincar\'{e}-Korn inequality \eqref{PKIforabcs},  we arrive at that
\begin{equation}\label{as8}
	\begin{aligned}
		&\mathcal{A}_{2}-r(\sqrt{1-r^2}-\f1{\sqrt{1-r^2}})\f d{dt}\Big(2(1-\de_3)\ma_{\mathsf{n}} \mathcal{Z}_x^{\f12},(\msA(T_{42})_{32}x_1x_3-\msA(T_{42})_{31}x_2x_3)\mathcal{Z}_x^{-(1-\de_3)}\Big)\\
		&\leq \f{\mathscr{C}_l(t)}4\|\mfS\na_x(\ma_{\mathsf{n}}\mathcal{Z}_x^{\f12})\mathcal{Z}_x^{-\f{1-\de_3}2}\|^2_{L^2_x}+ C_{\de_3}
		\Big[ \|\mfS\na_x(\mc_{\mathsf{n}}\mathcal{Z}_x^{\f1 2 })\mathcal{Z}_x^{-\f{1-\de_1}2}\|^2_{L^2_x}+\|\na^{sym}_\mfS(\mb_{\mathsf{n}}\mathcal{Z}_x^{\f12})\mathcal{Z}_x^{-\f{1-\de_2}2}\|^2_{L^2_x} \\
		&+  \|\mathcal{P}_v^{10}(\mathbb{I}-\mP)f\|^2_{\mathcal{H}^{1,\de_1}_{x}L^2_v}+\sum_{i=1}^{13}\int_{\R^3_x}\mathcal{Z}_x^{2\de_1}(g,e_i)^2_{L^2_v}dx+ \|\mathcal{P}_v^{10}(\mathbb{I}-\mP)f\|_{\mathcal{H}^{1,\de_1}_{x}L^2_v}\|\mathbb{P}f\|_{\mathcal{H}^{1,\de_3/2}_{x}L^2_v}\Big].
	\end{aligned}
\end{equation}

\noindent\underline{Step 3: Estimate of $\mathcal{A}_3$.}  
 Using integration by parts, \eqref{as7},\eqref{PKIforabcs} and Lemma \ref{estimateofT}, we can deduce that
\begin{equation}\label{as9}
	\begin{aligned}
		\mathcal{A}_3&\leq C_{\de_3} (\|\na^{sym}_\mfS(\mb_{\mathsf{n}}\mathcal{Z}_x^{\f12})\mathcal{Z}_x^{-\f{1-\de_2}2}\|_{L^2_x}+\|\mathcal{P}_v^{10}(\mathbb{I}-\mP)f\|_{\mathcal{H}^{1,\de_1}_{x}L^2_v})\\
		&\times\big(\|\mathbb{P}f\|_{\mathcal{H}^{1,\de_3/2}_{x}L^2_v}
		+\|\mathcal{P}_v^{10}(\mathbb{I}-\mP)f\|_{\mathcal{H}^{1,\de_1}_{x}L^2_v}+\sum_{i=1}^{13}\big(\int_{\R^3_x}\mathcal{Z}_x^{2\de_1}(g,e_i)^2_{L^2_v}dx\big)^{\f12}\big).
	\end{aligned}
\end{equation}

Patching together the estimates   \eqref{as6}, \eqref{as8} and \eqref{as9}, we can get the desired results and then complete the proof of this lemma.
\end{proof}

Now, we are ready to give the proof of Theorem \ref{macroestimate} with $N=1$.
\begin{proof}[Proof of Theorem \ref{macroestimate}: $N=1$.]
Thanks to Lemma \ref{cs}, Lemma \ref{bs}, Lemma \ref{as} and the facts that $0<\de_3<\f18, \de_1=2\de_3,\de_2=\f54\de_3$, 
if 
\begin{equation}\label{F1abc}
	\begin{aligned}
\mathcal{F}_{1,\mc_{\mathsf{n}}}&:=(\mfS\na_x(\mc_{\mathsf{n}}\mathcal{Z}_x^{\f12}),-\sfT_{52}\mathcal{Z}_x^{-(1-\de_1)})_{L^2_x},\quad\mathcal{F}_{1,\mb_{\mathsf{n}}}:=(\na^{sym}_\mfS(\mb_{\mathsf{n}}\mathcal{Z}_x^{\f12}),(-\sfT_{42}+\mc_{\mathsf{n}} \mathcal{Z}_x^{\f12})\mathcal{Z}_x^{-(1-\de_2)})_{L^2_x},\\
\mathcal{F}_{1,\ma_{\mathsf{n}}}&:=(\mfS\na_x(\ma_{\mathsf{n}}\mathcal{Z}_x^{\f12}),(\sfT_{52}+\mb_{\mathsf{n}}\mathcal{Z}_x^{\f12})\mathcal{Z}_x^{-(1-\de_3)})_{L^2_x}-r(\f1{\sqrt{1-r^2}}-\sqrt{1-r^2})\\
&\qquad\times(2(1-\de_3)\ma_{\mathsf{n}} \mathcal{Z}_x^{\f12},((\msA(\sfT_{42}))_{32}x_1x_3-(\msA(\sfT_{42}))_{31}x_2x_3)\mathcal{Z}_x^{-(1-\de_3)})_{L^2_x},
	\end{aligned}
\end{equation}
then for any $\vep_1,\vep_2\in(0,1)$ with $\vep_2<\vep_1$, it holds that
\begin{equation}\label{thmeq1}
	\begin{aligned}
&\f d{dt}( \mathcal{F}_{1,\mc_{\mathsf{n}}}+ \vep_1\mathcal{F}_{1,\mb_{\mathsf{n}}}+\vep_2 \mathcal{F}_{1,\ma_{\mathsf{n}}})+\f{\mathscr{C}_l(t)}2\Big(\|\mfS\na_x(\mc_{\mathsf{n}}\mathcal{Z}_x^{\f1 2 })\mathcal{Z}_x^{-\f{1-\de_1}2}\|^2_{L^2_x}+\vep_1\|\na^{sym}_\mfS(\mb_{\mathsf{n}}\mathcal{Z}_x^{\f12})\mathcal{Z}_x^{-\f{1-\de_2}2}\|^2_{L^2_x}\\
+&\vep_2\|\mfS\na_x(\ma_{\mathsf{n}}\mathcal{Z}_x^{\f12})\mathcal{Z}_x^{-\f{1-\de_3}2}\|^2_{L^2_x}\Big)\leq C_{\de_3}\Big(\|\mathcal{P}_v^{10}(\mathbb{I}-\mP)f\|^2_{\mathcal{H}^{1,\de_1}_{x}L^2_v}+\sum_{i=1}^{13}\int_{\R^3_x}\mathcal{Z}_x^{2\de_1}(g,e_i)^2_{L^2_v}dx\Big)\\
+&C_{\de_3}(\vep_1+\vep_2) \|\mfS\na_x(\mc_{\mathsf{n}}\mathcal{Z}_x^{\f12})\mathcal{Z}_x^{-\f{1-\de_1}2}\|^2_{L^2_x}+C_{\de_3}\vep_2 \|\na^{sym}_\mfS(\mb_{\mathsf{n}}\mathcal{Z}_x^{\f12})\mathcal{Z}_x^{-\f{1-\de_2}2}\|^2_{L^2_x}+C_{\de_3}\sum_{i=1}^4\mathcal{G}_{i},
	\end{aligned}
\end{equation}
where  \[\mathcal{G}_1:=\|\mathcal{P}_v^{10}(\mathbb{I}-\mP)f\|_{\mathcal{H}^{1,\de_1}_{x}L^2_v} \|\mathbb{P}f\|_{\mathcal{H}^{1,\de_3/2}_{x}L^2_v},\] \[\mathcal{G}_2:=\vep_1(\|\mfS\na_x(\mc_{\mathsf{n}}\mathcal{Z}_x^{\f12})\mathcal{Z}_x^{-\f{1-\de_1}2}\|_{L^2_x}+\|\mathcal{P}_v^{10}(\mathbb{I}-\mP)f\|_{\mathcal{H}^{1,\de_1}_{x}L^2_v})\|\mathbb{P}f\|_{\mathcal{H}^{1,\de_3/2}_{x}L^2_v},\] \[\mathcal{G}_3:=\vep_2 (\|\na^{sym}_\mfS(\mb_{\mathsf{n}}\mathcal{Z}_x^{\f12})\mathcal{Z}_x^{-\f{1-\de_2}2}\|_{L^2_x}+\|\mathcal{P}_v^{10}(\mathbb{I}-\mP)f\|_{\mathcal{H}^{1,\de_1}_{x}L^2_v})\|\mathbb{P}f\|_{\mathcal{H}^{1,\de_3/2}_{x}L^2_v}.\]

Thanks to Poincar\'{e}-Korn inequality, Lemma \ref{abceqtoabcs} and the fact that $\mathscr{C}_l(t)$ has a positive lower bound, there exists a constant  $\ka>0$ dependent of $\de_i(i=1,2,3)$ and $r$ such that 
\begin{multline}\label{fabc}
		\f{\mathscr{C}_l(t)}2\Big(\|\mfS\na_x(\mc_{\mathsf{n}}\mathcal{Z}_x^{\f1 2 })\mathcal{Z}_x^{-\f{1-\de_1}2}\|^2_{L^2_x}+\vep_1\|\na^{sym}_\mfS(\mb_{\mathsf{n}}\mathcal{Z}_x^{\f1 2 })\mathcal{Z}_x^{-\f{1-\de_2}2}\|^2_{L^2_x}+\vep_2\|\mfS\na_x(\ma_{\mathsf{n}}\mathcal{Z}_x^{\f1 2 })\mathcal{Z}_x^{-\f{1-\de_3}2}\|^2_{L^2_x}\Big) \\
		\geq \vep_2\ka (\|\mc \mathcal{Z}_x^{\f{\de_1}2}\|^2_{L^2_x}+\|\na_x\mc \mathcal{Z}_x^{\f{\de_1}2}\|^2_{L^2_x}+\|\mb \mathcal{Z}_x^{\f{\de_2}2}\|^2_{L^2_x}+\|\na_x\mb \mathcal{Z}_x^{\f{\de_2}2}\|^2_{L^2_x}+\|\ma \mathcal{Z}_x^{\f{\de_3}2}\|^2_{L^2_x}+\|\na_x\ma \mathcal{Z}_x^{\f{\de_3}2}\|^2_{L^2_x}).
\end{multline}
Using Young's inequality repeatedly, we have 
\begin{equation*}
	\begin{aligned}
\mathcal{G}_1&\leq \f{\vep_2\ka}{8C_{\de_3}}\|\mathbb{P}f\|^2_{\mathcal{H}^{1,\de_3/2}_{x}L^2_v}+C_{\de_3} \|\mathcal{P}_v^{10}(\mathbb{I}-\mP)f\|^2_{\mathcal{H}^{1,\de_1}_{x}L^2_v},\\
\mathcal{G}_2&\leq \f{\vep_2\ka}{8C_{\de_3}}\|\mathbb{P}f\|^2_{\mathcal{H}^{1,\de_3/2}_{x}L^2_v}+C_{\de_3} \f{8\vep_1^2}{\ka\vep_2}\|\mfS\na_x(\mc_{\mathsf{n}}\mathcal{Z}_x^{\f12})\mathcal{Z}_x^{-\f{1-\de_1}2}\|^2_{L^2_x}+ C_{\de_3}\|\mathcal{P}_v^{10}(\mathbb{I}-\mP)f\|^2_{\mathcal{H}^{1,\de_1}_{x}L^2_v},\\
\mathcal{G}_3&\leq \f{\vep_2\ka}{8C_{\de_3}}\|\mathbb{P}f\|^2_{\mathcal{H}^{1,\de_3/2}_{x}L^2_v}+C_{\de_3} \f{8\vep_2}{\ka}\|\na^{sym}_\mfS(\mb_{\mathsf{n}}\mathcal{Z}_x^{\f12})\mathcal{Z}_x^{-\f{1-\de_2}2}\|^2_{L^2_x}+C_{\de_3}\|\mathcal{P}_v^{10}(\mathbb{I}-\mP)f\|^2_{\mathcal{H}^{1,\de_1}_{x}L^2_v}.
\end{aligned}
\end{equation*}

Plug them into \eqref{thmeq1}, we can derive that
\begin{equation*}
	\begin{aligned}
		&\f d{dt}( \mathcal{F}_{1,\mc_{\mathsf{n}}}+ \vep_1\mathcal{F}_{1,\mb_{\mathsf{n}}}+\vep_2 \mathcal{F}_{1,\ma_{\mathsf{n}}})+(\f{\mathscr{C}_l(t)}2-(\vep_1+\vep_2)C_{\de_3}- \f{8\vep_1^2}{\ka\vep_2}C_{\de_3})\|\mfS\na_x(\mc_{\mathsf{n}}\mathcal{Z}_x^{\f1 2 })\mathcal{Z}_x^{-\f{1-\de_1}2}\|^2_{L^2_x}\\
		&+(\f{\mathscr{C}_l(t)}2\vep_1-\vep_2C_{\de_3}-\f{8\vep_2}\ka C_{\de_3})\|\na^{sym}_\mfS(\mb_{\mathsf{n}}\mathcal{Z}_x^{\f12})\mathcal{Z}_x^{-\f{1-\de_2}2}\|^2_{L^2_x}+\f{\mathscr{C}_l(t)}2\vep_2\|\mfS\na_x(\ma_{\mathsf{n}}\mathcal{Z}_x^{\f12})\mathcal{Z}_x^{-\f{1-\de_3}2}\|^2_{L^2_x}\\
		&-\f{3\vep_2\ka}8\|\mathbb{P}f\|^2_{\mathcal{H}^{1,\de_3/2}_{x}L^2_v}\leq C_{\de_3}\Big(\|\mathcal{P}_v^{10}(\mathbb{I}-\mP)f\|^2_{\mathcal{H}^{1,\de_1}_{x}L^2_v}+\sum_{i=1}^{13}\int_{\R^3_x}\mathcal{Z}_x^{2\de_1}(g,e_i)^2_{L^2_v}dx\Big) .
	\end{aligned}
\end{equation*}
 Let $\vep_2=\vep_1^{3/2}<\f2{\mathscr{C}_l(t)}\min\{(\f{\ka}{192C_{\de_3}})^3,(\f1{24C_{\de_3}})^3,1\}$. Then it holds that 
 \begin{equation*}
 	\begin{aligned}
 		\f{\mathscr{C}_l(t)}2-\vep_1C_{\de_3}-\vep_2C_{\de_3}- \f{8\vep_1^2}{\ka\vep_2}C_{\de_3}>\f{7\mathscr{C}_l(t)}{16},~\f{\mathscr{C}_l(t)}2\vep_1-\vep_2C_{\de_3}-\f{8\vep_2}\ka C_{\de_3}>\f{7\mathscr{C}_l(t)}{16} \vep_1.
	\end{aligned}
\end{equation*}
  Additionally, in virtue of \eqref{fabc}, we have 
 \begin{equation*}
	\begin{aligned}
&\f {3\vep_2\ka}8\|\mathbb{P}f\|^2_{\mathcal{H}^{1,\de_3/2}_{x}L^2_v}
\leq  \f 38 \f{\mathscr{C}_l(t)}2 (\|\mfS\na_x(\mc_{\mathsf{n}}\mathcal{Z}_x^{\f1 2 })\mathcal{Z}_x^{-\f{1-\de_1}2}\|^2_{L^2_x}\\&+\vep_1\|\na^{sym}_\mfS(\mb_{\mathsf{n}}\mathcal{Z}_x^{\f1 2 })\mathcal{Z}_x^{-\f{1-\de_2}2}\|^2_{L^2_x}+\vep_2\|\mfS\na_x(\ma_{\mathsf{n}}\mathcal{Z}_x^{\f1 2 })\mathcal{Z}_x^{-\f{1-\de_3}2}\|^2_{L^2_x}).
	\end{aligned}
\end{equation*}

Finally  we are led to that
\begin{equation}\label{0order}
	\begin{aligned}
		&\f d{dt} \mathcal{F}_{1,\de_3}+	\f{\mathscr{C}_l(t)}4\Big(\underbrace{\|\mfS\na_x(\mc_{\mathsf{n}}\mathcal{Z}_x^{\f1 2 })\mathcal{Z}_x^{-\f{1-\de_1}2}\|^2_{L^2_x}+\vep_1\|\na^{sym}_\mfS(\mb_{\mathsf{n}}\mathcal{Z}_x^{\f12})\mathcal{Z}_x^{-\f{1-\de_2}2}\|^2_{L^2_x}+\vep_2\|\mfS\na_x(\ma_{\mathsf{n}}\mathcal{Z}_x^{\f12})\mathcal{Z}_x^{-\f{1-\de_3}2}\|^2_{L^2_x}}_{:=\mathcal{D}_{1}}\Big)\\
		&\leq C_{\de_3}\Big(\|\mathcal{P}_v^{10}(\mathbb{I}-\mP)f\|^2_{\mathcal{H}^{1,\de_1}_{x}L^2_v}+\sum_{i=1}^{13}\int_{\R^3_x}\mathcal{Z}_x^{2\de_1}(g,e_i)^2_{L^2_v}dx\Big),
	\end{aligned}
\end{equation}
where $\mathcal{F}_{1,\de_3}:=\mathcal{F}_{1,\mc_{\mathsf{n}}}+ \vep_1\mathcal{F}_{1,\mb_{\mathsf{n}}}+\vep_2 \mathcal{F}_{1,\ma_{\mathsf{n}}}$.  
On one hand,   Poincar\'{e} inequality \eqref{PKIforabcs}, Lemma \ref{abceqtoabcs} and Lemma \ref{mapf} imply that $\mathcal{D}_{1}\gs \|\mathbb{P}f\|^2_{\mathcal{H}^{1,\de_3/2}_xL^2_v}$. On the other hand,
from \eqref{F1abc} and Lemma \ref{mapf}, one may easily check that $|\mathcal{F}_{1,\de_3}|\leq C_{\de_3}\|\mathcal{P}_v^{10}f\|^2_{\mathcal{H}^{1,\de_1}_{x}L^2_v}$. Then we complete the proof of Theorem \ref{macroestimate} with $N=1$.
\end{proof}

Next we will provide a  proof for Theorem \ref{macroestimate} in general case. We begin with the derivation of the equation for the derivatives of $(\mathbf{a}, \mathbf{b}, \mathbf{c})$. To do that,  we introduce some special notations.

$\bullet$ Let $\al=(\al_1,\al_2,\al_3)$. We  define $\alpha<0$ to mean that there exists $1\le i\le 3$ such that $\alpha_i<0$. In this way, we will slightly abuse the notation by stating that
$\pa^\alpha_x:=0$ if $\al<0$.     

$\bullet$ A vector derivative and a matrix derivative are defined by $\pa_x^{\al-\om}:=(\pa_x^{\al-\om_1},\pa_x^{\al-\om_2},\pa_x^{\al-\om_3})^\tau$ and  $(\pa_x^{\al\pm\om})_{ij}:=\pa_x^{\al-\om_i+\om_j}$ respectively.

$\bullet$  $(\mfS\pa_x^{\al-\om})_i:=\sum_{j=1}^3\mfS_{ij}\pa^{\al-\om_j}_x$,  $\mathrm{tr}(\mfR\pa_x^{\al\pm\om}):=\sum_{i,j=1}^3\mfR_{ij}\pa_x^{\al-\om_j+\om_i}$ and $\Big((\mfS\pa_x^{\al-\om})h^\tau-((\mfS\pa_x^{\al-\om})h^\tau)^\tau\Big)_{ij}\\:=(\mfS\pa_x^{\al-\om})_jh_i-(\mfS\pa_x^{\al-\om})_ih_j$ for any vector function $h=(h_1,h_2,h_3)$.

Now, we can derive the macroscopic equation for $(\partial^\alpha_x\mathbf{a}, \partial^\alpha_x\mathbf{b}, \partial^\alpha_x\mathbf{c})$ with $\alpha \in \mathbb{Z}^3_+$. In fact, we have
\ben 
&& \bullet\quad \pa_t((\pa^\al_x\ma) \mathcal{Z}_x^{\f12})=-\mathscr{C}_l(t)(\mfS\na_x\cdot \pa^\al_x\mb)\mathcal{Z}_x^{\f12}-\mathscr{C}_l(t)\mfR x\cdot\na_x((\pa^\al_x\ma) \mathcal{Z}_x^{\f12})+\pa^\al_x(\sfT_{11}\mathcal{Z}_x^{-\f12})\mathcal{Z}_x^{\f12}\label{macroeqNgeq2a}\\
&&\qquad\qquad\qquad\qquad-\mathscr{C}_l(t)(\mathrm{tr}(\mfR \pa_x^{\al\pm\om})\ma) \mathcal{Z}_x^{\f12};\notag\\
&& \bullet\quad\pa_t((\pa^\al_x\mb) \mathcal{Z}_x^{\f12})=-\mathscr{C}_l(t)\mfS\na_x((\pa^\al_x\ma) \mathcal{Z}_x^{\f12})-2\mathscr{C}_l(t)(\mfS\na_x\pa^\al_x\mc)\mathcal{Z}_x^{\f12}-\mathscr{C}_l(t)(\mfR\pa^\al_x\mb) \mathcal{Z}_x^{\f12}\label{macroeqNgeq2b}\\
&&\qquad\quad-\mathscr{C}_l(t)(\mfR x\cdot\na_x)((\pa^\al_x\mb) \mathcal{Z}_x^{\f12})+\pa^\al_x(\sfT_{21}\mathcal{Z}_x^{-\f12})\mathcal{Z}_x^{\f12}-\mathscr{C}_l(t)(\mfS\pa^{\al-\om}_x\ma) \mathcal{Z}_x^{\f12}-\mathscr{C}_l(t)(\mathrm{tr}(\mfR\pa_x^{\al\pm\om})\mb) \mathcal{Z}_x^{\f12};\notag\\
&& \bullet\quad\pa_t((\pa^\al_x\mc) \mathcal{Z}_x^{\f12})=-\f13 \mathscr{C}_l(t)\mfS\na_x\cdot ((\pa^\al_x\mb) \mathcal{Z}_x^{\f12})-\mathscr{C}_l(t)\mfR x\cdot\na_x ((\pa^\al_x\mc) \mathcal{Z}_x^{\f12})+\pa^\al_x(\sfT_{31}\mathcal{Z}_x^{-\f12})\mathcal{Z}_x^{\f12}\label{macroeqNgeq2c1}\\
&&\qquad\qquad\qquad-\f13\mathscr{C}_l(t)(\mfS\pa^{\al-\om}_x\cdot\mb) \mathcal{Z}_x^{\f12}-\mathscr{C}_l(t)(\mathrm{tr}(\mfR\pa_x^{\al\pm\om})\mc) \mathcal{Z}_x^{\f12}\notag ;\\
&& \bullet\quad\pa_t((\pa^\al_x\mc) \mathcal{Z}_x^{\f12}) \mathbb{I}_{3\times 3}=-\mathscr{C}_l(t)\na^{sym}_\mfS((\pa^\al_x\mb) \mathcal{Z}_x^{\f12})-\mathscr{C}_l(t)\mfR x\cdot\na_x ((\pa^\al_x\mc) \mathcal{Z}_x^{\f12})\mathbb{I}_{3\times 3}+\pa^\al_x(\sfT_{41}\mathcal{Z}_x^{-\f12})\mathcal{Z}_x^{\f12}\label{macroeqNgeq2c2}\\
&&\quad+\pa_t \pa^\al_x(\sfT_{42}\mathcal{Z}_x^{-\f12})\mathcal{Z}_x^{\f12}-\f12\mathscr{C}_l(t)\Big((\mfS\pa_x^{\al-\om})(\mb \mathcal{Z}_x^{\f12})^\tau -\big((\mfS\pa_x^{\al-\om})(\mb \mathcal{Z}_x^{\f12})^\tau\big)^\tau\Big)-\mathscr{C}_l(t)(\mathrm{tr}(\mfR\pa_x^{\al\pm\om})\mc) \mathcal{Z}_x^{\f12}\mathbb{I}_{3\times 3};\notag
\\
&& \bullet\quad\mathscr{C}_l(t)\mfS\na_x ((\pa^\al_x\mc) \mathcal{Z}_x^{\f12})=\pa^\al_{x}(\sfT_{51}\mathcal{Z}_x^{-\f12})\mathcal{Z}_x^{\f12}+\pa_t\pa^\al_x(\sfT_{52}\mathcal{Z}_x^{-\f12})\mathcal{Z}_x^{\f12}-\mathscr{C}_l(t)(\mfS\pa^{\al-\om}_x\mc) \mathcal{Z}_x^{\f12}.\label{macroeqNgeq2c3}
\een

\begin{proof}[Proof of Theorem \ref{macroestimate}: $N\geq2$.] It is easy to verify that for any $|\al|=N-1\geq1$,
	\begin{equation}\label{vanishpaabc}
		\int_{\R^3} ((\pa^\al_x\mc)\mathcal{Z}_x^{\f12},(\pa^\al_x\mb)\mathcal{Z}_x^{\f12},(\pa^\al_x\ma)\mathcal{Z}_x^{\f12})\mathcal{Z}_x^{-\f12}dx
		=0,\quad\int_{\R^3}\na^{skew}_\mfS((\pa^\al_x\mb)\mathcal{Z}_x^{\f12})\mathcal{Z}_x^{-\f12}dx=0.
	\end{equation}
 Similarly to the case $|\al|=0$ ($N=1$), we will estimate $(\pa^\al_x\ma,\pa^\al_x\mb,\pa^\al_x\mc)$ by two steps.

\underline{Step 1: Estimate of $\pa^\al_x\mc$.} Thanks to \eqref{macroeqNgeq2c3}, we first have
\begin{equation*}
	\begin{aligned}
		&\f d{dt}(\mfS\na_x ((\pa^\al_x\mc) \mathcal{Z}_x^{\f12}),-\pa^\al_x(\sfT_{52}\mathcal{Z}_x^{-\f12})\mathcal{Z}_x^{\f12}\mathcal{Z}_x^{-(1-\de_1)})_{L^2_x}\\
		=&-\mathscr{C}_l(t) \|\mfS\na_x((\pa^\al_x\mc) \mathcal{Z}_x^{\f12})\mathcal{Z}_x^{-\f{1-\de_1}2}\|^2_{L^2_{x,v}}+(\mfS\na_x ((\pa^\al_x\mc) \mathcal{Z}_x^{\f12})\mathcal{Z}_x^{-(1-\de_1)},\pa^\al_{x}(\sfT_{51}\mathcal{Z}_x^{-\f12})\mathcal{Z}_x^{\f12}\\
		&-\mathscr{C}_l(t)(\mfS\pa^{\al-\om}_x\mc) \mathcal{Z}_x^{\f12})+\Big(\mfS\na_x\big(-\f13 \mathscr{C}_l(t)\mfS\na_x\cdot ((\pa^\al_x\mb) \mathcal{Z}_x^{\f12})-\mathscr{C}_l(t)\mfR x\cdot\na_x ((\pa^\al_x\mc) \mathcal{Z}_x^{\f12})+\pa^\al_x(\sfT_{31}\mathcal{Z}_x^{-\f12})\\
		&\times \mathcal{Z}_x^{\f12}-\f13(\mathscr{C}_l(t)\mfS\pa^{\al-\om}_x\cdot\mb) \mathcal{Z}_x^{\f12}-\mathscr{C}_l(t)(\mathrm{tr}(\mfR\pa_x^{\al\pm\om})\mc) \mathcal{Z}_x^{\f12}\big),-\pa^\al_x(\sfT_{52}\mathcal{Z}_x^{-\f12})\mathcal{Z}_x^{\f12}\mathcal{Z}_x^{-(1-\de_1)}\Big).
	\end{aligned}
\end{equation*}
Comparing to the case $|\al|=0$, we have three new terms which contain \[(\mfS\pa^{\al-\om}_x\mc) \mathcal{Z}_x^{\f12},~~(\mfS\pa^{\al-\om}_x\cdot\mb) \mathcal{Z}_x^{\f12}\quad \mbox{and} \quad (\mathrm{tr}(\mfR\pa_x^{\al\pm\om})\mc) \mathcal{Z}_x^{\f12}.\]
 Using Cauchy-Schwarz inequality and integration by parts, we can derive that
\ben
	 &&\f d{dt}(\mfS\na_x ((\pa^\al_x\mc) \mathcal{Z}_x^{\f12}),-\pa^\al_x(\sfT_{52}\mathcal{Z}_x^{-\f12})\mathcal{Z}_x^{\f12}\mathcal{Z}_x^{-(1-\de_1)})_{L^2_x}\ls -\f{\mathscr{C}_l(t)} 2\|\mfS\na_x((\pa^\al_x\mc) \mathcal{Z}_x^{\f12})\mathcal{Z}_x^{-\f{1-\de_1}2}\|^2_{L^2_{x,v}}\nonumber\\
		\notag&&+ \|\pa^\al_x(\sfT_{51}\mathcal{Z}_x^{-\f12})\mathcal{Z}_x^{\f{\de_1}2}\|^2_{L^2_x}+\|\mathcal{P}_x\na_x\pa^\al_x(\sfT_{52}e^{-\f12|x|^2})e^{\f{2\de_1-\de_3}2|x|^2}\|_{L^2_x}(\|\mathbb{P}f\|_{\mathcal{H}^{|\al|+1}_{\de_3/2,x}L^2_v}+\|\pa^\al_x(\sfT_{31}\mathcal{Z}_x^{-\f12})\mathcal{Z}_x^{\de_3}\|_{L^2_x})\\
		\label{pac}\notag&&+ \|\mathcal{P}_x\na_x\pa^\al_x(\sfT_{52}\mathcal{Z}_x^{-\f12})\mathcal{Z}_x^{\f{2\de_1-\de_3}2}\|^2_{L^2_x}+(\|(\mfS\pa_x^{\al-\om}\mc) \mathcal{Z}_x^{\f{\de_1}2}\|^2_{L^2_x}+\|(\mfS\pa_x^{\al-\om}\cdot \mb) \mathcal{Z}_x^{\f{\de_2}2}\|^2_{L^2_x}\nonumber\\&&+\|(\mathrm{tr}(\mfR\pa_x^{\al\pm\om})\mc) \mathcal{Z}_x^{\f{\de_1}2}\|^2_{L^2_x}) 
		\ls -\f{\mathscr{C}_l(t)} 2\|\mfS\na_x(\pa^\al_x\mc \mathcal{Z}_x^{\f12})\mathcal{Z}_x^{-\f{1-\de_1}2}\|^2_{L^2_{x,v}}+ \Big(\|\mathcal{P}_v^{10}(\mathbb{I}-\mP)f\|^2_{\mathcal{H}^{|\al|+1,\de_1}_{x}L^2_v}\nonumber\een\ben&&+\sum_{i=1}^{13}\int_{\R^3_x}\mathcal{Z}_x^{2\de_1}(\pa^\al_xg,e_i)^2_{L^2_v}dx\Big) +\|\mathcal{P}_v^{10}(\mathbb{I}-\mP)f\|_{\mathcal{H}^{|\al|+1,\de_1}_{x}L^2_v}\|\mathbb{P}f\|_{\mathcal{H}^{|\al|+1,\de_3/2}_{x}L^2_v}\nonumber\\
		&&+  (\|(\mfS\pa_x^{\al-\om}\mc) \mathcal{Z}_x^{\f{\de_1}2}\|^2_{L^2_x}+\|(\mfS\pa_x^{\al-\om}\cdot \mb) \mathcal{Z}_x^{\f{\de_2}2}\|^2_{L^2_x}+\|(\mathrm{tr}(\mfR\pa_x^{\al\pm\om})\mc )\mathcal{Z}_x^{\f{\de_1}2}\|^2_{L^2_x}),
\een
where we use Lemma \ref{estimateofT} in the final step.

\underline{Step 2: Estimates of $\pa^\al_x\mb$ and $\pa^\al_x\ma$.} By the similar argument, one may check that
\begin{equation}\label{pab}
	\begin{aligned}
		&\f d{dt}(\na^{sym}_\mfS((\pa^\al_x\mb) \mathcal{Z}_x^{\f12}),(-\pa^\al_x(\sfT_{42}\mathcal{Z}_x^{-\f12})\mathcal{Z}_x^{\f12}+(\pa^\al_x\mc) \mathcal{Z}_x^{\f12})\mathcal{Z}_x^{-(1-\de_2)})_{L^2_x}\ls  -\f{\mathscr{C}_l(t)} 2\\
		&\times\|\na^{sym}_\mfS((\pa^\al_x\mb) \mathcal{Z}_x^{\f12}) \mathcal{Z}_x^{-\f{1-\de_2}2}\|^2_{L^2_x}+ \|\mfS\na_x((\pa^\al_x\mc) \mathcal{Z}_x^{\f12})\mathcal{Z}_x^{-\f{1-\de_1}2}\|^2_{L^2_x}+\big(\|\mathcal{P}_v^{10}(\mathbb{I}-\mP)\|^2_{\mathcal{H}^{|\al|+1,\de_1}_{x}L^2_v}\\
		&+\sum_{i=1}^{13}\int_{\R^3_x}\mathcal{Z}_x^{2\de_1}(\pa^\al_xg,e_i)^2_{L^2_v}dx\big)+ \big(\|\mfS\na_x((\pa^\al_x\mc) \mathcal{Z}_x^{\f12})\mathcal{Z}_x^{-\f{1-\de_1}2}\|_{L^2_x}+\|\mathcal{P}_v^{10}(\mathbb{I}-\mP)f\|_{\mathcal{H}^{|\al|+1,\de_1}_{x}L^2_v}\big)\\
		&\times\|\mathbb{P}f\|_{\mathcal{H}^{|\al|+1,\de_3/2}_{x}L^2_v}+(\|(\pa_x^{\al-\om}\mb) \mathcal{Z}_x^{\f{\de_2}2}\|^2_{L^2_x}+\|(\mathrm{tr}(\mfR\pa_x^{\al\pm\om})\mc) \mathcal{Z}_x^{\f{\de_2}2}\|^2_{L^2_x}),
	\end{aligned}
\end{equation}
and
\begin{equation}\label{paa}
	\begin{aligned}
		&\f d{dt}(\mfS\na_x((\pa^\al_x\ma) \mathcal{Z}_x^{\f12}),(\pa^\al_x\mb) \mathcal{Z}_x^{\f12}\mathcal{Z}_x^{-(1-\de_3)})_{L^2_x}\ls  -\f{\mathscr{C}_l(t)} 2\|\mfS\na_x((\pa^\al_x\ma) \mathcal{Z}_x^{\f12})\mathcal{Z}_x^{-\f{1-\de_3}2}\|^2_{L^2_x}+ \big(\|\na^{sym}_\mfS((\pa^\al_x\mb) \mathcal{Z}_x^{\f12})\\
		&\times  \mathcal{Z}_x^{-\f{1-\de_2}2}\|^2_{L^2_x}+ \|\mfS\na_x((\pa^\al_x\mc) \mathcal{Z}_x^{\f12})\mathcal{Z}_x^{-\f{1-\de_1}2}\|^2_{L^2_x}\big)
		+ (\|\mathcal{P}_v^{10}(\mathbb{I}-\mP)f\|^2_{\mathcal{H}^{|\al|+1,\de_1}_{x}L^2_v}+\sum_{i=1}^{13}\int_{\R^3_x}\mathcal{Z}_x^{2\de_1}(g,e_i)^2_{L^2_v}dx)\\
		&+ \big(\|\na^{sym}_\mfS((\pa^\al_x\mb) \mathcal{Z}_x^{\f12}) \mathcal{Z}_x^{-\f{1-\de_2}2}\|_{L^2_x}+\|\mathcal{P}_v^{10}(\mathbb{I}-\mP)f\|_{\mathcal{H}^{|\al|+1,\de_1}_{x}L^2_v}\big) \|\mathbb{P}f\|_{\mathcal{H}^{|\al|+1,\de_3/2}_{x}L^2_v}\\ 
		&+ (\|(\pa_x^{\al-\om}\ma) \mathcal{Z}_x^{\f{\de_3}2}\|^2_{L^2_x}+\|(\mathrm{tr}(\mfR\pa_x^{\al\pm\om})\ma) \mathcal{Z}_x^{\f{\de_3}2}\|^2_{L^2_x}+\|(\mathrm{tr}(\mfR\pa_x^{\al\pm\om})\mb) \mathcal{Z}_x^{\f{\de_2}2}\|^2_{L^2_x}).
	\end{aligned}
\end{equation}

Patching together \eqref{pac},\eqref{pab} and \eqref{paa} and choosing suitable constants $\vep_{N,1}$ and $\vep_{N,2}$, we can obtain a new energy inequality similar to \eqref{0order}, i.e., if 
\beno   \widetilde{\mathcal{F}}_{N,\de_3}&:=&(\mfS\na_x ((\pa^\al_x\mc) \mathcal{Z}_x^{\f12}),-\pa^\al_x(\sfT_{52}\mathcal{Z}_x^{-\f12})\mathcal{Z}_x^{\f12}\mathcal{Z}_x^{-(1-\de_1)})_{L^2_x}+\vep_{N,1}(\na^{sym}_\mfS((\pa^\al_x\mb) \mathcal{Z}_x^{\f12}),(-\pa^\al_x(\sfT_{42}\mathcal{Z}_x^{-\f12})\\&&\times \mathcal{Z}_x^{\f12}
+(\pa^\al_x\mc) \mathcal{Z}_x^{\f12})\mathcal{Z}_x^{-(1-\de_2)})_{L^2_x}+\vep_{N,2}(\mfS\na_x((\pa^\al_x\ma) \mathcal{Z}_x^{\f12}),(\pa^\al_x\mb) \mathcal{Z}_x^{\f12}\mathcal{Z}_x^{-(1-\de_3)})_{L^2_x},\\
 \mathcal{D}_N&:=&\sum_{|\al|=N-1}(\|\mfS\na_x((\pa^\al_x\mc) \mathcal{Z}_x^{\f1 2 })\mathcal{Z}_x^{-\f{1-\de_1}2}\|^2_{L^2_x}+\vep_{N,1}\|\na^{sym}_\mfS((\pa^\al_x\mb) \mathcal{Z}_x^{\f12})\mathcal{Z}_x^{-\f{1-\de_2}2}\|^2_{L^2_x}\\
		&&+\vep_{N,2}\|\mfS\na_x((\pa^\al_x\ma) \mathcal{Z}_x^{\f12})\mathcal{Z}_x^{-\f{1-\de_3}2}\|^2_{L^2_x}),
\eeno
then it holds that
\begin{equation}\label{Norder}
	\begin{aligned}
		&\f d{dt} \widetilde{\mathcal{F}}_{N,\de_3}+\f{\mathscr{C}_l(t)} 4\mathcal{D}_N\ls\Big(\|\mathbb{P}f\|^2_{\mathcal{H}^{0,\de_3/2}_{x}L^2_v}+\|\mathcal{P}_v^{10}(\mathbb{I}-\mP)f\|^2_{\mathcal{H}^{N,\de_1}_{x}L^2_v}+\sum_{i=1}^{13}\int_{\R^3_x}\mathcal{Z}_x^{2\de_1}(\pa^\al_xg,e_i)^2_{L^2_v}dx\Big)\\
		&+ \sum_{|\al|=N-1}(\|(\mfS\pa_x^{\al-\om}\mc) \mathcal{Z}_x^{\f{\de_1}2}\|^2_{L^2_x}+\|(\mathrm{tr}(\mfR\pa_x^{\al\pm\om})\mc) \mathcal{Z}_x^{\f{\de_1}2}\|^2_{L^2_x}+\|(\mfS\pa_x^{\al-\om}\mb) \mathcal{Z}_x^{\f{\de_2}2}\|^2_{L^2_x}\\
		&+\|(\mathrm{tr}(\mfR\pa_x^{\al\pm\om})\mb) \mathcal{Z}_x^{\f{\de_2}2}\|^2_{L^2_x}+\|(\pa_x^{\al-\om}\ma) \mathcal{Z}_x^{\f{\de_3}2}\|^2_{L^2_x}+\|(\mathrm{tr}(\mfR\pa_x^{\al\pm\om})\ma) \mathcal{Z}_x^{\f{\de_3}2}\|^2_{L^2_x}).
	\end{aligned}
\end{equation}
Here we have $|\widetilde{\mathcal{F}}_{N,\de_3}|\leq C_{\de_3,N}\|\mathcal{P}_v^{10}f\|^2_{\mathcal{H}^{N,\de_1}_{x}L^2_v}$.
Comparing to the case $\al=0$, we have to control the additional terms  in the right hand side of \eqref{Norder}. Let
\begin{equation*}
	\begin{aligned}
\mathcal{I}_N:=& \|\mathbb{P}f\|^2_{\mathcal{H}^{0,\de_3/2}_{x}L^2_v}+\sum_{|\al|=N-1}(\|(\mfS\pa_x^{\al-\om}\mc) \mathcal{Z}_x^{\f{\de_1}2}\|^2_{L^2_x}+\|(\mathrm{tr}(\mfR\pa_x^{\al\pm\om})\mc) \mathcal{Z}_x^{\f{\de_1}2}\|^2_{L^2_x}+\|(\mfS\pa_x^{\al-\om}\mb) \mathcal{Z}_x^{\f{\de_2}2}\|^2_{L^2_x}\\&+\|(\mathrm{tr}(\mfR\pa_x^{\al\pm\om})\mb) \mathcal{Z}_x^{\f{\de_2}2}\|^2_{L^2_x}+\|(\pa_x^{\al-\om}\ma) \mathcal{Z}_x^{\f{\de_3}2}\|^2_{L^2_x}+\|(\mathrm{tr}(\mfR\pa_x^{\al\pm\om})\ma) \mathcal{Z}_x^{\f{\de_3}2}\|^2_{L^2_x}).
	\end{aligned}
\end{equation*}
Thanks to Lemma \ref{lemKPI} and the condition \eqref{vanishpaabc}, we have 
\ben\label{lowbdDN} \mathcal{D}_N&\gs& \sum_{|\al|=N-1}(\|(|\mathcal{P}_x (\pa^\al_x\mc)|+| \na_x(\pa^\al_x\mc)| )\mathcal{Z}_x^{ \f{\de_1}2}\|^2_{L^2_x}+ \|(|\mathcal{P}_x (\pa^\al_x\mb)|+| \na_x(\pa^\al_x\mb)| )\mathcal{Z}_x^{ \f{\de_2}2}\|^2_{L^2_x}\notag\\
		&&+ \|(|\mathcal{P}_x (\pa^\al_x\ma)|+| \na_x(\pa^\al_x\ma)| )\mathcal{Z}_x^{ \f{\de_3}2}\|^2_{L^2_x}),  \een 
which implies that
\begin{equation}\label{I2N}
	\begin{aligned}
\mathcal{I}_N&\leq C_{\de_3,N}\sum_{|\al|\leq N-1} (\|(\pa^\al_x\mc)\mathcal{Z}_x^{\f{\de_1}2}\|^2_{L^2_x}+\|(\pa^\al_x\mb)\mathcal{Z}_x^{\f{\de_2}2}\|^2_{L^2_x}+\|(\pa^\al_x\ma)\mathcal{Z}_x^{\f{\de_3}2}\|^2_{L^2_x})\ls \sum_{k=1}^{N-1}\mathcal{D}_{k}.
	\end{aligned}
\end{equation}

We can complete the proof by induction argument.  Suppose that there exists a energy functional $\mathcal{F}_{N-1}$ verifies all the properties stated in Theorem \ref{macroestimate}. Let $\mathcal{F}_{N}:=\mathcal{F}_{N-1}+\vep_N\widetilde{\mathcal{F}}_N$. Then by \eqref{lowbdDN} and \eqref{I2N}, it is easy to check that the new energy functional   $\mathcal{F}_{N}$ satisfies the properties in Theorem \ref{macroestimate} if $\vep_N$ is sufficiently small. This ends the proof.
\end{proof}

\section{Proof of Theorem  \ref{Thmnonstability}}\label{Section nonlinear stability}
To prove our main results, we first recall that \eqref{pertubNSlandaucauchy} can be rewritten as 
\begin{equation}\label{co=1}
\partial_t f + Tf =\mathscr{C}\big(Q(G, f)+Q(f,\mathcal{M})\big),
\end{equation}
where  $T:=\mathscr{C}_l(t)(\mfS v \cdot \nabla_x - \mfS x \cdot \nabla_v + \mfR x \cdot \nabla_x - \mfR v \cdot \nabla_v)$ and $\mathcal{M}=\f1{(2\pi)^3}e^{-\f12(|x|^2+|v|^2)}$.
Taking derivative $\pa^\al_\be:=\pa^\al_x\pa^\be_v$ on \eqref{co=1}($|\al|+|\be|= N$) with $1\le N\le5$, we get that
\begin{equation}\label{albeEq}
\pa_t \pa^\al_\be f+\pa^\al_\be (Tf)=\mathscr{C}\bigg[Q(G,\pa^\al_\be f)+\sum_{|\al_1|+|\be_1|\geq1}C_\al^{\al_1}C_{\be}^{\be_1}Q(\pa^{\al_1}_{\be_1}G,\pa^{\al_2}_{\be_2}f)+\sum_{\substack{\al_1+\al_2=\al\\ \be_1+\be_2=\be}}C_\al^{\al_1}C_{\be}^{\be_1} Q(\pa^{\al_1}_{\be_1}f,\pa^{\al_2}_{\be_2}\mathcal{M})\bigg].
\end{equation}
 
 \subsection{Energy space and  dissipation functional} To familiarize ourselves with the energy space and its corresponding dissipation functional, we incorporate the use of a space-time cutoff function and space-velocity weights as follows:
\smallskip

 \noindent$\bullet$  Let $\widetilde{\phi}:\R\rightarrow [0,1]$ is a non-increasing and smooth function satisfying that $\widetilde{\phi}(z)=1$ if $z\leq0$, $ \widetilde{\phi}(z)\in (0,1)$ if $z\in (0,1)$ and $\widetilde{\phi}(z)=0$ if $z\geq1$. Moreover, $|(\widetilde{\phi})'|\leq 10$. The space-time cutoff function is defined by 
\ben\label{depsi}
\psi(t,x):=\widetilde{\phi}^2\Big(\mathcal{Z}_x^{\f14}-q_0\ln(t+e)\Big),\quad \widetilde{\psi}_{\pm}(t,x):=\widetilde{\phi}\Big(\mathcal{Z}_x^{\f14}-q_0\ln(t+e)\mp 2\Big)
\een 
with $q_0=q_0(r)\in(\f12,1)$, to be chosen in the later. It is easy to see that $\psi \mathcal{Z}_x^{\f14}\ls \<t\>^{q_0}$ and $(1-\psi)\mathcal{Z}_x^{-\f14}\ls \<t\>^{-q_0}$. Now we may  split the term $Q(G,\pa^\al_\be f)$  by 
\ben\label{befo}
Q(G,\pa^\al_\be f)&=&Q(\psi G,\pa^\al_\be f)+Q((1-\psi) G,\pa^\al_\be f)\\
\notag&=&(2\pi)^{-\f32}\psi \mathcal{Z}_x^{-\f12}Q(\mu,\pa^\al_\be f)+\psi Q(f,\pa^\al_\be f)+Q((1-\psi) G,\pa^\al_\be f).
\een

\noindent$\bullet$ Let $\al,\be\in\Z_+^3, \a\in \R$.  Recalling \eqref{exponentialfunction} and \eqref{exponentialfunctionxv}, 
we define the special exponential space-velocity weight as follows:
\begin{equation}\label{deW}
	\begin{aligned}
(W^{\al,\be}_\a)^2(x,v):=\mathcal{N}\mathcal{Z}_{x,v}^{\a-\eta (|\al|+|\be|)}+\mathcal{Z}_{x,v}^{\a-\c-\vth-\eta (|\al|+|\be|)} \mfS^{-1} x\cdot v,
	\end{aligned}
\end{equation}
where  $\mathcal{N},\c,\vth$ are defined in Lemma \ref{mixturegainweight} and $\eta\ll (|\al|+|\be|)^{-1}(\a-\f12)$. One may easily check that
\[0\leq (W^{\al,\be}_\a)^2\sim \mathcal{Z}_{x,v}^{\a-\eta (|\al|+|\be|)}.\]

\noindent$\bullet$ Let $\al,\be\in\Z_+^3, \a\in (\f12,1)$. We define 
\ben\label{DissLan} &&\|f\|_{\mathcal{D}_v}^2:=\|f\|^2_{H^1_{-\f32}}+\|  \mathbb{T}_{\S^2} f\|^2_{L^2_{-\f32}}+\| f\|^2_{L^2_{-\f12}}; \mathbf{E}_\a^{\al,\be}(f):=\|W^{\al,\be}_\a\pa^\al_\be f\|^2_{L^2_{x,v}};\\
  \label{Dalbe}
&&\mathbf{D}_\a^{\al,\be}(f):=\|\psi^{\f12}\mathcal{Z}_x^{-1/4}W^{\al,\be}_\a\na_v\pa^\al_\be f\|^2_{L^2_xL^2_{-\f32}}+\| \mathcal{Z}_x^{-1/4}W^{\al,\be}_\a\pa^\al_\be f\|^2_{L^2_xL^2_{-\f12}}\nonumber\\
&&+\|\psi^{\f12}\mathcal{Z}_x^{-1/4}W^{\al,\be}_\a\mathbb{T}_{\S^2}\pa^\al_\be f\|^2_{L^2_xL^2_{-\f32}}+\|\mathcal{P}_{x,v}^1\mathcal{Z}_{x,v}^{\f12(\a-\c-\vth-\eta (|\al|+|\be|))} \pa^\al_\be f\|^2_{L^2_{x,v}},\\
&&\mathbf{D}_0^{\al,\be}(f):=\|\psi^{\f12}\mathcal{Z}_v^{1/4}(\mathbb{I}-\mathbb{P})\pa^\al_\be f\|^2_{L^2_x\mathcal{D}_v}+\|(1-\psi)^{\f12}\mathcal{Z}_v^{1/4}(\mathbb{I}-\mathbb{P})\pa^\al_\be f\|^2_{L^2_xL^2_{-1/2}},\\
&&\label{EDNI} \mathbf{E}^{\mathbf{I}}_{\a,N}(f):=\sum_{|\al|+|\be|\le N}C(\al,\be)\mathbf{E}_\a^{\al,\be}(f);\quad \mathbf{D}^{\mathbf{I}}_{\a,N}(f):=\sum_{|\al|+|\be|\le N} C(\al,\be)\mathbf{D}_\a^{\al,\be}(f); \\
&&\label{EDNII} \mathbf{E}^{\mathbf{II}}_{0,N}(f):=\sum_{|\al|+|\be|\le N} \|\mathcal{Z}_{x,v}^{\f14}\pa^\al_\be f\|_{L^2}^2;\quad \mathbf{D}^{\mathbf{II}}_{0,N}(f):=\sum_{|\al|+|\be|\le N} \mathbf{D}_0^{\al,\be}(f).\een 
Here $C(\al,\be)$ is a positive function on $\Z^3_+\times \Z^3_+$, which will be fixed in the proof of Lemma \ref{fixcalbe}. 
\medskip

Before performing the energy estimates, we present a auxiliary lemma.
\begin{lem}\label{H2paal}
	Let $\psi(t,x)$ be defined in \eqref{depsi} and $1\le m\le 5$. It holds that
	\ben
	&&\|\psi\mathcal{Z}_x^{\f12}f\|_{H^2_xL^2_6}\ls \<t\>^{q_0} \sum_{|\al|\leq 2}\|\mathcal{P}^2_x\mathcal{Z}_{x,v}^{\f14}\pa^\al_xf\|_{L^2_{x,v}};\label{H2pa}\\
	&&\|\mathcal{P}_v^2(1-\psi)f\|_{H^m_{x,v}}\ls \<t\>^{-2q_0(\f12+\de_1-5\eta)}\sum_{|\al|+|\be|\leq 5}\|W_{\f12+\de_1}^{\al,\be}\pa^\al_\be f\|_{L^2_{x,v}},~\de_1>0.\label{H2pa2}
	\een
\end{lem}
\begin{proof}
We begin with the proof of \eqref{H2pa}. By definition, it holds that
\ben\label{51}
\|\psi\mathcal{Z}_x^{\f12}f\|^2_{H^2_xL^2_6}\ls\sum_{|\al_1|+|\al_2|+|\al_3|\leq 2}\int_{\R^3_x\times\R^3_v} |\pa^{\al_1}_x\psi\pa^{\al_2}_x\mathcal{Z}_x^{\f12}\pa^{\al_3}_x f\mathcal{P}_v^{6}|^2dvdx.
\een
Note that $\widetilde{\psi}_+ \mathcal{Z}_x^{\f14}\ls \<t\>^{q_0}$ and    $\mathcal{P}_v^{6}\ls \mathcal{Z}_v^{\f14}$, which implies that \[|\pa^{\al_1}_x\psi\pa^{\al_2}_x\mathcal{Z}_x^{\f12}\mathcal{P}_v^{6}|\ls |\widetilde{\psi}_+\mathcal{P}_x^2\mathcal{Z}_x^{\f12}\mathcal{P}_v^{6}|\ls \<t\>^{q_0} |\mathcal{P}_x^2\mathcal{Z}_{x,v}^{\f14}|.\]
Thus we can obtain \eqref{H2pa} by plugging it into \eqref{51}.

For \eqref{H2pa2}, we first note that
\beno
\|\mathcal{P}_v^2(1-\psi)f\|^2_{H^m_{x,v}}\ls\sum_{\substack{|\al_1|+|\al_2|+|\al_3|\\+|\be_1|+|\be_2|+|\be_3|\le m}}\int_{\R^3_x\times\R^3_v}|\pa^{\al_1}_{\be_1}\mathcal{P}_v^2\pa^{\al_2}_{\be_2}(1-\psi) \pa^{\al_3}_{\be_3}f|^2dvdx.
\eeno
Observe that $|\pa^{\al_1}_{\be_1}\mathcal{P}_v^2\pa^{\al_2}_{\be_2}(1-\psi)|\ls\<t\>^{-2q_0(\f12+\de_1-5\eta)}(W^{\al_3,\be_3}_{\f12+\de_1})$. Thus \eqref{H2pa2} holds and it ends the proof of this lemma.
\end{proof}

\subsection{Energy method(I)} Thanks to \eqref{albeEq} and \eqref{befo}, basic energy method will imply that 
\beno
	&&\f12\pa_t(\pa^\al_\be f,(W^{\al,\be}_\a)^2\pa^\al_\be f)_{L^2_{x,v}}=-(\pa^\al_\be(Tf),(W^{\al,\be}_\a)^2\pa^\al_\be f)_{L^2_{x,v}} 
	+\f{\mathscr{C}}{(2\pi)^{\f32}}(\psi \mathcal{Z}_x^{-1/2}Q(\mu,\pa^\al_\be f),(W^{\al,\be}_\a)^2\pa^\al_\be f)_{L^2_{x,v}}\\&&+ \mathscr{C}(\psi Q(f,\pa^\al_\be f),(W^{\al,\be}_\a)^2\pa^\al_\be f)_{L^2_{x,v}}+\mathscr{C}(Q((1-\psi) G,\pa^\al_\be f),(W^{\al,\be}_\a)^2\pa^\al_\be f)_{L^2_{x,v}}+\mathscr{C}\Big(\sum_{|\al_1|+|\be_1|\geq1}C_\al^{\al_1}C_{\be}^{\be_1}\\
	&&\times \big(Q(\pa^{\al_1}_{\be_1}G,\pa^{\al_2}_{\be_2}f),(W^{\al,\be}_\a)^2\pa^\al_\be f\big)_{L^2_{x,v}}+\sum_{\substack{\al_1+\al_2=\al\\ \be_1+\be_2=\be}}C_\al^{\al_1}C_{\be}^{\be_1}\big(Q(\pa^{\al_1}_{\be_1}f,\pa^{\al_2}_{\be_2}\mathcal{M}),(W^{\al,\be}_\a)^2\pa^\al_\be f\big)_{L^2_{x,v}}\Big)\eqdefa\sum_{i=1}^6\Ga_i.
\eeno

\subsubsection{Estimates of $\Ga_1$ and $\Ga_2$.} We first have
\begin{lem}
Let $\a\in(\f12,1)$.There exists a positive constant $\lambda_\a$ such that
	\ben\label{Ga1lam}
	\Ga_1+\Ga_2\leq -\lam_\a\mathbf{D}_\a^{\al,\be}(f)+C_\a\<t\>^{-2q_0}\|W^{\al,\be}_\a \pa^\al_\be f\|^2_{L^2_{x,v}}+C_\a\|\psi^{\f12}\pa^\al_\be f\|^2_{L^2_{x,v}}+([T,\pa^\al_\be] f,(W_\a^{\al,\be})^2\pa^\al_\be f)_{L^2_{x,v}},
	\een
	where  $q_0$ is defined in \eqref{depsi}.

\end{lem}
\begin{proof} Let us give the estimates term by term.
\smallskip

\noindent\underline{Step 1: Estimate of $\Ga_{1}$}. We observe that
\beno
\Ga_1&=&-(T\pa^\al_\be f,(W_\a^{\al,\be})^2\pa^\al_\be f)_{L^2_{x,v}}+([T,\pa^\al_\be] f,(W_\a^{\al,\be})^2\pa^\al_\be f)_{L^2_{x,v}}\\
&=&-\f{\mathscr{C}_l(t)}2(\pa^\al_\be f,(|x|^2-|v|^2-2(\mfR\mfS^{-1}x,v))\mathcal{Z}_{x,v}^{\a-\c-\vth-\eta(|\al|+|\be|)}\pa^\al_\be f)_{L^2_{x,v}}+([T,\pa^\al_\be] f,(W_\a^{\al,\be})^2\pa^\al_\be f)_{L^2_{x,v}}.
\eeno

 \noindent\underline{Step 2: Estimate of $\Ga_{2}$}. It is easy to verify that
\beno 
 &&\f{(2\pi)^{\f32}}{\mathscr{C}}\Ga_2=\mathcal{N}(\psi \mathcal{Z}_x^{-1/2}Q(\mu,\pa^\al_\be f),\pa^\al_\be f\mathcal{Z}_{x,v}^{\a-\eta (|\al|+|\be|)})_{L^2_{x,v}}\\&&+(\psi \mathcal{Z}_x^{-1/2}Q(\mu,\pa^\al_\be f), 
\pa^\al_\be f\mathcal{Z}_{x,v}^{\a-\c-\vth-\eta (|\al|+|\be|)}\mfS^{-1}x\cdot v)_{L^2_{x,v}}
\eqdefa\Ga_{2,1}+\Ga_{2,2}.
\eeno

Noticing that we can spilt
\ben\label{split}
&&\Ga_{2,2}=\Big(\psi \mathcal{Z}_x^{-1/2}Q(\mu,\pa^\al_\be f \mathcal{Z}_{x,v}^{\f12(\a-\c-\vth-\eta (|\al|+|\be|))}), 
\pa^\al_\be f\mfS^{-1}x\cdot v \mathcal{Z}_{x,v}^{\f12(\a-\c-\vth-\eta (|\al|+|\be|))}\Big)_{L^2_{x,v}}+\Big(\psi \mathcal{Z}_x^{-1/2}\\
\notag&&\times\big(\mathcal{Z}_{x,v}^{\f12(\a-\c-\vth-\eta (|\al|+|\be|))}Q(\mu,\pa^\al_\be f )-Q(\mu,\pa^\al_\be f \mathcal{Z}_{x,v}^{\f12(\a-\c-\vth-\eta (|\al|+|\be|))})\big), 
\pa^\al_\be f\mfS^{-1}x\cdot v \mathcal{Z}_{x,v}^{\f12(\a-\c-\vth-\eta (|\al|+|\be|))}\Big)_{L^2_{x,v}}.
\een

\noindent Thus applying Lemma \ref{coer}, Lemma \ref{ghf} and Lemma \ref{expQ} for $\Ga_{2,1}$ and $\Ga_{2,2}$, we get that
\beno
&&\Ga_{2}\leq -\lam_\a \mathcal{N}\big(\mathbf{D}_\a^{\al,\be}(f)-\|\sqrt{1-\psi} \mathcal{Z}_x^{-1/4}W^{\al,\be}_\a\pa^\al_\be f\|^2_{L^2_xL^2_{-\f12}}-\|\mathcal{P}_{x,v}^1\mathcal{Z}_{x,v}^{\f12(\a-\c-\vth-\eta (|\al|+|\be|))} \pa^\al_\be f\|^2_{L^2_{x,v}}\big) \\&&\qquad\quad+C_\a\|\psi^{\f12}\mathcal{Z}_{x}^{\f{\a-\eta(|\al|+|\be|)}2-\f14}\pa^\al_\be f\|^2_{L^2_{x,v}}+ C_\a \|\psi^{\f12}\mathcal{P}_x^{\f12}\mathcal{P}_v^{\f12}\pa^\al_\be f\mathcal{Z}_{x,v}^{\f12(\a-\c-\vth-\eta (|\al|+|\be|))}\mathcal{Z}_x^{-1/4}\|_{L^2_xH^1_v}^2\\
&&\leq -\f12\lam_\a \mathcal{N}\big(\mathbf{D}_\a^{\al,\be}(f) -\|\mathcal{P}_{x,v}^1\mathcal{Z}_{x,v}^{\f12(\a-\c-\vth-\eta (|\al|+|\be|))} \pa^\al_\be f\|^2_{L^2_{x,v}}\big)+C_\a\lr{t}^{-2q_0}\|W^{\al,\be}_\a\pa^\al_\be f\|^2_{L^2_{x,v}} \\&&\qquad\quad+C_\a\|\psi^{\f12}\mathcal{Z}_{x}^{\f{\a-\eta(|\al|+|\be|)}2-\f14}\pa^\al_\be f\|^2_{L^2_{x,v}}. 
\eeno
Here we use the facts  $(W^{\al,\be}_\a)^2\sim e^{(\a-\eta (|\al|+|\be|))(|x|^2+|v|^2)}$, $(1-\psi)e^{-\f{|x|^2}2}\leq C\<t\>^{-2q_0}$, $\mathcal{N}\gg1$,
\beno
 &&[W_\a^{\al,\be}, \mathbb{T}_{\S^2}]f=\f12(W^{\al,\be}_\a)^{-1}\mathcal{Z}_{x,v}^{(\a-\c-\epsilon-\eta (|\al|+|\be|))} \mfS^{-1}   x\times v f,~~\mbox{and}\\
&&C_\a \|\psi^{\f12}\mathcal{P}_x^{\f12}\mathcal{P}_v^{\f12}\pa^\al_\be f\mathcal{Z}_{x,v}^{\f12(\a-\c-\vth-\eta (|\al|+|\be|))}\mathcal{Z}_x^{-1/4}\|_{L^2_xH^1_v}^2\ls  \mathbf{D}_\a^{\al,\be}(f)-\|\mathcal{P}_{x,v}^1\mathcal{Z}_{x,v}^{\f12(\a-\c-\vth-\eta (|\al|+|\be|))} \pa^\al_\be f\|^2_{L^2_{x,v}} .
\eeno

Putting together the estimates for $\Ga_1$ and $\Ga_2$, by Lemma \ref{mixturegainweight}, we deduce that there exists a constant $\tilde{\lam}_\a>0$ such that
\beno 
\Ga_1+\Ga_2-([T,\pa^\al_\be] f,(W_\a^{\al,\be})^2\pa^\al_\be f)_{L^2_{x,v}}\le  -\tilde{\lam}_\a \mathbf{D}_\a^{\al,\be}(f)+C_\a\lr{t}^{-2q_0}\|W^{\al,\be}_\a\pa^\al_\be f\|^2_{L^2_{x,v}} +C_\a \|\psi^{\f12}\mathcal{Z}_{x}^{\f{\a-\eta(|\al|+|\be|)}2-\f14}\pa^\al_\be f\|^2_{L^2_{x,v}}.\eeno 
 To get the desired result, we observe that 
$\c+\vth<\f12$ which together with Lemma \ref{interpolationineq} yield that 
\begin{equation*}
	\begin{aligned}
	\|\psi^{\f12}\mathcal{Z}_x^{\f{\a-\eta(|\al|+|\be|)}2-\f14}\pa^\al_\be f\|^2_{L^2_{x,v}}\leq \f{\tilde{\lam}_\a} 2\|\mathcal{P}_{x,v}^1\mathcal{Z}_x^{\f {\a-\c-\vth -\eta(|\al|+|\be|)}2}\pa^\al_\be f\|^2_{L^2_{x,v}}+C_\a\|\psi^{\f12}\pa^\al_\be f\|^2_{L^2_{x,v}}.
	\end{aligned}
\end{equation*} 
This is enough to conclude our result.
\end{proof}

\subsubsection{Estimates of $\Ga_3$ and $\Ga_4$.} We have
\begin{lem} \label{lemGa23}
	For any $0<\de_1\leq\a-\f12$ with $\a\in(\f12,1)$ and $|\al|+|\be|=N\leq5$, it holds that
\begin{equation}\label{Ga23}
	\begin{aligned}
&|\Ga_3|+|\Ga_4|
	\leq (\lam_\a/4+C_\a\<t\>^{2q_0}\sum_{|\al|+|\be|\le 4}\| \mathcal{P}_x^2\mathcal{Z}_{x,v}^{1/4}\pa^{\al}_{\be} f\|^2_{L^2_{x,v}})\mathbf{D}_\a^{\al,\be}(f)+C_{\a}\Big(\<t\>^{-2q_0(1-\eta)}\\
	&\qquad\qquad\quad+\<t\>^{-2q_0(\f12+\de_1-5\eta)}\sum_{|\al|+|\be|\leq 4}\|W_{\f12+\de_1}^{\al,\be}\pa^\al_\be f\|_{L^2_{x,v}}\Big)\|W^{\al,\be}_\a \pa^\al_\be f \|^2_{L^2_{x,v}}.
	\end{aligned}
\end{equation}
\end{lem}
\begin{proof} We recall that 
\beno  \f{1}{\mathscr{C}}\Ga_3=\mathcal{N}(\psi  Q(f,\pa^\al_\be f),\pa^\al_\be f\mathcal{Z}_{x,v}^{\a-\eta (|\al|+|\be|)})+(\psi Q(f,\pa^\al_\be f), 
\pa^\al_\be f\mathcal{Z}_{x,v}^{\a-\c-\vth-\eta (|\al|+|\be|)}\mfS^{-1}x\cdot v):=\mathcal{N}\Ga_{3,1}+\Ga_{3,2}. \eeno 
Using the similar split as \eqref{split}, we have 
\beno
&&\Ga_{3,1}=\big(\psi  Q(f,\pa^\al_\be f \mathcal{Z}_{x,v}^{\f12(\a-\eta (|\al|+|\be|))}),\pa^\al_\be f\mathcal{Z}_{x,v}^{\f12(\a-\eta (|\al|+|\be|))}\big)\\
&&+\big(\psi  \big(\mathcal{Z}_{x,v}^{\f12(\a-\eta (|\al|+|\be|))} Q(f,\pa^\al_\be f )-Q(f,\pa^\al_\be f \mathcal{Z}_{x,v}^{\f12(\a-\eta (|\al|+|\be|))})\big),\pa^\al_\be f\mathcal{Z}_{x,v}^{\f12(\a-\eta (|\al|+|\be|))}\big),\\
&& \Ga_{3,2}=\Big(\psi Q(f,\pa^\al_\be f \mathcal{Z}_{x,v}^{\f12(\a-\c-\vth-\eta (|\al|+|\be|))}), 
\pa^\al_\be f\mfS^{-1}x\cdot v \mathcal{Z}_{x,v}^{\f12(\a-\c-\vth-\eta (|\al|+|\be|))}\Big)+\Big(\psi \\
\notag&&\times\big(\mathcal{Z}_{x,v}^{\f12(\a-\c-\vth-\eta (|\al|+|\be|))}Q(f,\pa^\al_\be f )-Q(f,\pa^\al_\be f \mathcal{Z}_{x,v}^{\f12(\a-\c-\vth-\eta (|\al|+|\be|))})\big), 
\pa^\al_\be f\mfS^{-1}x\cdot v \mathcal{Z}_{x,v}^{\f12(\a-\c-\vth-\eta (|\al|+|\be|))}\Big).
\eeno
Then thanks to  Lemma \ref{ghf}, Lemma \ref{expQ} and the fact that $\psi\widetilde{\psi}_+=\psi$(due to \eqref{depsi}), we have
 \ben\label{nonestityp1} |\Ga_3|\ls \|\widetilde{\psi}_+ \mathcal{Z}_x^{1/2}f\|_{H^2_xL^2_6}\mathbf{D}_\a^{\al,\be}(f)&\ls& \<t\>^{q_0}\sum_{|\al|+|\be|\leq 4}\|\mathcal{P}_x^2\mathcal{Z}_{x,v}^{1/4}\pa^\al_\be f \|_{L^2_{x,v}}\mathbf{D}_\a^{\al,\be}(f)\\
 &\ls&\notag (\lam_\a/4+C_\a\<t\>^{2q_0}\sum_{|\al|+|\be|\le 4}\| \mathcal{P}_x^2\mathcal{Z}_{x,v}^{1/4}\pa^{\al}_{\be} f\|^2_{L^2_{x,v}})\mathbf{D}_\a^{\al,\be}(f),
   \een 
 where we use \eqref{H2pa}.

Since $\Ga_4=\mathscr{C}(Q((1-\psi) G,\pa^\al_\be f),(W^{\al,\be}_\a)^2\pa^\al_\be f)$ and $G\ge0$, 
by Lemma \ref{Fff}, we have 
 \begin{equation*}
	\begin{aligned}
&|\Ga_4|\ls \|\mathcal{P}_v^2(1-\widetilde{\psi}_-)G\|_{H^4_{x,v}}\|(1-\psi)^{\f12}W^{\al,\be}_\a\pa^\al_\be f \|^2_{L^2_{x,v}}\\
 &\ls ~(\|\mathcal{P}_v^2(1-\widetilde{\psi}_-)\mathcal{M}\|_{H^4_{x,v}}+\|\mathcal{P}_v^2(1-\widetilde{\psi}_-)f\|_{H^4_{x,v}})\|(1-\psi)^{\f12}W^{\al,\be}_\a\pa^\al_\be f \|^2_{L^2_{x,v}},
 	\end{aligned}
\end{equation*}
where we use the fact  that $(1-\psi)(1-\widetilde{\psi}_-)=(1-\psi)$ due to \eqref{depsi}. Due to \eqref{H2pa2} and 
  \\$\mathcal{P}^5_v\mathcal{P}_x^3(1-\widetilde{\psi}_-)|\pa^\al_\be \mathcal{M}|\leq C_\eta\<t\>^{-2q_0(1-\eta)}$ ,  we  obtain that
 \begin{equation}\label{Gamma3}
	\begin{aligned}
		|\Ga_4|\leq&C_{\a}\Big(\<t\>^{-2q_0(1-\eta)}+\<t\>^{-2q_0(\f12+\de_1-5\eta)}\sum_{|\al|+|\be|\leq 4}\|W_{\f12+\de_1}^{\al,\be}\pa^\al_\be f\|_{L^2_{x,v}}\Big)\|W^{\al,\be}_\a \pa^\al_\be f \|^2_{L^2_{x,v}}.
	\end{aligned}
\end{equation}
We conclude the desired results by putting together all the estimates. This ends the proof.
\end{proof}

\subsubsection{Estimates of $\Ga_5$.} We will prove
\begin{lem}
For any $0<\de_1\leq\a-\f12$  with $\a\in(\f12,1)$ and $|\al|+|\be|=N\le5$, it holds that
\begin{equation}\label{Ga4}
	\begin{aligned}
&|\Ga_{5}|\leq ~(\lam_\a/4+C_\a\<t\>^{2q_0}\sum_{|\tilde{\al}|+|\tilde{\be}|\le 5}\| \mathcal{P}_x^2\mathcal{Z}_{x,v}^{1/4}\pa^{\tilde{\al}}_{\tilde{\be}} f\|^2_{L^2_{x,v}})\mathbf{D}_\a^{\al,\be}(f)+C_\a\sum_{|\al_2|+|\be_2|\leq N-1 }\mathbf{D}_\a^{\al_2,\be_2}(f)\\
&+C_{\a}\Big(\<t\>^{-2q_0(1-\eta)}+\<t\>^{-2q_0(\f12+\de_1-5\eta)}\sum_{|\al|+|\be|\leq 5}\|W_{\f12+\de_1}^{\al,\be}\pa^\al_\be f\|_{L^2_{x,v}}\Big)\sum_{|\al|+|\be|\leq 5}\|W^{\al,\be}_\a \pa^\al_\be f \|^2_{L^2_{x,v}}.
\end{aligned}
\end{equation}
\end{lem}
\begin{proof} Recalling the definition, we only  provide a proof for the typical term  $(Q(\pa^{\al_1}_{\be_1}G,\pa^{\al_2}_{\be_2}f),(W^{\al,\be}_\a)^2\pa^\al_\be f)$ with $|\al_1|+|\be_1|\geq1$. We split it into two parts: $\Ga_{5,1}$ and $\Ga_{5,2}$, where 
\beno \Ga_{5,1}:=( Q(\psi\pa^{\al_1}_{\be_1}G,\pa^{\al_2}_{\be_2}f),(W^{\al,\be}_\a)^2\pa^\al_\be f),\quad \Ga_{5,2}:=( Q((1-\psi)\pa^{\al_1}_{\be_1}G,\pa^{\al_2}_{\be_2}f),(W^{\al,\be}_\a)^2\pa^\al_\be f). \eeno

\noindent $\bullet$ \underline{Estimate of $\Ga_{5,1}$.} It is easy to check that   
\begin{equation*}
	\begin{aligned}
\Ga_{5,1} &= ( Q(\psi\pa^{\al_1}_{\be_1}\mathcal{M},\pa^{\al_2}_{\be_2}f),(W^{\al,\be}_\a)^2\pa^\al_\be f)+( Q(\psi\pa^{\al_1}_{\be_1}f,\pa^{\al_2}_{\be_2}f),(W^{\al,\be}_\a)^2\pa^\al_\be f):=\Ga_{5,1,1}+\Ga_{5,1,2}.
\end{aligned}
\end{equation*}

\noindent For $\Ga_{5,1,1}$, using the similar split as $\Ga_{3,1}$ and $\Ga_{3,2}$, then Lemma \ref{ghf} and Lemma \ref{expQ} imply that
\begin{equation*}
\begin{aligned}
|\Ga_{5,1,1}|\leq&~ C_\a\int_{\R_x^3} \psi\|\pa^{\al_1}_{\be_1}\mathcal{M}\|_{L^2_6}\big(\|W^{\al,\be}_\a\mathbb{T}_{\S^2}\pa^{\al_2}_{\be_2} f \|_{L^2_{-\f32}}+\|W^{\al,\be}_\a\pa^{\al_2}_{\be_2} f\|_{H^1_{-\f32}}+\|W^{\al,\be}_\a\pa^{\al_2}_{\be_2} f\|_{L^2_{-\f12}}\big)\\
&\times\big(\|W^{\al,\be}_\a\mathbb{T}_{\S^2}\pa^\al_\be f \|_{L^2_{-\f32}}+\|W^{\al,\be}_\a\pa^\al_\be f\|_{H^1_{-\f32}}+\|W^{\al,\be}_\a\pa^\al_\be f\|_{L^2_{-\f12}}\big)dx.
\end{aligned}
\end{equation*}
Since $\|\pa^{\al_1}_{\be_1}\mathcal{M}\|_{L^2_6}\leq C_\eta e^{-(\f12-\f\eta 4)|x|^2}$ and $W^{\al,\be}_\a\leq C_\a e^{-\f \eta2(|x|^2+|v|^2)}W_\a^{\al_2,\be_2}$ with $|\al_1|+|\be_1|\geq1$,   we   obtain that
\begin{equation*}
\begin{aligned}
|\Ga_{5,1,1}|&\leq   \f {\lam_\a}4\mathbf{D}_\a^{\al,\be}(f)+C_{\a} \sum_{|\al_2|+|\be_2|\leq N-1 }\mathbf{D}_\a^{\al_2,\be_2}(f).
\end{aligned}
\end{equation*}
 
We can copy the above argument for $\Ga_{5,1,2}$   to get that
\begin{equation*}
	\begin{aligned}
|\Ga_{5,1,2}|\leq& ~C_\a\int_{\R_x^3} \psi\|\pa^{\al_1}_{\be_1}f\|_{L^2_6}\big(\|W^{\al,\be}_\a\mathbb{T}_{\S^2}\pa^{\al_2}_{\be_2} f \|_{L^2_{-\f32}}+\|W^{\al,\be}_\a\pa^{\al_2}_{\be_2} f\|_{H^1_{-\f32}}+\|W^{\al,\be}_\a\pa^{\al_2}_{\be_2} f\|_{L^2_{-\f12}}\big)\\
&\times\big(\|W^{\al,\be}_\a\mathbb{T}_{\S^2}\pa^\al_\be f \|_{L^2_{-\f32}}+\|W^{\al,\be}_\a\pa^\al_\be f\|_{H^1_{-\f32}}+\|W^{\al,\be}_\a\pa^\al_\be f\|_{L^2_{-\f12}}\big)dx.
\end{aligned}
\end{equation*}
If $|\alpha_1|+|\beta_1|\le 3$, we apply the Sobolev embedding theorem to $\pa^{\al_1}_{\be_1}f$. And if $|\alpha_1|+|\beta_1|\ge 4$, we apply the Sobolev embedding theorem to $\pa^{\al_2}_{\be_2}f$. Then we get that 
 
\beno 
&&|\Ga_{5,1,2}|\ls  \sum_{1\le|\al_1|+|\be_1|\le3} \|\widetilde{\psi}_+\mathcal{Z}_{x}^{1/2}\pa^{\al_1}_{\be_1}f\|_{H^2_xL^2_6}\mathbf{D}_\a^{\al_2,\be_2}(f)^{1/2}\mathbf{D}_\a^{\al,\be}(f)^{1/2}\\
&&+\sum_{|\al_1|+|\be_1|\ge4}\|\widetilde{\psi}_+\mathcal{Z}_{x}^{1/2}\pa^{\al_1}_{\be_1}f\|_{L^2_xL^2_6} \Big(\|\psi^{\f12}\mathcal{Z}_x^{-\f14}W^{\al,\be}_\a\mathbb{T}_{\S^2}\pa^{\al_2}_{\be_2} f \|_{H^2_xL^2_{-\f32}}\\
&&+\|\psi^{\f12}\mathcal{Z}_x^{-\f14}W^{\al,\be}_\a\na_v\pa^{\al_2}_{\be_2} f\|_{H^2_xL^2_{-\f32}}+\|\psi^{\f12}\mathcal{Z}_x^{-\f14}W^{\al,\be}_\a\pa^{\al_2}_{\be_2} f\|_{H^2_xL^2_{-\f12}}\Big)\mathbf{D}_\a^{\al,\be}(f)^{1/2}.
\eeno 
 Due to Cauchy-Schwarz inequality and  \eqref{H2pa}, the first term has the upper bound as follows:
 \beno
 \mathbf{D}_\a^{\al_2,\be_2}(f)+ \<t\>^{2q_0} (\sum_{|\tilde{\al}|+|\tilde{\be}|\le 5}\| \mathcal{P}_x^2\mathcal{Z}_{x,v}^{1/4}\pa^{\tilde{\al}}_{\tilde{\be}} f\|^2_{L^2_{x,v}})\mathbf{D}_\a^{\al,\be}(f).
 \eeno
For the second term, we only  give the estimate for the term $\|\psi^{\f12}\mathcal{Z}_x^{-\f14}W^{\al,\be}_\a\na_v\pa^{\al_2}_{\be_2} f\|^2_{H^2_xL^2_{-\f32}}$. The others can be treated similarly. In fact, similar to \eqref{H2pa}, we get that 
\ben\label{upperDV1}
\|\psi^{\f12}\mathcal{Z}_x^{-\f14}W^{\al,\be}_\a\na_v\pa^{\al_2}_{\be_2} f\|^2_{H^2_xL^2_{-\f32}}
&\ls& \sum_{|\al_2|+|\be_2|\leq N-1} \Big(\<t\>^{-2q_0(1-\eta)}\|W^{\al,\be}_\a \pa^{\al_2}_{\be_2}f\|^2_{L^2_{x,v}}+\mathbf{D}_\a^{\al_2,\be_2}(f)\Big),
\een 
where we use the facts that $W^{\al,\be}_\a\leq C_\a e^{-\f \eta2(|x|^2+|v|^2)}W_\a^{\al_2,\be_2}$ with $|\al_1|+|\be_1|\geq1$ and $\mathcal{P}_x^2(1-\tilde{\psi}_-)^{\f12}\mathcal{Z}_x^{-\f14}\ls \<t\>^{-q_0(1-\eta)}$. Thus the second term can be bounded by 
\beno
\<t\>^{2q_0} (\sum_{|\tilde{\al}|+|\tilde{\be}|\le 5}\| \mathcal{P}_x^2\mathcal{Z}_{x,v}^{1/4}\pa^{\tilde{\al}}_{\tilde{\be}} f\|^2_{L^2_{x,v}})\mathbf{D}_\a^{\al,\be}(f)+\sum_{|\al_2|+|\be_2|\leq N-1} \Big(\<t\>^{-2q_0(1-\eta)}\|W^{\al,\be}_\a \pa^{\al_2}_{\be_2}f\|^2_{L^2_{x,v}}+\mathbf{D}_\a^{\al_2,\be_2}(f)\Big).
\eeno 

Hence we conclude that
\begin{equation}\label{Ga41}
	\begin{aligned}
|\Ga_{5,1}|~\leq&~ (\f{\lam_\a}4+C_\a\<t\>^{2q_0}\sum_{|\tilde{\al}|+|\tilde{\be}|\le 5}\| \mathcal{P}_x^2\mathcal{Z}_{x,v}^{1/4}\pa^{\tilde{\al}}_{\tilde{\be}} f\|^2_{L^2_{x,v}})\mathbf{D}_\a^{\al,\be}(f)\\
&+C_\a\sum_{|\al_2|+|\be_2|\leq N-1} \Big(\<t\>^{-2q_0(1-\eta)}\|W^{\al,\be}_\a \pa^{\al_2}_{\be_2}f\|^2_{L^2_{x,v}}+\mathbf{D}_\a^{\al_2,\be_2}(f)\Big).
\end{aligned}
\end{equation}

\noindent  $\bullet$ \underline{Estimate of $\Ga_{5,2}$.} We further split $\Ga_{5,2}$ into two cases:  $|\al_1|+|\be_1|=1$ and $|\al_1|+|\be_1|\ge2$.
 If $|\al_1|+|\be_1|=1$, then due to Lemma \ref{IBP}, we have  
\begin{equation*}
	\begin{aligned} |\Ga_{5,2}|&\le C_\a (\|(1-\widetilde{\psi}_-)\mathcal{P}_{x,v}^3G\|_{H^4_{x,v}}+\|(1-\widetilde{\psi}_-)\mathcal{P}_{v}^3G\|_{H^5_{x,v}})\\
		\times&\Big(\|(1-\psi)^{\f12}W^{\al,\be}_\a\pa^{\al_2}_{\be_2+\tilde{\be}} f\|^2_{L^2_{x,v}}+\|(1-\psi)^{\f12}W^{\al,\be}_\a\pa^\al_\be f\|^2_{L^2_{x,v}}\Big),
	\end{aligned}
\end{equation*}
where $\tilde{\be}\in \Z^3_+$ satisfy $|\tilde{\be}|\leq1$.  From this together with  \eqref{Gamma3}, we deduce that
 \begin{equation*}
	\begin{aligned}
		|\Ga_{5,2}|\leq&C_{\a}\Big(\<t\>^{-2q_0(1-\eta)}+\<t\>^{-2q_0(\f12+\de_1-5\eta)}\sum_{|\al|+|\be|\leq 5}\|W_{\f12+\de_1}^{\al,\be}\pa^\al_\be f\|_{L^2_{x,v}}\Big)\sum_{|\al|+|\be|\leq N}\|W^{\al,\be}_\a \pa^\al_\be f \|^2_{L^2_{x,v}}.
	\end{aligned}
\end{equation*}

When $2\le |\al_1|+|\be_1|\le N-2$,  due to Lemma \ref{ghf} and Lemma \ref{expQ}, we have 
 \begin{equation*}
	\begin{aligned}
|\Ga_{5,2}|\leq& C_\a \|(1-\widetilde{\psi}_-)\mathcal{P}_v^6G\|_{H^N_{x,v}}\|W^{\al,\be}_\a\pa^{\al_2}_{\be_2} f\|_{L^2_xH^2_{v}}\|W^{\al,\be}_\a\pa^\al_\be f\|_{L^2_{x,v}}.
	\end{aligned}
\end{equation*} In the case of $|\al_1|+|\be_1|\ge N-1$, we derive that 
\begin{equation*}
	\begin{aligned}
|\Ga_{5,2}|\leq& C_\a \|(1-\widetilde{\psi}_-)\mathcal{P}_v^6G\|_{H^N_{x,v}}\|W^{\al,\be}_\a\pa^{\al_2}_{\be_2} f\|_{H^4_{x,v}}\|W^{\al,\be}_\a\pa^\al_\be f\|_{L^2_{x,v}}.
	\end{aligned}
\end{equation*}

 We can conclude that
 \begin{equation*}
	\begin{aligned}
		|\Ga_{5,2}|\leq&C_{\a}\Big(\<t\>^{-2q_0(1-\eta)}+\<t\>^{-2q_0(\f12+\de_1-5\eta)}\sum_{|\al|+|\be|\leq 5}\|W_{\f12+\de_1}^{\al,\be}\pa^\al_\be f\|_{L^2_{x,v}}\Big)\sum_{|\al|+|\be|\le 5}\|W^{\al,\be}_\a \pa^\al_\be f \|^2_{L^2_{x,v}}.
	\end{aligned}
\end{equation*}
 
 Combining the estimates of $\Ga_{5,1}$ and $\Ga_{5,2}$, we can get the upper bound of $\Ga_5$. It ends the proof of the lemma.
\end{proof}

\subsubsection{Estimates of $\Ga_6$.}  We want to show
\begin{lem}
	For any $0<\de_1\leq\a-\f12$  with $\a\in(\f12,1)$ and $|\al|+|\be|=N\le 5$, it holds that
	 \begin{equation}\label{Ga5}
		\begin{aligned}
|\Ga_6|&\leq \f{\lam_\a}4 \mathbf{D}_\a^{\al,\be}(f)+C_\a\sum_{|\al_1|+|\be_1|\leq N-1}\mathbf{D}_\a^{\al_1,\be_1}(f)\\
+&C_\a \<t\>^{-2q_0(1-5\eta)}\sum_{|\al|+|\be|\leq 5}\|W^{\al,\be}_\a \pa^\al_\be f\|^2_{L^2_{x,v}}+ C_\a\sum_{|\al_1|\leq|\al|,|\be_1|\leq|\be|}\|\pa^{\al_1}_{\be_1}f\|^2_{L^2_{x,v}}.
	\end{aligned}
\end{equation}
\end{lem}
\begin{proof} 
We only deal with the typical term $\big(Q(\pa^{\al_1}_{\be_1}f,\pa^{\al_2}_{\be_2}\mathcal{M}),(W^{\al,\be}_\a)^2\pa^\al_\be f\big)_{L^2_{x,v}}$. Let
\beno
\Ga_{6,1}:=\big(Q(\psi \pa^{\al_1}_{\be_1}f,\pa^{\al_2}_{\be_2}\mathcal{M}),(W^{\al,\be}_\a)^2\pa^\al_\be f\big),\quad \Ga_{6,2}:=\big(Q((1-\psi)\pa^{\al_1}_{\be_1}f,\pa^{\al_2}_{\be_2}\mathcal{M}),(W^{\al,\be}_\a)^2\pa^\al_\be f\big).
\eeno

For $\Ga_{6,2}$,  by Lemma \ref{ghf} and Lemma \ref{expQ}, we can derive that
\beno
|\Ga_{6,2}|\ls \| \mathcal{Z}_x^{-(\f12-\f\a2 )}\mathcal{P}_x^{8}\mathcal{P}_v^{6}(1-\psi)^{\f12}\pa^{\al_1}_{\be_1} f\|_{L^2_{x,v}}\|W^{\al,\be}_\a \pa^\al_\be f\|_{L^2_{x,v}}\ls \<t\>^{-2q_0(1-5\eta)}\sum_{|\al|+|\be|\leq 5}\|W^{\al,\be}_\a \pa^\al_\be f\|^2_{L^2_{x,v}}.
\eeno
For $\Ga_{6,1}$, also   by Lemma \ref{ghf} and Lemma \ref{expQ}, we have 
\beno
|\Ga_{6,1}|&\ls& \|\mathcal{Z}_x^{\f\a2-\f14-\f\eta 2(|\al|+|\be|) }\mathcal{P}_x^8\mathcal{P}_v^6\psi^{\f12}\pa^{\al_1}_{\be_1}f\|_{L^2_{x,v}}(\mathbf{D}_\a^{\al,\be}(f))^{\f12}\\
&\leq& \f{\lam_\a}4 \mathbf{D}_\a^{\al,\be}(f)+C_\a\sum_{|\al_1|+|\be_1|\leq N-1}\mathbf{D}_\a^{\al_1,\be_1}(f)+C_\a\sum_{|\al_1|+|\be_1|\leq 5} \|\pa^{\al_1}_{\be_1}f\|^2_{L^2_{x,v}},
\eeno
where we use Lemma \ref{interpolationineq} in the last step since $\a-\f12<\a-\c-\vartheta$. 
The desired result follows  the estimates of $\Ga_{6,1}$ and $\Ga_{6,2}$. And 
 this ends the proof.
\end{proof}

\subsubsection{Energy estimates} We are in a position to prove
\begin{lem}\label{fixcalbe}
For any $0<\de_1\leq\a-\f12$  with $\a\in(\f12,1)$, it holds that 
\ben\label{EN}
	&& \f d{dt}\mathbf{E}^{\mathbf{I}}_{\a,5}(f)+\f{\lambda_\a}8\mathbf{D}^{\mathbf{I}}_{\a,5}(f)\leq C_\a \<t\>^{2q_0}\mathbf{E}^{\mathbf{I}}_{\f12+6\eta,5}(f)\mathbf{D}^{\mathbf{I}}_{\a,5}(f)+C_\a\big(\<t\>^{-2q_0(1-5\eta)}\nonumber\\ &&\qquad+\mathbf{E}^{\mathbf{I}}_{\f12+\de_1,5}(f)\big)\mathbf{E}^{\mathbf{I}}_{\a,5}(f)
		+C_\a (\sum_{|\al|+|\be|=5}\|\psi^{\f12}\pa^\al_\be f\|^2_{L^2_{x,v}}+\|\psi^{\f12}f\|^2_{L^2_{x,v}}).
 \een
\end{lem}
\begin{proof}  Patch  together \eqref{Ga1lam},\eqref{Ga23},\eqref{Ga4} and \eqref{Ga5}, and sum up with respect to $|\al|+|\be|=N\leq 5$. Due to Lemma \ref{trans}, i.e.,
	$\sum\limits_{|\al|+|\be|=N}([T,\pa^\al_\be] f,(W_\a^{\al,\be})^2\pa^\al_\be f)_{L^2_{x,v}}=0$, and $6\eta<\de_1$(see \eqref{epsiloneta}), \[\<t\>^{-2q_0(\f12+\de_1-5\eta)}\sum\limits_{|\al|+|\be|\leq 5}\|W^{\al,\be}_{\f12+\de_1}\pa^\al_\be f\|_{L^2_{x,v}}\ls \<t\>^{-2q_0(1-5\eta)}+\sum\limits_{|\al|+|\be|\leq 5}\|W^{\al,\be}_{\f12+\de_1}\pa^\al_\be f\|^2_{L^2_{x,v}},\] we can derive that
	\begin{equation}\label{E1}
	\begin{aligned}
&\pa_t\sum_{|\al|+|\be|=N}\|W^{\al,\be}_\a\pa^\al_\be f\|^2_{L^2_{x,v}}+\f{\lam_\a}4\sum_{|\al|+|\be|= N}\mathbf{D}_\a^{\al,\be}(f)\leq C_\a\<t\>^{2q_0}(\sum_{|\tilde{\al}|+|\tilde{\be}|\le 5}\| \mathcal{P}_x^2\mathcal{Z}_{x,v}^{1/4}\pa^{\tilde{\al}}_{\tilde{\be}} f\|^2_{L^2_{x,v}})\sum_{|\al|+|\be|= N}\mathbf{D}_\a^{\al,\be}(f)\\
&+C_{\a}\Big(\<t\>^{-2q_0(1-5\eta)}+\sum_{|\al|+|\be|\leq 5}\|W_{\f12+\de_1}^{\al,\be}\pa^\al_\be f\|^2_{L^2_{x,v}}\Big) \sum_{|\al|+|\be|\le 5}\|W^{\al,\be}_\a \pa^\al_\be f\|^2_{L^2_{x,v}}+C_\a(\|\psi^{\f12}f\|^2_{L^2_{x,v}} \\
& +\sum_{|\al|+|\be|=5}\|\psi^{\f12}\pa^\al_\be f\|^2_{L^2_{x,v}})  + C_\a\sum_{|\al_2|+|\be_2|\leq N-1 }\mathbf{D}_\a^{\al_2,\be_2}(f).
	\end{aligned}
\end{equation}
We remark that the terms in the last line of \eqref{E1} is used to control the last term in \eqref{Ga5} thanks to the interpolation inequality(see Lemma \ref{interpolationineq}) and the trick used in \eqref{Gamma3}, i.e., 
\begin{equation*}
	\begin{aligned}
		\sum_{|\al|+|\be|\leq 5}\|\pa^\al_\be f\|^2_{L^2_{x,v}} 
		\ls& (\sum_{|\al|+|\be|=5}\|\psi^{\f12}\pa^\al_\be f\|^2+\|\psi^{\f12}f\|^2_{L^2_{x,v}})+\<t\>^{-2q_0(1+2\de_1-5\eta)}\sum_{|\al|+|\be|=5}\|W_{\f12+\de_1}^{\al,\be}\pa^\al_\be f\|^2_{L^2_{x,v}},
	\end{aligned}
\end{equation*}
where the term $\|W_{\f12+\de_1}^{\al,\be}\pa^\al_\be f\|^2_{L^2_{x,v}}$ can be bounded by $\|W_{\a}^{\al,\be}\pa^\al_\be f\|^2_{L^2_{x,v}}$ since $\a\geq\f12+\de_1$.  

Recalling the definition of $\mathbf{E}^{\mathbf{I}}_{\a,5}(f)$(see \eqref{EDNI}), we set $C:\Z^3_+\times\Z^3_+\rightarrow \R^+_{>0}$ verifies that (i).   $C(\al,\be)< C(\al_1,\be_1)$ if $|\al_1|+|\be_1|<|\al|+|\be|$; (ii). $C(\al,\be)C_\a<(\lam_\a/8) C(\al_1,\be_1)$. Then   \eqref{E1} will imply the desired result thanks to  
 the fact that $\sum_{|\al|+|\be|\le 5}\| \mathcal{P}_x^2\mathcal{Z}_{x,v}^{1/4}\pa^\al_\be f\|^2_{L^2_{x,v}}\ls \mathbf{E}^{\mathbf{I}}_{\f12+6\eta,5}(f)$.
\end{proof}

\subsection{Energy method(II)} We focus on the energy method for \eqref{co=1} in the standard perturbation framework. In our presentation, we first get that 
\begin{equation}\label{albeeq3}
	\begin{aligned}
&\f12\pa_t(\pa^\al_\be f,\mathcal{Z}_{x,v}^{1/2}\pa^\al_\be f)_{L^2_{x,v}}+ (\pa^\al_\be(Tf),\mathcal{Z}_{x,v}^{1/2}\pa^\al_\be f)_{L^2_{x,v}}
=\mathscr{C}\bigg[(2\pi)^{-\f32}(\psi  L(\pa^\al_\be f),\mathcal{Z}_{v}^{1/2}\pa^\al_\be f)_{L^2_{x,v}}+(Q(\psi f,\pa^\al_\be f),\\
&\mathcal{Z}_{x,v}^{1/2}\pa^\al_\be f)_{L^2_{x,v}}+(Q((1-\psi)G,\pa^\al_\be f),\mathcal{Z}_{x,v}^{1/2}\pa^\al_\be f)_{L^2_{x,v}}+\big(\sum_{|\al_1|+|\be_1|\geq 1}C_\al^{\al_1}C_\be^{\be_1}Q(\pa^{\al_1}_{\be_1}G, \pa^{\al_2}_{\be_2}f\big)+Q((1-\psi)\pa^\al_\be f,\\&\mathcal{M}), \mathcal{Z}_{x,v}^{1/2}\pa^\al_\be f)_{L^2_{x,v}}  +\sum_{|\al_2|+|\be_2|\geq 1}C_\al^{\al_1}C_\be^{\be_1}Q(\pa^{\al_1}_{\be_1}f,\pa^{\al_2}_{\be_2}\mathcal{M}), \mathcal{Z}_{x,v}^{1/2}\pa^\al_\be f)_{L^2_{x,v}}\big)\bigg]:=\mathscr{C}\sum_{i=1}^5\mathcal{I}_{i}.
	\end{aligned}
\end{equation}
\subsubsection{Estimates of $\mathcal{I}_i$} Let us give the estimates term by term.   Thanks to Lemma \ref{coer}\eqref{i-p} and the fact that $\mathcal{I}_2$ and $\mathcal{I}_3$ are very similar to $\Ga_3$ and $\Ga_4$, we first get that 
\beno  &&\mathcal{I}_1\leq -\lam \|\psi^{\f12} \mathcal{Z}_{v}^{1/4}(\mathbb{I}-\mP)\pa^\al_\be f\|^2_{L^2_{x}\mathcal{D}_v}\leq -\lam \mathbf{D}_0^{\al,\be}(f)+C_\a \<t\>^{-2q_0(1-5\eta)}\mathbf{E}^{\mathbf{I}}_{\a,5}(f),\\
&&|\mathcal{I}_2|+|\mathcal{I}_3|\ls \lr{t}^{q_0}(\mathbf{E}^{\mathbf{I}}_{\f12+6\eta,5}(f))^{\f12} \|\psi^{\f12} \mathcal{Z}_{v}^{1/4}\pa^\al_\be f\|_{L^2_{x}\mathcal{D}_v}^2+(\<t\>^{-2q_0(1-5\eta)}+\mathbf{E}^{\mathbf{I}}_{\f12+\de_1,5}(f)) \mathbf{E}^{\mathbf{I}}_{\a,5}(f).
\eeno

 For $\mathcal{I}_4$, we may copy the argument used for $\Ga_5$ and $\Ga_6$ to get that
\beno 
 |\mathcal{I}_4|&\ls& \sum_{|\al_1|+|\be_1|\ge1}\int_{\R_x^3} \psi\|\pa^{\al_1}_{\be_1}G\|_{L^2_6}\|\mathcal{Z}_{x,v}^{1/4}\pa^{\al_2}_{\be_2} f\|_{\mathcal{D}_v}\|\mathcal{Z}_{x,v}^{1/4}\pa^{\al}_{\be} f\|_{\mathcal{D}_v}dx\\
 &&+\Big(\<t\>^{-2q_0(1-5\eta)}+\sum_{|\al|+|\be|\leq 5}\|W_{\f12+\de_1}^{\al,\be}\pa^\al_\be f\|^2_{L^2_{x,v}}\Big)\sum_{|\al|+|\be|\le 5}\|W^{\al,\be}_\a \pa^\al_\be f \|^2_{L^2_{x,v}}.
\eeno
\noindent$\bullet$ Recalling that $G=\mathcal{M}+f$, we first observe that  for any $R>0$ and $0<\epsilon_1,\epsilon_2\ll1$, 
\beno  &&\int_{\R_x^3} \psi\|\pa^{\al_1}_{\be_1}\mathcal{M}\|_{L^2_6}\|\mathcal{Z}_{x,v}^{1/4}\pa^{\al_2}_{\be_2} f\|_{\mathcal{D}_v}\|\mathcal{Z}_{x,v}^{1/4}\pa^{\al}_{\be} f\|_{\mathcal{D}_v}dx\ls \|\psi^{\f12}\mathcal{P}_x^{|\al|}\mathcal{Z}_{v}^{1/4}\pa^{\al_2}_{\be_2} f\|_{L^2_x\mathcal{D}_v}\|\psi^{\f12}\mathcal{Z}_{v}^{1/4}\pa^{\al}_{\be} f\|_{L^2_x\mathcal{D}_v}\\ 
&&\ls (R^{|\al|} \|\psi^{\f12}\mathcal{Z}_{v}^{1/4}\pa^{\al_2}_{\be_2} f\|_{L^2_x\mathcal{D}_v}+C_\a e^{-\f12(\a-\f12-\f\eta 2(|\al_2|+|\be_2|))R^2}\mathbf{D}_\a^{\alpha_2,\be_2}(f)^{\f12})\|\psi^{\f12}\mathcal{Z}_{v}^{1/4}\pa^{\al}_{\be} f\|_{L^2_x\mathcal{D}_v}\\
&&
\ls (\epsilon_1+\epsilon_2) \|\psi^{\f12}\mathcal{Z}_{v}^{1/4}\pa^{\al}_{\be} f\|^2_{L^2_x\mathcal{D}_v}+\epsilon_1 \mathbf{D}_\a^{\alpha_2,\be_2}(f)+|\ln \epsilon_1|^6\epsilon_2^{-1}\|\psi^{\f12}\mathcal{Z}_{v}^{1/4}\pa^{\al_2}_{\be_2} f\|_{L^2_x\mathcal{D}_v}^2.
\eeno

\noindent$\bullet$ If $1\le |\al_1|+|\be_1|\le 3$, by \eqref{H2pa}, it holds that
\beno  &&\int_{\R_x^3} \psi\|\pa^{\al_1}_{\be_1}f\|_{L^2_6}\|\mathcal{Z}_{x,v}^{1/4}\pa^{\al_2}_{\be_2} f\|_{\mathcal{D}_v}\|\mathcal{Z}_{x,v}^{1/4}\pa^{\al}_{\be} f\|_{\mathcal{D}_v}dx\ls \|\widetilde{\psi}_+\mathcal{Z}_x^{\f12}\pa^{\al_1}_{\be_1}f\|_{H^2_xL^2_6}\|\psi^{\f12} \mathcal{Z}_{v}^{1/4}\pa^{\al_2}_{\be_2} f\|_{L^2_x\mathcal{D}_v} \\ 
&&\times\|\psi^{\f12}\mathcal{Z}_{v}^{1/4}\pa^{\al}_{\be} f\|_{L^2_x\mathcal{D}_v}\ls  \epsilon_2^{-1}\lr{t}^{2q_0}\mathbf{E}^{\mathbf{I}}_{\f12+6\eta,5}(f)\|\psi^{\f12}\mathcal{Z}_{v}^{1/4}\pa^{\al}_{\be} f\|_{L^2_x\mathcal{D}_v}^2+\epsilon_2\|\psi^{\f12} \mathcal{Z}_{v}^{1/4}\pa^{\al_2}_{\be_2} f\|_{L^2_x\mathcal{D}_v}^2. \eeno 
While if $ |\al_1|+|\be_1|\ge 4$, similar to \eqref{upperDV1}, it is not difficult to see that
\beno  &&\int_{\R_x^3} \psi\|\pa^{\al_1}_{\be_1}f\|_{L^2_6}\|\mathcal{Z}_{x,v}^{1/4}\pa^{\al_2}_{\be_2} f\|_{\mathcal{D}_v}\|\mathcal{Z}_{x,v}^{1/4}\pa^{\al}_{\be} f\|_{\mathcal{D}_v}dx\ls \|\widetilde{\psi}_+\mathcal{Z}_x^{\f12}\pa^{\al_1}_{\be_1}f\|_{L^2_xL^2_6}\|\psi^{\f12} \mathcal{Z}_{v}^{1/4}\pa^{\al_2}_{\be_2} f\|_{H^2_x\mathcal{D}_v}\\
&&\times\|\psi^{\f12}\mathcal{Z}_{v}^{1/4}\pa^{\al}_{\be} f\|_{L^2_x\mathcal{D}_v}\ls (\lr{t}^{2q_0}\mathbf{E}_{\f12+6\eta,5}(f))^{\f12}  (\mathbf{D}^{\mathbf{I}}_{\a,N-1}(f)+\<t\>^{-2q_0(1-\eta)}\mathbf{E}^{\mathbf{I}}_{\a,5}(f))^{\f12} \|\psi^{\f12}\mathcal{Z}_{v}^{1/4}\pa^{\al}_{\be} f\|_{L^2_x\mathcal{D}_v}\\&&
\ls  \epsilon_2 \mathbf{D}^{\mathbf{I}}_{\a,N-1}(f) +  \epsilon_2^{-1}\lr{t}^{2q_0}\mathbf{E}_{\f12+6\eta,5}(f) \|\psi^{\f12}\mathcal{Z}_{v}^{1/4}\pa^{\al}_{\be} f\|_{L^2_x\mathcal{D}_v}^2 +\<t\>^{-2q_0(1-\eta)}\mathbf{E}^{\mathbf{I}}_{\a,5}(f).  \eeno 

 We conclude that for any $\epsilon_1,\epsilon_2\ll1$,
\beno |\mathcal{I}_4|&\ls&  (\epsilon_2^{-1}\lr{t}^{2q_0}\mathbf{E}^{\mathbf{I}}_{\f12+6\eta,5}(f)+\epsilon_1+\epsilon_2) \|\psi^{\f12}\mathcal{Z}_{v}^{1/4}\pa^{\al}_{\be} f\|_{L^2_x\mathcal{D}_v}^2+(\epsilon_1+\epsilon_2)\mathbf{D}^{\mathbf{I}}_{\a,N-1}(f)+(|\ln\epsilon_1|^6\epsilon_2^{-1}+\epsilon_2)\\&&\times\sum_{|\al_2|+|\be_2|\le N-1}\|\psi^{\f12}\mathcal{Z}_{v}^{1/4}\pa^{\al_2}_{\be_2} f\|_{L^2_x\mathcal{D}_v}^2 +\Big(\<t\>^{-2q_0(1-5\eta)}+\mathbf{E}^{\mathbf{I}}_{\f12+\delta_1,5}(f) \Big)\mathbf{E}^{\mathbf{I}}_{\a,5}(f). \eeno 

 For $\mathcal{I}_5$, by the similar trick as $\mathcal{I}_4$, we may derive that

\beno
|\mathcal{I}_5|&\ls&    (\epsilon_1+\epsilon_2) \|\psi^{\f12}\mathcal{Z}_{v}^{1/4}\pa^{\al}_{\be} f\|_{L^2_x\mathcal{D}_v}^2+\epsilon_1\mathbf{D}^{\mathbf{I}}_{\a,N-1}(f)+|\ln\epsilon_1|^6\epsilon_2^{-1}\sum_{|\al_2|+|\be_2|\le N-1}\|\psi^{\f12}\mathcal{Z}_{v}^{1/4}\pa^{\al_2}_{\be_2} f\|_{L^2_x\mathcal{D}_v}^2 \\&& +\ \<t\>^{-2q_0(1-5\eta)} \mathbf{E}^{\mathbf{I}}_{\a,5}(f).  
\eeno

Note that by Lemma \ref{trans}, it holds that $\sum\limits_{|\al|+|\be|=N}(\pa^\al_\be (Tf),\mathcal{Z}_{x,v}^{1/2}\pa^\al_\be f)_{L^2_{x,v}}=0$. Summarizing the above estimates, we conclude that  for $1\le N\le 5$,
\ben\label{EstPerN1} 
 &&\f{d}{dt}\sum_{|\al|+|\be|=N}\|\mathcal{Z}_{x,v}^{\f14}\pa^\al_\be f\|_{L^2}^2+\lambda   \sum_{|\al|+|\be|=N}(\|\psi^{\f12} \mathcal{Z}_{v}^{1/4}(\mathbb{I}-\mP)\pa^\al_\be f\|^2_{L^2_{x}\mathcal{D}_v}+\|(1-\psi)^{\f12} \mathcal{Z}_{v}^{1/4}(\mathbb{I}-\mP)\pa^\al_\be f\|^2_{L^2_{x}L^2_{-1/2}})\nonumber\\
 &&\ls (\epsilon_2^{-1}\lr{t}^{2q_0}\mathbf{E}^{\mathbf{I}}_{\f12+6\eta,5}(f)+\epsilon_1+\epsilon_2) \|\psi^{\f12}\mathcal{Z}_{v}^{1/4}\pa^\al_\be f\|_{L^2_{x}\mathcal{D}_v}^2+ (\epsilon_1+\epsilon_2)\mathbf{D}^{\mathbf{I}}_{\a,N-1}(f)+ (|\ln\epsilon_1|^6\epsilon_2^{-1}+\epsilon_2)\nonumber\\&&\times\sum_{|\al_2|+|\be_2|\le N-1}\|\psi^{\f12}\mathcal{Z}_{v}^{1/4}\pa^{\al_2}_{\be_2} f\|_{L^2_x\mathcal{D}_v}^2 +C_{\epsilon_1,\epsilon_2}\Big(\<t\>^{-2q_0(1-5\eta)}+\mathbf{E}^{\mathbf{I}}_{\f12+\delta_1,5}(f) \Big)\mathbf{E}^{\mathbf{I}}_{\a,5}(f).
\een
 To control the third term in the right-hand side, we apply interpolation inequality in Lemma \ref{interpolationineq} to get that
 \beno \|\psi^{\f12}\mathcal{Z}_{v}^{1/4}\pa^{\al_2}_{\be_2} f\|_{L^2_x\mathcal{D}_v}^2\ls \epsilon \mathbf{D}_{\a,N}^{\mathbf{I}}(f)
 +C_\epsilon(\|\psi^{\f12}\na_v\pa^{\al_2}_{\be_2} f\|_{L^2_xL^2_{-3/2}}^2+\|\psi^{\f12}\mathbb{T}_{\SS^2}\pa^{\al_2}_{\be_2} f\|_{L^2_xL^2_{-3/2}}^2+\|\psi^{\f12}\pa^{\al_2}_{\be_2} f\|_{L^2_xL^2_{-1/2}}^2). \eeno 
Since $|\al_2|+|\be_2|\le N-1$, it holds that 
\beno &&\|\psi^{\f12}\na_v\pa^{\al_2}_{\be_2} f\|_{L^2_xL^2_{-3/2}}^2\le \|\pa^{\al_2}_{\be_2}\na_v (\psi^{\f12} f)\|_{L^2_xL^2_{-3/2}}^2+\|[\psi^{\f12},\na_v\pa^{\al_2}_{\be_2}] f\|_{L^2_xL^2_{-3/2}}^2
\\&&\ls\epsilon\|\mathcal{P}_v^{-3/2}\psi^{\f12}\na_vf\|_{H^N_{x,v}}^2+C_\epsilon\|\mathcal{P}_v^{-3/2}\psi^{\f12}f\|_{L^2_{x,v}}^2+  \<t\>^{-2q_0(1-5\eta)}\mathbf{E}_{\a,N}^{\mathbf{I}}(f)\\ &&\ls \<t\>^{-2q_0(1-5\eta)}\mathbf{E}_{\a,N}^{\mathbf{I}}(f)
+\epsilon\mathbf{D}_{\a,N}^{\mathbf{I}}(f)+ C_\epsilon\| \psi^{\f12}f\|_{L^2_xL^2_{-3/2} }^2.\eeno 
Similar argument can be applied to obtain that 
\beno \|\psi^{\f12}\mathcal{Z}_{v}^{1/4}\pa^{\al_2}_{\be_2} f\|_{L^2_x\mathcal{D}_v}^2\ls\<t\>^{-2q_0(1-5\eta)}\mathbf{E}_{\a,N}^{\mathbf{I}}(f)
+\epsilon\mathbf{D}_{\a,N}^{\mathbf{I}}(f)+ C_\epsilon\| \psi^{\f12}f\|_{L^2_xL^2_{-1/2} }^2.\eeno

Now substituting the above estimate into \eqref{EstPerN1} and noting that $f=\mathbb{P}f+(\mathbb{I}-\mathbb{P})f$, we get that
\ben\label{EstPerN3} 
 &&\f{d}{dt}\mathbf{E}^{\mathbf{II}}_{0,N}(f)+(\lambda-\epsilon^{-1}C\lr{t}^{2q_0}\mathbf{E}^{\mathbf{I}}_{\f12+6\eta,5}(f)-\epsilon)\mathbf{D}^{\mathbf{II}}_{0,N}(f)\ls (\epsilon^{-1}\lr{t}^{2q_0}\mathbf{E}^{\mathbf{I}}_{\f12+6\eta,5}(f)+\epsilon) \|\mathbb{P}f\|^2_{\mathcal{H}^{N,0}_{x}L^2_v}\nonumber\\
 &&+ \epsilon\mathbf{D}^{\mathbf{I}}_{\a,N}(f)+C_{\epsilon}  \Big(\<t\>^{-2q_0(1-5\eta)}+\mathbf{E}^{\mathbf{I}}_{\f12+\delta_1,5}(f) \Big)\mathbf{E}^{\mathbf{I}}_{\a,5}(f)+ C_{\epsilon}\| \psi^{\f12}f\|_{L^2_xL^2_{-1/2} }^2.
\een
If  $N=0$, one can easily derive that
\ben \label{EstPerN0}
&&\f{d}{dt}\mathbf{E}^{\mathbf{II}}_{0,0}(f)+(\lambda-C\lr{t}^{2q_0}\mathbf{E}_{\f12+6\eta,5}(f))  \mathbf{D}^{\mathbf{II}}_{0,0}(f)\ls  \big(\epsilon^{-1}\lr{t}^{q_0}\mathbf{E}_{\f12+6\eta,5}^{\mathbf{I}}(f)+\epsilon\big)\|\psi^{\f12} \mathbb{P}f\|_{L^2_{x,v}}^2\nonumber  \\
&&+\Big(\<t\>^{-2q_0(1-5\eta)}+\mathbf{E}^{\mathbf{I}}_{\f12+\delta_1,5}(f) \Big)\mathbf{E}^{\mathbf{I}}_{\a,5}(f).
\een

\subsubsection{Micro-macro combination(I)}  By Theorem 
\ref{macroestimate} with $\de:=\f{\a-\c}{16}$, then for $N\leq 5$,
\begin{equation*}
	\begin{aligned}
	\f d{dt} \mathcal{F}_{N,\f{\a-\c}{16}}+\|\mP f\|^2_{\mathcal{H}^{N,\f{\a-\c}{32}}_{x}L^2_v}&\ls  \|\mathcal{P}_v^{10}\mathbb{(I-P)}f\|^2_{\mathcal{H}^{N,\f {\a-\c}{8}}_{x}L^2_v} + \sum_{|\al|\leq N-1}\sum_{i=1}^{13}\int_{\R^3_x}e^{\f{\a-\c}4|x|^2}(\pa^\al_xQ(f,f),e_i)^2_{L^2_v}dx.
	\end{aligned}
\end{equation*}
Here the energy functional $\mathcal{F}_{N,\f{\a-\c}{16}}$ verifies  $|\mathcal{F}_{N,\f{\a-\c}{16}}|\leq C_{N}\|\mathcal{P}_v^{10} f\|^2_{\mathcal{H}^{N,\f{\a-\c}{8}}_{x}L^2_v}$.  We also notice that
\begin{equation*}
	\begin{aligned}
		&\Big|\sum_{|\al|\leq N-1}\sum_{i=1}^{13}\int_{\R^3_x}e^{\f{\a-\c}4|x|^2}(\pa^\al_xQ(f,f),e_i)^2_{L^2_v}dx\Big|\ls
		\mathbf{E}^{\mathbf{I}}_{\f12+\delta_1,5}(f)\mathbf{E}^{\mathbf{I}}_{\a,5}(f).
	\end{aligned}
\end{equation*}

Thanks to Lemma \ref{interpolationineq},  we can derive that
\beno
\|\mathcal{P}_v^{10}\mathbb{(I-P)}f\|^2_{\mathcal{H}^{4,\f {\a-\c}{8}}_{x}L^2_v}&=&\sum_{|\al|\leq 4}\|\mathcal{P}_v^{10}\mathcal{Z}^{\f {\a-\c}{8}}_x \mathbb{(I-P)}\pa^\al_xf\|^2_{L^2_{x,v}}\leq\epsilon \sum_{|\al|\leq 4}\|\mathcal{Z}_{x,v}^{\f {\a-\c}{4}} \mathbb{(I-P)}\pa^\al_xf\|^2_{L^2_{x,v}}\\
&&+C_\epsilon\|\mathbb{(I-P)}f\|^2_{H^4_xL^2_{-\f12}}
\leq\epsilon \mathbf{D}^{\mathbf{I}}_{\a,5}(f)+C_\epsilon \|\mathbb{(I-P)}f\|^2_{H^4_xL^2_{-\f12}}\\
&&\ls \epsilon(\mathbf{D}^{\mathbf{I}}_{\a,5}(f)+\mathbf{D}^{\mathbf{II}}_{0,5}(f))+C_\epsilon\|(\mathbb{I}-\mathbb{P})f\|_{L^2_xL^2_{-1/2}}^2.
\eeno 
By the same manner, it holds that
\beno
\|\mathcal{P}_v^{10}\mathbb{(I-P)}f\|^2_{\mathcal{H}^{5,\f {\a-\c}{8}}_{x}L^2_v}\ls \epsilon \mathbf{D}_{\a,5}^{\mathbf{I}}(f)+C_{\epsilon} \mathbf{D}_{0,5}^{\mathbf{II}}(f).
\eeno
Thus, we can derive that 
 \ben
 &&\f d{dt} \mathcal{F}_{4,\f{\a-\c}{16}}+\|\mP f\|^2_{\mathcal{H}^{4,\f {\a-\c}{32}}_{x}L^2_v}
 \ls \epsilon(\mathbf{D}^{\mathbf{I}}_{\a,5}(f)+\mathbf{D}^{\mathbf{II}}_{0,5}(f)) +C_{\epsilon}\|(\mathbb{I}-\mathbb{P})f\|_{L^2_xL^2_{-1/2}}^2+\mathbf{E}^{\mathbf{I}}_{\f12+\delta_1,5}(f)\mathbf{E}^{\mathbf{I}}_{\a,5}(f);\label{F4}\notag\\
 &&\\
 &&\f d{dt} \mathcal{F}_{5,\f{\a-\c}{16}}+\|\mP f\|^2_{\mathcal{H}^{5,\f{\a-\c}{32}}_{x}L^2_v}
 \ls \epsilon \mathbf{D}^{\mathbf{I}}_{\a,5}(f)+C_\epsilon \mathbf{D}^{\mathbf{II}}_{0,5}(f) +\mathbf{E}^{\mathbf{I}}_{\f12+\delta_1,5}(f)\mathbf{E}^{\mathbf{I}}_{\a,5}(f).\label{F5}
  \een 

Thanks to \eqref{EstPerN0}, \eqref{EstPerN3} and \eqref{F4}, there exist constants $0<B_1\ll B_2\ll B_3<\infty$ such that  
\ben\label{MMC1} &&\f{d}{dt}\big(B_1\mathbf{E}^{\mathbf{II}}_{0,4}+B_2\mathcal{F}_{4,\f{\a-\c}{16}}+B_3\mathbf{E}^{\mathbf{II}}_{0,0}\big)
+B_1(\lambda-C\epsilon_1^{-1}\lr{t}^{2q_0}\mathbf{E}^{\mathbf{I}}_{\f12+6\eta,5}(f)-\epsilon_1)\mathbf{D}^{\mathbf{II}}_{0,4}(f)+(B_2-(B_1\epsilon_1^{-1}\\
\nonumber&&+B_3\epsilon^{-1}_3)\lr{t}^{2q_0} \mathbf{E}^{\mathbf{I}}_{\f12+6\eta,5}(f)-(C_{\epsilon_1}+\epsilon_1)B_1-\epsilon_3B_3)\|\mP f\|^2_{\mathcal{H}^{4,\f{\a-\c}{32}}_{x}L^2_v}+\big(B_3(\lambda-C\lr{t}^{2q_0}\mathbf{E}^{\mathbf{I}}_{\f12+6\eta,5}(f))- (C_{\epsilon_1}B_1 \\ &&+B_2C_{\epsilon_2})) \mathbf{D}^{\mathbf{II}}_{0,0}(f)\ls   (B_1\epsilon_1+B_2\epsilon_2)(\mathbf{D}^{\mathbf{I}}_{\a,5}(f)+\mathbf{D}^{\mathbf{II}}_{0,5}(f))+C_{\epsilon_1,\epsilon_2,\epsilon_3}\Big(\<t\>^{-2q_0(1-5\eta)}+\mathbf{E}^{\mathbf{I}}_{\f12+\delta_1,5}(f) \Big)\mathbf{E}^{\mathbf{I}}_{\a,5}(f).\nonumber
\een 

 Let $M\gg2$. We choose 
 \ben\label{constconstrain1} B_1:=1,B_2>(C_{\epsilon_1}+\epsilon_1)B_1+2\, \mbox{with}\, \epsilon_1<\f\lam4, B_3>4(M+C_{\epsilon_1} B_1+C_{\epsilon_2}B_2)/\lam, \epsilon_3<1/B_3,\een  
 which imply that
 $\lam-\epsilon_1>\lam/2,~~B_2-(\epsilon_1+C_{\epsilon_1})B_1-\epsilon_3B_3>1$, and $B_3\lam-(C_{\epsilon_1}B_1+C_{\epsilon_2}B_2) > B_3\lam/2$.
Now \eqref{MMC1} can be rewritten as 
\ben\label{MMC2} &&\f{d}{dt}\big(B_1\mathbf{E}^{\mathbf{II}}_{0,4}+B_2\mathcal{F}_{4,\f{\a-\c}{16}}+B_3\mathbf{E}^{\mathbf{II}}_{0,0}\big)
+B_1\big(\lambda/2-C\epsilon_1^{-1}\lr{t}^{2q_0}\mathbf{E}^{\mathbf{I}}_{\f12+6\eta,5}(f)\big)\mathbf{D}^{\mathbf{II}}_{0,4}(f)+\big(1-(B_1\epsilon_1^{-1}+B_3\epsilon_3^{-1})\nonumber\\&&\times\mathbf{E}^{\mathbf{I}}_{\f12+6\eta,5}(f)\lr{t}^{2q_0}\big) \|\mP f\|^2_{\mathcal{H}^{4,\f{\a-\c}{32}}_{x}L^2_v}+\big(  B_3(\lambda/2-C\lr{t}^{2q_0}\mathbf{E}^{\mathbf{I}}_{\f12+6\eta,5}(f)) \mathbf{D}^{\mathbf{II}}_{0,0}(f)\ls    (B_1\epsilon_1+B_2\epsilon_2)(\mathbf{D}^{\mathbf{I}}_{\a,5}(f)\nonumber \\ &&+\mathbf{D}^{\mathbf{II}}_{0,5}(f))+C_{\epsilon_1,\epsilon_2,\epsilon_3}\Big(\<t\>^{-2q_0(1-5\eta)}+\mathbf{E}^{\mathbf{I}}_{\f12+\delta_1,5}(f) \Big) \mathbf{E}^{\mathbf{I}}_{\a,5}(f). \een

\subsubsection{Micro-macro combination(II)}  By \eqref{EstPerN3} and \eqref{F5}, by choosing $B_5\ll B_4$, we have
\ben\label{MMC3} &&\f{d}{dt}\big(B_4\mathbf{E}^{\mathbf{II}}_{0,5}+B_5\mathcal{F}_{5,\f{\a-\c}{16}}\big)+\big(
B_4(\lambda-C\epsilon_4^{-1}\lr{t}^{2q_0}\mathbf{E}^{\mathbf{I}}_{\f12+6\eta,5}(f)-\epsilon_4)-B_5C_{\epsilon_5}\big)\mathbf{D}^{\mathbf{II}}_{0,5}(f)
+(B_5\nonumber\\&&-B_4(\epsilon_4^{-1}\lr{t}^{2q_0}\mathbf{E}^{\mathbf{I}}_{\f12+6\eta,5}(f)+\epsilon_4))\|\mP f\|^2_{\mathcal{H}^{5,\f{\a-\c}{32}}_{x}L^2_v} \ls   (\epsilon_4B_4+\epsilon_5B_5)\mathbf{D}^{\mathbf{I}}_{\a,5}(f)+B_4C_{\epsilon_4} \| \psi^{\f12} f\|_{L^2_xL^2_{-1/2} }^2\notag  \\
&&+C_{\epsilon_4,\epsilon_5} \Big(\<t\>^{-2q_0(1-5\eta)}+\mathbf{E}^{\mathbf{I}}_{\f12+\delta_1,5}(f) \Big)\mathbf{E}^{\mathbf{I}}_{\a,5}(f).\een  
Now we choose
 \ben\label{constconstrain2} B_5:=2, B_4:=8(C_{\epsilon_5}+M)/\lam, \epsilon_4<\min\{\f{1}{2B_4},\lambda/4\}.\een Then we have
 $B_5-\epsilon_4B_4>1$ and $\f14B_4\lambda- \epsilon_4 B_4 - B_5C_{\epsilon_5}\ge0$, 
which imply that  
\ben\label{MMC4} &&\f{d}{dt}\big(B_4\mathbf{E}^{\mathbf{II}}_{0,5}+B_5\mathcal{F}_{5,\f{\a-\c}{16}}\big)+
 B_4\big(\lambda/2-C\epsilon_4^{-1}\lr{t}^{2q_0}\mathbf{E}^{\mathbf{I}}_{\f12+6\eta,5}(f)\big)\mathbf{D}^{\mathbf{II}}_{0,5}(f)
\\&&\qquad\qquad+(1-B_4\epsilon_4^{-1}\lr{t}^{2q_0}\mathbf{E}^{\mathbf{I}}_{\f12+6\eta,5}(f))\|\mP f\|^2_{\mathcal{H}^{5,\f{\a-\c}{32}}_{x}L^2_v} \nonumber\\&&\ls   (\epsilon_4B_4+\epsilon_5B_5)\mathbf{D}^{\mathbf{I}}_{\a,5}(f)+ B_4C_{\epsilon_4} \| \psi^{\f12} f\|_{L^2_xL^2_{-1/2} }^2 +C_{\epsilon_4,\epsilon_5} \Big(\<t\>^{-2q_0(1-5\eta)}+\mathbf{E}^{\mathbf{I}}_{\f12+\delta_1,5}(f) \Big)\mathbf{E}^{\mathbf{I}}_{\a,5}(f).\nonumber\een

\subsubsection{Micro-macro combination(III)} 
Observe that \beno
\sum\limits_{|\al|+|\be|=5}\|\psi^{\f12}\pa^\al_\be f\|^2_{L^2_{x,v}}\ls \mathbf{D}^{\mathbf{II}}_{0,5}(f)+\|\mP f\|^2_{\mathcal{H}^{5,\f{\a-\c}{32}}_{x}L^2_v}~~\mbox{and}~~\|\psi^{\f12} f\|^2_{L^2_{x,v}}\ls \mathbf{D}^{\mathbf{II}}_{0,0}(f)+\|\mP f\|^2_{\mathcal{H}^{4,\f{\a-\c}{32}}_{x}L^2_v}.
\eeno
Then from \eqref{EN}, \eqref{MMC2} and \eqref{MMC4}, we are led to
that there exist constants   $1\ll B_7\ll B_6$  such that
\beno &&\f{d}{dt}\Big[\mathbf{E}^{\mathbf{I}}_{\a,5}(f)+B_6\big(B_1\mathbf{E}^{\mathbf{II}}_{0,4}+B_2\mathcal{F}_{4,\f{\a-\c}{16}}+B_3\mathbf{E}^{\mathbf{II}}_{0,0}\big)+B_7\big(B_4\mathbf{E}^{\mathbf{II}}_{0,5}+B_5\mathcal{F}_{5,\f{\a-\c}{16}}\big)\Big]+\Big[\f{\lambda_\a}8- B_6(B_1\epsilon_1+B_2\epsilon_2) \\&&-B_7(B_4\epsilon_4+B_5\epsilon_5)-C_\a \<t\>^{2q_0}\mathbf{E}^{\mathbf{I}}_{\f12+6\eta,5}(f)\Big]\mathbf{D}^{\mathbf{I}}_{\a,5}(f)+\Big[
B_7B_4\big(\lambda/2-C\epsilon_4^{-1}\lr{t}^{2q_0}\mathbf{E}^{\mathbf{I}}_{\f12+6\eta,5}(f)\big)-C_\a-B_6\\
\\&&\times(B_1\epsilon_1+B_2\epsilon_2)\Big]\mathbf{D}^{\mathbf{II}}_{0,5}(f)
+\Big[B_7(1-B_4\epsilon_4^{-1}\lr{t}^{2q_0}\mathbf{E}^{\mathbf{I}}_{\f12+6\eta,5}(f))-C_\a\Big]\|\mP f\|^2_{\mathcal{H}^{5,\f{\a-\c}{32}}_{x}L^2_v} +\Big[B_6B_1(\lambda/2-C\epsilon_1^{-1}\eeno\beno&&\times\lr{t}^{2q_0}\mathbf{E^{\mathbf{I}}}_{\f12+6\eta,5}(f))\Big]\mathbf{D}^{\mathbf{II}}_{0,4}(f)+\Big[B_6(1-(B_1\epsilon_1^{-1}+B_3\epsilon_3^{-1})\lr{t}^{2q_0}\mathbf{E}^{\mathbf{I}}_{\f12+6\eta,5}(f) )-C_\a-B_7B_4C_{\epsilon_4} \Big]\\&&\times\|\mP f\|^2_{\mathcal{H}^{4,\f{\a-\c}{32}}_{x}L^2_v} +\Big[B_6\big(B_3(\lambda/2-C\lr{t}^{2q_0}\mathbf{E}^{\mathbf{I}}_{\f12+6\eta,5}(f)) \big)-C_\a-B_7B_4C_{\epsilon_4} \Big] \mathbf{D}^{\mathbf{II}}_{0,0}(f)
\\&&\ls C_{\epsilon_1,\epsilon_2,\epsilon_3,\epsilon_4,\epsilon_5}\Big(\<t\>^{-2q_0(1-5\eta)} +\mathbf{E}^{\mathbf{I}}_{\f12+\delta_1,5}(f) \Big)\mathbf{E}^{\mathbf{I}}_{\a,5}(f).
\eeno
 
Choose 
 \ben\label{constconstrain3} B_7:=4C_\a, B_6:=4C_\a+4B_7B_4C_{\epsilon_4},\een then by the choice of $B_3$ and $B_4$, we get that
  $B_7B_4\lambda/4\ge C_\a, B_6B_3\lambda/4\ge  C_\a+B_7B_4C_{\epsilon_4}$,
which yields that 
\ben\label{Engfinal} &&\f{d}{dt}\Big[\mathbf{E}^{\mathbf{I}}_{\a,5}(f)+B_6\big(B_1\mathbf{E}^{\mathbf{II}}_{0,4}+B_2\mathcal{F}_{4,\f{\a-\c}{16}}+B_3\mathbf{E}^{\mathbf{II}}_{0,0}\big)+B_7\big(B_4\mathbf{E}^{\mathbf{II}}_{0,5}+B_5\mathcal{F}_{5,\f{\a-\c}{16}}\big)\Big]+\Big[\f{\lambda_\a}8- B_6(B_1
\epsilon_1 +B_2\epsilon_2)\nonumber \\&&-B_7(B_4\epsilon_4+B_5\epsilon_5)-C_\a \<t\>^{2q_0}\mathbf{E}^{\mathbf{I}}_{\f12+6\eta,5}(f)\Big]\mathbf{D}^{\mathbf{I}}_{\a,5}(f)+\Big[
B_7B_4\big(\lambda/4-C\epsilon_4^{-1}\lr{t}^{2q_0}\mathbf{E}^{\mathbf{I}}_{\f12+6\eta,5}(f)\big) -B_6(B_1\epsilon_1\nonumber\\&&+B_2\epsilon_2)\Big]\mathbf{D}^{\mathbf{II}}_{0,5}(f)
+\Big[B_7(1/2-B_4\epsilon_4^{-1}\lr{t}^{2q_0}\mathbf{E}^{\mathbf{I}}_{\f12+6\eta,5}(f))\Big]\|\mP f\|^2_{\mathcal{H}^{5,\f{\a-\c}{32}}_{x}L^2_v} +\Big[B_6B_1(\lambda/2-C\epsilon_1^{-1}\lr{t}^{2q_0}\nonumber\\ &&\times\mathbf{E}^{\mathbf{I}}_{\f12+6\eta,5}(f))\Big]\mathbf{D}^{\mathbf{II}}_{0,4}(f) +\Big[B_6(\f12-(B_1\epsilon_1^{-1}+B_3\epsilon_3^{-1})\lr{t}^{2q_0}\mathbf{E}^{\mathbf{I}}_{\f12+6\eta,5}(f) ) \Big]\|\mP f\|^2_{\mathcal{H}^{4,\f{\a-\c}{32}}_{x}L^2_v}+\Big[B_6\big(B_3(\lambda/4\nonumber
\\&&-C\lr{t}^{2q_0} \mathbf{E}^{\mathbf{I}}_{\f12+6\eta,5}(f)) \big) \Big]\mathbf{D}^{\mathbf{II}}_{0,0}(f)\ls C_{\epsilon_1,\epsilon_2,\epsilon_3,\epsilon_4,\epsilon_5}\Big(\<t\>^{-2q_0(1-5\eta)} +\mathbf{E}^{\mathbf{I}}_{\f12+\delta_1,5}(f) \Big)\mathbf{E}^{\mathbf{I}}_{\a,5}(f).
\een 

Now we can determine the choice of $\epsilon_1,\epsilon_2,\epsilon_3,\epsilon_4$ and $\epsilon_5$. By \eqref{constconstrain2} and  \eqref{constconstrain3}, we first observe that $B_5$ and $B_7$ are universal constants. From this, we set that $\epsilon_5:= (64B_5B_7)^{-1}\lambda_\a$. It implies that $C_{\epsilon_5}$ is a universal constant and so is $B_4$ thanks to \eqref{constconstrain2}. By setting $\epsilon_4:=\min\{(64B_4B_7)^{-1}\lambda_\a , 1/(4B_4), \lambda/8\}$, we get that $C_{\epsilon_4}$ is a universal constant and so is $B_6$ by \eqref{constconstrain3}. Now we can set $\epsilon_1:= \min\{(64B_1B_6)^{-1}\lam B_7B_4,(64B_1B_6)^{-1}\lambda_\a, \lambda/8\}$. This implies that $B_2$ is a universal constant by \eqref{constconstrain1}. Let $\epsilon_2:= \min\{(64B_2B_6)^{-1}\lam B_7B_4,(64B_2B_6)^{-1}\lambda_\a\}$. By  \eqref{constconstrain1}, we finally fix  $B_3\gg1$ and $\epsilon_3=(2B_3)^{-1}$. These imply that 
 
\begin{lem}
Let $0<\de_1,\eta\ll1 $ and   $\a\in[\f12+\de_1,1)$. There exists computable constants  $C_i,c_i(1\le i\le 6)$ and $B_j(1\le j\le 7)$ such that if
\beno \mathbb{E}(\a;f):= \mathbf{E}^{\mathbf{I}}_{\a,5}(f)+B_6\big(B_1\mathbf{E}^{\mathbf{II}}_{0,4}+B_2\mathcal{F}_{4,\f{\a-\c}{16}}+B_3\mathbf{E}^{\mathbf{II}}_{0,0}\big)+B_7\big(B_4\mathbf{E}^{\mathbf{II}}_{0,5}+B_5\mathcal{F}_{5,\f{\a-\c}{16}}\big),\eeno 
then
\ben\label{bfE}
&&\f{d}{dt} \mathbb{E}(\a; f)+\big[c_1-C_1\lr{t}^{2q_0}\mathbf{E}^{\mathbf{I}}_{\f12+6\eta,5}(f)\big]\mathbf{D}^{\mathbf{I}}_{\a,5}(f)+\big[c_2-C_2\lr{t}^{2q_0}\mathbf{E}^{\mathbf{I}}_{\f12+6\eta,5}(f)\big]\mathbf{D}^{\mathbf{II}}_{0,5}(f)\nonumber\\&&+\big[c_3-C_3\lr{t}^{2q_0}\mathbf{E}^{\mathbf{I}}_{\f12+6\eta,5}(f)\big]\|\mP f\|^2_{\mathcal{H}^{5,\f{\a-\c}{32}}_{x}L^2_v}+\big[c_4-C_4\lr{t}^{2q_0}\mathbf{E}^{\mathbf{I}}_{\f12+6\eta,5}(f)\big]\mathbf{D}^{\mathbf{II}}_{0,4}(f)\nonumber\\&&+\big[c_5-C_5\lr{t}^{2q_0}\mathbf{E}^{\mathbf{I}}_{\f12+6\eta,5}(f)\big]\|\mP f\|^2_{\mathcal{H}^{4,\f{\a-\c}{32}}_{x}L^2_v}+\big[c_6-C_6\lr{t}^{2q_0}\mathbf{E}^{\mathbf{I}}_{\f12+6\eta,5}(f)\big]\mathbf{D}^{\mathbf{II}}_{0,0}(f)\nonumber \\&&
\ls  \Big(\<t\>^{-2q_0(1-5\eta)} +\mathbb{E}(\f12+\delta_1;f) \Big)\mathbb{E}(\a;f),
\een  
and we have $0\leq \mathbb{E}(\a;f) \sim \mathbf{E}^{\mathbf{I}}_{\a,5}(f)$.
\end{lem}
\begin{proof} Initially, it should be noted that \eqref{bfE} can be deduced from \eqref{Engfinal} based on our choice of $B_j$ for $1\leq j\leq 7$ and $\epsilon_i$ for $1\leq i\leq 5$. Additionally, we possess the flexibility to set $B_3$ to be arbitrarily large. To finalize the lemma's proof, it is imperative to demonstrate that $0\leq \mathbb{E}(\alpha; f)\sim \mathbf{E}^{\mathbf{I}}_{\alpha,5}(f)$.

Thanks to \eqref{constconstrain2} and Theorem \ref{macroestimate}, we have $B_4\gg B_5$ and $|\mathcal{F}_{5,\f{\a-\c}{16}}|\ls \mathbf{E}^{\mathbf{II}}_{0,5}$, which imply that  
\beno
B_7\big(B_4\mathbf{E}^{\mathbf{II}}_{0,5}+B_5\mathcal{F}_{5,\f{\a-\c}{16}}\big)\geq \f12 B_7B_4\mathbf{E}^{\mathbf{II}}_{0,5}.
\eeno
By the  interpolation inequality that $|\mathcal{F}_{4,\f{\a-\c}{16}}|\ls \|\mathcal{P}_v^{10}f\|^2_{\mathcal{H}_x^{4,\f{\a-\c}8}L^2_v}\leq \epsilon \mathbf{E}^{\mathbf{II}}_{0,5}+C_\epsilon \mathbf{E}^{\mathbf{II}}_{0,0}$, we get that 
\beno
|B_6B_2\mathcal{F}_{4,\f{\a-\c}{16}}|\leq  \f14 B_7B_4\mathbf{E}^{\mathbf{II}}_{0,5}+C_{B_2,B_4,B_6,B_7}\mathbf{E}^{\mathbf{II}}_{0,0}.
\eeno
Now choose   \( B_3  \) to be sufficinelty large, then we derive that
\beno
\mathbb{E}(\a;f)\geq \mathbf{E}^{\mathbf{I}}_{\a,5}(f)+\f14 B_7B_4 \mathbf{E}^{\mathbf{II}}_{0,5}+(B_6B_3-C_{B_2,B_4,B_6,B_7})\mathbf{E}^{\mathbf{II}}_{0,0}\geq \mathbf{E}^{\mathbf{I}}_{\a,5}(f).
\eeno

From the upper bounds of $\mathbf{E}^{\mathbf{II}}_{0,N}$ and $\mathcal{F}_{N,\f{\a-\c}{16}}$, one may easily check that \( \mathbb{E}(\a;f)\ls \mathbf{E}^{\mathbf{I}}_{\a,5}(f) \). This ends the proof of the lemma.
\end{proof}

\subsection{Proof of Theorem \ref{ThmGstofNSlandau}}\label{(1+t)-1} Now, we are ready to give the proof of Theorem \ref{ThmGstofNSlandau}. 
\begin{proof}[Proof of Theorem \ref{ThmGstofNSlandau}] The whole proof will be separated into several steps.
We   postpone the proof of the existence, uniqueness and the non-negativity of the solution to \eqref{NSlandauCauchy} and focus on the {\it a priori} estimates for the global stability of \eqref{pertubNSlandaucauchy}. 
\smallskip

\noindent\underline{Step 1: Proof of global stability.} To prove the desired result, we shall use the continuity argument. Let $\delta\in (0,\f14-\f{\c}2)$ and $\eta_0\in(0,\f{1-4\delta}{4\c}-\f12)$.  We set $q_0:=\f{1-4\de}{4\c}-\eta_0$ and  
\begin{equation}\label{epsiloneta}
 	\delta_1,\vth \in(0,\f{\eta_0}{200}), \quad \eta\in \Big(0, \min\big\{\f{\de_1}6,\f15-\f1{10q_0}\big\}\Big).
 \end{equation}
Here $\vth,\eta$ are defined in \eqref{deW}. We assume that 
\ben\label{DefT*} T^*:=\sup\Big\{T>0\Big|\sup_{t\in[0,T]}(\mathbb{E}(1-2\delta;f)+\lr{t}^{2q_0}\mathbb{E}(1/2+\de_1;f))\le \varepsilon^{\f12}\Big\}.\een

Set \beno \mathbb{E}_{\mathsf{T}}(f):=\mathbb{E}(1-2\delta;f)+\lr{t}^{2q_0}\mathbb{E}(\f12+\de_1;f). \eeno
Then applying \eqref{bfE} with $\a=1-2\de$ and $\a=\f12+\de_1$, by the definition of $T^*$, we drive that for any $t\in[0,T^*]$,  
we get that
\beno
&&\f d{dt}\mathbb{E}_{\mathsf{T}}(f)+\f{c_1}2\big[\mathbf{D}^{\mathbf{I}}_{1-2\de,5}(f)+\lr{t}^{2q_0}\mathbf{D}^{\mathbf{I}}_{\f12+\de_1,5}(f)\big]\ls  \<t\>^{-2q_0(1-5\eta)}(\mathbb{E}_{\mathsf{T}}(f)+\mathbb{E}^2_{\mathsf{T}}(f))+\lr{t}^{2q_0-1}\mathbb{E}(\f12+\de_1;f).
\eeno
 
Due to the the definition   \eqref{Dalbe} and \eqref{EDNI}, we have  
\begin{equation*}
	\mathbf{D}^{\mathbf{I}}_{\f12+\de_1,5}(f)\geq  \sum_{|\al|+|\be|\leq 5}\|\mathcal{Z}_{x,v}^{\f12(\f12+\de_1-\c-\vth-5\eta)}\pa^\al_\be f\|^2_{L^2_{x,v}},\quad \mathbb{E}(1-2\de;f)\geq  \sum_{|\al|+|\be|\leq 5}\|\mathcal{Z}_{x,v}^{\f12(1-2\de-5\eta)}\pa^\al_\be f\|^2_{L^2_{x,v}},
\end{equation*}
from which together with Lemma \ref{interpolationineq} imply that
\beno
\mathbb{E}(\f12+\de_1;f)\leq (\mathbf{D}^{\mathbf{I}}_{\f12+\de_1,5}(f))^{\th}(\mathbb{E}(1-2\de;f))^{1-\th},\quad\mbox{with}\quad\th=\f{\f12-2\de-\de_1-5\eta}{\f12-2\de-\de_1+\c+\vth}.
\eeno
By Young inequality, we have  
\beno
 \lr{t}^{2q_0-1}\mathbb{E}(\f12+\de_1;f) \leq\f{c_1}4\<t\>^{2q_0}\mathbf{D}^{\mathbf{I}}_{\f12+\de_1,5}(f)+C\<t\>^{2q_0-\f1{1-\th}}\mathbb{E}(1-2\de;f).
\eeno
This yields that
\beno
&&\f d{dt}\mathbb{E}_{\mathsf{T}}(f) \ls  (\<t\>^{-2q_0(1-5\eta)}+\<t\>^{2q_0-\f1{1-\th}})(\mathbb{E}_{\mathsf{T}}(f)+\mathbb{E}^2_{\mathsf{T}}(f)).
\eeno 
Thanks to the choice of \eqref{epsiloneta}, we have
\begin{equation}\label{q0}
	-2q_0(1-5\eta)<-1,\quad 2q_0-\f1{1-\th}=2q_0-\f{\f12-2\de-\de_1+\c+\vth}{\c+\vth+5\eta}<-1.
\end{equation}
The Gr\"{o}nwall inequality implies that  for any $t\in[0,T^*]$, $\mathbb{E}_{\mathsf{T}}(f(t))\ls \mathbb{E}_{\mathsf{T}}(f_0)\sim \varepsilon$. This in particular means that $T^*=+\infty$ and moreover,
\ben \sup_{t>0}\big(\mathbb{E}(1-2\delta;f(t))+\lr{t}^{2q_0}\mathbb{E}(\f12+\de_1;f(t))\big)\ls \mathbb{E}_{\mathsf{T}}(f_0)\ls \mathbb{E}(1-2\delta;f_0).\label{anstzenergy}  \een
This ends the proof of global stability.
\smallskip

\noindent\underline{Step 2: Construction of the solution.}
 Let us  recall that the function space $\mathcal{H}^N_{x,v}$ defined as follows:
\beno
\mathcal{H}^N_{x,v}:=\Big\{f\in H^N_{x,v} |\|f\|^2_{\mathcal{H}^N_{x,v}}:=\sum_{|\al|+|\be|\leq N}\|\mathcal{Z}_{x,v}^{\f12( 1-2\de -\eta(|\al|+|\be|))}\pa^\al_\be f\|^2_{L^2_{x,v}}<+\infty\Big\}.
\eeno
We first have 

\begin{prop}\label{lGPWApplinear} Given $0\le F=F(t,x,v)\in L^\infty([0,T];\mathcal{H}^5_{x,v}),0\leq G_0(x,v)\in \mathcal{H}^5_{x,v}$ and  $\sfT:=\mathscr{C}_l(t)(\mfS v \cdot \nabla_x - \mfS x \cdot \nabla_v + \mfR x \cdot \nabla_x - \mfR v \cdot \nabla_v)$. The linear equation 
\ben\label{AppLinearNSlandau}  \begin{aligned} \partial_t G  + \sfT G=\mathscr{C}Q(F,G),\quad G|_{t=0}=G_0,\end{aligned} \een
admits a unique and non-negative solution $G=G(t,x,v)\in L^\infty([0,T];\mathcal{H}^5_{x,v})$. Moreover, it holds that
\ben\label{energyestimate} &&\|G\|_{L^\infty([0,T];\mathcal{H}^5_{x,v})} \ls C(T;\|F\|_{L^\infty([0,T];\mathcal{H}^5_{x,v})})\|G_0\|_{\mathcal{H}^5_{x,v}}. \een 
\end{prop}\begin{proof} We first remark that the linear equation can be solved by the following approximation equation:
\ben\label{EpAppAppLinearNSlandau}  \begin{aligned} &\partial_t G^\epsilon   + \sfT G^\epsilon+\epsilon \mathscr{P} G^\epsilon=\mathscr{C}Q(F,G^\epsilon),\quad G^\eps|_{t=0}=G_0,\end{aligned}  \een where $\mathscr{P}:=a_1(-\triangle_v)^2+a_2(-\triangle_v)\mathcal{P}_v^2+a_3\mathcal{P}_v^{4}$ with $0<a_1\ll a_2\ll a_3$ satisfies that
\ben\label{propertyP} (\mathscr{P}G, G\mathcal{Z}_{x,v}^\a)_{L^2_{x,v}}\sim \|(-\triangle_v G) \mathcal{Z}_{x,v}^{\a/2}\|_{L^2_{x,v}}^2+\|(\na_vG) \mathcal{Z}_{x,v}^{\a/2}\mathcal{P}_v^1\|_{L^2_{x,v}}^2+\|G \mathcal{Z}_{x,v}^{\a/2}\mathcal{P}_v^2\|_{L^2_{x,v}}^2, \een  for $\a\in[0,1]$. 
One may solve it by using the following Picard iteration scheme:
\ben\label{AppAppLinearNSlandau} \left\{ \begin{aligned} &\partial_t G^n     + \sfT G^n+\epsilon \mathscr{P} G^n=\mathscr{C}Q(F,G^{n-1}),\\&G^n|_{t=0}=G_0,\quad G^{0}:=G_0.\end{aligned} \right. \een
We omit the details here. To get the  uniform bound for $G^\epsilon$, 
 thanks to \eqref{propertyP}, Lemma \ref{Fff}, Lemma \ref{ghf}, Lemma \ref{expQ} and Lemma \ref{IBP}, we easily derive that 
\beno &&\f{d}{dt}\|G^\epsilon\|_{\mathcal{H}^5_{x,v}}^2+\sum_{|\al|+|\be|\leq 5}\iint (F*a):(\na_v (\mathcal{Z}_{x,v}^{\f12( 1-2\de -\eta(|\al|+|\be|)}\pa^\al_\be G^\eps))\otimes \na_v (\mathcal{Z}_{x,v}^{\f12( 1-2\de -\eta(|\al|+|\be|)}\pa^\al_\be G^\eps))) dxdv\\&&+C_{a_1,a_2,a_3}\epsilon(\|(-\triangle)_vG^\epsilon\|_{\mathcal{H}^5_{x,v}}^2+\|(\na_vG^\epsilon)\mathcal{P}_v^1\|_{\mathcal{H}^5_{x,v}}^2+\| G^\epsilon \mathcal{P}_v^2\|_{\mathcal{H}^5_{x,v}}^2)  \ls (1+\|F\|_{\mathcal{H}^5_{x,v}}) \|G^\epsilon\|_{\mathcal{H}^5_{x,v}}^2.\eeno
From this together with Gr\"{o}nwall inequality, we conclude the following energy estimates: 
\beno \|G^\epsilon\|_{L^\infty([0,T];\mathcal{H}^5_{x,v})} \ls C(T;\|F\|_{L^\infty([0,T];\mathcal{H}^5_{x,v})})\|G_0\|_{\mathcal{H}^5_{x,v}}.\eeno 

To prove the non-negativity, we observe that if $G^\epsilon_\pm=\pm\max\{\pm G^\epsilon,0\}$, then   the basic energy method implies that
\beno \f12\f{d}{dt}\|G^\epsilon_-\|_{L^2_{x,v}}^2=(-\sfT G^\epsilon-\epsilon\mathscr{P}G^\epsilon+\mathscr{C}Q(F,G^\epsilon),G^\epsilon_-)_{L^2_{x,v}}\leq(\mathscr{C}Q(F,G^\epsilon_-),G^\epsilon_-)_{L^2_{x,v}}\ls \|F\|_{\mathcal{H}^5_{x,v}} \|G^\epsilon_-\|_{L^2_{x,v}}^2.\eeno
Note that $G^\epsilon_-|_{t=0}=0$, therefore by Gr\"{o}nwall inequality, $G^\epsilon(t)$ is non-negative for any $t\in[0,T]$. 

Finally, we can conclude our results by the vanishing limit $\epsilon\rightarrow0$ and the uniqueness is a natural byproduct of the energy estimates.
\end{proof}
 
 To solve the equation \eqref{NSlandauCauchy}, we will consider the following approximation equation: 
\ben\label{AppNSlandau} 
\begin{aligned}
\partial_t G^n + \mathsf{T}G^n=\mathscr{C}Q(G^{n-1},G^n),\quad
G^n|_{t=0}=\mathcal{M}+f_0(x,v)\ge0,\quad G^0:=\mathcal{M}.
\end{aligned}
\een 
Applying Proposition \ref{lGPWApplinear} to \eqref{AppNSlandau}, we may derive that there exists a common lifespan $T=T(G_0)<\infty$ such that for any $n\in\N$,  $G^n\ge0$ and
\beno   &&\|G^n\|_{L^\infty([0,T];\mathcal{H}^5_{x,v})} \ls C(T;\|G_0\|_{\mathcal{H}^5_{x,v}}). \eeno 
Let $g^n=G^{n+1}-G^{n}$. Then we have 
\beno  \begin{aligned} \partial_t g^n  +\mathsf{T}g^n=\mathscr{C}Q(G^{n},g^n)+\mathscr{C}Q(g^{n-1},G^n),\quad g^n|_{t=0}=0.\end{aligned} \eeno
Again by Lemma \ref{Fff}, Lemma \ref{ghf}, Lemma \ref{expQ} and Lemma \ref{IBP}, we get that 
\ben\label{CSinH3}
\f{d}{dt}\|g^n\|_{\mathcal{H}^3_{x,v}}^2\ls \|G^n\|_{\mathcal{H}^5_{x,v}}(\|g^n\|_{\mathcal{H}^3_{x,v}}^2+\|g^{n-1}\|_{\mathcal{H}^3_{x,v}}\|g^n\|_{\mathcal{H}^3_{x,v}}),
 \een 
 which is enough to conclude that $\{g^n\}_{n\ge1}$ is a Cauchy sequence that converges to zero in $L^\infty([0,T];\mathcal{H}^3_{x,v})$. This implies the existence of non-negative solution in $L^\infty([0,T];\mathcal{H}^5_{x,v})$ for \eqref{NSlandauCauchy}. The uniqueness can be proved similarly thanks to \eqref{CSinH3}. This completes the proof of the local well-posedness of \eqref{NSlandauCauchy}. 

Next, let $\om_1=(1,0,0),\om_2=(0,1,0)$ and $\om_3=(0,0,1)$ and define difference operators as follows:
\beno
D^{\tilde{\al}}_{h,x}f(x)=\f{f(x+h\tilde{\al})-f(x)} h,\quad D^{\tilde{\be}}_{h,v}f(v)=\f{f(v+h\tilde{\be})-f(v)} h,~~\tilde{\al},\tilde{\be}=\om_i,i=1,2,3.
\eeno
Denote $D^{\tilde{\al}}_{\tilde{\be},h}:=D^{\tilde{\al}}_{x,h}D^{\tilde{\be}}_{v,h}$ with $|\tilde{\al}|+|\tilde{\be}|=1$. It is not difficult to check that
\ben\label{comT}
&& \big|\sum_{\substack{|\al|+|\be|=4\\|\tilde{\al}|+|\tilde{\be}|=1}}([\pa^\al_\be D^{\tilde{\al}}_{\tilde{\be},h},\mathsf{T}]G,\mathcal{Z}_{x,v}^{( 1-2\de -5\eta)}D^{\tilde{\al}}_{\tilde{\be},h}\pa^\al_\be G)_{L^2_{x,v}}\big|\ls \|G\|^2_{\mathcal{H}^5_{x,v}},\\
\mbox{and}&& D^{\tilde{\al}}_{\tilde{\be},h}Q(G,G)=Q(G,D^{\tilde{\al}}_{\tilde{\be},h}G)+Q(D^{\tilde{\al}}_{\tilde{\be},h}G,G(\cdot+\tilde{\al}h,\cdot+\tilde{\be}h)).\label{chainrule}
\een
Thanks to Lemma \ref{Fff}, Lemma \ref{ghf}, Lemma \ref{expQ}, Lemma \ref{IBP}, and \eqref{chainrule}, we easily derive that
\ben\label{Gepsilon} \notag&&\f{d}{dt}(\|G\|_{\mathcal{H}^4_{x,v}}^2+\sum_{\substack{|\al|+|\be|=4\\|\tilde{\al}|+|\tilde{\be}|=1}}\|\mathcal{Z}_{x,v}^{\f12( 1-2\de -5\eta)}D^{\tilde{\al}}_{\tilde{\be},h}\pa^\al_\be G\|_{L^2_{x,v}}^2)+\sum_{\substack{|\al|+|\be|=4\\|\tilde{\al}|+|\tilde{\be}|=1}}([\pa^\al_\be D^{\tilde{\al}}_{\tilde{\be},h},\mathsf{T}]G,\mathcal{Z}_{x,v}^{( 1-2\de -5\eta)}D^{\tilde{\al}}_{\tilde{\be},h}\pa^\al_\be G)_{L^2_{x,v}}\\
\notag&&+\sum_{|\al|+|\be|\leq 4}\iint (G*a):(\na_v (\mathcal{Z}_{x,v}^{\f12( 1-2\de -\eta(|\al|+|\be|))}\pa^\al_\be G))\otimes \na_v (\mathcal{Z}_{x,v}^{\f12( 1-2\de -\eta(|\al|+|\be|))}\pa^\al_\be G))) dxdv\\
&&+\sum_{\substack{|\al|+|\be|=4\\|\tilde{\al}|+|\tilde{\be}|=1}}\iint (G*a):(\na_v (\mathcal{Z}_{x,v}^{\f12( 1-2\de -5\eta)}D^{\tilde{\al}}_{\tilde{\be},h}\pa^\al_\be G))\otimes \na_v (\mathcal{Z}_{x,v}^{\f12( 1-2\de -5\eta)} D^{\tilde{\al}}_{\tilde{\be},h}\pa^\al_\be G))) dxdv \\
\notag&\ls&(1+\|G\|_{\mathcal{H}^5_{x,v}})\|G\|^2_{\mathcal{H}^5_{x,v}}.\een
By Fatou Lemma, \eqref{comT} and the fact that $ G\in L^\infty([0,T];\mathcal{H}^5_{x,v})$, taking $h\rightarrow 0$ implies that 
\beno  && \sum_{|\al|+|\be|\leq 5}\int_0^T\iint (G*a):(\na_v (\mathcal{Z}_{x,v}^{\f12( 1-2\de -\eta(|\al|+|\be|))}\pa^\al_\be G))\otimes (\na_v (\mathcal{Z}_{x,v}^{\f12( 1-2\de -\eta(|\al|+|\be|))}\pa^\al_\be G)) dt\\&&\ls C(T,\|G_0\|_{\mathcal{H}^5_{x,v}}). 
\eeno
This in particular implies that all the {\it a priori} estimates employed in Step 1 are valid. We end the proof.
\end{proof}

\subsection{Proof of Theorem \ref{Thmnonstability}} Now we can give a proof to Theorem \ref{Thmnonstability}  thanks to Theorem \ref{ThmGstofNSlandau}.
\begin{proof}[Proof of Theorem \ref{Thmnonstability}] We begin with the proof of asymptotic stability. Recalling \eqref{RelationFG}, we have  
\beno
G(t,x,v)=m^{-1}\left(\det \mathsf{S}\right)^{-1} F\left(t,\mathbf{x},\mathbf{v}\right),\quad \mathcal{M}(x,v)=m^{-1}\left(\det \mathsf{S}\right)^{-1}M(t,\mathbf{x},\mathbf{v}),
\eeno
with $\mathbf{x}=\sqrt{C(t)}\mathsf{S}^{-1}\mathsf{U}x+y\cos t+z\sin t$ and $\mathbf{v}=\frac{\mathsf{U}v-B(t)\mathsf{S}^{-1}\mathsf{U}x+\mathsf{R}\mathsf{S}^{-1}\mathsf{U}x}{\sqrt{C(t)}}-y\sin t+z\cos t$. One may easily check that the Jacobian between $(x,v)$ and $(\mathbf{x},\mathbf{v})$ is invertible and bounded uniformly in $t$. Indeed, we have 
\beno 
m\left(\det \mathsf{S}\right)(\f\pa{\pa x},\f\pa{\pa v})=(\f\pa{\pa \mathbf{x}},\f\pa{\pa \mathbf{v}})
\begin{pmatrix}
	\sqrt{C(t)}\mathsf{S}^{-1}\mathsf{U} & 0 \\
	\f{-B(t)\mathsf{S}^{-1}\mathsf{U}+\mathsf{R}\mathsf{S}^{-1}\mathsf{U}}{\sqrt{C(t)}} & \f{\mathsf{U}}{\sqrt{C(t)}}
\end{pmatrix}.
\eeno 
Thus it holds that
\beno
\sum_{|\tilde{\al}|+|\tilde{\be}|\leq 5}\pa^{\tilde{\al}}_x\pa^{\tilde{\be}}_v=\sum_{|\al|+|\be|\leq 5}\mathsf{C}_{\al,\be}^{\tilde{\al},\tilde{\be}}\pa^\al_{\mathbf{x}}\pa^\be_{\mathbf{v}},\quad\sum_{|\al|+|\be|\leq 5}\pa^\al_{\mathbf{x}}\pa^\be_{\mathbf{v}} =\sum_{|\tilde{\al}|+|\tilde{\be}|\leq 5}C^{\al,\be}_{\tilde{\al},\tilde{\be}}\pa^{\tilde{\al}}_x\pa^{\tilde{\be}}_v,
\eeno
where the constants $\mathsf{C}_{\al,\be}^{\tilde{\al},\tilde{\be}}$ and $C^{\al,\be}_{\tilde{\al},\tilde{\be}}$ are all bounded. It leads to 
\beno
\sum_{|\tilde{\al}|+|\tilde{\be}|\leq 5}\|\mathcal{M}^{\a}\pa^{\tilde{\al}}_x\pa^{\tilde{\be}}_v f(t)\|_{L^2_{x,v}}= \sum_{|\tilde{\al}|+|\tilde{\be}|\leq 5}\|\mathcal{M}^{\a}\pa^{\tilde{\al}}_x\pa^{\tilde{\be}}_v( G(t)-\mathcal{M})\|_{L^2_{x,v}}\sim	\sum_{|\al|+|\be|\leq 5}\|M^{\a}(t)\pa^\al_{\mathbf{x}}\pa^\be_{\mathbf{v}}( F(t)-M(t))\|_{L^2_{\mathbf{x},\mathbf{v}}}
\eeno
for any $\a\in \R$ and $t\in[0,\infty)$. Hence, the asymptotic stability can be directly deduced from Theorem \ref{ThmGstofNSlandau}.

\medskip

Next, we prove the Lyapunov stability. Let $M\in \mathscr{M}$ have the form \eqref{explicitformM0} with parameters $(m,y,z,a,b,c,\mathsf{R})$. Suppose that $F_0$ satisfies $F_0\geq0, \Psi_1(F_0)>0$ and 
\beno
\sum_{|\al|+|\be|\leq 5}\|M^{-1+2\de}(0)\pa^\al_\be( F_0-M(0))\|_{L^2_{x,v}}\leq \vep,
\eeno
for some $\de\in (\f16-\f\c 3)$ and $\vep\ll1$. By Theorem \ref{biinjectionofMF0}, there exists an unique $M_1\in \mathscr{M}$ such that $\Psi(M_1)=\Psi(F_0)$. Assume that $M_1$ has the form \eqref{explicitformM0} with parameters $(m_1,y_1,z_1,a_1,b_1,c_1,\mathsf{R}_1)$. Then
\begin{equation}\label{MM1}
	\begin{aligned}
\|\Psi(M_1)-\Psi(M)\|_2&=\|\Psi(F_0)-\Psi(M(0))\|_2\ls \|(1+|x|^2+|v|^2)(F_0-M(0))\|_{L^1_{x,v}}\\
&\ls\sum_{|\al|+|\be|\leq 5}\|M^{-1+2\de}(0)\pa^\al_\be( F_0-M(0))\|_{L^2_{x,v}}\leq \vep.
	\end{aligned}
\end{equation}
Thus Theorem \ref{biinjectionofMF0} gives that
\begin{multline}\label{abcR1}
 \|(m-m_1,y-y_1,z-z_1,a-a_1, b-b_1,c-c_1,\mathsf{R}-\mathsf{R}_1)\|_2\ls \vep\\
 \mbox{and}\quad \f M {M_1}+\f{M_1}M\ls e^{C\vep(1+|x|^2+|v|^2)}\ls \min\{M^{-C\vep},M_1^{-C\vep}\}.
\end{multline}
By virtue of \eqref{abcR1} and Mean Value Theorem, one may get that
\ben\label{MM4}
|M_1^{-1}\pa^\al_\be(M-M_1)|&\ls& C(\Psi(M))\|(m-m_1,y-y_1,z-z_1,a-a_1, b-b_1,c-c_1,\mathsf{R}-\mathsf{R}_1)\|_2M_1^{-C\vep}\notag\\
&\ls& \vep M_1^{-C\vep}.
\een

Since $C\vep\ll1$, it holds that $\de+\f{C\vep}2\in (0,\f14-\f\c 2)$. Thanks to  \eqref{abcR1} and \eqref{MM4}, we can obtain that 
\begin{equation}\label{MM2}
	\begin{aligned}
&\|M_1^{-1+2\de+C\vep}(0)\pa^\al_\be( F_0-M_1(0))\|_{L^2_{x,v}}\\
\leq&~~~\|(M_1/M)^{-1+2\de+C\vep}(0)M^{-1+2\de+C\vep}(0)\pa^\al_\be( F_0-M(0))\|_{L^2_{x,v}}+\|M_1^{-1+2\de+C\vep}(0)\pa^\al_\be(M(0)-M_1(0))\|_{L^2_{x,v}}\\
\ls&~~~  \|M^{-1+2\de}(0)\pa^\al_\be( F_0-M(0))\|_{L^2_{x,v}}+\vep \|M_1^{\de}(0)\|_{L^2_{x,v}}\ls \vep. 
\end{aligned}
\end{equation}
Thus the asymptotic stability ensures the global existence of solution $F(t)$ with initial data $F_0$ satisfying
	\ben\label{MM3}
	 \sum_{|\al|+|\be|\leq 5} \|M_1^{-1+3\de}(t)\pa^\al_\be (F(t)-M_1(t))\|_{L^2_{x,v}} \lesssim \vep.
\een
Thanks to \eqref{MM3} and the same trick used in \eqref{MM2}, we can derive that
	\beno \sum_{|\al|+|\be|\leq 5} \|M^{-1+4\de}(t)\pa^\al_\be (F(t)-M(t))\|_{L^2_{x,v}} \lesssim \vep.
\eeno
This ends the proof of Theorem \ref{Thmnonstability}.
 \end{proof}

\section{Appendix}\label{Section Appendix}
In this appendix, we will provide  detail proofs for the claims in the previous sections. We begin with a detailed proof for \eqref{entropyinvariant}.

\begin{prop}
Let \mbox{$F$ be an entropy-invariant solution of \eqref{kineticwithexternalpotential}} with $\mathcal{C}_{coll}(F,F)=Q(F,F)$. Then  
\beno  \mathscr{D}(F)(t)=0, \forall t\ge0 \Leftrightarrow \mathcal{C}_{coll}(F,F)(t)=0, \forall t\ge0. \eeno 
\end{prop}
\begin{proof}  We first assume that $\mathcal{C}_{coll}(F,F)(t)=0, \forall t\ge0$. This implies that $F$ takes the form \eqref{localMaxwellian}. From this together with \eqref{entropydisspationoflandau}, we conclude that  $\mathscr{D}(F)(t)=0, \forall t\ge0$. 
For the inverse direction, suppose that $ \mathscr{D}(F)(t)=0, \forall t\ge0$.
Let  $\varphi$ be a  test function. It is easy to verify that
 \beno
&& \int_{\R^6} Q(F,F) (x,v)\, \varphi(x,v) \, dvdx= - \frac12 \, \sum_{i=1}^3\sum_{j=1}^3 \iint_{\R^9} a_{ij}(v-v_{*}) \\
&&\quad\times
 \left\{ \frac{\partial_j F}{F}(v) - \frac{\partial_j F}{F}(v_{*})  \right\} \left\{ \partial_i \varphi(v) - \partial_i \varphi(v_{*})  \right\}   
 F(v) F(v_{*}) \, dv \, dv_{*}dx.
\eeno
Since for all $\xi,\eta\in\R^3$,  it holds that $|(A\xi,\eta)_{\R^3}|^2\le (A\xi,\xi)_{\R^3}(A\eta,\eta)_{\R^3}$ if $A$ is semi-positive definite. Then it implies that
\beno 
&&\bigg|\int_{\R^6} Q(F,F) (x,v)\, \varphi(x,v) \, dvdx\bigg|\le \mathscr{D}(F)^{\f12}\bigg(\f12\int_{\R^9}  a_{ij}(v-v_{*}) \\
&&\quad\times
 \left\{ \partial_j \varphi(v) - \partial_j \varphi(v_{*})  \right\}    \left\{ \partial_i \varphi(v) - \partial_i \varphi(v_{*})  \right\}   
 F(v) F(v_{*}) \, dv \, dv_{*}dx\bigg)^{\f12}.
\eeno 
We conclude the desired result and end the proof.
\end{proof}

Next we give a detailed proof of \eqref{conservationofNSlandau} in Remark \ref{rmk140}.

\begin{prop}Let $G(t,x,v)$ be a global solution to \eqref{NSlandauCauchy} with the initial data $G_0=G_0(x,v)$. Then 
\ben\label{conservationofNSlandau1}
\f{d}{dt}\int_{\R^3_x\times\R^3_v} G(t, x, v)\begin{bmatrix} 1 \\ v   \\ \mfS^{-1}x \\ |v + \mfR\mfS^{-1}x|^2 \\ \mfS^{-1}x \cdot v \\ |\mfS^{-1}x|^2 \\   (v + \mfR\mfS^{-1}x) \wedge \mfS^{-1}x   \end{bmatrix} dxdv=0.
\een
\end{prop}
\begin{proof} To prove the desired result, we follow the notations used in (\ref{FtoF1}-\ref{F4toF5}).
Thanks to  \eqref{FtoF1}, \eqref{F1toF2}, we first  get that
\ben\label{F2con1}
\int_{\R^3_x\times\R^3_v}\begin{bmatrix} 1 \\ x\cos t - v\sin t \\ x\sin t + v\cos t \end{bmatrix}F_2(t, x, v)dxdv = \begin{bmatrix} 1 \\ 0 \\ 0 \end{bmatrix}=\int_{\R^3_x\times\R^3_v}\begin{bmatrix} 1 \\ x\cos t - v\sin t \\ x\sin t + v\cos t \end{bmatrix}M_2(t, x, v)dxdv;
\een  
Thanks to Proposition \ref{properM0100}, we further derive that
\ben\label{F2con2}
	&& \notag\quad\int_{\R^3_x\times\R^3_v}\begin{bmatrix} |x\cos t - v\sin t|^2 \\ (x\cos t - v\sin t) \cdot (x\sin t + v\cos t) \\ |x\sin t + v\cos t|^2 \\ (x\cos t - v\sin t) \wedge (x\sin t + v\cos t) \end{bmatrix}F_2(t, x, v)dxdv\\
	 &=& \int_{\R^3_x\times\R^3_v}(|x|^2 , x \cdot v , |v|^2 ,x \wedge v)^\tau M_2(0, x, v)dxdv
	=(c \mathrm{tr}\mathsf{Q}^{-1}, -b \mathrm{tr}\mathsf{Q}^{-1} , a \mathrm{tr}\mathsf{Q}^{-1} , -2\mathsf{Q}^{-1}\mathsf{R})^\tau.
\een

In what follows, we will frequently use the following equality:
\ben\label{FJocobi} \int_{\R_x^3\times\R_v^3} (F\circ  \Phi_1)(x,v)(P\circ \Phi_1)(x,v)dxdv=\int_{\R_x^3\times\R_v^3} F(X,V)P(X,V) dXdV, \een 
if $\Phi_1$ is a transformation from $(x,v)\in\R^3\times\R^3$ to $(X,V)\in\R^3\times\R^3$ with a unit Jacobian.
By the relation \eqref{F2toF3}, it is easy to verify that  
\begin{equation*}
	\Phi_1: (x, v) \to (\sqrt{C(t)}x, \frac{v - B(t)x + \mathsf{R}x}{\sqrt{C(t)}}),
\end{equation*}
has a unit Jacobian. Since
\begin{equation*}
	\begin{aligned}
		x\cos t - v\sin t &\to \sqrt{C(t)}x\cos t - \frac{v - B(t)x + \mathsf{R}x}{\sqrt{C(t)}}\sin t = \frac{(c\cos t - b\sin t)x - (v + \mathsf{R}x)\sin t}{\sqrt{C(t)}};\\
		x\sin t + v\cos t &\to \sqrt{C(t)}x\sin t + \frac{v - B(t)x + \mathsf{R}x}{\sqrt{C(t)}}\cos t = \frac{(a\sin t - b\cos t)x + (v + \mathsf{R}x)\cos t}{\sqrt{C(t)}},\\
	\end{aligned}
\end{equation*}
 \eqref{F2con1}, \eqref{F2con2} and \eqref{FJocobi} will imply that
\ben  &&\int_{\R_x^3\times\R_v^3} F_3(t, x, v)dxdv=1;\nonumber \\
		  &&\int_{\R_x^3\times\R_v^3}[(c\cos t - b\sin t)x - (v + \mathsf{R}x)\sin t]F_3(t, x, v)dxdv=0;\label{conservF1}\\
	  &&\int_{\R_x^3\times\R_v^3}[(a\sin t - b\cos t)x + (v + \mathsf{R}x)\cos t]F_3(t, x, v)dxdv=0;\label{conservF2}\\
	 && \int_{\R_x^3\times\R_v^3}|(c\cos t - b\sin t)x - (v + \mathsf{R}x)\sin t|^2F_3(t, x, v)dxdv= cC(t)\mathrm{tr}\mathsf{Q}^{-1}; \label{conservF3}\\
	&&\int_{\R_x^3\times\R_v^3}[(c\cos t - b\sin t)x - (v + \mathsf{R}x)\sin t] \cdot [(a\sin t - b\cos t)x + (v + \mathsf{R}x)\cos t]F_3(t, x, v)dxdv\nonumber \\&&= -bC(t)\mathrm{tr}\mathsf{Q}^{-1}; \label{conservF4}\\
		&&  \int_{\R_x^3\times\R_v^3}|(a\sin t - b\cos t)x + (v + \mathsf{R}x)\cos t|^2F_3(t, x, v)dxdv=aC(t)\mathrm{tr}\mathsf{Q}^{-1};\label{conservF5}\\
		&&\int_{\R_x^3\times\R_v^3}[(c\cos t - b\sin t)x - (v + \mathsf{R}x)\sin t] \wedge [(a\sin t - b\cos t)x + (v + \mathsf{R}x)\cos t]F_3(t, x, v)dxdv\nonumber\\&&= -2C(t)\mathsf{Q}^{-1}\mathsf{R}.\label{conservF6} 
\een

  We first observe that   $\cos t \times \eqref{conservF1} + \sin t \times \eqref{conservF2}$ implies  that
\beno
\int_{\R_x^3\times\R_v^3} xF_3(t, x, v)dxdv = \int_{\R_x^3\times\R_v^3} vF_3(t, x, v)dxdv = 0.
\eeno
Secondly, the combination of $\cos t \times \eqref{conservF3} + \sin t \times \eqref{conservF4}$ and $\sin t \times \eqref{conservF5} + \cos t\times \eqref{conservF4}$ yield that 
\ben
 &&\int_{\R_x^3\times\R_v^3}[(c\cos t - b\sin t)|x|^2 - \sin tx \cdot v]F_3(t, x, v)dxdv=(c\cos t - b\sin t)\mathrm{tr}\mathsf{Q}^{-1};\label{conservF7}\\
 &&\int_{\R_x^3\times\R_v^3}[(a\sin t - b\cos t)|x|^2 + \cos tx \cdot v]F_3(t, x, v)dxdv=(a\sin t - b\cos t)\mathrm{tr}\mathsf{Q}^{-1}. \label{conservF8}
\een
Thirdly,   $\cos t \times \eqref{conservF7} + \sin t \times \eqref{conservF8} $ will lead to that
\beno
\int_{\R_x^3\times\R_v^3}|x|^2F_3(t, x, v)dxdv=\mathrm{tr}\mathsf{Q}^{-1}~~\mbox{and}~~ \int_{\R_x^3\times\R_v^3} x \cdot vF_3(t, x, v)dxdv = 0.
\eeno
Putting together $ \eqref{conservF3} + \eqref{conservF5} $ and   \eqref{conservF6}, we have 
\beno
\int_{\R_x^3\times\R_v^3}|v + \mathsf{R}x|^2F_3(t, x, v)dxdv = (ac - b^2)\mathrm{tr}\mathsf{Q}^{-1}~~\mbox{and}~~  \int_{\R_x^3\times\R_v^3} x \wedge (v +  \mathsf{R}x)F_3(t, x, v)dxdv = -2\mathsf{Q}^{-1} \mathsf{R}.
\eeno
 
In conclusion, we  obtain that
\beno
	&&\int_{\R_x^3\times\R_v^3}(1,x,v,|x|^2,x\cdot v,|v+\mathsf{R} x|^2,x\wedge(v+\mathsf{R}x))^\tau F_3(t, x, v)dxdv \\
	&&=( 1 ,0 , 0 ,\mathrm{tr}\mathsf{Q}^{-1} ,0, (ac - b^2)\mathrm{tr}\mathsf{Q}^{-1} ,-2\mathsf{Q}^{-1}\mathsf{R} )^{\tau}.
\eeno
 
Thanks to \eqref{F3toF4}, this gives that 
\beno
	&&\int_{\R^3_x\times\R^3_v}( 1 , \mathsf{S}^{-1}x, v , |\mathsf{S}^{-1}x|^2 , \mathsf{S}^{-1}x \cdot v ,|v + \mathsf{R}\mathsf{S}^{-1}x|^2 , \mathsf{S}^{-1}x \wedge (v + \mathsf{R}\mathsf{S}^{-1}x) )^\tau F_4(t, x, v)dxdv \\
	&&= ( 1, 0 , 0 , \mathrm{tr}\mathsf{Q}^{-1} , 0 , (ac - b^2)\mathrm{tr}\mathsf{Q}^{-1} ,-2\mathsf{Q}^{-1}\mathsf{R})^\tau.
\eeno
From this together with \eqref{F4toF5}, we conclude that
\begin{multline}\label{conF5}
	\int_{\R^3_x\times\R^3_v}( 1 ,\mfS^{-1}x, v , |\mfS^{-1}x|^2 , \mfS^{-1}x \cdot v ,|v + \mfR\mfS^{-1}x|^2 , \mfS^{-1}x \wedge (v + \mfR\mfS^{-1}x) )^\tau F_5(t, x, v)dxdv \\
	= ( 1, 0 , 0 , (ac-b^2)\mathrm{tr}\mathsf{Q}^{-1} , 0 , (ac - b^2)\mathrm{tr}\mathsf{Q}^{-1} ,-2(ac-b^2)^{\f12}\mathsf{Q}^{-1}\mathsf{R})^\tau,
\end{multline}
where $\mfS$ and $\mfR$ are defined in \eqref{deofS} and \eqref{deofR}. This completes the proof of \eqref{conservationofNSlandau1}.
\end{proof}

\smallskip

Next, we give some interpolation inequalities. One may easily get the proof by combining Cauchy-Schwarz inequality, Young inequality and Lemma 4.15 in \cite{CHJ}.
\begin{lem}\label{interpolationineq}
	Let $a_i,b_i\in\R,i=1,2,3,\th\in(0,1)$ verifying $a_1=a_2\th+a_3(1-\th)$ and $b_1=b_2\th+b_3(1-\th)$, then for any function $f(v)$ or $f(x,v)$, we have 
	\beno
	&&\|\<\cdot\>^{a_1}e^{b_1\<\cdot\>^2}f\|_{L^2}\leq \|\<\cdot\>^{a_2}e^{b_2\<\cdot\>^2}f\|^\th_{L^2}\|\<\cdot\>^{a_3}e^{b_3\<\cdot\>^2}f\|^{1-\th}_{L^2}\leq \epsilon \|\<\cdot\>^{a_2}e^{b_2\<\cdot\>^2}f\|_{L^2}+C_{\epsilon,\th}\|\<\cdot\>^{a_3}e^{b_3\<\cdot\>^2}f\|_{L^2},\\
	&&\|f\|_{H^{a_1}_{b_1}}=\|\<D\>^{a_1}\<\cdot\>^{b_1}f\|_{L^2}\sim\|\<\cdot\>^{b_1}\<D\>^{a_1}f\|_{L^2}~~\mbox{and}~~\|f\|_{H^{a_1}_{b_1}}\leq \|f\|^\th_{H^{a_2}_{b_2}}\|f\|^{1-\th}_{H^{a_3}_{b_3}}\ls \epsilon \|f\|_{H^{a_2}_{b_2}}+C_{\epsilon,\th}\|f\|_{H^{a_3}_{b_3}}.
	\eeno
\end{lem}

\medskip

At the end of this paper, we provide a brief derivation of  macroscopic equation (\ref{macroeq}-\ref{defofT2}).
\begin{prop}\label{abcequa}
Let $f$ be a solution to \eqref{pertubNSlandaucauchy}.  
Then (\ref{macroeq}-\ref{defofT2}) hold ture.
\end{prop}
\begin{proof}
 Substituting $f=\mathbb{P}f+(\mathbb{I-P})f$ into \eqref{pertubNSlandaucauchy}, we can derive that
 \begin{equation}\label{abceq}
	\begin{aligned}
		&(\pa_t \ma +\pa_t\mb\cdot v +\pa_t\mc(|v|^2-3))\mu+\mathscr{C}_l(t)(\mfS v+\mfR x)\cdot(\na_x \ma+\na_x(\mb\cdot v)+\na_x\mc(|v|^2-3))\mu\\
		&-\mathscr{C}_l(t)(\mfS x+\mfR v)\cdot(\ma\na_v\mu+\na_v(b\cdot v\mu)+\mc\na_v((|v|^2-3)\mu))\\
		=&-(\pa_t+\mathscr{C}_l(t)(\mfS v\cdot\na_x-\mfS x\cdot\na_v+\mfR x\cdot\na_x-\mfR v\cdot\na_v))(\mathbb{I-P})f+\mathscr{C}_1e^{-\f12|x|^2}L((\mathbb{I-P})f)+g,
	\end{aligned}
\end{equation}
where $g:=\mathscr{C}_2Q(f,f)$. It is noteworthy that 
\begin{equation}\label{basicpro1}
	\begin{aligned}
		&(\pa_t(\mathbb{(I-P)}f),\phi(v))_{L^2_v}=0,\quad ((\mfS v)\cdot\na_x(\mathbb{(I-P)}f),1)_{L^2_v}=0,\quad ((\mfR x)\cdot\na_x(\mathbb{(I-P)}f),\phi(v))_{L^2_v}=0,\\
		&((\mfS x)\cdot\na_v(\mathbb{(I-P)}f),\phi(v))_{L^2_v}=0,\quad ((\mfR v)\cdot\na_v(\mathbb{(I-P)}f),\phi(v))_{L^2_v}=0,~~ \mbox{for}~~\phi(v)=1,v,|v|^2,
	\end{aligned}
\end{equation}
thanks to the facts that 
\begin{equation*}
	\begin{aligned}
		(\mathbb{(I-P)}f,\phi(v))_{L^2_v}=0,\quad (L(\mathbb{(I-P)}f),\phi(v))_{L^2_v}=0,\quad(\mfR v,v)_{\R^3}=0,~~ \mbox{for}~~\phi(v)=1,v,|v|^2.
	\end{aligned}
\end{equation*}

Let us  prove the first and the fourth equation in \eqref{macroeq} as typical cases. 
Taking inner product of \eqref{abceq} with $1$ in $L^2(\R^3_v)$, and using \eqref{msA}, we can get that
\beno
&&\pa_t \ma+\mathscr{C}_l(t)(\mfS\na_x\cdot\mb)+\mathscr{C}_l(t)\mfR x\cdot\na_x\ma-(g,1)_{L^2_v}\\
&=&\big(-(\pa_t+\mathscr{C}_l(t)(\mfS v\cdot\na_x-\mfS x\cdot\na_v+\mfR x\cdot\na_x-\mfR v\cdot\na_v))(\mathbb{I-P})f+\mathscr{C}_1e^{-\f12|x|^2}L((\mathbb{I-P})f),1\big)_{L^2_v}.
\eeno
Noting that \eqref{basicpro1} implies that the right-hand side term equals $0$, we can multiply both sides by \( e^{\frac{1}{2}|x|^2} \) to obtain the first equation in \eqref{macroeq} with $\sfT_{11}$ defined in \eqref{defofT1}.

Next, taking inner product of \eqref{abceq} with $\f12v_iv_j,i,j=1,2,3,i\neq j$ in $L^2(\R^3_v)$, and using \eqref{msA}, we can get that
\beno
&&\f12\mathscr{C}_l(t)((\mfS x)_i\mb_j+(\mfS x)_j\mb_i)+\pa_t\big((\mathbb{I-P})f,\f12v_iv_j\big)_{L^2_v}\\
&=&-\big(\mathscr{C}_l(t)(\mfS v\cdot\na_x-\mfS x\cdot\na_v+\mfR x\cdot\na_x-\mfR v\cdot\na_v))(\mathbb{I-P})f+\mathscr{C}_1e^{-\f12|x|^2}L((\mathbb{I-P})f)+g,\f12v_iv_j\big)_{L^2_v}
\eeno
Analogously, taking inner product of \eqref{abceq} with $\f12(v_i^2-1),i=1,2,3$ in $L^2(\R^3_v)$, we can obtain that
\beno
&&\pa_t\mc+\mathscr{C}_l(t)((\mfS x)_i\mb_i+\mfR x\cdot\na_x \mc)+\pa_t\big((\mathbb{I-P})f,\f12(v_i^2-1)\big)_{L^2_v}\\
&=&-\big(\mathscr{C}_l(t)(\mfS v\cdot\na_x-\mfS x\cdot\na_v+\mfR x\cdot\na_x-\mfR v\cdot\na_v))(\mathbb{I-P})f+\mathscr{C}_1e^{-\f12|x|^2}L((\mathbb{I-P})f)+g,\f12(v_i^2-1)\big)_{L^2_v}.
\eeno
Observing that $\mfR x\cdot\na_x(e^{\f12|x|^2})=0$ and $\int_{\R^3_x}(\mfS\na_x)_i\mb_j dx=0$ for any $i,j=1,2,3$, we can multiply both sides of two equations in the above by $e^{\f12|x|^2}$, yielding the fourth equation in \eqref{macroeq} with $\sfT_{41}$ and $\sfT_{42}$ defined in \eqref{defofT2}.

To reiterate the procedure described above, compute the inner product of \eqref{abceq} with \( v \),  then with \( \frac{|v|^2 - 3}{6} \) and finally with \( \frac{v(|v|^2 - 5)}{10} \). By subsequently applying \eqref{msA} and \eqref{basicpro1},  we can obtain the remaining equations outlined in \eqref{macroeq}.
\end{proof}


\begin{thebibliography}{99}

\bibitem{AMUXY}  R. Alexandre,  Y.Morimoto, S. Ukai, C.-J. Xu and T. Yang, Global existence and full regularity of the Boltzmann equation without angular cutoff, Comm. Math. Phys. 304 (2011), no. 2, 513-581.

\bibitem{Bakry Rate 2008}
D. Bakry, P. Cattiaux and A. Guillin,
Rate of convergence for ergodic continuous Markov processes: Lyapunov versus Poincar\'{e}. 
J. Funct. Anal., 254(2008), no. 3, 727-759.


 \bibitem{BGGL} C. Bardos,  I.M. Gamba, F. Golse, C. D. Levermore, 
Global solutions of the Boltzmann equation over  $\R^n$ near global Maxwellians with small mass,  Comm. Math. Phys. 346 (2016), no. 2, 435-467. 


\bibitem{Boltzmann} L. Boltzmann,  \"Uber die Aufstellung und Integration der Gleichungen, welche die Molekularbewegung in Gasen bestimmen. Kk Hof-und Staatsdruckerei, 1876.

\bibitem{BC} R. Bosi,  and M. J. Caceres, The BGK model with external confining potential: existence, longtime behavior and time-periodic Maxwellian equilibria, J. Stat. Phys. 136, 297-330 (2009).

\bibitem{CHJ}  C. Cao, L.-B. He, J. Ji, Propagation of moments and sharp convergence rate for inhomogeneous noncutoff Boltzmann equation with soft potentials, SIAM J. Math. Anal. 56 (2024), no. 1, 1321-1426.
 
\bibitem{KJFSC}
K. Carrapatoso, J. Dolbeault, F. H\'erau, S. Mischler and C. Mouhot,
{ Weighted Korn and Poincar\'e-Korn inequalities in the Euclidean space and associated operators,}
Archive for Rational Mechanics and Analysis, 243(2022), 1565-1596 .

\bibitem{KJFSCS}  K. Carrapatoso, J. Dolbeault, F. H\'erau, S. Mischler,  C. Mouhot   and C. Schmeiser, Special macroscopic modes and hypocoercivity. J. Eur. Math. Soc. (2024), published online first.


\bibitem{KS}
K. Carrapatoso and S. Mischer.
  Landau equation for very soft and Coulomb potentials near Maxwellians. 
Ann. PDE(2017) 3:1.
 
\bibitem{Ce} C. Cercignani, The Boltzmann equation and its applications. Applied Mathematical Sciences 67, Springer-Verlag, New York (1988).

\bibitem{DV1} L. Desvillettes   and  C. Villani,   On the trend to global equilibrium in spatially inhomogeneous entropy-dissipating systems: the linear Fokker-Planck equation,
Comm. Pure Appl. Math. 54 (2001), no. 1, 1-42.

\bibitem{DV2}  L. Desvillettes and C. Villani,  On the trend to global equilibrium for spatially inhomogeneous kinetic systems: the Boltzmann equation. Invent. Math. 159(2005), 245-316.

\bibitem{PM}
P. Degond and M. Lemou,
{ Dispersion relations for the linearized Fokker-Planck equation},
Archive for Rational Mechanics and Analysis, 138(1997) 137-167.



\bibitem{DMS} J. Dolbeault, C. Mouhot and C. Schmeiser,  Hypocoercivity for linear kinetic equations conserving mass. Trans. Amer. Math. Soc. 367, 3807-3828 (2015).

 

\bibitem{Douc Subgeometric 2009}
R. Douc, G. Fort, and A. Guillin,
Subgeometric rates of convergence of f-ergodic strong Markov processes. 
Stochastic Process. Appl., 119 (2009), no. 3, 897-923.



\bibitem{Grad} H. Grad,  On Boltzmann's H-theorem. J. Soc. Indust. Appl. Math. 13, 259-277 (1965).

 \bibitem{GS}  P. T. Gressman and R. M. Strain,  Global classical solutions of the Boltzmann equation without angular cut-off, J. Amer. Math. Soc. 24 (2011), no. 3, 771-847.

\bibitem{GMM} M. P. Gualdani,  S. Mischler, and C. Mouhot, Factorization of non-symmetric operators and exponential H-theorem,
M\'em. Soc. Math. Fr. (N.S.)(2017), no. 153, 137 pp.

 
\bibitem{Guo1} Y. Guo,   The Landau equation in a periodic box, Comm. Math. Phys. 231, 391-434 (2002).

\bibitem{Guo2} Y. Guo,  The Boltzmann equation in the whole space, Indiana Univ. Math. J. 53 (2004), no. 4, 1081-1094.


 \bibitem{HN1}  B. Helffer   and F. Nier,   Hypoelliptic estimates and spectral theory for Fokker-Planck operators and Witten Laplacians. Lecture Notes in Mathematics 1862, Springer-Verlag, Berlin (2005).


\bibitem{HL} L.-B. He, Sharp bounds for Boltzmann and Landau collision operators, Ann. Sci. Ec. Norm. Super. 
 51 (2018), no.4, 1285-1373.


\bibitem{HN2} F. H\'erau   and F. Nier, Isotropic hypoellipticity and trend to equilibrium for the Fokker-Planck equation with a high-degree potential,
Arch. Ration. Mech. Anal. 171 (2004), no. 2, 151-218.




\bibitem{Kim The 2018}

C. Kim and D. Lee. The Boltzmann equation with specular boundary condition in convex domains. Communications on Pure and Applied Mathematics 71, no. 3 (2018): 411-504.


 


 \bibitem{Leoni Giovanni (2017)}
G. Leoni. A First Course in Sobolev Spaces. Graduate Studies in Mathematics, Volume 105. American Mathematical Society. 2009.





\bibitem{Levermore1}
C. D. Levermore,  
 Global Maxwellians over all space and their relation to conserved quantities of classical kinetic equations,  preprint.





\bibitem{LYY} T.-P. Liu, T. Yang  and  S.-H. Yu,
Energy method for Boltzmann equation. 
Phys. D 188(2004), no.3-4, 178-192.


  \bibitem{Luk} L. Jonathan, Stability of vacuum for the Landau equation with moderately soft potentials. 
Ann. PDE 5 (2019), no. 1, Paper No. 11, 101 pp.



\bibitem{Mischler Exponential 2016}
S. Mischler, C. Mouhot,
Exponential stability of slowing decaying solutions to the Kinetic-Fokker-Planck equation.
Arch. Ration. Mech. Anal., 221(2): 677-723, 2016. 

 


\bibitem{Tabata Decay 1993}

M. Tabata,  Decay of solutions to the Cauchy problem for the linearized Boltzmann equation with some external-force potential. Japan J. Indust. Appl. Math. 10, 2 (1993), 237-253.  

\bibitem{Tabata Decay 1994}

M. Tabata,   Decay of solutions to the Cauchy problem for the linearized Boltzmann equation with an unbounded external-force potential. Transport Theory Statist. Phys. 23, 6 (1994), 741-780. 



\bibitem{uhlenbeck-ford}
G. E. Uhlenbeck and  G. W. Ford, Lectures in statistical mechanics,
Lectures in Applied Mathematics (Proceedings of the Summer Seminar, Boulder, Colorado, 1960), Vol. I
American Mathematical Society, Providence, RI, 1963, x+181 pp.


 
 
\bibitem{V2} C. Villani, Hypocoercivity. Mem. Amer. Math. Soc. 202, iv+141 (2009).


\bibitem{Wu Large 2001}
L. Wu. Large and moderate deviations and exponential convergence for stochastic damping Hamiltonian systems. Stochastic Process. Appl., 91(2): 205-238, 2001.




\end{thebibliography}
 \end{document}